\documentclass[a4paper, 11pt]{amsart}

\usepackage{amsmath,amssymb,enumitem,verbatim,stmaryrd,xcolor,microtype,graphicx,mathscinet}
\usepackage[T1]{fontenc}
\usepackage[utf8]{inputenc}
\usepackage[english]{babel}
\usepackage[top=3.48cm,bottom=3.48cm,left=3.2cm,right=3.2cm]{geometry}
\usepackage[bookmarksdepth=2,linktoc=page,colorlinks,linkcolor={red!80!black},citecolor={red!80!black},urlcolor={blue!80!black}]{hyperref}
\usepackage{tikz}\usetikzlibrary{matrix,arrows,decorations.markings}
\usepackage{tikz-cd}
\pgfrealjobname{matroidfamilies}
\usepackage[all]{xy}
\usepackage[mathcal]{euscript}                               
\usepackage{mathptmx}                                        
\usepackage[backgroundcolor=yellow,linecolor=yellow,textsize=footnotesize]{todonotes} \makeatletter \providecommand \@dotsep{5} \def\listtodoname{List of Todos} \def\listoftodos{\@starttoc{tdo}\listtodoname} 
 \makeatother 
\usepackage{listings}

\usepackage{etoolbox}
\makeatletter
\patchcmd{\@startsection}{\@afterindenttrue}{\@afterindentfalse}{}{}             
\patchcmd{\part}{\bfseries}{\bfseries\LARGE}{}{}
\patchcmd{\section}{\scshape}{\bfseries}{}{}\renewcommand{\@secnumfont}{\bfseries} 
\patchcmd{\@settitle}{\uppercasenonmath\@title}{\large}{}{}
\patchcmd{\@setauthors}{\MakeUppercase}{}{}{}
\addto{\captionsenglish}{} 
\addto{\captionsenglish}{} 
\addto{\captionsenglish}{} 
\makeatother

\usepackage{fancyhdr}

\theoremstyle{plain}
\newtheorem{thm}{Theorem}[section]
\newtheorem{cor}[thm]{Corollary}
\newtheorem{lemma}[thm]{Lemma}
\newtheorem{prop}[thm]{Proposition}
\newtheorem{thmA}{Theorem}  
\newtheorem*{claim}{Claim}

\theoremstyle{definition}
\newtheorem{df}[thm]{Definition}
\newtheorem{rem}[thm]{Remark}
\newtheorem{ex}[thm]{Example}

\theoremstyle{remark}

\DeclareRobustCommand{\gobblefour}[5]{}    

\DeclareFontFamily{OT1}{pzc}{}                                
\DeclareFontShape{OT1}{pzc}{m}{it}{<-> s * [1.10] pzcmi7t}{}
\DeclareMathAlphabet{\mathpzc}{OT1}{pzc}{m}{it}
\DeclareSymbolFont{sfoperators}{OT1}{bch}{m}{n} \DeclareSymbolFontAlphabet{\mathsf}{sfoperators} \makeatletter\def\operator@font{\mathgroup\symsfoperators}\makeatother 
\DeclareSymbolFont{cmletters}{OML}{cmm}{m}{it}              
\DeclareSymbolFont{cmsymbols}{OMS}{cmsy}{m}{n}
\DeclareSymbolFont{cmlargesymbols}{OMX}{cmex}{m}{n}
\DeclareMathSymbol{\myjmath}{\mathord}{cmletters}{"7C}     \let\jmath\myjmath 
\DeclareMathSymbol{\myamalg}{\mathbin}{cmsymbols}{"71}     \let\amalg\myamalg
\DeclareMathSymbol{\mycoprod}{\mathop}{cmlargesymbols}{"60}\let\coprod\mycoprod
\DeclareMathSymbol{\myalpha}{\mathord}{cmletters}{"0B}     \let\alpha\myalpha 
\DeclareMathSymbol{\mybeta}{\mathord}{cmletters}{"0C}      \let\beta\mybeta
\DeclareMathSymbol{\mygamma}{\mathord}{cmletters}{"0D}     \let\gamma\mygamma
\DeclareMathSymbol{\mydelta}{\mathord}{cmletters}{"0E}     \let\delta\mydelta
\DeclareMathSymbol{\myepsilon}{\mathord}{cmletters}{"0F}   \let\epsilon\myepsilon
\DeclareMathSymbol{\myzeta}{\mathord}{cmletters}{"10}      \let\zeta\myzeta
\DeclareMathSymbol{\myeta}{\mathord}{cmletters}{"11}       \let\eta\myeta
\DeclareMathSymbol{\mytheta}{\mathord}{cmletters}{"12}     \let\theta\mytheta
\DeclareMathSymbol{\myiota}{\mathord}{cmletters}{"13}      \let\iota\myiota
\DeclareMathSymbol{\mykappa}{\mathord}{cmletters}{"14}     \let\kappa\mykappa
\DeclareMathSymbol{\mylambda}{\mathord}{cmletters}{"15}    \let\lambda\mylambda
\DeclareMathSymbol{\mymu}{\mathord}{cmletters}{"16}        \let\mu\mymu
\DeclareMathSymbol{\mynu}{\mathord}{cmletters}{"17}        \let\nu\mynu
\DeclareMathSymbol{\myxi}{\mathord}{cmletters}{"18}        \let\xi\myxi
\DeclareMathSymbol{\mypi}{\mathord}{cmletters}{"19}        \let\pi\mypi
\DeclareMathSymbol{\myrho}{\mathord}{cmletters}{"1A}       \let\rho\myrho
\DeclareMathSymbol{\mysigma}{\mathord}{cmletters}{"1B}     \let\sigma\mysigma
\DeclareMathSymbol{\mytau}{\mathord}{cmletters}{"1C}       \let\tau\mytau
\DeclareMathSymbol{\myupsilon}{\mathord}{cmletters}{"1D}   \let\upsilon\myupsilon
\DeclareMathSymbol{\myphi}{\mathord}{cmletters}{"1E}       \let\phi\myphi
\DeclareMathSymbol{\mychi}{\mathord}{cmletters}{"1F}       \let\chi\mychi
\DeclareMathSymbol{\mypsi}{\mathord}{cmletters}{"20}       \let\psi\mypsi
\DeclareMathSymbol{\myomega}{\mathord}{cmletters}{"21}     \let\omega\myomega
\DeclareMathSymbol{\myvarepsilon}{\mathord}{cmletters}{"22}\let\varepsilon\myvarepsilon
\DeclareMathSymbol{\myvartheta}{\mathord}{cmletters}{"23}  \let\vartheta\myvartheta
\DeclareMathSymbol{\myvarpi}{\mathord}{cmletters}{"24}     \let\varpi\myvarpi
\DeclareMathSymbol{\myvarrho}{\mathord}{cmletters}{"25}    \let\varrho\myvarrho
\DeclareMathSymbol{\myvarsigma}{\mathord}{cmletters}{"26}  \let\varsigma\myvarsigma
\DeclareMathSymbol{\myvarphi}{\mathord}{cmletters}{"27}    \let\varphi\myvarphi

\DeclareMathOperator{\Spec}{Spec}
\DeclareMathOperator{\Sch}{Sch}
\DeclareMathOperator{\OBSch}{OBSch}

\DeclareMathOperator{\OBMod}{OBMod}
\DeclareMathOperator{\Hom}{Hom}

\DeclareMathOperator{\Aut}{Aut}
\DeclareMathOperator{\Gr}{Gr}
\DeclareMathOperator{\Mat}{Mat}
\DeclareMathOperator{\Dress}{Dress}
\DeclareMathOperator{\MacPh}{MacPh}
\DeclareMathOperator{\cMat}{\mathpzc{Mat}}
\DeclareMathOperator{\cPl}{\mathpzc{Pl}}
\DeclareMathOperator{\Proj}{Proj}
\DeclareMathOperator{\Pic}{Pic}
\DeclareMathOperator{\PicGen}{PicGen}
\DeclareMathOperator{\colim}{colim}

\DeclareMathOperator{\rk}{rk}
\DeclareMathOperator{\im}{im}

\DeclareMathOperator{\Cr}{Cr}

\DeclareMathOperator{\OBlpr}{{OBlpr}}
\DeclareMathOperator{\HypRings}{{HypRings}}
\DeclareMathOperator{\HypFields}{{HypFields}}
\DeclareMathOperator{\PartFields}{{PartFields}}
\DeclareMathOperator{\FuzzRings}{{FuzzRings}}
\DeclareMathOperator{\Tracts}{{Tracts}}
\DeclareMathOperator{\Sets}{Sets}

\DeclareMathOperator{\SRings}{{SRings}}

\DeclareMathOperator{\OBlprSp}{{OBlprSp}}
\DeclareMathOperator*{\hypersum}{\,\raisebox{-2.2pt}{\larger[2]{$\boxplus$}}\,}

\renewcommand\emph[1]{\textbf{#1}}         
\newcommand\A{{\mathbb A}}
\newcommand\B{{\mathbb B}}
\newcommand\C{{\mathbb C}}
\newcommand\F{{\mathbb F}}
\newcommand\G{{\mathbb G}}
\newcommand\K{{\mathbb K}}
\newcommand\N{{\mathbb N}}
\renewcommand\P{{\mathbb P}}

\newcommand\R{{\mathbb R}}
\renewcommand\S{{\mathbb S}}
\newcommand\T{{\mathbb T}}
\newcommand\U{{\mathbb U}}
\newcommand\Z{{\mathbb Z}}

\newcommand\cB{{\mathcal B}}

\newcommand\cF{{\mathcal F}}
\newcommand\cG{{\mathcal G}}

\newcommand\cI{{\mathcal I}}
\newcommand\cJ{{\mathcal J}}

\newcommand\cL{{\mathcal L}}
\newcommand\cM{{\mathcal M}}

\newcommand\cO{{\mathcal O}}
\newcommand\cP{{\mathcal P}}

\newcommand\cR{{\mathcal R}}

\newcommand\cX{{\mathcal X}}

\newcommand\fm{{\mathfrak m}}
\newcommand\fp{{\mathfrak p}}
\newcommand\fq{{\mathfrak q}}
\newcommand\un{{\underline{n}}}

\newcommand\Fun{{\F_1}}
\newcommand\Funsq{{\F_{1^2}}}
\newcommand\Funpm{{\F_1^\pm}}
\newcommand\id{\textup{id}}

\newcommand\ppos{{\textup{ppos}}}

\newcommand\res{\textup{res}}

\newcommand\op{\textup{op}}
\newcommand\oblpr{\textup{oblpr}}

\newcommand\blpr{\textup{blpr}}
\newcommand\fuzz{\textup{fuzz}}
\newcommand\tract{\textup{tract}}
\newcommand\found{\textup{found}}

\newcommand\univ{\textup{univ}}

\renewcommand\top{\textup{top}}
\newcommand\smallvee{{\mbox{\larger[-8]$\vee$}}}
\newcommand\sign{\textup{sign}}

\renewcommand\={\equiv}
\renewcommand\geq{\geqslant}
\renewcommand\leq{\leqslant}
\newcommand{\gen}[1]{\langle #1 \rangle}
\newcommand{\biggen}[1]{\big\langle #1 \big\rangle}
\newcommand{\bpquot}[2]{#1\!\sslash\!#2}
\newcommand{\bpgenquot}[2]{#1\!\sslash\!\gen{#2}}
\newcommand{\bpbiggenquot}[2]{#1\big/\!\!\!\!\big/\biggen{#2}}
\newcommand{\hyperplus}{\,\raisebox{-1.1pt}{\larger[-0]{$\boxplus$}}\,}
\newdir{ >}{{}*!/-5pt/@^{(}}\newcommand{\arincl}[1]{\ar@{ >->}@<-0,0ex>#1} 
\newcommand{\dashedrightarrow}{\;\tikz\draw[densely dashed,->] (0,0) -- (0.48,0);\;}

\newcommand\NN{{\mathbb N}}
\newcommand\ZZ{{\mathbb Z}}
\def\isomap{{\buildrel \sim\over\longrightarrow}} 

\title{The moduli space of matroids}

\author{Matthew Baker}
\address{\rm Matthew Baker, School of Mathematics, Georgia Institute of Technology, Atlanta, USA}
\email{\href{mailto:mbaker@math.gatech.edu}{mbaker@math.gatech.edu}}

\author{Oliver Lorscheid}
\address{\rm Oliver Lorscheid, University of Groningen, the Netherlands, and IMPA, Rio de Janeiro, Brazil}
\email{\href{mailto:oliver@impa.br}{oliver@impa.br}}

\begin{document}


\begin{abstract} 
In \cite{Baker-Bowler19}, Nathan Bowler and the first author introduced a category of algebraic objects called {\em tracts} and defined the notion of (weak and strong) {\em matroids over a tract}. In the first part of the paper, we summarize and clarify the connections to other algebraic objects which have previously been used in connection with matroid theory.  For example, we show that both partial fields and hyperfields are fuzzy rings, that fuzzy rings are tracts, and that these relations are compatible with previously introduced matroid theories. We also show that fuzzy rings are ordered blueprints in the sense of the second author. Thus fuzzy rings lie in the intersection of tracts with ordered blueprints; we call the objects of this intersection {\em idylls}.

We then turn our attention to constructing moduli spaces for (strong) matroids over idylls.  We show that, for any non-empty finite set $E$, the functor taking an idyll $F$ to the set of isomorphism classes of rank-$r$ strong $F$-matroids on $E$ is representable by an ordered blue scheme $\Mat(r,E)$.  We call  $\Mat(r,E)$ the {\em moduli space of rank-$r$ matroids on $E$}.  The construction of $\Mat(r,E)$ requires some foundational work in the theory of ordered blue schemes; in particular, we provide an analogue for ordered blue schemes of the ``Proj'' construction in algebraic geometry, and we show that line bundles and their global sections control maps to projective spaces, much as in the usual theory of schemes. 

Idylls themselves are field objects in a larger category which we call {\em $\Funpm$-algebras}; roughly speaking, idylls are to $\Funpm$-algebras as hyperfields are to hyperrings.  We define matroid bundles over ordered blue $\Funpm$-schemes and show that $\Mat(r,E)$ represents the functor taking an ordered blue $\Funpm$-scheme $X$ to the set of isomorphism classes of rank-$r$ (strong) matroid bundles on $E$ over $X$.  This characterizes $\Mat(r,E)$ up to (unique) isomorphism.  

Finally, we investigate various connections between the space $\Mat(r,E)$ and known constructions and results in matroid theory.  For example, a classical rank-$r$ matroid $M$ on $E$ corresponds to a morphism $\Spec ({\mathbb K}) \to \Mat(r,E)$, where ${\mathbb K}$ (the ``Krasner hyperfield'') is the final object in the category of idylls.  The image of this morphism is a point of $\Mat(r,E)$ to which we can canonically attach a residue idyll $k_M$, which we call the {\em universal idyll} of $M$.  We show that morphisms from the universal idyll of $M$ to an idyll $F$ are canonically in bijection with strong $F$-matroid structures on $M$.  Although there is no corresponding moduli space in the weak setting, we also define an analogous idyll
 $k_M^w$ which classifies weak $F$-matroid structures on $M$.  We show that the unit group of $k_M^w$ can be canonically identified with the {\em Tutte group} of $M$, originally introduced by Dress and Wenzel.
We also show that the sub-idyll $k_M^f$ of $k_M^w$ generated by ``cross-ratios'', which we call  the {\em foundation} of $M$, parametrizes rescaling classes of weak $F$-matroid structures on $M$, and its unit group is coincides with the {\em inner Tutte group} of $M$.  As sample applications of these considerations, we show that a matroid $M$ is regular if and only if its foundation is the regular partial field (the initial object in the category of idylls), and a non-regular matroid $M$ is binary if and only if its foundation is the field with two elements.  From this, we deduce for example a new proof of the fact that a matroid is regular if and only if it is both binary and orientable.
\end{abstract}

\maketitle

\vfill
\noindent\line(1,0){60}\\[0pt]
\begin{footnotesize}
 \noindent\textbf{Acknowledgements.} We thank Nathan Bowler and Rudi Pendavingh for useful discussions, and the anonymous referee for his or her detailed and helpful report. The first author was supported by NSF Grant DMS-1529573 and a Simons Foundation Fellowship.
\end{footnotesize}


\pagebreak

\begin{small} \tableofcontents \end{small}


\section{Introduction}
\label{intro}

One of the most ubiquitous, and useful, moduli spaces in mathematics is the Grassmannian variety $\Gr(r,n)$ of $r$-dimensional subspaces of a fixed $n$-dimensional vector space.
In Dress's paper \cite{Dress86} (and much later, using a different formalism, in \cite{Baker-Bowler19}), one finds that there is a precise sense in which rank-$r$ matroids on an $n$-element set $E$ are analogous to points of the Grassmannian $\Gr(r,n)$.  More precisely, in the language of \cite{Baker-Bowler19}, both can be considered as {\em matroids over hyperfields}, or more generally {\em matroids over tracts}.\footnote{Tracts are more general than both hyperfields and fuzzy rings in the sense of \cite{Dress86}; see section \ref{section: the requisites} below for an in-depth discussion of the relationship between these and other algebraic structures.}  So it seems natural to wonder if there is a ``moduli space of matroids''.  More precisely, one can ask if there is some ``geometric'' object $\Mat(r,E)$ whose ``points'' over any tract $F$ are precisely the $F$-matroids of rank $r$ on $E$ in the sense of \cite{Baker-Bowler19}.  With some small technical caveats (such as the fact that we deal with a slightly restricted class of tracts and work with {\em strong} $F$-matroids as opposed to weak ones), we answer this question affirmatively in the present paper.  We also explore in detail how various properties of the moduli space $\Mat(r,E)$ are related to more ``classical'' considerations in matroid theory.

\medskip

What kind of object should $\Mat(r,E)$ be?  In modern algebraic geometry, one thinks of the Grassmannian $\Gr(r,n)$ as representing a certain moduli functor from schemes to sets.\footnote{More precisely, $\Gr(r,n)$ represents the functor taking a scheme $X$ to the set of isomorphism classes of surjections from ${\mathcal O}_X^{\oplus n}$ onto a locally free ${\mathcal O}_X$-module of rank $n-r$.}
This is the point of view we wish to take here, but clearly schemes would not suffice for our purposes since there is no way to encode the algebra of tracts in the language of commutative rings.  It turns out that the second author's theory of {\em ordered blueprints} and {\em ordered blue schemes} \cite{Lorscheid15} is well-suited to the task at hand.  Indeed, as we show, a certain nice subcategory of tracts -- which we call {\em idylls}\footnote{In an earlier version of this text, what we now call `idylls' were called `pastures'. We now reserve the term `pasture' for a slightly different notion, see Definition~\ref{def:pasture}.} -- contains the category of hyperfields (as well as the more general category of fuzzy rings) and embeds as a full subcategory of ordered blueprints.  We can then use the theory developed in \cite{Lorscheid15}, together with a few new results and constructions, to define a suitable moduli functor and prove that it is representable by an ordered blue scheme.\footnote{We note that, in broad outline, Eric Katz had already envisioned using the theory of blueprints to represent moduli spaces of matroids in section 9.7 of \cite{Katz16}.} 

\subsection{Structure of the paper}


This paper is divided into three parts, each having a different flavor: the first part is algebraic, the second geometric, and the third combinatorial. Each part is largely independent from the others except for certain common definitions. In particular, the reader who is mainly interested in the applications to matroid theory should be able to start reading sections \ref{section: Applications to realization spaces} and \ref{section: cross ratios and rescaling classes} immediately after looking up the necessary definitions in sections \ref{Part1Summary} and \ref{Part2Summary}.
We have combined the algebraic, geometric, and combinatorial aspects of our theory into a single paper because we believe that the resulting ``big picture'' might lead to interesting new insights and developments in algebra/algebraic geometry and/or matroid theory.

\medskip

In Part 1, which comprises sections \ref{section: the requisites} and \ref{section: Comparision of matroid theories}, we compare various algebraic structures and different notions of matroids over these structures.  The main goal of section \ref{section: the requisites} is to clarify precisely how hyperrings / hyperfields, partial fields, fuzzy rings, tracts, and idylls relate to ordered blueprints.  
We also describe the important category of {\em $\Funpm$-algebras}, which itself contains the category of idylls; the new feature of $\Funpm$-algebras is that they possess an element $\epsilon$ which plays the role of $-1$.  (The element $\epsilon$ is needed, for example, in order to be able to write down the Pl\"ucker relations.)
In section \ref{section: Comparision of matroid theories}, we define matroids over idylls, and more generally $\Funpm$-algebras, and compare this notion to the existing notions of matroids over tracts, fuzzy rings, etc.

\medskip

In Part 2, which comprises sections \ref{section: Projective geometry for ordered blueprints} and \ref{section: matroid bundles}, we construct moduli spaces of strong matroids over $\Funpm$-algebras.  These moduli spaces are constructed as ordered blue subschemes of a certain projective space, and their construction requires developing some foundational material on the ``Proj'' construction, line bundles, and maps to projective spaces in the context of ordered blue schemes.

More precisely, we define matroid bundles 
over ordered blue $\Funpm$-schemes and show that 
the functor taking an ordered blue $\Funpm$-scheme $X$ to the set of isomorphism classes of rank-$r$ matroid bundles on $E$ over $X$ is representable by a (unique up to unique isomorphism) ordered blue $\Funpm$-scheme $\Mat(r,E)$.  

\medskip

In Part 3, which comprises sections \ref{section: Applications to realization spaces} and \ref{section: cross ratios and rescaling classes}, we relate certain properties of moduli spaces of matroids to known constructions and results in matroid theory.  For example, we use moduli spaces to associate, in a natural way, a {\em universal idyll} $k_M$ to each (classical) matroid $M$.  We show that morphisms from the universal idyll of $M$ to an idyll $F$ are canonically in bijection with strong $F$-matroid structures on $M$.  Although there is no corresponding moduli space in the weak setting, we also define an analogous 
idyll $k_M^w$,  which classifies weak $F$-matroid structures on $M$, and a sub-idyll $k_M^f$ of $k_M^w$ (which we call  the {\em foundation} of $M$) which parametrizes {\em rescaling classes} of weak $F$-matroid structures on $M$.
The unit group of $k_M^w$ (resp. $k_M^f$) can be canonically identified with the {\em Tutte group} (resp. the {\em inner Tutte group}) of $M$; these groups were originally introduced by Dress and Wenzel via explicit presentations by generators and relations.  

As sample applications of such considerations, we characterize regular and binary matroids in terms of their foundations and show that a matroid is regular if and only if it is both binary and representable over some idyll with $\epsilon \neq 1$.  Examples of such idylls include fields of characteristic different from $2$ and the hyperfield of signs ${\mathbb S}$, so in particular we obtain a new proof of the fact that a matroid is regular if and only if it is both binary and orientable.

\medskip

We now provide a more detailed overview of each of the three parts of the paper.

\subsection{Part 1: Idylls, ordered blueprints, and matroids} \label{Part1Summary}

Our first goal, which is modest but necessary, is to tame the zoo of terminology which we are forced to deal with in order to clarify the relationship between ordered blueprints and various algebraic structures which have already appeared in the literature, as well as various notions of matroids over such objects.

\subsubsection{Matroids over tracts}

In \cite{Baker-Bowler19}, Nathan Bowler and the first author introduce a new category of algebraic objects called {\em tracts} and define a notion of {\em matroids over tracts}.   Examples of tracts include hyperfields in the sense of Krasner and partial fields in the sense of Semple and Whittle.  For example, matroids over the Krasner hyperfield ${\mathbb K}$ are just matroids, matroids over the hyperfield of signs ${\mathbb S}$ are oriented matroids, matroids over the tropical hyperfield ${\mathbb T}$ are valuated matroids, and matroids over a field are linear subspaces.
Matroids over tracts generalize matroids over fuzzy rings in the sense of Dress (\cite{Dress86}). 

\medskip

Actually, there are two different notions of matroid over a tract $F$, called {\em weak} and {\em strong} $F$-matroids.  Over many tracts of interest, including fields and the hyperfields ${\mathbb K}, {\mathbb S}$, and ${\mathbb T}$, weak and strong matroids coincide.  However, the two notions are different in general.  For both weak and strong $F$-matroids, the results of 
\cite{Baker-Bowler19} provide {\em cryptomorphic} axiomatizations of $F$-matroids in terms of circuits, Grassmann-Pl{\"u}cker functions, and dual pairs.  The subsequent work of Laura Anderson (\cite{Anderson19}) also provides a cryptomorphic axiomatization of {\bf strong} $F$-matroids in terms of vectors or covectors.


\medskip

More formally, a tract is a pair $(G,N_G)$ consisting of an abelian 
group $G$ (written multiplicatively), together with a subset $N_G$ (called the {\em nullset} of the tract) of the group semiring ${\mathbb N}[G]$ satisfying:
\begin{itemize}
\item[(T1)] The zero element of ${\mathbb N}[G]$ belongs to $N_G$, and the identity element $1$ of $G$ is not in $N_G$.
\item[(T2)] $N_G$ is closed under the natural action of $G$ on ${\mathbb N}[G]$.
\item[(T3)] There is a unique element $\epsilon$ of $G$ with $1 + \epsilon \in N_G$.
\end{itemize}  

One thinks of $N_G$ as those linear combinations of elements of $G$ which ``sum to zero''.
We let $F = G \cup \{ 0 \} \subset {\mathbb N}[G]$, and we often refer to the tract $(G,N_G)$ simply as $F$.  

\medskip

Tracts form a category in a natural way: a morphism $(G,N_G) \to (G',N_{G'})$ of tracts corresponds to a homomorphism $G \to G'$ which takes $N_G$ to $N_{G'}$.
The Krasner hyperfield ${\mathbb K}$ (identified with its corresponding tract, which is $(\{1\},\{0,2,3,\dotsc\})$) is a final object in the category of tracts.

\subsubsection{Idylls and ordered blueprints: a first glance} \label{sec:firstglance}

Although the axiom (T2) suffices for establishing all of the cryptomorphisms in \cite{Baker-Bowler19}, from a ``geometric'' point of view it is more natural to replace axiom (T2) with the stronger axiom:

\begin{itemize}
\item[(P)]  The nullset of $F$ is an {\bf ideal} in ${\mathbb N}[G]$, i.e., it is closed under addition and if $\alpha \in {\mathbb N}[G]$ and $\beta \in N_G$ then $\alpha \beta \in N_G$.
\end{itemize}  

We define an {\bf idyll} to be a tract satisfying $(P)$, i.e., a tract whose nullset is an ideal.\footnote{Note that if $F=(G,N_G)$ is a tract and $N_G$ is closed under addition, then axiom (T2) guarantees that $N_G$ is in fact an ideal and thus $F$ is an idyll.}

\medskip

One advantage of working with idylls is that they can be naturally thought of as {\em ordered blueprints}.
The theory of ordered blueprints, developed by the second author, has a rich geometric theory associated to it.
There is a speculative remark in \cite{Baker-Bowler19} to the effect that ordered blue schemes might be a suitable geometric category for defining moduli spaces of matroids over tracts.\footnote{In the special case of matroids over hyperfields, one could attempt to construct such moduli spaces as {\em hyperring schemes} in the sense of J.~Jun (\cite{Jun18}), but this is potentially problematic for a few reasons: (i) The category of hyperring schemes does not appear to admit fiber products; (ii) the structure sheaf of a hyperring scheme as defined by Jun has some undesirable properties, e.g., the hyperring of global sections of the structure sheaf on $\Spec (R)$ is not always equal to $R$; and (iii) the theory of hyperring schemes is not as well developed as the theory of ordered blue schemes.  In any case, it is highly desirable to fit not only matroids over hyperfields but also {\em matroids over partial fields} into our theory, and for this ordered blue schemes fit the bill quite well.}
One of the main goals of the present paper is to turn this speculation into a rigorous theorem, at least in the case of {\em strong matroids over idylls}.
The other main goal is to give applications of this algebro-geometric point of view to more traditional questions and ideas in matroid theory.

\subsubsection{The relationship between various algebraic structures}

Loosely speaking, the relationship between hyperfields, tracts, idylls, ordered blueprints, and other algebraic structures mentioned in this Introduction can be depicted as follows (for a more precise statement, see Theorem \ref{thm: hyperfuzz to bluetracts} and the remarks in Section~\ref{sec:HyperfuzzRemarks}):

\begin{figure}[htb]
 \centering
 \includegraphics{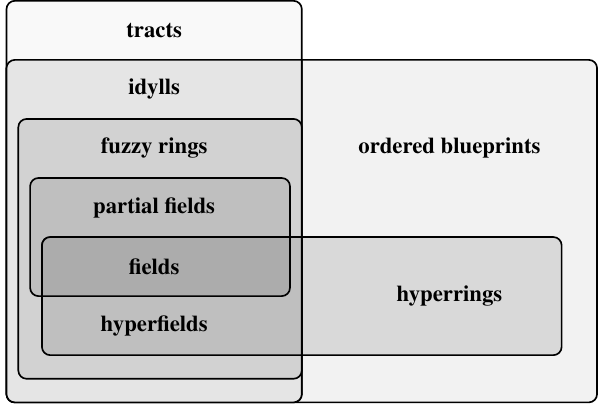}
 \caption{Comparison of different algebraic structures}
 \label{fig: label}
\end{figure}

Note that we consider idylls as both tracts and ordered blueprints, which makes sense since there is an adjunction between the categories of tracts and ordered blueprints that retricts to an equivalence precisely for idylls. In this sense, an idyll can be thought of an object that is both a tract and an ordered blueprint; cf.\ Theorem \ref{thm: hyperfuzz to bluetracts} for more details.

We now turn to giving a more precise definition of ordered blueprints and matroids over them.

\subsubsection{Ordered blueprints}

An {\bf ordered semiring} is a semiring $R$ together with a partial order $\leq$ that is compatible with multiplication and addition. (See section \ref{subsection: ordered blueprints} for a more precise definition.)

An {\bf ordered blueprint} is a triple $B=(B^\bullet,B^+,\leq)$ where $(B^+,\leq)$ is an ordered semiring and $B^\bullet$ is a multiplicative subset of $B^+$ which generates $B^+$ as a semiring and contains $0$ and $1$. 

A {\bf morphism of ordered blueprints $(B^\bullet_1,B^+_1,\leq_1)$ and $(B^\bullet_2,B^+_2,\leq_2)$} is an order-preserving morphism $f:B^+_1\to B^+_2$ of semirings with $f(B^\bullet_1)\subset B^\bullet_2$. 

\medskip

We denote the category of ordered blueprints by $\OBlpr$.

\begin{ex}
A {\bf hyperfield} is an algebraic structure similar to a field, but where addition is allowed to be multivalued (see Section~\ref{subsection: hyperfields and hyperrings} for a precise definition).  
We can identify a hyperfield $F$ with an ordered blueprint $F^\oblpr$ as follows:
\begin{itemize}
\item The associated semiring $(F^\oblpr)^+$ is the free semiring ${\mathbb N}[F^\times]$ over the multiplicative group $F^\times$.
\item The underlying monoid $(F^\oblpr)^\bullet$ is $(F,\cdot)$.
\item The partial order $\leq$ of $(F^\oblpr)^+$ is generated by the relations $0 \leq \sum a_i$ whenever $0 \in \boxplus a_i$.
\end{itemize}
\end{ex}

\begin{ex}
A {\bf partial field} $P$ is a certain equivalence class of pairs $(G,R)$ consisting of a commutative ring $R$ with $1$ and a subgroup $G \leq R^\times$ containing $-1$.
(See section \ref{subsection: partial fields} for a more precise definition.)
We can identify a partial field $P$ with an ordered blueprint $P^\oblpr$ as follows:
\begin{itemize}
\item The associated semiring $(P^\oblpr)^+$ is $\NN[G]$.
\item The underlying monoid $(P^\oblpr)^\bullet$ is $G \cup \{ 0 \}$.
\item The partial order $\leq$ is generated by the 3-term relations $0 \leq a+b+c$ whenever $a,b,c \in G$ satisfy $a+b+c = 0$ in $R$.
\end{itemize}
\end{ex}

An example of particular interest is the ordered blueprint associated with the regular partial field $\U_0=(\{-1,0,1\},\Z)$, whose associated ordered blueprint $\Funpm$
corresponds to the submonoid $\{ 0, 1, -1 \}$ of $\ZZ$ together with the partial order generated by 
$0 \leq 1 + (-1)$.  

The ordered blueprints associated to hyperfields and partial fields are in fact {\bf ordered blue fields}, meaning that $B = B^\times \bigcup \{ 0 \}$, where $B^\times$ denotes the set of invertible elements of $B$.  They are also {\em $\Funpm$-algebras}, a notion which will be defined shortly.

The category of ordered blueprints has an initial object called $\Fun$ with associated semiring $\NN$, underlying monoid $\{ 0,1 \}$ (with the usual multiplication), and partial order given by equality.  


\subsubsection{Some properties of ordered blueprints}

The category of ordered blueprints admits push\-outs: given morphisms $B\to C$ and $B\to D$ of ordered blueprints, one can form their {\bf tensor product} $C\otimes_B D$, which satisfies the universal property of a fiber coproduct.

\medskip

One can also form the localization $S^{-1}B$ of an ordered blueprint $B$ with respect to any multiplicative subset $S$, and it has the usual universal property.

\subsubsection{\texorpdfstring{$\Funpm$}{F1+-}-algebras}

An {\bf $\Funpm$-algebra} is an ordered blueprint $B$ together with a morphism $\Funpm\to B$. Equivalently, an $\Funpm$-algebra is an ordered blueprint $B$ together with an element $\epsilon$ of $B$ that satisfies $0\leq 1+\epsilon$. This element $\epsilon$ plays the role of $-1$ in this theory, and is crucial for defining structures such as matroids. We denote the full subcategory of $\Funpm$-algebras by $\OBlpr_\Funpm$.  



\subsubsection{Idylls as \texorpdfstring{$\Funpm$}{F1+-}-algebras}

The ordered blueprints associated to hyperfields and partial fields are $\Funpm$-algebras.  More generally, if $F=(G,N_G)$ is any {\em idyll} in the sense of section~\ref{sec:firstglance}, we can consider $F$ as an $\Funpm$-algebra $F^\oblpr$ as follows:

\begin{itemize}
\item The associated semiring $(F^\oblpr)^+$ is $\NN[G]$.
\item The underlying monoid $(F^\oblpr)^\bullet$ is $G \cup \{ 0 \}$.
\item The partial order $\leq$ is generated by the relations $0 \leq \sum a_i$ whenever $a_i \in G$ satisfy $\sum a_i \in N_G$ in ${\mathbb N}[G]$.
\end{itemize}

One can characterize ordered blueprints of the form $F^\oblpr$ for some idyll $F$ among all ordered blueprints in a simple way: they are precisely the $\Funpm$-algebras that of the form $(B^\times\cup\{0\},\N[B^\times],\leq)$ for which $0\leq 1+a$ only if $a=\epsilon$ and that are {\bf purely positive}, meaning that $\leq$ is generated by elements of the form $0 \leq \sum a_i$.

\subsubsection{Matroids over \texorpdfstring{$\Funpm$}{F1+-}-algebras}

Let $B$ be an $\Funpm$-algebra, let $E$ be a finite totally ordered set, and let $r \in \NN$.
We denote by $\binom Er$ the family of all $r$-element subsets of $E$. 

\medskip

A {\bf Grassmann-Pl{\"u}cker function} of rank $r$ on $E$ with coefficients in $B$ is a function $\varphi : \binom{E}{r} \to B$ such that: 
\begin{itemize}
\item $\varphi(I) \in B^\times$ for some $I \in \binom{E}{r}$.
\item $\varphi$ satisfies the {\bf Pl{\"u}cker relations}
\[
0 \leq \sum_{k=0}^r \epsilon^k \varphi(I \backslash i_k) \varphi(J \cup i_k)
\]
whenever $J \in \binom{E}{r-1}$ and $I = \{ i_0,\ldots,i_r \} \in \binom{E}{r+1}$ with $i_0 < \cdots < i_r$.
(We set $\varphi(J \cup i) = 0$ if $i \in J$.)
\end{itemize}

We say that two Grassmann-Pl{\"u}cker functions $\varphi, \varphi' : \binom{E}{r} \to B$ are {\bf equivalent} if $\varphi = a \varphi'$ for some $a \in B^\times$.

\medskip

A {\bf $B$-matroid of rank $r$ on $E$} is an equivalence class of Grassmann-Pl{\"u}cker functions.
We denote by $\cMat_B(r,E)$ the set of all $B$-matroids of rank $r$ on $E$.

\medskip

If $F$ is an idyll, an $F$-matroid of rank $r$ on $E$ in the above sense is the same thing as a strong $F$-matroid of rank $r$ on $E$ in the sense of \cite{Baker-Bowler19}.  In this case, we can characterize (strong) $F$-matroids of rank $r$ on $E$ in several different cryptomorphic ways, e.g. in terms of {\em circuits, dual pairs}, or {\em vectors} (see \cite{Anderson19,Baker-Bowler19} or sections \ref{subsubsection: cryptomorphisms for matroids over tracts} and \ref{sec:vectoraxioms} below).

\medskip

The definition of $\cMat_B(r,E)$ is functorial: if $f:B \to C$ is a morphism of $\Funpm$-algebras, there is an induced map $f_* : \cMat_B(r,E) \to \cMat_C(r,E)$.

\medskip

If $F$ is an idyll and $f:F \to {\mathbb K}$ is the canonical morphism to the final object ${\mathbb K}$ (which is shorthand for ${\mathbb K}^\oblpr$) of the category of idylls, the push-forward $\underline{M} := f_*(M)$ is a ${\mathbb K}$-matroid, i.e. a matroid in the usual sense.  We call $\underline{M}$ the {\bf underlying matroid} of $M$.

\medskip

If $M'$ is a matroid, we say that $M'$ is weakly (resp. strongly) {\bf representable} over an idyll $F$ if $M' = f_*(M)$ for some weak (resp. strong) $F$-matroid $M$.  This generalizes the usual notion of representability over fields, or more generally partial fields (for which the notions of weak and strong $F$-matroids coincide).

\subsection{Part 2: Constructing moduli spaces of matroids}  \label{Part2Summary}

As discussed above, we wish to construct a moduli space $\Mat(r,E)$ of rank-$r$ matroids on $E$ as a ordered blue scheme (over $\Funpm$) which represents a certain functor.  In order to formulate precisely what this means, and in particular to specify which moduli functor we wish to represent, we first provide the reader with a gentle introduction to the theory of ordered blue schemes.

\subsubsection{Ordered blue schemes}

One constructs the category of ordered blue schemes, starting from ordered blueprints, much in the same way that one constructs the category of schemes starting from commutative rings.  We give just a brief synopsis here; see section~\ref{subsection: Ordered blue schemes} for further details.

\medskip

Let $B$ be an ordered blueprint.

\medskip

A {\bf monoid ideal} of $B$ is a subset $I$ of $B$ such that $0\in I$ and $IB=I$ where $IB=\{ab\mid a\in I, b\in B\}$. 

\medskip

A {\bf prime ideal} of $B$ is a monoid ideal whose complement is a multiplicative subset. 

\medskip

The {\bf spectrum} $\Spec B$ of $B$ is constructed as follows:

\begin{itemize}
\item The topological space of $X=\Spec B$ consists of the prime ideals of $B$, and comes with the topology generated by the \emph{principal opens}
\[
 U_h \ = \ \{\, \fp\in\Spec B \,|\, h\notin \fp\, \}
\]
for $h \in B$.
\item  The {\bf structure sheaf $\cO_{X}$} is the unique sheaf on $X$ with the property that $\cO_{X}(U_h) = B[h^{-1}]$ for all $h \in B$.  The stalk of $\cO_{X}$ at a point $x \in X$ corresponding to $\fp$ is $B_\fp$.
\end{itemize}

An {\bf ordered blueprinted space} is a topological space $X$ together with a sheaf $\cO_X$ in $\OBlpr$.  Such spaces form a category $\OBlprSp$.
A morphism $f:B\to C$ of ordered blueprints defines a morphism $f^\ast:\Spec C\to \Spec B$ of $\OBlpr$-spaces.
This defines the contravariant functor $\Spec: \ \OBlpr \ \longrightarrow \ \OBlprSp$ whose essential image is the category of {\bf affine ordered blue schemes}. 

\medskip

An {\bf ordered blue scheme} is an $\OBlpr$-space that has an open covering by affine ordered blue schemes $U_i$. A {\bf morphism of ordered blue schemes} is a morphism of $\OBlpr$-spaces. We denote the category of ordered blue schemes by $\OBSch$.

\medskip

An {\bf ordered blue $\Funpm$-scheme} is an ordered blue scheme $X$ for which $\cO_X(U)$ has the structure of an $\Funpm$-algebra for every open subset $U$ of $X$. We denote the full subcategory of ordered blue $\Funpm$-schemes by $\OBSch_\Funpm$.

\subsubsection{Some properties of ordered blue schemes}

Ordered blue schemes possess many familiar properties from the world of schemes.  For example:
\begin{itemize}
\item The global section functor $\Gamma:\OBSch \to \OBlpr$ defined by $\Gamma(X,\cO_X) :=\cO_X(X)$ is a left inverse to $\Spec$. In particular, 
$B \cong \Gamma (\Spec B)$. 
\item The category $\OBSch$ contains fibre products, and in the affine case $\Spec(B) \times_{\Spec(D)} \Spec(C) \cong \Spec (B\otimes_D C)$. 
\end{itemize}

Various familiar objects from algebraic geometry have analogues in the context of ordered blue schemes; for example, one can define an {\bf invertible sheaf} on an ordered blue scheme $X$ to be a sheaf which is locally isomorphic to the structure sheaf $\cO_X$ of $X$.  There is a tensor product operation which turns the set $\Pic X$ of isomorphism classes of invertible sheaves on $X$ into an abelian group.

\medskip
Similarly, one can define, for each $n \in {\mathbb N}$ and each ordered blueprint $B$, the {\bf projective $n$-space} ${\mathbb P}^n_B$ as an ordered blue scheme over $\Spec (B)$.

\subsubsection{Families of matroids}

Let $X$ be an ordered blue $\Funpm$-scheme.  A {\bf Grassmann--Pl{\"u}cker function} of rank $r$ on $E$ over $X$ is an invertible sheaf $\cL$ on $X$ together with a map $\varphi : \binom{E}{r} \to \Gamma(X,\cL)$ such that $\{ \varphi(I) \}_{I \in \binom{E}{r}}$ generate $\cL$ and the $\varphi(I)$ satisfy the Pl{\"u}cker
relations in $\Gamma(X,\cL^{\otimes 2})$ (see Definition \ref{def: matroid bundles} for a more precise definition).

\medskip

Two such functions $(\cL,\varphi)$ and $(\cL',\varphi')$ are said to be {\bf isomorphic} if there is an isomorphism from $\cL$ to $\cL'$ taking $\varphi$ to $\varphi'$.

\medskip

A {\bf matroid bundle} of rank $r$ on the set $E$ over $X$ is an isomorphism class of Grassmann--Pl{\"u}cker functions.

\medskip

If $X=\Spec(B)$ is an affine ordered blue $\Funpm$-scheme, it turns out that a matroid bundle over $X$ is the same thing as a $B$-matroid.


\subsubsection{The moduli functor of matroids}

One can extend the (covariant) functor taking an $\Funpm$-algebra $B$ to $\cMat(r,E)(B)$
to a (contravariant) functor $\cMat(r,E) :  \OBSch_\Funpm \to \Sets$ taking $X$ to the set of matroid bundles of rank $r$ on $E$ over $X$.

\medskip

We prove the following theorem, cf.\ Theorem \ref{thm: moduli space of matroids}: 

\begin{thmA}
The moduli functor $\cMat(r,E)$ is representable by an ordered blue $\Funpm$-scheme $\Mat(r,E)$.  In particular, for every ordered blue $\Funpm$-scheme $X$
there is a natural bijection 
\[
\Hom_{\Funpm} (X,\Mat(r,E)) \isomap \cMat(r,E)(X).
\]
\end{thmA}


The moduli space $\Mat(r,E)$ is constructed as an ordered blue subscheme of $\P^N_\Funpm$, where $N=\#\binom Er-1$.  (This is analogous to the Pl{\"u}cker embedding of the Grassmannian $\Gr(r,n)$.)  However, making this precise requires developing some foundational material on line bundles, the ``Proj'' construction, etc. in the context of ordered blue schemes.

\subsection{Part 3: Applications to matroid theory}  \label{Part3Summary}

We conclude this introduction by providing a more detailed overview of Part 3 of the paper, in which we connect various algebraic structures related to the moduli spaces $\Mat(r,E)$ to concepts such as realization spaces, cross ratios, rescaling classes, and universal partial fields which have been previously studied in the matroid theory literature.

\subsubsection{Universal idylls}

Given a (classical) matroid $M$, we can associate to $M$ a {\em universal idyll} $k_M$, which is derived from a certain ``residue ordered blue field'' of the matroid space $\Mat(r,E)$. 

\medskip

More precisely, a classical matroid $M$ corresponds to a morphism $\chi_M:\Spec\K\to\Mat(r,E)$ which we call the {\bf characteristic morphism} of $M$.
Topologically, $\Spec\K$ is a point, and the image point $x_M$ in the ordered blue $\Funpm$-scheme $\Mat(r,E)$ of the characteristic morphism has an associated residue idyll, much as every point of a (classical) scheme has an associated residue field.
We call the residue idyll $k_M$ of $x_M$ the {\bf universal idyll} of $M$.  


\subsubsection{Realization spaces}

Let $K$ be a field.  The realization space over $K$ of a rank-$r$ matroid $M$ on $E=\{1,\ldots,n\}$ is the subset of the Grassmannian $\Gr(r,n)$ consisting of sub-vector spaces of $K^n$ whose associated matroid is $M$. Such realization spaces have been used for proving that several moduli spaces, such as Hilbert schemes and moduli spaces of curves, can have arbitrarily complicated singularities, cf.\ \cite{Vakil06}. 

\medskip

Given a matroid $M$ and an idyll $F$, the {\bf realization space} $\cX_M(F)$ is the set of isomorphism classes of $F$-matroids whose underlying matroid is $M$.
More precisely, let $\Delta:\binom Er\to \K$ be a Grassmann-Pl\"ucker function, $M$ the corresponding matroid, and $\chi_M: \Spec(\K)\to\Mat(r,E)$ its characteristic morphism. 
The canonical map from $F$ to $\K$ (which takes $0$ to $0$ and every nonzero element of $F$ to $1$) induces a natural map $\Phi: \ \Mat(r,E)(F) \ \longrightarrow \ \Mat(r,E)(\K)$ taking an $F$-matroid to its underlying matroid. 
With this notation, the realization space $\cX_M(F)$ of $M$ over $F$ is the fibre of $\Phi$ over $\chi_M$. 

\medskip

Realization spaces are functorial with respect to morphisms of idylls.

\medskip

The functor from idylls to sets taking an idyll $F$ to the realization space $\cX_M(F)$ is represented by the universal idyll $k_M$.  In other words, there is a canonical bijection 
\[
{\rm Hom}(k_M,F) \isomap \cX_M(F)
\]
which is functorial in $F$.

\subsubsection{The weak matroid space}

So far we have been talking more or less exclusively about {\em strong matroids} in the sense of \cite{Baker-Bowler19}.  However, there is also a notion of {\em weak matroids} over an idyll $F$ which is quite important in many contexts.

\medskip

A {\bf weak Grassmann-Pl\"ucker function} of rank $r$ on $E$ with coefficients in an idyll $F$ is a function
 \[\textstyle
  \Delta: \ \binom Er \ \longrightarrow \ F
 \]
 whose support is the set of bases of a matroid and which satisfies the $3$-term Pl\"ucker relations
 \[
  0 \ \leq \ \Delta(I_{1,2}) \, \Delta(I_{3,4}) \ + \ \epsilon \, \Delta(I_{1,3}) \, \Delta(I_{2,4}) \ + \ \Delta(I_{1,4}) \, \Delta(I_{2,3})
 \]
 for every $(r-2)$-subset $I$ of $E$ and all $i_1<i_2<i_3<i_4$ with $i_1,i_2,i_3,i_4\notin I$, where $I_{k,l}=I\cup\{i_k,i_l\}$.
 
\medskip

Two weak Grassmann-Pl\"ucker functions $\Delta$ and $\Delta'$ are \emph{equivalent} if $\Delta=a\Delta'$ for some element $a\in F^\times$. 
 
 \medskip
 
A {\bf weak $F$-matroid of rank $r$ on $E$} is an equivalence class $M$ of weak Grassmann-Pl\"ucker functions $\Delta$ of rank $r$ on $E$ with coefficients in $F$. We denote the set of all weak $F$-matroids of rank $r$ on $E$ by $\cMat^w(r,E)(F)$. 

\medskip


The weak matroid space $\Mat^w(r,E)$ is defined analogously to $\Mat(r,E)$, with the important difference that we only impose $3$-term Pl\"ucker relations; see section~\ref{subsection: the weak matroid space} for a precise definition.\footnote{Note that the natural injection $\cMat^w(r,E)(F)\to\Mat^w(r,E)(F)$ fails in general to be surjective. The reason is that the additional condition that the support of a weak Grassmann-Pl\"ucker function must be the set of bases of a matroid is not always satisfied by functions representing $F$-points of $\Mat^w(r,E)$.} 
For a matroid $M$, we define the {\bf universal pasture} $k_M^w$ as the residue idyll of the space of weak matroids at the point corresponding to $M$.
We can also define the {\bf weak realization space}  $\cX_M^w(F)$ of $M$ over $F$ to be the set of all weak $F$-matroids whose underlying matroid is $M$.

\medskip

As in the strong case, there is a canonical bijection
\[
{\rm Hom}(k_M^w,F) \isomap \cX^w_M(F)
\]
which is functorial in $F$.

\subsubsection{Cross ratios}

Four points on a projective line over a field $K$ correspond to a point of the Grassmannian $\Gr(2,4)$ over $K$, and their cross ratio can be expressed in terms of the Pl\"ucker coordinates of this point. This reinterpretation allows for a generalization of cross ratios to higher Grassmannians and also to non-realizable matroids.

\medskip

Let $F$ be an idyll and $M$ be a matroid of rank $r$ on $E$. The {\em cross ratios} of $M$ in $F$ are indexed by the set $\Omega_M$ of $4$-tuples $\cI=(I,i_1,i_2,i_3,i_4)\in\binom E{r-2}\times E^4$ for which $I_{1,2}$, $I_{2,3}$, $I_{3,4}$ and $I_{4,1}$ are bases of $M$, where $I_{k,l}=I\cup\{i_k,i_l\}$. 


\medskip

Let $F$ be an idyll and let $M$ be a weak $F$-matroid defined by the weak Grassmann-Pl\"ucker function $\Delta:\binom Er\to F$. The {\bf cross ratio function} of $M$ is the function $\Cr_M :\Omega_M\to F^\times$ that sends an element $\cI=(I,i_1,i_2,i_3,i_4)$ of $\Omega_M$ to
 \[
 \Cr_M(\cI) \ = \ \frac{\Delta(I_{1,2}) \cdot \Delta(I_{3,4})}{\Delta(I_{2,3}) \cdot \Delta(I_{4,1})}.
 \]

One checks easily that this depends only on the equivalence class of $\Delta$, and is thus a well-defined function of $M$.


\subsubsection{Foundations}

Let $B$ be an $\Funpm$-algebra. A \emph{fundamental element of $B$} is an element $a\in B$ such that $0\leq a+b+\epsilon$ 
for some $b\in B$.
 
 \medskip
 
The \emph{foundation of $B$} is the subblueprint $B^\found$ of $B$ generated by the fundamental elements of $B$.  Taking foundations is a functorial construction.

\medskip

The relevance of this notion for matroid theory is that cross ratios of weak $F$-matroids are fundamental elements of $F$ (which is a simple consequence of the 3-term Pl{\"u}cker relations).

\medskip

If $M$ is a matroid, we define the {\bf foundation of $M$}, denoted $k_M^f$, to be the foundation of the universal pasture $k_M^w$ of $M$.
Since the functor taking an idyll $F$ to the weak realization space $\cX^w_M(F)$ is represented by $k_M^w$, and cross ratios of weak $F$-matroids are fundamental elements of $F$, there is a natural {\bf universal cross ratio function} $\Cr_M^\univ:\Omega_M\to k_M^f$.

\medskip

We prove that the foundation $k_M^f$ of $M$ is generated by the universal cross ratios $\Cr_M^\univ(\Omega_M)$ over $\Funpm$. 

\medskip

We also show (cf. Theorem~\ref{thm: weakly representable over F and morphism from the fundament}) that for every matroid $M$ and idyll $F$, the following are equivalent: (a) $M$ is weakly representable over $F$; (b) $M$ is weakly representable over $F^\found$; and (c) there exists a morphism $k_M^f\to F$.


\subsubsection{The Tutte group and inner Tutte group}

The Tutte group $\T_M$ of a matroid $M$ was introduced by Dress and Wenzel in \cite{Dress-Wenzel89} as a tool for studying the representability of matroids by algebraically encoding results such as Tutte's homotopy theorem (cf.\ \cite{Tutte58a} and \cite{Tutte58b}).

\medskip

The Tutte group is usually defined in terms of generators and relations, and several ``cryptomorphic'' presentations of this group are known.  Our approach allows for an intrinsic definition of the Tutte group $\T_M$ as the unit group of the universal pasture $k_M^w$ (see section~\ref{subsection: the Tutte group} for details).

\medskip

Dress and Wenzel also define a certain subgroup $\T_M^{(0)}$ of the Tutte group $\T_M$ which they call the {\em inner Tutte group}.
Using their results, we show that the natural isomorphism $(k_M^w)^\times\to\T_M$ restricts to an isomorphism $(k_M^f)^\times\to\T_M^{(0)}$.
In other words, the inner Tutte group of $M$ is the unit group of the foundation of $M$.



\subsubsection{Rescaling classes}

Let $M$ be a matroid of rank $r$ on $E$.
The importance of the foundation of $M$ is that it represents the functor which takes an idyll $F$ to the set of {\em rescaling classes} of $F$-matroids with underlying matroid $M$, in the same way the universal idyll (resp. universal pasture) of $M$ represents the realization space $\cX_M(F)$ (resp. weak realization space $\cX_M^w(F)$).

\medskip

Let $F$ be an idyll, and let $T(F)$ be the group of functions $t:E\to F^\times$.
The {\bf rescaling class} of an $F$-matroid $M$ is the $T(F)$-orbit of $M$ in $\Mat^w(r,E)(F)$, where $T(F)$ acts on a weak Grassmann-Pl\"ucker function $\Delta:\binom Er\to F$ by the formula 
\[
 t.\Delta(I) \ = \ \prod_{i\in I}t(i) \cdot\Delta(I).
\]

Rescaling classes are the natural generalization to matroids over arbitrary idylls
of {\em reorientation classes} for oriented matroids, where two (realizable) oriented matroids are considered \emph{reorientation equivalent} if they correspond to the isomorphic real hyperplane arrangements.  (For non-realizable matroids, there is a similar assertion involving pseudosphere arrangements; this is part of the famous "Topological Representation Theorem" of Folkman and Lawrence, cf.\ \cite{Folkman-Lawrence78} or \cite[section 5.2]{Bjorner-LasVergnas-Sturmfels-White-Ziegler99}.)

\medskip

If we fix a matroid $M$ and an idyll $F$, we define the {\bf rescaling class space} $\cX^f_M(F)$ to be the set of rescaling classes of weak $F$-matroids with underlying matroid $M$.
There is a canonical bijection
\[
\Hom(k_M^f,F) \isomap  \cX_M^f(F) 
\]
which is functorial in $F$.

\medskip

As a sample motivation for considering rescaling classes over more general idylls than just the hyperfield of signs, we mention that while it is true that a matroid $M$ is regular (i.e., representable over the rational numbers by a totally unimodular matrix) if and only if $M$ is representable over $\Funpm$, there are in general many non-isomorphic $\Funpm$-matroids whose underlying matroid is a given regular matroid $M$.
However, there is always precisely one rescaling class over $\Funpm$.  In other words, regular matroids are the same thing as rescaling classes of $\Funpm$-matroids.

\subsubsection{Foundations of binary and regular matroids}
\label{subsubsection: intro - foundations of binary and regular matroids}

A \emph{binary matroid} is a matroid that is representable over the finite field $\F_2$ with two elements.

\medskip

We show that a matroid is regular if and only if its foundation is $\Funpm$, and binary if and only if its foundation is either $\Funpm$ or $\F_2$.
We recover from these observations a new proof of the well-known facts that (a) a matroid is regular if and only if it is representable over every field; and (b) a binary matroid is either representable over every field or not representable over any field of characteristic different from $2$.

\medskip

We also use these observations to give new and conceptual proofs of the following facts: (a) a binary matroid has at most one rescaling class over every idyll (compare with \cite[Thm.\ 6.9]{Wenzel91}); and (b) every matroid has at most one rescaling class over $\F_3$ (cf.~\cite{Brylawski-Lucas73}).

\medskip

In addition, we show (cf. Theorem~\ref{thm: characteriztaion of regular matroids}) that a matroid $M$ is regular if and only if $M$ is binary and weakly representable over some idyll $F$ with $1\neq\epsilon$ .
This implies, for example (taking $F$ to be the hyperfield of signs) the well-known fact that a matroid is regular if and only if it is both binary and orientable.

\subsubsection{Relation to the universal partial field of Pendavingh and van Zwam}

The universal partial field $\P_M$ of a matroid $M$ was introduced by Pendavingh and van Zwam in \cite{Pendavingh-vanZwam10a}.
It has the property that a matroid $M$ is representable over a partial field $P$ if and only if there is a partial field homomorphism $\P_M\to P$.

\medskip

We show that there is a partial field $P_{M,0}$ naturally derived from the universal pasture $k_M^w$ of $M$ with the property that for every partial field $P$ there is a natural bijection
\[
\Hom(P_{M,0}, P) \isomap \cX_M(P^\oblpr) 
\]
which is functorial in $P$.  

\medskip

The universal partial field of Pendavingh and van Zwam is isomorphic to the partial subfield $\P_M$ of $P_{M,0}$ generated by the cross ratios of $M$.
We prove that for every partial field $P$ there is a natural and functorial bijection
\[
\Hom(\P_{M}, P) \isomap \cX_M^f(P^\oblpr).
\]

\medskip

One disadvantage of the universal partial field is that it doesn't always exist: there are matroids (e.g. the V\'amos matroid) which are not representable over any partial field.
However, every matroid is representable over some idyll, so the foundation $k_M^f$ of $M$ gives us information about representations of $M$ even when the universal partial field is undefined.  

\medskip

Our classification of binary and regular matroids in terms of their foundations also yields a classification of such matroids in terms of their universal partial fields:
a matroid is regular if and only if its universal partial field is $\Funpm$, and binary if and only if its universal partial field is $\Funpm$ or $\F_2$.  




\bigskip
\part{Idylls, ordered blueprints, and matroids}
\nopagebreak[4]
\bigskip

\section{The interplay between partial fields, hyperfields, fuzzy rings, tracts, and ordered blueprints}
\label{section: the requisites}

Our approach to matroid bundles utilizes an interplay between tracts and ordered blueprints, as introduced by the first author and Bowler in \cite{Baker-Bowler19} and the second author in \cite{Lorscheid15}, respectively. Tracts and ordered blueprints are common generalizations of other algebraic structures that appear in matroid theory, such as partial fields, hyperfields, and fuzzy rings. 

In this section, we review the definitions of all of the aforementioned notions and explain their interdependencies. Our exposition culminates in Theorem \ref{thm: hyperfuzz to bluetracts}, which exhibits a diagram of comparison functors between the corresponding categories.


\subsection{Semirings}
\label{subsection: semirings}

Since many of the following concepts are based on semirings and derived notions, we begin with an exposition of semirings. All of our structures will be commutative and, following the practice of the literature in commutative algebra and algebraic geometry, we omit the adjective ``commutative'' when speaking of semirings, monoids, and other structures.

A \emph{monoid} is a commutative semigroup with a neutral element. A monoid morphism is a multiplicative map that preserves the neutral element. In this text, a \emph{semiring} is a set $R$ together with two binary operations $+$ and $\cdot$ and with two constants $0$ and $1$ such that the following axioms are satisfied:
\begin{enumerate}[label={(SR\arabic*)}]
 \item\label{SR1} $(R,+,0)$ is a monoid;
 \item\label{SR2} $(R,\cdot,1)$ is a monoid;
 \item\label{SR3} $0\cdot a=0$ for all $a\in R$;
 \item\label{SR4} $a(b+c)=ab+ac$ for all $a,b,c\in R$.
\end{enumerate}
A \emph{morphism of semirings} is a map $f:R_1\to R_2$ between semirings $R_1$ and $R_2$ such that 
\[
 f(0) \ = \ 0, \qquad f(1) \ = \ 1, \qquad f(a+b) \ = \ f(a)+f(b) \qquad \text{and} \qquad f(ab) \ = \ f(a)f(b) 
\]
for all $a,b\in R_1$. We denote the category of semirings by $\SRings$.

Let $R$ be a semiring. An \emph{ideal of $R$} is a subset $I$ such that $0\in I$, $a+b\in I$ and $ac\in I$ for all $a,b\in I$ and $c\in R$. An ideal is \emph{proper} if it is not equal to $R$. 

Given any subset $S=\{a_i\}_{i\in I}$ of elements of $R$, we define the ideal $\gen{S}=\gen{a_i}_{i\in I}$ generated by $S$ as the smallest ideal of $R$ containing $S$, which is equal to
\[
 \gen{S} \quad = \quad \bigcap_{\substack{\text{ideals }J\\\text{with }S\subset J}} J \quad = \quad \big\{ \,\sum_{i\in I} b_ia_i \in R \, \big| \, b_i\in R\text{ with }b_i=0\text{  for almost all }i\in I \, \big\}.
\]

The \emph{group of units of $R$} is the group $R^\times$ of all multiplicatively invertible elements of $R$. A \emph{semifield} is a semiring $R$ such that $R=R^\times\cup\{0\}$.

Note that an ideal $I$ is proper if and only if $I\cap R^\times=\emptyset$. A semiring $R$ is a semifield if and only if $\{0\}$ and $R$ are the only ideals of $R$.

\begin{ex}
 Every (commutative and unital) ring is a semiring. Examples of semirings that are not rings are the natural numbers $\N$ and the nonnegative real numbers $\R_{\geq0}$. Examples of a more exotic nature are the \emph{tropical numbers} $\R_{\geq0}$ together with the usual multiplication and the tropical addition $a+b=\max\{a,b\}$, and the \emph{Boolean numbers} $\B=\{0,1\}$ with $1+1=1$, which appears simultaneous as a subsemiring and as a quotient of the tropical numbers.
\end{ex}

\subsubsection{Monoid semirings}

Let $A$ be a multiplicatively written monoid and $R$ a semiring. The monoid semiring $R[A]$ consists of all finite formal $R$-linear combinations $\sum r_a a$ of elements $a$ of $A$, i.e.\ almost all $r_a\in R$ are zero. The addition and multiplication of $R[A]$ are defined by the formulas
\[
 \Big(\sum_{a\in A} r_aa\Big) + \Big(\sum_{a\in A} s_aa\Big) \ = \ \sum_{a\in A} (r_a+s_a)a \quad \text{and} \quad \Big(\sum_{a\in A} r_aa\Big) \cdot \Big(\sum_{a\in A} s_aa\Big) \ = \ \sum_{a\in A} \Big(\hspace{-0pt}\sum_{bc=a}\hspace{-0pt} r_bs_c\Big) a,
\]
respectively. The zero of $R[A]$ is $0=\sum 0\cdot a$ and its multiplicative identity is $1=\sum r_aa$ with $r_1=1$ and $r_a=0$ for $a\neq1$. 

This construction comes with an inclusion $A\to R[A]$, and we often identify $A$ with its image in $R[A]$, i.e.\ we write $b$ or $1\cdot b$ for the element $\sum r_aa$ of $R[A]$ with $r_b=1$ and $r_a=0$ for $a\neq b$. In the case $R=\N$, every element of $\N[A]$ is a sum $\sum a_i$ of elements $a_i$ in $A$.

\begin{ex}
 Polynomial semirings are particular examples of monoid semirings. Let $R$ be a semiring and let $A=\{\prod_{i=1}^nT_i^{e_i} \; | \; e_i\in\N\}$ be the monoid of all monomials in $T_1,\dotsc, T_n$. Then $R[A]$ is the polynomial semiring $R[T_1,\dotsc,T_n]$.
\end{ex}

\subsection{Partial fields}
\label{subsection: partial fields}

In \cite{Semple-Whittle96}, Semple and Whittle introduced partial fields as a tool for studying representability questions about matroids. The theory of matroid representations over partial fields was developed further by Pendavingh and van Zwam in \cite{Pendavingh-vanZwam10a} and \cite{Pendavingh-vanZwam10b}; also cf.\ van Zwam's thesis \cite{vanZwam09}. Loosely speaking, a partial field can be thought of as a set $P$ together with distinguished elements $0$ and $1$, a map $\cdot:P\times P\to P$, and a partially defined map $+:P\times P\to P$ satisfying:
\begin{itemize}
\item $(P,\cdot,1)$ is a commutative monoid in which every nonzero element is invertible
\item $+$ is associative and commutative with neutral element $0$
\item every element $a\in P$ has a unique additive inverse $-a$ 
\item multiplication distributes over addition. 
\end{itemize}

Morphisms of partial fields are defined to be structure preserving maps.

It is somewhat involved to make the requirements on $+$ rigorous; in particular, the formulation of the associativity of $+$ involves binary rooted trees with labelled leaves.

Van Zwam gives in \cite[section 2.1]{vanZwam09} a simpler but equivalent description of partial fields in terms of a ring $R$ together with a subgroup $G$ of the unit group $R^\times$ of $R$ that contains $-1$. The downside of this approach is that morphisms are not structure preserving maps; in particular, the isomorphism type of the ambient ring $R$ is not determined by a partial field. There is, however, a distinguished ambient ring for every partial field, which has better properties than other choices for $R$, cf.\ \cite[Thm.\ 2.6.11]{vanZwam09}. This latter observation leads us to the following hybrid of Semple-Whittle's and van Zwam's definitions.

Let $\Z[G]$ be the group ring of a group $G$, which comes together with the inclusion $G\cup\{0\}\to\Z[G]$, sending $a\in G$ to $1\cdot a$ and $0$ to $0$.

A \emph{partial field $P=(P^\times,\pi_P)$} is a commutative group $P^\times$ together with a surjective ring homomorphism $\pi_P:\Z[P^\times]\to R_P$ such that
\begin{enumerate}[label={(PF\arabic*)}]
 \item\label{PF1} the composition $P^\times\cup\{0\}\longrightarrow\Z[P^\times]\stackrel{\pi_P}\longrightarrow R_P$ is injective;
 \item\label{PF2} for every $a\in P^\times$, there is a unique element $b\in P^\times$ such that $\pi_P(a+b)=0$;
 \item\label{PF3} the kernel of $\pi_P$ is generated by all elements $a+b+c$ with $a,b,c\in P^\times\cup\{0\}$ such that $\pi_P(a+b+c)=0$. 
\end{enumerate}

We can recover the motivating properties of a partial field from these axioms: 

\begin{itemize}
\item we define $P^\times\cup\{0\}$ as the underlying set of $P$, i.e.\ $P$ is a subset of the ambient ring $R_P$; 
\item we write $a+b=c$ for elements $a,b,c\in P$ with $\pi_P(a+b+(-c))=0$, which defines the partial addition $+$ of $P$; 
\item if $\pi_P(a+b)=0$, then $\pi_P(b)=\pi_P((-1)\cdot a)$, i.e.\ every element $a\in P$ has a unique additive inverse with respect to the partial addition of $P$, which we denote by $-a$.
\end{itemize}


A \emph{morphism $f:P_1\to P_2$ of partial fields} is a group homomorphism $P^\times_1\to P^\times_2$ that extends to a ring homomorphism $R_{P_1}\to R_{P_2}$. Using the notation introduced above, this is the same as a map $f:P_1\to P_2$ such that $f(0)=0$, $f(1)=1$, $f(ab)=f(a)f(b)$ for all $a,b\in P_1$ and $f(a)+f(b)=f(c)$ if $a+b=c$ with $a,b,c\in P_1$. We denote the category of partial fields by $\PartFields$.

\begin{ex}\label{ex: partial fields}
 A field $K$ can be identified with the partial field $(K^\times,\pi_K)$, where $\pi_K:\Z[K^\times]\to K$ is the surjective ring homomorphism induced by the identity map $K^\times\cup\{0\}\to K$. Note that with this identification a field homomorphism is the same as a morphism between the associated partial fields.
 
 The \emph{regular partial field $\U_0$} consists of the group $\U_0^\times=\{\pm1\}$ and the surjective ring homomorphism $\pi_{\U_0}:\Z[\U_0^\times]\to\Z$ mapping $\pm 1$ to the corresponding elements in $\Z$. Note that it is an initial object in the category of partial fields. Later on, this partial field will be reincarnated as the ordered blue field $\Funpm$, cf.\ Example \ref{ex: Funpm-algebras}.
 
 For an extensive list of other examples, we refer to \cite{vanZwam09}.
\end{ex}


\subsection{Hyperfields and hyperrings}
\label{subsection: hyperfields and hyperrings}

The notion of an algebraic structure in which addition is allowed to be multi-valued goes back to Fr\'ed\'eric Marty, who introduced hypergroups in 1934 (\cite{Marty36}).  Later on, in the mid-1950's, Marc Krasner (\cite{Krasner56}) developed the theory of hyperrings and hyperfields in the context of approximating non-Archimedean fields, and in the 1990's Murray Marshall (\cite{Marshall96}) explored connections to the theory of real spectra and spaces of orderings.  Subsequent advocates of hyperstructures included Oleg Viro (\cite{Viro11}, in connection with tropical geometry) and Connes and Consani (\cite{Connes-Consani11}, in connection with geometry over $\Fun$). 

A \emph{commutative hypergroup} is a set $G$ together with a distinctive element $0$ and a \emph{hyperaddition}, which is a map
\[
 \hyperplus: \quad G\times G \quad \longrightarrow \quad \cP(G)
\]
into the power set $\cP(G)$ of $G$, such that:
\begin{enumerate}[label={(HG\arabic*)}]
 \item\label{HG1} $a\hyperplus b$ is not empty,    \hfill\textit{(nonempty sums)}
 \item\label{HG2} $\bigcup_{d\in b\hyperplus c} a\hyperplus d = \bigcup_{d\in a\hyperplus b} d\hyperplus c$, \hfill\textit{(associativity)}
 \item\label{HG3} $0\hyperplus a=a\hyperplus 0=\{a\}$,  \hfill\textit{(neutral element)}
 \item\label{HG4} there is a unique element $-a$ in $G$ such that $0\in a\hyperplus (-a)$,   \hfill\textit{(inverses)}
 \item\label{HG5} $a\hyperplus b=b\hyperplus a$,    \hfill\textit{(commutativity)}
 \item\label{HG6} $c\in a\hyperplus b$ if and only if $(-a)\in (-c)\hyperplus b$    \hfill\textit{(reversibility)}
\end{enumerate}
for all $a,b,c\in G$. Note that thanks to commutativity and associativity, it makes sense to define hypersums of several elements $a_1,\dotsc, a_n$ unambiguously by the recursive formula

\[
 \hypersum_{i=1}^n \ a_i \ = \  \bigcup_{b\in\hypersum_{i=1}^{n-1} a_i} b\hyperplus a_n.
\]


A \emph{(commutative) hyperring} is a set $R$ together with distinctive elements $0$ and $1$ and with maps $\hyperplus:R\times R\to\cP(R)$ and $\cdot:R\times R\to R$ such that 
\begin{enumerate}[label={(HR\arabic*)}]
 \item\label{HR1} $(R,\hyperplus,0)$ is a commutative hypergroup,
 \item\label{HR2} $(R,\cdot,1)$ is a commutative monoid,
 \item\label{HR3} $0\cdot a=a\cdot 0=0$,
 \item\label{HR4} $a\cdot(b\hyperplus c)=ab\hyperplus ac$
\end{enumerate}
for all $a,b,c\in R$ where $a\cdot(b\hyperplus c)=\{ad|d\in b\hyperplus c\}$. Note that the reversibility axiom \ref{HG6} for the hyperaddition follows from the other axioms of a hyperring. 

A \emph{morphism of hyperrings} is a map $f:R_1\to R_2$ between hyperrings such that 
\[
 f(0) \ = \ 0, \qquad f(1) \ = \ 1, \qquad f(a\hyperplus b) \ \subset \ f(a)\hyperplus f(b) \qquad \text{and} \qquad f(ab) \ = \ f(a)\cdot f(b) 
\]
for all $a,b\in R_1$ where $f(a\hyperplus b)=\{f(c)|c\in a\hyperplus b\}$. We denote the category of hyperrings by $\HypRings$.

The unit group $R^\times$ of a hyperring $R$ is the group of all multiplicatively invertible elements in $R$. A \emph{hyperfield} is a hyperring $K$ such that $0\neq1$ and $K=K^\times\cup\{0\}$. We denote the full subcategory of hyperfields in $\HypRings$ by $\HypFields$.

\begin{ex}\label{ex: hyperfields}
 Every ring $R$ can be considered as a hyperring by defining $a\hyperplus b=\{a+b\}$. If $R$ is a field, the corresponding hyperring is a hyperfield.
 
 The \emph{Krasner hyperfield} is the hyperfield $\K=\{0,1\}$ whose addition is characterized by $1\hyperplus 1=\{0,1\}$. Note that all other sums and products are determined by the hyperring axioms. It is a terminal object in $\HypFields$. 

 The \emph{tropical hyperfield $\T$} was introduced by Viro in \cite{Viro11}. Its multiplicative monoid consists of the non-negative real numbers $\R_{\geq0}$, together with the usual multiplication, and its hyperaddition is defined by the rule $a\hyperplus b=\max\{a,b\}$ if $a\neq b$ and $a\hyperplus a =[0,a]$. The tropical hyperfield has a particular importance for valuations and tropical geometry, since a nonarchimedean absolute value $v:K\to \R_{\geq0}$ on a field $K$ is the same thing as a morphism $v:K\to\T$ of hyperfields.
 
 The \emph{sign hyperfield $\S$} is the multiplicative monoid $\S=\{0,\pm1\}$ together with the hyperaddition characterized by $1\hyperplus 1=\{1\}$, $(-1)\hyperplus (-1)=\{-1\}$, and $1\hyperplus(-1)=\{-1,0,1\}$. Note that with this definition, the sign map $\R\to\S$ becomes a morphism of hyperfields.
 
 There is a more general construction of hyperfields, as quotients of fields by a multiplicative subgroup, which covers all of the previous examples. Let $K$ be a field and $G$ a multiplicative subgroup of $K^\times$. Then the quotient $K/G$ of $K$ by the action of $G$ on $K$ by multiplication carries a natural structure of a hyperfield: we have $(K/G)^\times=K^\times/G$ as an abelian group and 
 \[
  [a]\hyperplus[b] \ = \ \big\{ \ [c] \ \big| \ c=a'+b'\text{ for some }a'\in[a], b'\in[b] \ \big\} 
 \]
 for classes $[a]$ and $[b]$ of $K/G$.
\end{ex}

\begin{ex}
 The Krasner hyperfield and the tropical hyperfield are instances of a construction that applies to all totally ordered idempotent semirings. Namely, every idempotent semiring $R$ comes with a natural partial order $\leq$ defined by declaring that $a\leq b$ if $a+b=b$. If $R$ is bipotent, i.e.\ $a+b\in\{a,b\}$ for all $a,b\in R$, then $R$ is totally ordered with respect to $\leq$.
 
When $R$ is totally ordered, we can define a hyperaddition on $R$ as 
 \[
  a \hyperplus b \ = \ \begin{cases} 
                        \{a+b\}                           & \text{if }a\neq b, \\ 
                        \{c\in R \, | \, c\leq a\} & \text{if }a=b. 
                       \end{cases}
 \]
 The set $R$, together with its usual multiplication and the hyperaddition $\hyperplus$, is a hyperring.  (The only nontrivial thing to check is associativity of the hyperaddition, which requires the total order of $R$).  If $R$ is a totally ordered idempotent {\it semifield}, this procedure turns $R$ into a hyperfield. 
 \end{ex}


\subsection{Fuzzy rings}
\label{subsection: fuzzy rings}

We review the definition of a fuzzy ring from \cite{Dress86} in a slightly simplified but evidently equivalent form. As a second step, we give a yet simpler description of the category of fuzzy rings by exhibiting a representative of a particularly simple form in each isomorphism class. We refer the reader to section \ref{subsection: semirings} for all preliminary definitions on semirings.

A \emph{fuzzy ring} is a possibly nondistributive semiring $R$, i.e.\ it might disobey axiom \ref{SR4} of a semiring, together with a proper ideal $I$ that satisfies the following axioms for all $a,b,c,d\in R$:
\begin{enumerate}[label={(FR\arabic*)}]
 \item \label{FR1} there is a unique $\epsilon\in R^\times$ such that $1+\epsilon\in I$;
 \item \label{FR2} $a+b,c+d\in I$ implies $ac+\epsilon bd\in I$.
 \item \label{FR3} $a\in R^\times$ implies $a(b+c)=ab+ac$;
 \item \label{FR4} $a+b(c+d)\in I$ implies $a+bc+bd\in I$;
\end{enumerate}
Here $R^\times=\{\, a\in R \, | \, ab=1\text{ for some }b\in R \, \}$ denotes the \emph{unit group of $R$}.

Since these axioms might look bewildering to the reader that sees them for the first time, we include a brief discussion to motivate them. The idea behind this definition is that the zero of a field becomes replaced by the ideal $I$; consequently additive inverses only exist ``up to $I$''. 

Axiom \ref{FR1} implies that every unit $a\in R^\times$ has a unique additive inverse in $R^\times$, which is $\epsilon a$. In particular, we have $\epsilon^2=1$. Axioms \ref{FR3} and \ref{FR4} are weakened forms of distributivity---these axioms hold automatically in a distributive semiring. Axiom \ref{FR2} is reminiscent of the arithmetic of quotient rings: $a\equiv \epsilon b\pmod{I}$ and $c\equiv\epsilon d\pmod{I}$ implies $ac\equiv bd\pmod{I}$ and thus $ab+\epsilon bd\in I$.

\begin{ex}
 A standard example is the fuzzy ring associated with a field $K$, which is given by $R=\{S\subset K \; | \; S\neq\emptyset\}$ with addition $S+T=\{s+t \; | \; s\in S,t\in T\}$ and multiplication $S\cdot T=\{st \; | \; s\in S,t\in T\}$ and whose ideal is $I=\{S\subset K \; | \; 0\in S\}$. Another example is the group semiring $\N[G]$ of an abelian group $G$ together with a proper ideal $I$ of $\N[G]$ such that \ref{FR1} and \ref{FR2} hold---axioms \ref{FR3} and \ref{FR4} are automatically satisfied since $\N[G]$ is distributive.
\end{ex}

A \emph{morphism $R_1\to R_2$ of fuzzy rings} is a group homomorphism $f:R_1^\times \to R_2^\times$ such that $\sum a_i\in I_1$ implies $\sum f(a_i)\in I_2$ for all $a_1,\dotsc, a_n\in R_1^\times$. We denote the category of fuzzy rings by $\FuzzRings$.

Note that Dress also defines homomorphisms of fuzzy rings in \cite{Dress86}. Since this latter notion does not have a particular meaning for matroid theory, we omit it from our exposition. Note further that morphisms and homomorphisms were renamed in \cite{Giansiracusa-Jun-Lorscheid16} as weak and strong morphisms, respectively.

It follows immediately from the definition that a morphism of fuzzy rings preserves $1$ and $\epsilon$. It is also apparent that it depends only on the respective subsets of $R_1$ and $R_2$ whose elements can be written as a sum of units. This observation leads to the following fact.

\begin{prop} \label{prop: simplifying fuzzy rings}
 Let $(R,I)$ be a fuzzy ring with unit group $G=R^\times$ and $\pi:\N[G]\to R$ the natural map that sends a finite formal sum $\sum a_i$ of elements $a_i\in G$ to its sum in $R$. Define $J=\pi^{-1}(I)$. Then $\overline R=(\N[G],J)$ is a fuzzy ring and the identity map $G\to R^\times$ defines an isomorphism of fuzzy rings $\overline R\to R$. 
\end{prop}

\begin{proof}
 We include a brief proof. For more details, we refer the reader to Appendix B of \cite{Baker-Bowler19}.
 
 Since $I\subset R$ is a proper ideal, $R^\times \cap I=\emptyset$ and thus $J\cap G=\emptyset$, i.e.\ $J$ is proper. It follows from the definition that $J$ is an ideal of $\N[G]$. Properties \ref{FR3} and \ref{FR4} of a fuzzy ring are satisfied since $\N[G]$ is distributive. Properties \ref{FR1} and \ref{FR2} follow immediately from the definition of $J$ as $\pi^{-1}(I)$ and the validity of \ref{FR1} and \ref{FR2} in $R$.

 It is also clear from the definition of $\overline R$ that the identity maps $\id:G\to R^\times$ and $\id:R^\times\to G$ are mutually inverse isomorphisms of fuzzy rings.
\end{proof}

\begin{cor} \label{cor: full subcategory of fuzzy rings}
 The category $\FuzzRings$ is equivalent to its full subcategory consisting of fuzzy rings of the form $(\N[G],I)$ for some abelian group $G$. Given an abelian group $G$ and a proper ideal $I\subset\N[G]$, the pair $(\N[G],I)$ is a fuzzy ring if and only if it satisfies \ref{FR1} and \ref{FR2}. A morphism between two fuzzy rings of the form $(\N[G_1],I_1)$ and $(\N[G_2],I_2)$ is the same as a semiring homomorphism $f:\N[G_1]\to\N[G_2]$ with $f(G_1)\subset G_2$ and $f(I_1)\subset I_2$.
\end{cor}

\begin{proof}
 The first claim follows immediately from Proposition \ref{prop: simplifying fuzzy rings}. The second claim follows since $\N[G]$ is distributive and thus satisfies \ref{FR3} and \ref{FR4} automatically. To conclude, a group homomorphism $f:G_1\to G_2$ extends uniquely to a semiring homomorphism $\tilde f:\N[G_1]\to \N[G_2]$. The homomorphism $f$ is a morphism of fuzzy rings if and only if $\sum a_i\in I_1$ implies $\sum f(a_i)\in I_2$, i.e.\ if $\tilde f(I_1)\subset I_2$. Thus the last claim.
\end{proof}

\begin{ex}\label{ex: fuzzy rings}
 As we will see in Theorem \ref{thm: hyperfuzz to bluetracts}, partial fields and hyperfields can be realized as fuzzy rings in a natural way. 
 They represent somewhat opposite ends of a spectrum: while the sum of any two elements of a hyperfield needs to contain at least $1$ element, the sum of two elements of a partial field is equal to at most $1$ element. 
 
 The fuzzy ring associated with the regular partial field $\U_0$ is initial in $\FuzzRings$ and it is not associated with any hyperfield. The fuzzy ring associated with the Krasner hyperfield $\K$ is terminal in $\FuzzRings$ and it is not associated with any partial field.
 
 We describe some examples of fuzzy rings that are neither partial fields nor hyperfields. Let $n$ be an even integer and $\mu_n$ a cyclic group with $n$ elements, generated by an element $\zeta_n$. Let $I$ be the ideal of $\N[\mu_n]$ that is generated by the elements
 \[
  \sum_{i=1}^{n/d} \ \zeta_n^{id}
 \]
 where $d$ ranges through the divisors of $n$ smaller than $n$. Then $F=(\N[\mu_n],I)$ is a fuzzy ring, which does not come from a hyperfield since $\zeta_n+\zeta_n$ is not defined. If $n$ is divisible by a prime larger than $3$, then $F$ does not come from a partial field since $I$ is not generated by only $3$-term sums.
 
 Another class of examples is the following. Let $G$ be a group with neutral element $1$ and $I$ the ideal of $\N[G]$ that is generated by $1+1$ and $1+1+1$. Then $F=(\N[G],I)$ is a fuzzy ring, which does not come from a partial field since it would require that $0=1+1+1=1+0=1$. If $G$ contains an element $a\neq 1$, then $F$ does not come from a hyperfield since $1\hyperplus a$ would have to be empty.

 Both of the above examples can be realized as partial hyperfields, as considered in \cite{Baker-Bowler19}, which are (roughly speaking) hyperfields that are allowed to disobey the nonemptiness axiom \ref{HG1}. Though partial hyperfields and fuzzy rings are closely related, there are technical differences between them stemming from the different roles of the associative law in the two settings. This is discussed in detail in \cite{Giansiracusa-Jun-Lorscheid16}; in particular, cf.\ Example 4.2.
\end{ex}


\subsection{Tracts}
\label{subsection: tracts}

In the recent joint work \cite{Baker-Bowler19}, the first author and Nathan Bowler distill from the aforementioned theories of partial fields, hyperfields, and fuzzy rings the notion of a {\it tract}, which seems to be both a natural setting for matroid theory and a relatively simple (yet quite general) algebraic structure.

A \emph{tract} $F=(F^\times,N_F)$ is an abelian group $F^\times$ together with a subset $N_F$ of the group semiring $\N[F^\times]$, called the \emph{nullset of $F$}, satisfying the following properties:

\begin{enumerate}[label={(T\arabic*)}]
 \item \label{T1} $0\in N_F$ and $1\notin N_F$;
 \item \label{T2} $aN_F=N_F$ for every $a\in F^\times$;
 \item \label{T3} there is a unique element $\epsilon\in F^\times$ such that $1+\epsilon\in N_F$.
\end{enumerate}
We sometimes write $F$ for the set $F^\times\cup\{0\}$, and we think of $N_F$ as the linear combinations of elements of $F$ which ``sum to zero''.

Let $F_i=(F_i^\times,N_{F_i})$ be tracts for $i=1,2$. A \emph{morphism $f:F_1\to F_2$ of tracts} is a group homomorphism $F_1^\times\to F_2^\times$ such that the induced homomorphism $\N[F_1^\times]\to \N[F_2^\times]$ maps $N_{F_1}$ to $N_{F_2}$. Equivalently, it is a map $f:F_1\to F_2$ such that $f(0)=0$, $f(1)=1$, $f(ab)=f(a)f(b)$ for all $a,b\in F_1$ and $\sum f(a_i)\in N_{F_2}$ for all $\sum a_i\in N_{F_1}$. We denote the category of tracts by $\Tracts$.

We recall some facts about tracts from \cite{Baker-Bowler19}.

\begin{lemma}
 Let $F$ be a tract and $\epsilon\in F^\times$ the unique element with $1+\epsilon\in N_F$. Then we have
 \begin{enumerate}
  \item $F\cap N_F=\{0\}$;
  \item $\epsilon^2=1$;
  \item $a+\epsilon a\in N_F$ for every $a\in F^\times$;
  \item if $a+b\in N_F$ with $a,b\in F^\times$, then $b=\epsilon a$.
 \end{enumerate}
\end{lemma}

\begin{ex}\label{ex: tracts}
 The tract $(\{1\},\N-\{1\})$ associated to the Krasner hyperfield $\K$ is terminal in $\Tracts$. More generally, every fuzzy ring $(\N[G],I)$ defines a tract $(G,I)$ as we will see in Theorem \ref{thm: hyperfuzz to bluetracts}. In fact, there are only two ways in which a tract can fail to come from a fuzzy ring.
 
 The first deviation of tracts from fuzzy rings lies in the fact that the nullset $N_F$ of a tract does not have to be an ideal of the semiring $\N[F^\times]$, but merely an $F^\times$-invariant subset. For example, the nullset $N_F=\{0,1+\epsilon\}$ of the initial object $F=(\{1,\epsilon\}, N_F)$ of $\Tracts$ is not an ideal of $\N[\{1,\epsilon\}]$.
 
 Another example is the following. If $F=(G,I)$ is a tract, for instance coming from a fuzzy ring $(\N[G],I)$, then we can consider the $3$-term truncation
 \[
  I_3 \ = \ \big\{ \, a+b+c \, \big| \, a+b+c\in I \text{ and }a,b,c\in G\cup\{0\} \, \big\}
 \]
 of $I$, which is not an ideal of $\N[G]$, but merely a $G$-invariant set. The resulting tract $(G,I_3)$ can be a useful gadget to compare strong $F$-matroids, which are defined by all Pl\"ucker relations, with weak $F$-matroids, which are defined in terms of just the $3$-term Pl\"ucker relations.
  
 The second deviation of tracts from fuzzy rings stems from the omission of axiom \ref{FR2} of a fuzzy ring. The following is an example of a tract whose nullset is an ideal, but which does not come from a fuzzy ring due to the failure of axiom \ref{FR2}.
 
 Let $F=(\F_5^\times,N_F)$ be the subtract of $\F_5=\{0,1,2,3,4\}$ that consists of the same underlying monoid, but whose nullset is the ideal $N_F$ of $\N[F^\times]$ generated by the elements $1+4$ and $1+1+3$. Then $F=(F^\times,N_F)$ is a tract with $\epsilon=4$. If $(\N[F^\times],N_F)$ was a fuzzy ring, then \ref{FR2} applied to $a=c=1+1$ and $b=d=3$ would imply that
 \[
  (1+1)(1+1)+4\cdot 3\cdot 3 \ = \ 1+1+1+1+1
 \]
 is an element of $N_F$, which is not the case. (Note that $4\cdot 3\cdot 3=36\equiv 1$ modulo $5$.)
\end{ex}


\subsection{Ordered blueprints}
\label{subsection: ordered blueprints}

As a general reference for ordered blueprints, we refer the reader to Chapter 5 of the second author's lecture notes \cite{Lorscheid18}.

An \emph{ordered semiring} is a semiring $R$ together with a partial order $\leq$ that is compatible with multiplication and addition, i.e.\ $x\leq y$ and $z\leq t$ imply $x+z\leq y+t$ and $xz\leq yt$ for all $x,y,z,t\in R$. 

An \emph{ordered blueprint} is a triple $B=(B^\bullet,B^+,\leq)$ where $(B^+,\leq)$ is an ordered semiring and $B^\bullet$ is a multiplicative subset of $B^+$ that generates $B^+$ as a semiring and contains $0$ and $1$. A \emph{morphism of ordered blueprints $(B^\bullet_1,B^+_1,\leq_1)$ and $(B^\bullet_2,B^+_2,\leq_2)$} is an order preserving morphism $f:B^+_1\to B^+_2$ of semirings with $f(B^\bullet_1)\subset B^\bullet_2$. We denote the category of ordered blueprints by $\OBlpr$.

Let $B=(B^\bullet,B^+,\leq)$ be an ordered blueprint. We call $B^\bullet$ the \emph{underlying monoid of $B$} and think of it as the underlying set, i.e.\ we write $a\in B$ for $a\in B^\bullet$. Note that a morphism $f:B_1\to B_2$ of ordered blueprints is determined by its restriction $f^\bullet:B^\bullet_1\to B^\bullet_2$ to the underlying monoids. We call $B^+$ be \emph{associated semiring of $B$}. We call $\leq$ the \emph{partial order of $B$}. 

Typically, we denote the elements of $B^\bullet$ by $a$, $b$, $c$ and $d$, and the elements of $B^+$ by either $x$, $y$, $z$ and $t$ or by $\sum a_i$, $\sum b_j$, $\sum c_k$ and $\sum d_l$ where assume that the $a_i$, $b_j$, $c_k$ and $d_l$ are in $B^\bullet$. Note that every element of $B^+$ is indeed a sum of elements in $B^\bullet$. 

The \emph{unit group $B^\times$} of an ordered blueprint is the commutative group of all multiplicatively invertible elements of $B$. An \emph{ordered blue field} is an ordered blueprint $B$ with $B^\bullet=B^\times\cup\{0\}$.

If we have $\sum a_i\leq\sum b_j$ and $\sum b_j\leq \sum a_i$, then we write $\sum a_i\=\sum b_j$. If this is the case for all relations in $B$, then we say that $B$ is an \emph{algebraic blueprint}.  (Algebraic blueprints are simply called ``blueprints'' in \cite{Lorscheid18b}.)

\begin{rem}
 Note that the partial order of $B$ corresponds to what is called the \emph{subaddition of $B$} in \cite{Lorscheid15} in the following way: the subaddition of $B$ is the preorder on the monoid semiring $\N[B^\bullet]$ that is the pullback of $\leq$ along the quotient map $\N[B^\bullet]\to B^+$.
\end{rem}

The following constructions provide a rich class of examples of ordered blueprints.

\subsubsection{Semirings and monoids}\label{subsubsection: semirings and monoids}
Every semiring $R$ defines the ordered blueprint $(R,R,=)$ where $=$ denotes the trivial partial order given by equality of elements. A \emph{monoid with zero}, which is a (multiplicatively written) commutative semigroup $A$ with neutral element $1$ and absorbing element $0$ (i.e.\ $0\cdot a=0$ for all $a\in A$), defines the ordered blueprint $(A,\N[A],=)$. In what follows, we identify semirings and monoids with their associated ordered blueprints and say that an ordered blueprint $B$ is a semiring or a monoid if it is isomorphic to an ordered blueprint coming from a semiring or monoid with zero, respectively.

\begin{ex}
 The monoid $\Fun=\{0,1\}$, sometimes referred to as the \emph{field with one element}, can be identified with the ordered blueprint \mbox{$(\{0,1\},\N,=)$}, and the Boolean semifield $\B=\{0,1\}$ (where $1+1=1$) can be identified with $(\{0,1\},\B,=)$.
\end{ex}

\subsubsection{Free algebras}
Given an ordered blueprint $B$ and a set $X=\{T_i\}_{i\in I}$, we can form the \emph{free ordered blueprint $B[X]=B[T_i]_{i\in I}$ over $B$}, which is defined as follows. The associated semiring $B[X]^+$ is the usual polynomial semiring over $B^+$ in the variables $T_i$. The underlying monoid $B[X]^\bullet$ is the subset of monomials $a\prod T_i^{n_i}$ with coefficients $a\in B$. The partial order $\leq$ of $B[X]^+$ is the smallest partial order that contains the partial order of $B^+$ and that is closed under multiplication and addition.

\begin{ex}
 The free algebra $\C[T]=(\C[T]^\bullet,\C[T]^+,=)$ is the ordered blueprint where $\C[T]^+$ is the usual polynomial ring and $\C[T]^\bullet$ is the subset of all terms of the form $aT^n$ with $a\in \C$ and $n\in\N$. The free algebra $\Fun[T]$ is the ordered blueprint where $\Fun[T]^+=\N[T]^+$ is the usual polynomial semiring over $\N$, endowed with the trivial order $=$, and where $\Fun[T]^\bullet$ consists of all monomials and $0$.
\end{ex}

\begin{rem}
 The reader might be alarmed by the fact that the notation $B[T]$ for the free ordered blueprint conflicts with the notation for corresponding notation for semirings: given a semiring $R$, the blueprint $(R[T],R[T],=)$ associated with the polynomial semiring $R[T]$ differs from the free blueprint $R[T]=(R[T]^\bullet,R[T]^+,\leq_{R[T]})$. More precisely, if $B=R$ and $B[T]$ is the free ordered blueprint, then the polynomial semiring $R[T]$ equals $B[T]^+$. However, it will be clear from the context to which construction we refer when we write $B[T]$ or $R[T]$. Sometimes we indicate that we mean the latter construction by writing $B[T]^+$.
\end{rem}

\subsubsection{Quotients by relations}
Given an ordered blueprint $B=(B^\bullet,B^+,\leq_B)$ and a set of relations $S=\{x_i\leq y_i\}_{i\in I}$, which we do not assume to be contained in $\leq_B$, we define the ordered blueprint $C=\bpgenquot BS$ as the following triple $(C^\bullet,C^+,\leq_C)$. Let $\leq'$ be the smallest preorder on $B^+$ that contains $\leq_B$ and $S$ and that is closed under multiplication and addition. We write $x\=y$ if $x\leq y$ and $y\leq x$. Then $\=$ is an equivalence relation on $B^+$, and we define $C^+$ as $B^+/\=$, which inherits naturally the structure of an ordered blueprint since $\leq'$ is closed under multiplication and addition. The preorder $\leq'$ induces a partial order $\leq_C$ on $C^+$, which turns $C^+$ into an ordered semiring. The multiplicative subset $C^\bullet$ is defined as the image of $B^\bullet$ under the quotient map $B^+\to C^+$. 

Given an ordered blueprint $B=(B^\bullet,B^+,\leq)$, we say that the partial order $\leq$ of $B$ is \emph{generated} by a set $S=\{x_i\leq y_i\}_{i\in I}$ of relations on $B^+$ if $\leq$ is the smallest preorder on $B^+$ that contains $S$ and that is closed under multiplication and addition.

\begin{ex}\label{ex: quotients of ordered blueprints}
 With this construction, we can define ordered blueprints like $\Funsq=\bpgenquot{\{0,1,-1\}}{0\= 1+(-1)}$ (note that in this case $\F_{1^2}^+=\Z$) and $\bpgenquot{\B[T_1,\dotsc,T_6]}{0\leq T_1T_6+T_2T_5+T_3T_4}$.
\end{ex}

\subsubsection{Subblueprints}\label{subsubsection: subblueprints}
Let $B=(B^\bullet,B^+,\leq_B)$ be an ordered blueprint. An \emph{(ordered) subblueprint} of $B$ is an ordered blueprint $C$ such that $C^\bullet$ is a submonoid of $B^\bullet$, such that the ambient semiring $C^+$ is the subsemiring of $B^+$ that is generated by $C^\bullet$ and such that the partial order of $C$ is the restriction of the partial order of $B^+$ to $C^+$. Note that by definition, every submonoid $C^\bullet$ of $B^\bullet$ determines a unique subblueprint $C$ of $B$.

\subsubsection{Tensor products}
Given two morphisms $B\to C$ and $B\to D$ of ordered blueprints, there exists a push-out of the diagram $C\leftarrow B\to D$, which is represented by the \emph{tensor product $C\otimes_BD$ of $C$ and $D$ over $B$}. The tensor product is constructed as follows.

The semiring $(C\otimes_BD)^+$ is the usual tensor product $C^+\otimes_{B^+}D^+$ of commutative semirings, whose elements are classes of finite sums $\sum c_i\otimes d_i$ of \emph{pure tensors $c_i\otimes d_i$} with respect to the usual identifications. The monoid $(C\otimes_BD)^\bullet$ is defined as the subset of all pure tensors of $(C\otimes_BD)^+$. The partial order on $(C\otimes_BD)^+$ is defined as the smallest partial order that is closed under addition and multiplication and that contains all relations of the forms
\[ \textstyle
 \sum a_i \otimes 1 \leq \sum c_k \otimes 1 \qquad \text{and} \qquad \sum 1\otimes b_j \leq \sum 1\otimes d_l
\]
for which $\sum a_i\leq\sum c_k$ in $C$ and $\sum b_j\leq \sum d_l$ in $D$, respectively.

\begin{ex}\label{ex: tensor products of ordered blueprints}
 The tensor product can be used to extend constants. For instance, we have $\Fun[T_1,\dotsc,T_n]\otimes_\Fun B=B[T_1,\dotsc,T_n]$ for every ordered blueprint $B$. Given a semiring $R$, the tensor product $R\otimes_\Fun\Funsq=R\otimes_\N\Z$ is the ring of differences associated with $R$. 
\end{ex}


\subsection{\texorpdfstring{$\Funpm$}{F1+-}-algebras and idylls}
\label{subsection: Funpm-algebras}

As we will see in Theorem \ref{thm: hyperfuzz to bluetracts}, the ordered blueprint that corresponds to the regular partial field $\U_0$ under the embedding of $\PartFields$ in $\OBlpr$ is $\Funpm= \bpgenquot{\{0,1,\epsilon\}}{0\leq1+\epsilon}$. 

An \emph{$\Funpm$-algebra} is an ordered blueprint $B$ together with a morphism $\alpha_B:\Funpm\to B$, which we call the \emph{structure map of $B$}. We denote the image $\alpha_B(\epsilon)$ of $\epsilon\in\Funpm$ by $\epsilon_B$, or simply by $\epsilon$ if there is no danger of confusion. This element plays the role of $-1$, and for a given for $a\in B$, we call $\epsilon a$ the \emph{weak inverse of $a$}. Note that we do not use the symbol $-1$ in this context since $1$ does not have an additive inverse in the ambient semiring $B^+$, i.e.\ $B^+$ is in general not a ring. ($\Funpm$ itself provides an example: $(\Funpm)^+=\N[1,\epsilon]$.)

A \emph{morphism of $\Funpm$-algebras} from an $\Funpm$-algebra $B$ to an $\Funpm$-algebra $C$ is a morphism $f:B\to C$ of ordered blueprints that commutes with the respective structure maps $\alpha_B$ and $\alpha_C$. Equivalently, it is a morphism of ordered blueprints $f:B\to C$ such that $f(\epsilon_B)=\epsilon_C$. This defines the category $\OBlpr_\Funpm$ of $\Funpm$-algebras. 


The weak inverse $\epsilon a$ of an element $a$ of an $\Funpm$-algebra is in general not uniquely determined by the relation $0\leq a+\epsilon a$. For instance, we have $0\leq 1\otimes 1+\epsilon\otimes 1$ and $0\leq 1\otimes 1+1\otimes\epsilon$ in $\Funpm\otimes_\Fun\Funpm$. For certain purposes, we want to impose the condition that $b=\epsilon a$ is uniquely determined by the realtion $0\leq a+b$. We call an $\Funpm$-algebra $\Funpm\to B$ with this property an \emph{$\Funpm$-algebra with unique weak inverses}\footnote{In an earlier version of this paper, $\Funpm$-algebra with unique weak inverses were called \emph{pasteurized ordered blueprints}.}. This defines a full subcategory $\OBlpr^\pm$ of $\OBlpr_\Funpm$. The following definition captures the class of ordered blueprints that come from tracts.

\begin{df}
 An ordered blueprint $B$ is \emph{purely positive} if its partial order is generated by relations of the form $0\leq \sum a_i$. An \emph{idyll} is a purely positive $\Funpm$-algebra $B$ with unique weak inverses that is an ordered blue field and for which the natural map $\N[B^\times]\to B^+$ is a bijection.
\end{df}

Since the element $\epsilon$ of an $\Funpm$-algebra $B$ with unique weak inverses is uniquely determined by $0\leq 1+\epsilon$, the structure map $\Funpm\to B$ is the unique morphism from $\Funpm$ to $B$. In other words, forgetting the structure map defines a fully faithful embedding $\OBlpr^\pm\to \OBlpr$. This embedding has a left adjoint and left inverse $(-)^\pm:\OBlpr\to\OBlpr^\pm$, which can be described as follows. Let $B=\bpquot A\cR$ be an ordered blueprint. Then we define the associated $\Funpm$-algebra with unique weak inverses as 
\[
 B^\pm \quad = \quad B\otimes_\Fun \Funpm \ /\!\!/ \ \gen{a\= a'\, |\, \text{there is a }b\in B\otimes_\Fun \Funpm\text{ such that }0\leq a+b\text{ and }0\leq a'+b}.
\]
A morphism $f:B\to C$ of ordered blueprints induces the morphism $f^\pm:B^\pm\to C^\pm$, defined by $f^\pm(a\otimes b)=f(a)\otimes b$.

\begin{ex}\label{ex: Funpm-algebras}
 We have already introduced $\Funpm=\bpgenquot{\{0,1,\epsilon\}}{0\leq1+\epsilon}$, which is an initial object of both $\OBlpr_\Funpm$ and $\OBlpr^\pm$. It corresponds to the regular partial field $\U_0$ under the embedding of $\PartFields$ in $\OBlpr^\pm$ that is explained in Theorem \ref{thm: hyperfuzz to bluetracts}.
 
 If $B$ is an algebraic blueprint, then $B$ is an $\Funpm$-algebra with unique weak inverses if and only if $0\=1+\epsilon$ for some $\epsilon\in B$. Indeed, we have $0\= a+\epsilon a$ for every $a\in B$ and $0\= a+b$ implies $b\=b+a+\epsilon a\=\epsilon a$. If $B$ contains the relation $0\=1+\epsilon$, then we write $-1$ for $\epsilon$ and say that $B$ is \emph{with inverses} or \emph{with $-1$}. Note that $B$ is with $-1$ if and only if it contains $\Funsq$ as a subblueprint. 
\end{ex}

\begin{rem}
 We would like to point out that the idea to consider ordered semirings $(B^+,\leq)$ together with a multiplicative subset $B^\bullet$ and an involution $\iota:a\mapsto\epsilon a$ has been considered previously in the papers \cite{Akian-Gaubert-Guterman14} by Akian, Gaubert and Guterman and \cite{Rowen16} by Rowen. Although the algebraic datum is the same as for $\Funpm$-algebras, the axioms vary in these publications, so that the resulting notions have a different level of generality. To avoid a digression into technicalities, we refrain from a detailed discussion of these differences.
\end{rem}


\subsection{From partial fields and hyperfields to tracts and ordered blueprints}
\label{subsection: from hyperrings and fuzzy rings to tracts and ordered blueprints}

Partial fields and hyperfields turn out to be particular examples of tracts and ordered blueprints, passing through the intermediate categories of fuzzy rings and hyperrings.  The latter objects are connected by an adjunction, as explained in the following:

\begin{thm}\label{thm: hyperfuzz to bluetracts}
 There is a diagram
 \[
  \beginpgfgraphicnamed{tikz/fig1}
  \begin{tikzcd}[column sep=3cm]
   \PartFields \ar{dr}[above=3pt]{(-)^\fuzz} &                                                                & \Tracts \ar[shift left=1.0ex]{dd}{(-)^\oblpr}  \\
                                             & \FuzzRings \ar{ur}[above=0pt]{(-)^\tract}\ar{dr}[above=3pt]{(-)^\oblpr} & \\
   \HypFields \ar{ur}[above=0pt]{(-)^\fuzz}\ar{r}[above=0pt]{\iota} & \HypRings  \ar{r}[above=0pt]{(-)^\oblpr} & \OBlpr^\pm \ar[shift left=1.0ex]{uu}{(-)^\tract}[swap]{\vdash}
  \end{tikzcd}
  \endpgfgraphicnamed
 \]
 of functors with the following properties:
 \begin{enumerate}
  \item\label{part1} the functors with domains $\PartFields$, $\HypFields$, $\HypRings$ and $\FuzzRings$ are fully faithful;
  \item\label{part2} $(-)^\tract:\OBlpr^\pm\to\Tracts$ is right adjoint to $(-)^\oblpr:\Tracts\to\OBlpr^\pm$ and this adjunction restricts to an equivalence between the essential images of the functors, which are the full subcategory of idylls and the full subcategory of tracts whose nullsets are ideals, respectively;
  \item\label{part3} the square and both triangles in the diagram commute.
 \end{enumerate}
\end{thm}

The functor $\iota:\HypFields\to\HypRings$ is the inclusion as a full subcategory and hence is fully faithful. In the rest of this section, we will construct the other functors of this diagram and prove the various assertions of Theorem \ref{thm: hyperfuzz to bluetracts}.


\subsubsection{From partial fields to fuzzy rings}\label{subsubsection: From partial fields to fuzzy rings}
Let $P$ be a partial field with unit group $P^\times$ and projection $\pi_P:\Z[P^\times]\to R_P$. We define the associated fuzzy ring $P^\fuzz$ as the pair $(\N[P^\times],I)$, where
\[\textstyle
 I \ = \ \big\{ \ \sum a_i\in\N[P^\times] \ \big| \ \pi_P(\sum a_i)=0 \ \big\}
\]
is the kernel of the composition $\N[P^\times]\stackrel\iota\longrightarrow\Z[P^\times]\stackrel{\pi_P}\longrightarrow R_P$ of semiring morphisms.

Given a morphism $f:P_1\to P_2$ of partial fields, we define the associated morphism of fuzzy rings $f^\fuzz:P_1^\fuzz\to P_2^\fuzz$ as the restriction of $f$ to $P_1^\times\to P_2^\times$. 

\begin{claim}
 The above description defines a fully faithful functor $(-)^\fuzz:\PartFields\to\FuzzRings$.
\end{claim}

\begin{proof}
 We begin by showing that $P^\fuzz=(\N[P^\times],I)$ is indeed a fuzzy ring. Since $0\neq 1$ in $R_P$, $I$ is a proper ideal. According to Corollary \ref{cor: full subcategory of fuzzy rings}, we have to verify only \ref{FR1} and \ref{FR2} in order to show that $(\N[P^\times],I)$ is a fuzzy ring. Axiom \ref{FR1} follows immediately from axiom \ref{PF2} of a partial field. 
 
 We are left with \ref{FR2}. For simplicity, we write $\pi=\pi_P$. Let $a,b,c,d\in \N[P^\times]$ and $a+b,c+d\in I$. Then $\pi(b)=\pi(-a)$ and $\pi(d)=\pi(-c)$. Since $\pi(\epsilon)=\pi(-1)=-1$, we have
 \[
  \pi(ac+\epsilon bd) \ = \ \pi(a)\pi(c)+\pi(\epsilon)\pi(b)\pi(d) \ = \ \pi(a)\pi(c)-\pi(a)\pi(c) \ = \ 0,
 \]
 i.e.\ $ac+\epsilon bd\in I$, as desired. This shows that $P^\fuzz$ is indeed a fuzzy ring.

 We continue by showing that $f^\fuzz$ is a morphism of fuzzy rings. Note that $f^\fuzz:P_1^\times\to P_2^\times$ is well-defined as a map since for $a\in P_1^\times$, we have $f(a)f(a^{-1})=f(aa^{-1})=f(1)=1$ and thus $f(a)\neq0$. Since $f(ab)=f(a)f(b)$, we conclude that $f^\fuzz$ is a group homomorphism.

 As the next step, we verify the additive axiom for $f^\fuzz$. For $i=1,2$, we denote by $I_i$ the kernel of $\N[P_i^\times]\stackrel{\iota_i}\longrightarrow\Z[P_i^\times]\stackrel{\pi_i}\longrightarrow R_{P_i}$, where we write $\pi_i=\pi_{P_i}$. Then $I_i=\ker\pi_i\cap\N[P_i^\times]$. By \ref{PF3}, $\ker\pi_i$ is generated by the elements $a+b+c$ for which $\pi_i(a+b+c)=0$, where $a,b,c\in P_i$. As a consequence, $I_i$ is generated by the same elements as an ideal of $\N[P_i^\times]$.

 Thus we have to verify that $f(a)+f(b)+f(c)\in I_2$ if $a+b+c\in I_1$ for $a,b,c\in P_1^\times$. The condition $a+b+c\in I_1$ means that $\pi(a+b+c)=0$ or, equivalently, $\pi_1(a+b)=-\pi(c)=\pi(\epsilon c)$. Thus $f(a)+f(b)=f(\epsilon c)$, i.e.\ $\pi_2(f(a)+f(b))=\pi_2(\epsilon f(c))$. We conclude that $f(a)+f(b)+f(c)\in I_2$, as desired. 

 This shows that $f^\fuzz:P_1^\fuzz\to P_2^\fuzz$ is a morphism of fuzzy rings. Since the restriction of maps is functorial, this completes the proof that $(-)^\fuzz:\PartFields\to\FuzzRings$ is a functor.

 Finally, we show that $(-)^\fuzz$ is fully faithful. Let $P_1$ and $P_2$ be partial fields and $g:P_1^\fuzz\to P_2^\fuzz$ a morphism between the associated fuzzy rings. If $g=f^\fuzz$ for a morphism $f:P_1\to P_2$ of partial fields, then $f$ is determined by the rules $f(0)=0$ and $f(a)=g(a)$ for $a\in P_1^\times$. This shows that $(-)^\fuzz$ is faithful.

 We verify that $f$ as defined above is indeed a morphism of partial fields. Evidently, $f(0)=0$, $f(1)=1$ and $f(ab)=f(a)f(b)$ for $a,b\in P_1$. Given $a+b=c$ in $P_1$, we have to show that $f(a)+f(b)=f(c)$. If $a+b=c$, then $a+b+\epsilon c\in I_1$. After omitting the zero terms in this sum, we can apply $f$ and see, after placing back the zero terms at the omitted positions, that $f(a)+f(b)+\epsilon f(c) \in I_2$. Thus $f(a)+f(b)=f(c)$ as required. This shows that $(-)^\fuzz$ is full and finishes the proof of the claim.
\end{proof}


\subsubsection{From hyperfields to fuzzy rings}\label{subsubsection: From hyperfields to fuzzy rings}
The inclusion $\HypFields\to\FuzzRings$ is the theme of \cite{Giansiracusa-Jun-Lorscheid16}. Since we agreed to work only with fuzzy rings of the shape $(\N[G],I)$, we have to adapt the construction from Theorem A in \cite{Giansiracusa-Jun-Lorscheid16}. The reader can easily convince himself that our variant yields a fuzzy ring isomorphic to the original one via the isomorphism from Proposition \ref{prop: simplifying fuzzy rings}.

Given a hyperfield $K$, the associated fuzzy ring $K^\fuzz=(\N[G],I)$ is defined by $G=K^\times$ and 
\[\textstyle
 I \ = \ \big\{\ \sum a_i\in \N[G] \ \big| \ a_i\in G\text{ and }0\in \hypersum a_i \ \big\}.
\]
Given a morphism $f:K_1\to K_2$ of hyperfields, we define the associated morphism $f^\fuzz:K_1^\fuzz\to K_2^\fuzz$ of fuzzy rings as the restriction of $f$ to $K_1^\times\to K_2^\times$. 

\begin{claim}
 This defines a fully faithful embedding $(-)^\fuzz:\HypFields\to\FuzzRings$.
\end{claim}

\begin{proof}
 We begin with the verification that $(-)^\fuzz$ is well-defined on objects, i.e.\ that $K^\fuzz$ is a fuzzy ring for every hyperfield $K$. It is clear that $G=K^\times$ is an abelian group. We continue with showing that $I$ is an ideal of $\N[G]$. Clearly $0\in I$. If $\sum a_i,\sum b_j\in I$, then $0\in\hypersum a_i$ and $0\in\hypersum b_j$. This implies that $0\in\hypersum a_i\hyperplus\hypersum b_j$ and thus $\sum a_i+\sum b_j\in I$, as desired. Given an element $\sum a_i\in \N[G]$ and $\sum b_j\in I$, we have $0\in\hypersum b_j$. By the distributivity of $K$, this implies that $0\in\hypersum_j a_ib_j$ for every $i$. Summing over all $i$ yields $0\leq \hypersum_{i,j} a_ib_j$ and thus $\sum a_ib_j\in I$, as desired. This shows that $I$ is an ideal.
 
 We proceed with the proof that $(\N[G],I)$ satisfies \ref{FR1} and \ref{FR2}. Axiom \ref{FR1} follows immediately from the existence and uniqueness of additive inverses in $K$. Namely, we have $0\in 1\hyperplus(-1)$ and thus $1+\epsilon\in I$ for $\epsilon=-1$. If $1+a\in I$, then $0\in 1\hyperplus a$, which means that $a=\epsilon$.
 
 In order to verify axiom \ref{FR2}, consider $a+b,c+d\in I$, i.e.\ $0\in a\hyperplus b$ and $0\in c\hyperplus d$. By the uniqueness of additive inverses, this means that $b=\epsilon a$ and $d=\epsilon d$. Thus $bd=\epsilon^2ac=ac$ and $0\in ac\hyperplus\epsilon ac=ac\hyperplus \epsilon bd$. This shows that $ab+\epsilon bd\in I$, as desired, and concludes the proof that $K^\fuzz$ is a fuzzy ring.
 
 We continue with the verification that $(-)^\fuzz$ is well-defined on morphisms, i.e.\ that $f^\fuzz$ is indeed a morphism of fuzzy rings for every hyperfield morphism $f:K_1\to K_2$. It is evident that $f^\fuzz:K_1^\times\to K_2^\times$ is a group homomorphism. Given an element $\sum a_i\in I_1$, i.e.\ $0\in \hypersum a_i$, we have $0\in\hypersum f(a_i)$. Thus $\sum f^\fuzz(a_i)\in I_2$, which shows that $f^\fuzz$ is a morphism of fuzzy rings. This verifies that $(-)^\fuzz$ is indeed a functor.
 
 We conclude with the proof that $(-)^\fuzz$ is fully faithful. Since $f(0)=0$, a morphism $f:K_1\to K_2$ is determined by its restriction $f^\fuzz:K_1^\times\to K_2^\times$. Thus $(-)^\fuzz$ is faithful.
 
 Let $g:K_1^\times\to K_2^\times$ be a morphism of fuzzy rings $K_1^\fuzz\to K_2^\fuzz$. Consider the extension of $g$ to $f:K_1\to K_2$ with $f(0)=0$. Then $f$ is a multiplicative map with $f(1)=1$ and $f(0)=0$. Given two elements $a,b\in K_1$ and $c\in a\hyperplus b$, we have $0\in a\hyperplus b\hyperplus \epsilon c$. Thus $a+b+\epsilon c\in I_1$ and $f(1)+f(b)+\epsilon f(c)\in I_2$. This means that $f(a)\hyperplus f(b)$ contains an additive inverse of $\epsilon f(c)$, which must be $f(c)$ by the uniqueness of additive inverses. Thus $f(c)\in f(a)\hyperplus f(b)$, which verifies that $f$ is a morphism of hyperfields. By construction, we have $f^\fuzz=g$, which shows that $(-)^\fuzz$ is full. This completes the proof of the claim.
\end{proof}


\subsubsection{From fuzzy rings to tracts}\label{subsubsection: From fuzzy rings to tracts}

Let $F=(\N[G],I)$ be a fuzzy ring. We define the associated tract $F^\tract$ as the pair $(G,I)$. Note that axioms \ref{T1} and \ref{T2} of a tract are satisfied since $I$ is a proper ideal of $\N[G]$ and axiom \ref{T3} follows from axiom \ref{FR1} of a fuzzy ring.

Let $F_i=(\N[G_i],I_i)$ be fuzzy rings for $i=1,2$. It is evident that a map $G_1\to G_2$ defines a morphism between the fuzzy rings $f:F_1\to F_2$ if and only if it is a morphism between associated tracts $f^\tract:F_1^\tract\to F_2^\tract$. Therefore the association $f\mapsto f^\tract$ turns $(-)^\tract$ into fully faithful functor from $\FuzzRings$ to $\Tracts$.


\subsubsection{From fuzzy rings to ordered blueprints}\label{subsubsection: From fuzzy rings to ordered blueprints}

Let $F=(\N[G],I)$ be a fuzzy ring and $G_0=G\cup\{0\}$, which is a submonoid of $\N[G]$. We define the associated ordered blueprint as
\[ \textstyle
 F^\oblpr \ = \ \bpgenquot{G_0}{ \, 0\leq \sum a_i \, | \, \sum a_i\in I \, }.
\]
Note that the associated semiring is the group semiring $\N[G]$ itself. 

Let $F_i=(\N[G_i],I_i)$ be fuzzy rings for $i=1,2$. Given a morphism $f:F_1\to F_2$ of fuzzy rings, which is a group homomorphism $f:G_1\to G_2$, we define the associated morphism $f^\oblpr:F_1^\oblpr\to F_2^\oblpr$ of order blueprints as the linear extension $\N[G_1]\to \N[G_2]$ of $f$ to the respective group semirings.

\begin{claim}
 The above description defines a fully faithful functor $(-)^\oblpr:\FuzzRings\to\OBlpr^\pm$.
\end{claim}

\begin{proof}
 Note that axiom \ref{FR1} of a fuzzy ring $F$ implies that $F^\oblpr$ is an $\Funpm$-algebra with unique weak inverses. Thus the image of $(-)^\oblpr$ is indeed contained in $\OBlpr^\pm$. 
 
 A group homomorphism $f:G_1\to G_2$ is a morphism of fuzzy rings $F_1\to F_2$ if and only if for every formal sum $\sum a_i\in I_1$, we have $\sum f(a_i)\in I_2$. Note that $f^\oblpr(G_1\cup\{0\})\subset G_2\cup\{0\}$ holds automatically. This condition on elements of $I_1$ can be reformulated as follows: for every $0\leq \sum a_i$ in $F_1^\oblpr$, we have $0\leq\sum f(a_i)$ in $F_2^\oblpr$. This shows that $f^\oblpr$ is indeed a morphism of ordered blueprints and that $(-)^\oblpr$ is fully faithful.
\end{proof}


\subsubsection{From hyperrings to ordered blueprints}\label{subsubsection: From hyperrings to ordered blueprints}

It is already mentioned in \cite{Lorscheid15} that the category of hyperrings embeds fully faithfully into the category of ordered blueprints. In this text, we will consider a modification of this embedding which seems to be more natural with respect to the relation of ordered blueprints with fuzzy rings and tracts. 

Given a hyperring $R$, we define the associated ordered blueprint as 
\[ \textstyle
 R^\oblpr \ = \ \bpgenquot{R^\bullet}{\, 0\leq \sum a_i\, | \, 0\in \hypersum a_i\text{ and }a_i\in R\, },  
\]
where $R^\bullet$ is the multiplicative monoid of $R$. Note that the associated semiring $R^{\oblpr,+}$ is the free monoid semiring $\N[R^\bullet]^+$ modulo the identification of $0\in R$ with the zero in $\N[R^\bullet]^+$. 

Let $f:R_1\to R_2$ be a map of hyperrings with $f(0)=0$, $f(1)=1$ and $f(ab)=f(a)f(b)$ for all $a,b\in R_1$. Let $f^\oblpr:R_1^{\oblpr,+}\to R_2^{\oblpr,+}$ be the linear extension of $f$ to the respective semirings. 

\begin{claim}
  The above description defines a fully faithful functor $(-)^\oblpr:\FuzzRings\to\OBlpr^\pm$.
\end{claim}

\begin{proof}
 Let $R$ be a hyperring. By axiom \ref{HG4}, the associated ordered blueprint $R^\oblpr$ is an $\Funpm$-algebra with unique weak inverses.
 
 Let $f:R_1\to R_2$ be a multiplicative map between hyperrings preserving $0$ and $1$. Its linear extension $f^\oblpr:R_1^{\oblpr,+}\to R_2^{\oblpr,+}$ to the respective semirings is well-defined since $f(0)=0$. It is clear that $f^\oblpr(R_1^\bullet)\subset R_2^\bullet$. Then $f$ is a morphism of hyperrings if and only if for all relations $0\in \sum a_i$ in $R_1$, we have $0\in \sum f(a_i)$ in $R_2$. By the definition of the associated ordered blueprint, this condition is equivalent to $f^\oblpr$ being order preserving. To summarize, $f$ is a morphism of hyperrings if and only if $f^\oblpr$ is a morphism of ordered blueprints. Finally note that $f$ is uniquely determined by $f^\oblpr$ as the restriction of $f^\oblpr$ to $R_1\to R_2$.

 This shows that $(-)^\oblpr:\HypRings\to\OBlpr^\pm$ is a fully faithful functor.
\end{proof}

This concludes the proof of part \eqref{part1} of Theorem \ref{thm: hyperfuzz to bluetracts}.
 

\subsubsection{From tracts to ordered blueprints}\label{subsubsection: From tracts to ordered blueprints}

Let $F=(F^\times,N_F)$ be a tract. We define the associated ordered blueprint $F^\oblpr$ as the triple $B=(B^\bullet,B^+,\leq)$ where $B^\bullet=F$, seen as a multiplicative subset of $B^+=\N[F^\times]$, and whose partial order $\leq$ is generated by the relations $0\leq \sum a_i$ for which $\sum a_i\in N_F$ with $a_i\in F^\bullet$. By axiom \ref{T3} of a tract, $F^\oblpr$ is an $\Funpm$-algebra with unique inverses.

Given a morphism $f:F_1\to F_2$ of tracts, we define the morphism $f^\oblpr:F_1^\oblpr\to F_2^\oblpr$ between the associated ordered blueprints as the linear extension of $f$ to $f^+:\N[F_1^\times]\to\N[F_2^\times]$. Note that $f^\oblpr(F_1)=f(F_1)\subset F_2$ and that $f^+$ is order preserving since for a generator $0\leq \sum a_i$ of the partial order of $F_1^\oblpr$, i.e.\ $\sum a_i\in N_{F_1}$, we have $\sum f(a_i)\in N_{F_2}$ and thus $0\leq f^+(\sum a_i)$ in $F_2^\oblpr$.

This defines the functor $(-)^\oblpr:\Tracts\to\OBlpr^\pm$.


\subsubsection{From ordered blueprints to tracts}\label{subsubsection: From ordered blueprints to tracts}

Let $B=\bpquot A\cR$ be an ordered blueprint. In case that $B$ is not trivial, i.e.\ $0\neq1$, we define the associated tract $B^\tract$ as the pair $F=(B^\times,N_B)$ where $N_B=\{\sum a_i \; |\; a_i\in B^\times \text{ and }0\leq \sum a_i\}$. Note that the underlying set of $F$ is $B^\times\cup\{0\}$. If $B=\{0=1\}$, then we define $B^\tract$ as the terminal tract $\K=(\{1\},\N-\{1\})$, which we denote by the same symbol $\K$ as the Krasner hyperfield.

Given a morphism $f:B_1\to B_2$ of ordered blueprints, we define the morphism $f^\tract$ as the restriction of $f$ to $B_1^\times\to B_2^\times$ if $B_1$ and $B_2$ are nontrivial, and as the unique morphism $B_1^\tract\to\K$ if $B_2$ is trivial. Note that if $B_1$ is trivial and admits a morphism to $B_2$, then $B_2$ is trivial as well.

\begin{claim}
 The above description yields a functor $(-)^\tract:\OBlpr^\pm\to\Tracts$.
\end{claim}

\begin{proof}
 Let $B$ be an $\Funpm$-algebra with unique inverses. We verify that $B^\tract=(B^\times,N_B)$ is indeed a tract. Since $\K=\{0\}^\tract$ is a tract, we can assume that $B$ is nontrivial, i.e.\ $0\neq1$. Since $0\leq 0$, we have $0\in N_B$. We have $1\in N_B$ only if $0\leq 1$ in $B$. But then $0\leq 1=1+0$ and thus $1$ is the weak inverse of $0$, i.e.\ $1=\epsilon0=0$, a contradiction. This verifies axiom \ref{T1} of a tract.
 
 Axiom \ref{T2} follows from the fact that $\leq$ is closed under multiplication. Axiom \ref{T3} follows since $B$ is with unique weak inverses.

 Let $f:B_1\to B_2$ be a morphism of $\Funpm$-algebras. If $B_2$ is trivial, then $f^\tract$ is evidently well-defined as the unique morphism $B_1^\tract\to\K$ into the terminal object. Therefore we can assume that $B_1$ and $B_2$ are nontrivial.
 
 Since $f(B_1^\times)\subset B_2^\times$ and since $f$ is multiplicative, $f^\tract$ is a group homomorphism. If $\sum a_i\in N_{B_1^\tract}$, i.e.\ $0\leq \sum a_i$ in $B_1$, then $0\leq \sum f(a_i)$ in $B_2$ and thus $\sum f^\tract(a_i)\in N_{B_2^\tract}$.  This shows that $f^\tract:B_1^\tract\to B_2^\tract$ is indeed a morphism of tracts and defines the functor $(-)^\tract:\OBlpr^\pm\to \Tracts$.
\end{proof}


\subsubsection{The adjunction between tracts and \texorpdfstring{$\Funpm$}{F1+-}-algebras}\label{subsubsection: The adjunction between tracts and Funpm-algebras}

We begin with the description of the unit of the adjunction. Let $F$ be a tract and $F^\oblpr=(F,\N[F^\times],\leq)$ the associated ordered blueprint. Then the underlying set of $\tilde F=(F^\oblpr)^\tract$ is equal to the underlying set of $F$, which is $F^\times\cup\{0\}$. By the definition of $\leq$, we have $0\leq\sum a_i$ for all elements $\sum a_i\in N_F$. Thus $N_F\subset N_{\tilde F}$. This means that the identity map $\eta_F:F\to (F^\oblpr)^\tract$ is a morphism of tracts. 

Note that a morphism $F\to B^\tract$ into a tract $B^\tract$ that comes from an $\Funpm$-algebra with unique inverses $B$ factors uniquely into $\eta_F$ followed by a morphism $(F^\oblpr)^\tract\to B^\tract$.

\begin{claim}
 Let $B$ be an $\Funpm$-algebra with unique inverses. A morphism $f:F\to B^\tract$ factors uniquely into $\eta_F$ followed by a morphism $\tilde f:(F^\oblpr)^\tract\to B^\tract$.
\end{claim}

\begin{proof}
 The uniqueness of this factorization is clear since $\eta_F$ is the identity map between the underlying sets of $F$ and $\tilde F=(F^\oblpr)^\tract$ and is thus an epimorphism. The nullset $N_{\tilde F}$ of $\tilde F$ is the ideal of $\N[F^\times]$ generated by the nullset $N_F$ of $F$. Since the nullset of $B^\tract$ is also an ideal, independently of whether $B$ is trivial or not, the map $f:F\to B^\tract$ defines a morphism $\tilde f:\tilde F\to B^\tract$ and we have $f=\tilde f\circ\eta_F$, as claimed.
\end{proof}

We continue with the description of the counit of the adjunction. Let $B$ be an $\Funpm$-algebra with unique inverses and $B^\tract=(B^\times, N_B)$ the associated tract. If $B=\{0\}$ is trivial, then $\{0\}$ is a terminal object in $\OBlpr^\pm$ and we define $\epsilon_{\{0\}}:(\{0\}^\tract)^\oblpr\to \{0\}$ as the unique morphism into $\{0\}$. Note that $(\{0\}^\tract)^\oblpr=\bpgenquot{\{0,1\}}{0\leq1+1,0\leq1+1+1}$.

If $B$ is nontrivial, then $(B^\tract)^\oblpr$ is the ordered blueprint $(B^\times\cup\{0\},\N[B^\times],\leq_{B^\tract})$ whose partial order $\leq_{B^\tract}$ is generated by the relations of the form $0\leq \sum a_i$ with $a_i\in B^\times$ that are contained in the partial order of $B$. Thus the inclusion map $B^\times\to B$ induces a morphism $\epsilon_B:(B^\tract)^\oblpr\to B$ of ordered blueprints.

\begin{claim}
 Let $F$ be a tract. A morphism $f:F^\oblpr\to B$ factors uniquely into a morphism $\tilde f:F^\oblpr\to (B^\tract)^\oblpr$ followed by $\epsilon_B$.
\end{claim}

\begin{proof}
 If $B=\{0\}$ is trivial, then $(B^\tract)^\oblpr=\K^\oblpr$ and there is a unique morphism $\tilde f:F^\oblpr\to \K^\oblpr$, which sends $0$ to $0$ and nonzero elements $a$ to $1$. This morphism satisfies the claim.
 
 If $B$ is not trivial, then the uniqueness of the factorization follows from the fact that $\epsilon_B$ is an injection and thus a monomorphism. Note that $F^\oblpr$ is a blue field whose partial order is generated by relations of the form $0\leq \sum a_i$. Since $\tilde B=(B^\tract)^\oblpr$ contains all units of $B$ and $0$, the image of $f:F^\oblpr\to B$ is contained in $\tilde B$, i.e.\ $f$ factors into $\epsilon_B\circ \tilde f$ as a multiplicative map where $\tilde f:F^\oblpr\to\tilde B$ is the restriction of $f$ to the codomain $\tilde B$. Since all relations of the form $0\leq \sum a_i$ of $B$ with $a_i\in \tilde B=B^\times\cup\{0\}$ hold in $\tilde B$, the map $\tilde f$ is indeed a morphism of ordered blueprints, which proves the claim.
\end{proof}

This yields the adjunction
\[
 \begin{array}{ccc}
  \Hom_{\Tracts}(F,B^\tract) & \stackrel{1:1}\longleftrightarrow & \Hom_{\OBlpr^\pm}(F^\oblpr,B) \\[5pt]
   \big[ \ F\stackrel f\longrightarrow B^\tract \ \big] & \longmapsto & \big[ \ F^\oblpr\stackrel{f^\oblpr}\longrightarrow (B^\tract)^\oblpr \stackrel{\epsilon_B}\longrightarrow B \ \big] \\[5pt]
   \big[ \ F \stackrel{\eta_F}\longrightarrow (F^\oblpr)^\tract \stackrel{f^\tract}\longrightarrow B^\tract\ \big] & \reflectbox{$\longmapsto$} & \big[ \ F^\oblpr\stackrel f \longrightarrow B \ \big]         
 \end{array}
\]
for every tract $F$ and every $\Funpm$-algebra with unique inverses $B$.


\subsubsection{The images of the adjunction}\label{subsubsection: the images of the adjunction}

Let $F=B^\tract$ be a tract that comes from an $\Funpm$-algebra with unique inverses $B$. Then the nullset $N_F$ of $F$ is an ideal of $\N[F^\times]$. Thus the partial order $\leq_F$ of the ordered blueprint $F^\oblpr$ is generated by the relations of the form $0\leq\sum a_i$ for which $\sum a_i\in N_F$.
After identifying the underlying sets of $F$ and $\tilde F=(F^\oblpr)^\tract$, the nullset $N_{\tilde F}$ of $\tilde F$ is equal to $N_F$. This shows that $\epsilon_F:F\to (F^\oblpr)^\tract$ is an isomorphism and that the essential image of $(-)^\oblpr:\OBlpr^\pm\to\Tracts$ consists of all tracts $F$ whose nullset is an ideal of $\N[F^\times]$.

Let $B=F^\oblpr$ be an ordered blueprint that comes from a tract $F$. Then $B$ is an ordered blue field, i.e.\ $B=B^\times\cup\{0\}$, $B^+=\N[B^\times]$ and the partial order of $B$ is generated by relations of the form $0 \leq \sum a_i$
where $a_i\in B$.

Thus $\eta_B:(B^\tract)^\oblpr\to B$ is a bijection and the partial orders of both ordered blue fields agree. This shows that $\eta_B:(B^\tract)^\oblpr\to B$ is an isomorphism of ordered blueprints and that the essential image of $(-)^\tract:\Tracts\to\OBlpr^\pm$ consists of all idylls.

In summary, the adjoint functors $(-)^\oblpr$ and $(-)^\tract$ restrict to mutually inverse equivalences of categories between the respective images of these functors. This completes the proof of part \eqref{part2} of Theorem \ref{thm: hyperfuzz to bluetracts}.


\subsubsection{Commutativity of the diagram}

We verify that the square (starting in $\HypFields$) and both triangles (starting in $\FuzzRings$) of the diagram in Theorem \ref{thm: hyperfuzz to bluetracts} commute.

\begin{claim}
 Let $K$ be a hyperfield. Then $K^\oblpr\simeq(K^\fuzz)^\oblpr$.
\end{claim}

\begin{proof}
 The associated ordered blueprint $K^\oblpr$ is $B=\bpgenquot{K}{0\leq\sum a_i \; | \; 0\in\hypersum a_i}$. The associated fuzzy ring $K^\fuzz$ is $F=(\N[K^\times],I)$ where $I=\{\sum a_i \; | \; 0\in\hypersum a_i\}$. The ordered blueprint associated with $F$ is $F^\oblpr=\bpgenquot{(F^\times\cup\{0\})}{0\leq \sum a_i \; | \; \sum a_i\in I}$. The image of the natural embedding $K\to \N[K^\times]$ is $F^\times\cup\{0\}$. This defines a bijection $K^\oblpr\to F^\oblpr$, which is multiplicative and preserves $0$ and $1$. We have $0\leq\sum a_i$ in $K^\oblpr$ if and only if $\sum a_i\in I$, which is the case if and only if $0\leq \sum a_i$ in $F^\oblpr$. This establishes the claimed isomorphism $K^\oblpr\simeq(K^\fuzz)^\oblpr$.
\end{proof}

\begin{claim}
 Let $F=(\N[F^\times],I)$ be a fuzzy ring. Then $F^\oblpr\simeq(F^\tract)^\oblpr$ and $F^\tract\simeq(F^\oblpr)^\tract$.
\end{claim}

\begin{proof}
 The associated ordered blueprint $F^\oblpr$ is $B=\bpgenquot{(F^\times\cup\{0\})}{0\leq\sum a_i \; | \; a_i\in I}$. The associated tract $F^\tract$ is $T=(F^\times,I)$. 
 
 The ordered blueprint associated with $T$ is $T^\oblpr=\bpgenquot{(F^\times\cup\{0\})}{0\leq\sum a_i\in I}$, which is equal to $B=F^\oblpr$ with respect to the identity map $F^\times\cup\{0\}\to F^\times\cup\{0\}$. Thus the first isomorphism of the claim.
 
 The tract associated with $B$ is $B^\tract=(B^\times,J)$ for $J=\{\sum a_i \; | \; 0\leq\sum a_i\}$. Under the identification of $B^\times=F^\times$ with $T^\times=F^\times$, we obtain $I=J$ and thus the second isomorphism of the claim.
\end{proof}

We leave the easy verification that the above isomorphisms are functorial to the reader. This concludes the proof of Theorem \ref{thm: hyperfuzz to bluetracts}.


\subsection{Conventions for the rest of the paper}
\label{sec:HyperfuzzRemarks}

The dominant objects in the upcoming sections are $\Funpm$-algebras. From now on, we will rely on the results of Theorem \ref{thm: hyperfuzz to bluetracts} and think of the categories $\PartFields$, $\FuzzRings$ and $\HypFields$ as full subcategories of $\OBlpr^\pm$. Accordingly, we say that an ordered blueprint is a \emph{partial field}, \emph{fuzzy ring} or a \emph{hyperring} if it is isomorphic to an object of the corresponding subcategory.

By a slight abuse of notation, we will denote the associated ordered blueprints by the same symbols as their avatars in the subcategories of hyperfields and partial fields. For instance, we denote the ordered blueprints associated with the Krasner hyperfield and the tropical hyperfield by $\K$ and $\T$, respectively. This means that we make the identifications
\[ \textstyle
 \K \ = \ \bpbiggenquot{\{0,1\}}{0\leq1+1,0\leq1+1+1}
\]
and
\[
 \T \ = \ \bpbiggenquot{\R_{\geq0}^\bullet}{0\leq \sum a_i\,\big|\,\max\{a_i\}=a_k=a_l\text{ for }k\neq l}. 
\]
Similarly, we have $\U_0=\Funpm$. We will proceed with the symbol $\Funpm$ for the regular partial field since this is more systematic from our point of view.

\subsubsection{Remark on a common generalization of ordered blueprints and tracts}
Since the category of tracts does not embed into the category of ordered blueprints, we lose a certain part of the theory of matroids over tracts when passing to ordered blueprints. In principle, it is possible to formulate a common generalization of ordered blueprints and tracts which would be compatible with matroid theory over any tract. This would, however, complicate the exposition of this text considerably and, at the time of writing, the authors are not aware of a reason that would justify such an effort. In particular, all types of matroids that appeared in the literature before \cite{Baker-Bowler19} are based on tracts whose underlying nullset is an ideal.


\section{Comparison of matroid theories}
\label{section: Comparision of matroid theories}

In this section, we review matroid theory in its different incarnations over partial fields, hyperfields, fuzzy rings and tracts and show that the full embeddings from Theorem \ref{thm: hyperfuzz to bluetracts} are compatible with the respective matroid theories. We follow the approach to matroids in terms of Grassmann-Pl\"ucker functions, which makes the compatibility of the different matroid theories most visible. We restrict ourselves to some brief remarks on existing cryptomorphisms in the literature. We have chosen a top-to-bottom approach, starting with matroids over tracts and allowing ourselves to streamline the definitions from the original sources slightly to make our exposition more coherent. 

For the rest of this section, we fix a finite, non-empty ordered set $E=\{1,\dotsc,n\}$ and a natural number $r\leq n$. 
We denote by $\binom Er$ the family of all $r$-element subsets of $E$. 

\subsection{Matroids over tracts}

Matroids over tracts were introduced in \cite{Baker-Bowler19}; we provide a brief summary of this theory.  

\subsubsection{Grassmann-Pl\"ucker functions}

Let $F=(F^\times, N_F)$ be a tract, let $E=\{1,\dotsc,n\}$ be a non-empty finite ordered set, and let $r$ be a natural number with $r\leq n$. 

\begin{df}
 A \emph{Grassmann-Pl\"ucker function of rank $r$ on $E$ with coefficients in $F$} is a function
 \[\textstyle
  \Delta: \ \binom Er \ \longrightarrow \ F
 \]
 that is not identically $0$ and satisfies the Pl\"ucker relations
\[
  \sum_{k=0}^r \epsilon^k \ \Delta\big(I-\{i_k\}\big) \ \Delta\big(J\cup\{i_k\}\big) \ \in \ N_F
 \]
 for every $(r-1)$-subset $J$ of $E$ and every $(r+1)$-subset $I=\{i_0,\dotsc,i_r\}$ of $E$ with $i_0<\dotsb<i_r$, where we define $\Delta(J\cup\{i_k\})=0$ if $i_k\in J$.
 
Two Grassmann-Pl\"ucker functions $\Delta$ and $\Delta'$ are \emph{equivalent} if $\Delta=a\Delta'$ for some element $a\in F^\times$. A \emph{(strong) $F$-matroid of rank $r$ on $E$} is an equivalence class $M$ of Grassmann-Pl\"ucker functions $\Delta$ of rank $r$ on $E$ with coefficients in $F$. We denote the set of all $F$-matroids of rank $r$ on $E$ by $\cMat(r,E)(F)$. 
\end{df}


The notion of $F$-matroids, where $F$ is a tract, includes as particular special cases various generalizations of matroids which have appeared previously in the literature. In particular, one obtains
\begin{itemize}
 \item matroids as $\K$-matroids where $\K$ is the Krasner hyperfield;
 \item oriented matroids as $\S$-matroids where $\S$ is the hyperfield of signs;
 \item valuated matroids as $\T$-matroids where $\T$ is the tropical hyperfield.
\end{itemize}

Furthermore, if $K$ is a field, a $K$-matroid of rank $r$ on $E$ is the same thing as a $K$-linear subspace of $K^E$ of dimension $r$.
For more details, we refer the reader once again to \cite{Baker-Bowler19}.

\subsubsection{Pushforwards and the underlying matroid}

A morphism $f:F_1\to F_2$ of tracts yields a map $f_\ast:\cMat(r,E)(F_1)\to\cMat(r,E)(F_2)$ by sending the class $M=[\Delta]$ of a Grassmann-Pl\"ucker function $\Delta:\binom Er\to F_1$ to the class $f_\ast(M)=[f\circ\Delta]$ of the composition $f\circ\Delta:\binom Er\to F_2$, which is a Grassmann-Pl\"ucker function with coefficients in $F_2$. We call $f_\ast(M)$ the \emph{pushforward} of the $F_1$-matroid $M$ along $f$. Note that pushforwards are clearly functorial, i.e.\ $(g\circ f)_\ast=g_\ast\circ f_\ast$ for tract morphisms $f:F_1\to F_2$ and $g:F_2\to F_3$.

The Krasner hyperfield $\K$ is terminal in the category of tracts, i.e.\ every tract $F$ admits a unique morphism $f:F\to\K$. Given an $F$-matroid $M$, we call the pushforward $f_\ast(M)$ the \emph{underlying matroid of $M$}.

\subsubsection{Weak Grassmann-Pl\"ucker functions}\label{subsubsection: weak Grassmann-Plucker functions}

Let $\Delta: \ \binom Er \ \longrightarrow \ F$ be a nonzero function.  We define the \emph{support} of $\Delta$ to be 
\[\textstyle
\underline{\Delta} := \{ B \in \binom Er \; | \; \Delta(B) \neq 0 \}.
\]

We say that $\Delta$ is a \emph{weak Grassmann-Pl\"ucker function of rank $r$ on $E$ with coefficients in $F$} if the following two conditions hold:
\begin{itemize}
\item[${\rm (GP1)}'$] $\underline{\Delta}$ is the set of bases of a matroid $\underline{M}$ of rank $r$ on $E$.
\item[${\rm (GP2)}'$] $\Delta$ satisfies the $3$-term Pl\"ucker relations, i.e., all relations of the form 
 \[
 \Delta(I_{1,2}) \, \Delta(I_{3,4}) \ + \ \epsilon \, \Delta(I_{1,3}) \, \Delta(I_{2,4}) \ + \ \Delta(I_{1,4}) \, \Delta(I_{2,3}) \quad \in \, N_F
 \]
 for every $(r-2)$-subset $I$ of $E$ and all $i_1<i_2<i_3<i_4$ with $i_1,i_2,i_3,i_4\notin I$, where $I_{k,l}=I\cup\{i_k,i_l\}$.
\end{itemize}

A \emph{weak $F$-matroid of rank $r$ on $E$} is an equivalence class $M$ of weak Grassmann-Pl\"ucker functions $\Delta$ of rank $r$ on $E$ with coefficients in $F$ (with respect to the same equivalence relation as above).

\begin{rem}\label{rem: strong versus weak matroids}
For many tracts of interest, the notions of weak and strong $F$-matroids agree, cf.\ section 3 in \cite{Baker-Bowler19} for more details.  In particular, weak and strong $F$-matroids agree when $F$ is a partial field or $F$ is one of the hyperfields $\K, \S,$ or $\T$.
\end{rem}

\subsubsection{Cryptomorphisms}\label{subsubsection: cryptomorphisms for matroids over tracts}
The main results of \cite{Baker-Bowler19} provide equivalent (``cryptomorphic'') descriptions of weak and strong $F$-matroids in terms of circuits and dual pairs. 
Since these cryptomorphisms do not not play a very important role in the present paper, we provide only a brief description here.  (We slightly simplify the exposition from \cite{Baker-Bowler19} by assuming, in the case of circuit axioms, that the underlying structure forms a matroid in the usual sense.)

If $F$ and $E$ are as above, we denote by $F^E$ the set of functions from $E$ to $F$, which carries a natural action of $F$ by pointwise multiplication.
The $F$-circuits and $F$-vectors of a (strong or weak) $F$-matroid will by definition be certain subsets of $F^E$.

The {\bf support} of $X \in F^E$, denoted $\underline{X}$ or ${\rm supp}(X)$, is the set of $e \in E$ such that $X(e) \neq 0$.
If $A \subseteq F^E$, we set ${\rm supp}(A) := \{ \underline{X} \; | \; X \in A \}$.

\begin{def} \label{defn:linspan}
The {\bf linear span} of $X_1,\ldots,X_k \in F^E$ is defined to be the set
of all $X \in F^E$ such that 
\[
c_1X_1 + \cdots + c_k X_k +\epsilon X \in (N_F)^E
\]
for some $c_1,\ldots,c_k \in F$.
\end{def}

\begin{def} \label{defn:orthogonality}
The {\bf inner product}
of $X=(x_1,\ldots,x_n)$ and $Y=(y_1,\ldots,y_n)$ in $F^E$ with respect to a fixed involution $x \mapsto \bar{x}$ on $F$ is defined to be 
\[
X \cdot Y := x_1 \cdot \bar{y}_1 + \cdots + x_n \cdot \bar{y}_n.  
\]
We say that $X,Y$ are {\bf orthogonal}, denoted $X \perp Y$, if $X \cdot Y \in N_F$.

\end{def}

Let $\underline{M}$ be a (classical) matroid with ground set $E$.  We call a subset ${\mathcal C}$ of $F^E$ an \emph{$F$-signature of $\underline{M}$} if:
\begin{itemize}
\item[(S1)] The support $\underline{\mathcal C}$ of ${\mathcal C}$ is the set of circuits of $\underline{M}$.
\item[(S2)] If $X \in {\mathcal C}$ and $\alpha \in F^\times$, then $\alpha \cdot X \in {\mathcal C}$.
\item[(S3)] If $X,Y \in {\mathcal C}$ and $\underline{X} = \underline{Y}$, there exists $\alpha \in F^\times$ such that $X = \alpha \cdot Y$.
\end{itemize}

\medskip\noindent
{\bf Circuits.} A subset ${\mathcal C}$ of $F^E$ is called the \emph{$F$-circuit set of a strong $F$-matroid of rank $r$ on $E$} if:

\begin{itemize}
\item[(C1)] There is a matroid $\underline{M}$ of rank $r$ on $E$ such that ${\mathcal C}$ is an $F$-signature of $\underline{M}$.
\item[(C2)] Let $B$ be a basis of $\underline{M}$, and for $e \not\in B$ let $X(e)$ be the unique element of ${\mathcal C}$ such that $X(e)=1$ and whose support is the fundamental circuit of $B \in \underline{M}$ with respect to $e$.  Then every $X \in {\mathcal C}$ is in the $F$-linear span of $\{ X_e \}_{e \in E}$.
\end{itemize} 

We call ${\mathcal C}$ the \emph{$F$-circuit set of a weak $F$-matroid of rank $r$ on $E$} if it satisfies (C1) and the following axiom:
\begin{itemize}
\item[${\rm (C2)}'$] Let $B$ be a basis of $\underline{M}$, let $e_1,e_2 \in E \backslash B$ be distinct, and for $i=1,2$ let $X_i$ be the unique element of ${\mathcal C}$ such that $X_i(e_i)=1$ and whose support is the fundamental circuit of $B \in \underline{M}$ with respect to $e_i$.  Then there exists $X \in {\mathcal C}$ belonging to the $F$-linear span of $X_1$ and $X_2$.
\end{itemize} 

These definitions are equivalent to the ones in \cite{Baker-Bowler19} (use Lemma 2.7 and Theorem 2.12 from {\it loc.~cit.}).

\medskip\noindent
{\bf Dual pairs.} 
Let $\underline{M}$ be a (classical) matroid of rank $r$ with ground set $E$.  We say that $({\mathcal C},{\mathcal D})$ is a {\bf strong dual pair of $F$-signatures of $\underline{M}$} if:
\begin{itemize}
\item[(DP1)] ${\mathcal C}$ is an $F$-signature of the matroid $\underline{M}$.
\item[(DP2)] ${\mathcal D}$ is an $F$-signature of the dual matroid $\underline{M}^*$.
\item[(DP3)] $X \perp Y$ whenever $X \in {\mathcal C}$ and $Y \in {\mathcal D}$.
\end{itemize}

We say that $({\mathcal C},{\mathcal D})$ is a {\bf weak dual pair of $F$-signatures of $\underline{M}$} if
${\mathcal C}$ and ${\mathcal D}$ satisfy (DP1),(DP2), and the following weakening of (DP3):
\begin{itemize}
\item[${\rm (DP3)}'$] $X\perp Y$ for every pair $X \in {\mathcal C}$ and $Y \in {\mathcal D}$ with $|\underline{X} \cap \underline{Y}| \leq 3$.
\end{itemize}

The results of \cite{Baker-Bowler19} imply:

\begin{thm} 
Let $E$ be a non-empty finite set, let $F$ be a tract endowed with an involution $x \mapsto \bar{x}$, and let $r$ be a positive integer. 
Then there are natural bijections between:

\begin{enumerate}
\item Equivalence classes of Grassmann-Pl{\"u}cker functions of rank $r$ on $E$ with coefficients in $F$.
\item $F$-circuit sets of strong $F$-matroids of rank $r$ on $E$.
\item Matroids $\underline{M}$ endowed with a strong dual pair of $F$-signatures.
\end{enumerate}

Similarly, there are natural bijections between:
\begin{enumerate}
\item Equivalence classes of weak Grassmann-Pl{\"u}cker functions of rank $r$ on $E$ with coefficients in $F$.
\item $F$-circuit sets of weak $F$-matroids of rank $r$ on $E$.
\item Matroids $\underline{M}$ endowed with a weak dual pair of $F$-signatures.
\end{enumerate}
\end{thm}

\subsubsection{Duality}

There is a duality theory for matroids over tracts which generalizes the established duality theory for matroids, oriented matroids, valuated matroids, etc.  
The results of \cite{Baker-Bowler19} show:

\begin{thm}
Let $E$ be a non-empty finite set with $|E|=n$, let $F$ be a tract, and let $M$ be a strong (resp. weak) $F$-matroid of rank $r$ on $E$ with strong (resp. weak) $F$-circuit set ${\mathcal C}$ and Grassmann-Pl{\"u}cker function (resp. weak Grassmann-Pl{\"u}cker function) $\Delta$.
Then there is a strong (resp. weak) $F$-matroid $M^*$ of rank $n-r$ on $E$, called the {\bf dual matroid} of $M$, with the following properties:
\begin{itemize}
\item A Grassmann-Pl{\"u}cker function (resp. weak Grassmann-Pl{\"u}cker function) $\Delta^*$ for $M^*$ is defined by the formula
\[
\Delta^*(I) = \sigma_I \cdot \Delta(I^c),
\]
where $I = \{ i_1,\ldots,i_{n-r} \} \subseteq E$ with $i_1 < \cdots < i_{n-r}$, $I^c = \{ i_1',\ldots,i_r' \}$ is the complement of $I$ in $E$ with $i_1' < \cdots < i_r'$, and $\sigma_I$ is the sign of the permutation taking $(1,2,\ldots,n)$ to $(i_1,\ldots,i_{n-r},i_1',\ldots,i_r')$.
\item The $F$-circuits of $M^*$ are the elements of ${\mathcal C}^* := {\rm SuppMin}({\mathcal C}^\perp - \{ 0 \})$, where ${\rm SuppMin}(S)$ denotes the elements of $S$ of minimal support.
\item The underlying matroid of $M^*$ is the dual of the underlying matroid of $M$, i.e., $\underline{M^*} = \underline{M}^*$.
\item $M^{**} = M$.
\end{itemize}
\end{thm}

The $F$-circuits of $M^*$ are called the \emph{$F$-cocircuits} of $M$, and vice-versa.

\subsubsection{Vector axioms for strong $F$-matroids} \label{sec:vectoraxioms}

The set of \emph{$F$-vectors} of an $F$-matroid $M$ is defined as the set $V(M)$ of all $X \in F^E$ such that $X \perp {\mathcal D}$ for all $F$-cocircuits ${\mathcal D}$ of $M$.
(Similarly, the set of \emph{$F$-covectors} of $M$ is the set of vectors of $M^*$.)

Laura Anderson has worked out a cyptomorphic axiomatization of strong $F$-matroids in terms of their vectors in \cite{Anderson19}.  We briefly recall her description.

If $W$ is a subset of $F^E$, a \emph{support basis for $W$} is a minimal subset of $E$ meeting every element of ${\rm supp}(W  \backslash \{ 0 \})$.

Let $B$ be a support basis for $W$.
A \emph{$B$-frame for $W$}\footnote{In \cite{Anderson19} this is called a \emph{reduced row-echelon form}.} is a collection $\Phi_B = \{ w_i^B \}_{i \in B}$ of elements of $W$ such that $w_i^B(j) = \delta_{ij}$ and
every $w \in W$ is in the $F$-linear span of $\Phi_B$.
It is not hard to see that a $B$-frame for $W$, if it exists, is unique.

We define a collection $\Phi = \{ \Phi_B \}$ of frames for $W$ to be \emph{tight} if $W$ is precisely the set of elements of $F^E$ which are in the $F$-linear span of $\Phi_B$ for all 
$\Phi_B \in \Phi$.

\medskip\noindent
{\bf Vectors.} 
A subset $W$ of $F^E$ is the \emph{$F$-vector set of a strong $F$-matroid of rank $r$ on $E$} if:

\begin{itemize}
\item[(V1)] Every support basis $B$ for $W \backslash \{ 0 \}$ admits a $B$-frame.
\item[(V2)] The collection of all such $B$-frames is tight.
\end{itemize}

\begin{thm} [Anderson]
There are natural bijections between strong $F$-matroids of rank $r$ on $E$ and subsets $W$ of $F^E$ satisfying (V1) and (V2).
\end{thm}

\subsection{Matroids over fuzzy rings}
Matroids over fuzzy rings were introduced in by Dress in \cite{Dress86}. The cryptomorphic description in terms of Grassmann-Pl\"ucker functions can be found in \cite{Dress-Wenzel91}. Let $(R,I)$ be a fuzzy ring.

A \emph{Grassmann-Pl\"ucker function of rank $r$ on $E$ with coefficients in $R$} is a function $\Delta: \binom Er \to R^\times\cup\{0\}$ that is not identically $0$ and satisfies the Pl\"ucker relations
\[
 \sum_{k=0}^r (-1)^k \ \Delta\big(I-\{i_k\}\big) \ \Delta\big(J\cup\{i_k\}\big) \ \in \ I
\]
for every $(r-1)$-subset $J$ of $E$ and every $(r+1)$-subset $I=\{i_0,\dotsc,i_r\}$ of $E$ with $i_0<\dotsb<i_r$.
 
Two Grassmann-Pl\"ucker functions $\Delta$ and $\Delta'$ are \emph{equivalent} if $\Delta=a\Delta'$ for some element $a\in R^\times$. An \emph{$R$-matroid of rank $r$ on $E$} is an equivalence class $M$ of Grassmann-Pl\"ucker functions $\Delta$ of rank $r$ on $E$ with coefficients in $R$. We denote the set of all $R$-matroids of rank $r$ on $E$ by $\cMat(r,E)(R)$. 

\subsubsection{Pushforwards}\label{subsubsection: pushforwards of matroids over fuzzy rings}
A morphism $f:R_1\to R_2$ of fuzzy rings induces a pushforward $f_\ast:\cMat(r,E)(R_1)\to\cMat(r,E)(R_2)$, which sends an $R_1$-matroid $M=[\Delta]$ to the $R_2$-matroid $f_\ast(M)=[f\circ\Delta]$.

Note that an isomorphism of fuzzy rings $f:R_1\to R_2$ induces a bijection $f_\ast:\cMat(r,E)(R_1)\to\cMat(r,E)(R_2)$ since an isomorphism preserves both the codomain $R_1^\times\cup\{0\}$ of Grassmann-Pl\"ucker functions as well as the Pl\"ucker relations, which are relations of the form $\sum a_i\in I_1$ with $a_i\in R_1^\times\cup\{0\}$.

\subsubsection{Compatibility with matroids over tracts}
A matroid over a fuzzy ring is the same thing as a matroid over the associated tract. More precisely:

\begin{prop}
 The functor $(-)^\tract:\FuzzRings\to\Tracts$ induces a functorial bijection $\cMat(r,E)(F)\to \cMat(r,E)(F^\tract)$ for every fuzzy ring $F$.
\end{prop}

\begin{proof}
 By the observations in section \ref{subsubsection: pushforwards of matroids over fuzzy rings} and Corollary \ref{cor: full subcategory of fuzzy rings}, we can restrict ourselves to fuzzy rings of the form $F=(\N[G],I)$ where $G$ is a group. Recall from section \ref{subsubsection: From fuzzy rings to tracts} that the associated tract is $F^\tract=(G,I)$. 
 
 A Grassmann-Pl\"ucker function $\Delta:\binom Er\to G\cup\{0\}$ is evidently a Grassmann-Pl\"ucker function with coefficients in $F^\tract$, and the equivalence relation on Grassmann-Pl\"ucker functions coincide for $F$ and $F^\tract$. Thus we obtain a bijection $\cMat(r,E)(F)\to \cMat(r,E)(F^\tract)$. Since pushforwards for both morphisms of fuzzy rings and morphisms of tracts are defined by composition, this association is functorial.
\end{proof}

\subsubsection{Cryptomorphisms}
The original definition of $F$-matroids in Dress' paper \cite{Dress86} was formulated in terms of sets of relations, which includes both circuit sets as minimal sets of relations and dependency sets as maximals set of relations. The descriptions of closure operators, flats, rank functions and duality are derived from this definition, but without exhibiting a cryptomorphic axiomatization which ensures an equivalence with the definition by dependency sets. 

The equivalence with classes of Grassmann-Pl\"ucker functions is the theme of the subsequent joint paper \cite{Dress-Wenzel91} with Wenzel. Interestingly enough, the proof of this equivalence appears to be quite different from the corresponding proof for matroids over tracts in \cite{Baker-Bowler19}: instead of utilizing a cryptomorphic description of matroids in terms of duality theory, the proof in \cite{Dress-Wenzel91} is based on the Tutte group of a matroid, cf.\ section \ref{subsection: the Tutte group}. It seems interesting for future generalizations, e.g.\ to matroid bundles, to gain a better understanding of the relation between these two seemingly different approaches.

\subsection{Matroids over hyperfields}

Matroids over hyperfields were introduced by the first author and Bowler in \cite{Baker-Bowler19}. 

Let $K$ be a hyperfield.  A \emph{Grassmann-Pl\"ucker function of rank $r$ on $E$ with coefficients in $K$} is a function $\Delta: \binom Er \to K$ that is not identically $0$ and satisfies the Pl\"ucker relations
\[
 0 \ \in \ \boxplus_{k=0}^r (-1)^k \ \Delta\big(I-\{i_k\}\big) \ \Delta\big(J\cup\{i_k\}\big)
\]
for every $(r-1)$-subset $J$ of $E$ and every $(r+1)$-subset $I=\{i_0,\dotsc,i_r\}$ of $E$ with $i_0<\dotsb<i_r$.
 
Two Grassmann-Pl\"ucker functions $\Delta$ and $\Delta'$ are \emph{equivalent} if $\Delta=a\Delta'$ for some element $a\in K^\times$. A \emph{$K$-matroid of rank $r$ on $E$} is an equivalence class $M$ of Grassmann-Pl\"ucker functions $\Delta$ of rank $r$ on $E$ with coefficients in $K$. We denote the set of all $K$-matroids of rank $r$ on $E$ by $\cMat(r,E)(K)$. 

\subsubsection{Pushforwards}

A morphism $f:K_1\to K_2$ of hyperfields induces a pushforward $f_\ast:\cMat(r,E)(K_1)\to\cMat(r,E)(K_2)$, which sends a $K_1$-matroid $M=[\Delta]$ to the $K_2$-matroid $f_\ast(M)=[f\circ\Delta]$.

\subsubsection{Compatibility with matroids over fuzzy rings}

The following fact is already covered in Theorem B of \cite{Giansiracusa-Jun-Lorscheid16}. For completeness, we include a short proof.

\begin{prop}
 The functor $(-)^\fuzz:\HypFields\to\FuzzRings$ induces a functorial bijection $\cMat(r,E)(K)\to \cMat(r,E)(K^\fuzz)$ for every hyperfield $K$.
\end{prop}

\begin{proof}
 Let $K$ be a hyperfield. Recall from section \ref{subsubsection: From hyperfields to fuzzy rings} that the associated fuzzy ring is $K^\fuzz=(\N[G],I)$ where $G=K^\times$ and
 \[\textstyle
  I \ = \ \big\{\ \sum a_i\in \N[G] \ \big| \ a_i\in G\text{ such that }0\in \hypersum a_i \ \big\}.
 \]

A Grassmann-Pl\"ucker function $\Delta:\binom Er\to K$ is evidently a Grassmann-Pl\"ucker function with coefficients in $K^\fuzz$, and the equivalence relation on Grassmann-Pl\"ucker functions coincides for $F$ and $F^\tract$. Thus we obtain a bijection $\cMat(r,E)(K)\to \cMat(r,E)(K^\fuzz)$. Since pushforwards for both hyperfield morphisms and morphisms of fuzzy rings are defined by composition, this association is functorial.
\end{proof}

\subsection{Matroids over partial fields}
\label{subsection: matroids over partial fields}

Strictly speaking, the concept of a matroid over a partial field has not been introduced in the literature yet, but partial fields were utilized to realize matroids as matrices with coefficients over a given partial field. However, our approach via Grassmann-Pl\"ucker functions suggests a definition of matroids over a partial field. We will explain in this section in which sense this definition is compatible with the notion of a representation of a matroid over a partial field.

\subsubsection{Representations of matroids over a partial field}
\label{subsubsection: representations of matroids over a partial field}

Partial fields and representations of matroids over such were introduced in Semple and Whittle's paper \cite{Semple-Whittle96}. Let $M$ be a matroid, which is the same as a $\K$-matroid. This means that $M$ is the class of a Grassmann-Pl\"ucker function $\Delta:\binom Er\to\K$, which is uniquely determined by $M$ since $\K^\times=\{1\}$. The set of bases of $M$ is $\cB=\{I\in \binom Er \; | \; \Delta(I)=1\}$.

Let $P$ be a partial field with unit group $P^\times$ and projection $\pi_P:\Z[P^\times]\to R_P$. As usual, we identify $P$ with the subset $P^\times\cup\{0\}$ of $R_P$.

An $r\times E$-matrix $A$ with coefficients in $R_P$ is a \emph{$P$-matrix} if all of its $r \times r$ minors are in $P$. Given a $P$-matrix $A=(a_{i,j})_{1\leq i\leq r,j\in E}$ and an $r$-subset $I$ of $E$, we denote by $A_I$ the square submatrix $(a_{i,j})_{1\leq i\leq r,j\in I}$ and by $\delta_I(A)=\det A_I$ the determinant of $A_I$.

A \emph{representation of $M$ over $P$} is a $P$-matrix $A$ of size $r\times E$ such that for every $r$-subset $I$ of $E$, the minor $\delta_I(A)$ is nonzero if and only if $I$ is a basis of $M$.

Conversely, let $f:P\to\K$ be the map that sends $0$ to $0$ and every nonzero element to $1$. It is shown in \cite[Thm.\ 3.6]{Semple-Whittle96} that every non-degenerate (meaning that some $r \times r$ minor is nonzero) $P$-matrix $A$ defines a matroid $M[A]$, which is represented by the Grassmann-Pl\"ucker function $\Delta:\binom Er\to\K$ with $\Delta(I)=f\big(\delta_I(A)\big)$.

\subsubsection{$P$-matroids}
\label{subsubsection: P-matroids}

Realizing a partial field $P$ as the fuzzy ring $P^\fuzz=(\N[P^\times],I)$ with
\[
 I \ = \ \big\{ \, \sum a_i\in\N[P^\times] \, \big| \, \pi_P(\sum a_i)=0 \, \big\} \ = \ \ker \big( \, \N[P^\times] \, \hookrightarrow \, \Z[P^\times] \, \stackrel{\pi_P}\longrightarrow \, R_P \, \big)
\]
leads to the following definition of a $P$-matroid.

\begin{df}
 A \emph{Grassmann-Pl\"ucker function of rank $r$ on $E$ with coefficients in $P$} is a function $\Delta: \binom Er \to P$ that is not identically $0$ and satisfies the Pl\"ucker relations
 \[
  \sum_{k=0}^r (-1)^k \ \Delta\big(I-\{i_k\}\big) \ \Delta\big(J\cup\{i_k\}\big) \ = \ 0
 \]
 in $R_P$ for every $(r-1)$-subset $J$ of $E$ and every $(r+1)$-subset $I=\{i_0,\dotsc,i_r\}$ of $E$ with $i_0<\dotsb<i_r$.
 
 Two Grassmann-Pl\"ucker functions $\Delta$ and $\Delta'$ are \emph{equivalent} if $\Delta=a\Delta'$ for some element $a\in P^\times$. A \emph{$P$-matroid of rank $r$ on $E$} is an equivalence class $M$ of Grassmann-Pl\"ucker functions $\Delta$ of rank $r$ on $E$ with coefficients in $P$. We denote the set of all $F$-matroids of rank $r$ on $E$ by $\cMat(r,E)(P)$. 
\end{df}

Since a partial field is embedded into its universal ring $R_P$, we can use the usual Grassmannian $\Gr(r,E)$ over $R_P$ to lay the relation to representations of matroids, which is as follows.

\begin{prop}\label{prop: P-matroids from P-matrices}
 Let $A$ be a non-degenerate $P$-matrix of size $r\times E$. Then the map $\Delta_A:\binom Er\to P$ defined by $\Delta(I)=\delta_I(A)$ is a Grassmann-Pl\"ucker function and $M[A]=f_\ast([\Delta_A])$. Conversely, every $P$-matroid is of the form $[\Delta_A]$ for some $P$-matrix $A$.
\end{prop}

\begin{proof}
 A $P$-matrix $A$ is, in particular, a matrix with coefficients in the ring $R_P$. Therefore the $r\times r$-minors $\delta_I(A)$ of $A$ are the homogeneous coordinates of a point in the Grassmannian $\Gr(r,E)(R_P)$ and thus satisfy the Pl\"ucker relations. Since the bases $I\in\binom Er$ of $M[A]$ are defined by the non-vanishing of $\delta_I(A)$, we obtain $M[A]=f_\ast([\Delta_A])$, as claimed.
 
 Conversely, let $M$ be a $P$-matroid that is represented by a Grassmann-Pl\"ucker function $\Delta:\binom Er\to P$. Then $[\Delta(I)]_{I\in\binom Er}$ are the homogeneous coordinates of a point the Grassmannian $\Gr(r,E)(R_P)$, which is covered by affine spaces whose coordinates correspond to entries of $r\times E$-matrices with coefficients in $R_P$. Thus $[\Delta(I)]_{I\in\binom Er}$ corresponds to an $r\times E$-matrix $A$ with coefficients in $R_P$. More precisely, this matrix can be described as follows.
 
 After multiplying by a suitable nonzero element of $P^\times$, we can assume that $\Delta(I_0)=1$ for some $r$-subset $I_0$ of $E$. After reordering $E$, we can assume that $I_0=\{1,\dotsc,r\}$. Then the matrix $A=(a_{i,j})_{1\leq i\leq r,j\in E}$ has the following shape. For $1\leq i,j\leq r$, we have $a_{i,j}=1$ if $i=j$ and $a_{i,j}=0$ if $i\neq j$. If $1\leq i\leq r$ and $r< j\leq n$, then we have $a_{i,j}=(-1)^{r-i}\Delta(I)$. This shows already that $A$ is a matrix with coefficients in $P$. Since $[\Delta(I)]_{I\in\binom Er}$ is uniquely determined by $A$, we conclude that $\Delta(I)=\delta_I(A)$ for all $I$ in $\binom Er$.
 
 We are left with showing that $A$ is indeed a $P$-matrix, i.e.\ that all minors of $A$ are in $P$. Consider a square submatrix $A_{\tilde I,\tilde J}=(a_{i,j})_{i\in \tilde I,j\in \tilde J}$ of $A$ where $\tilde I\subset I_0=\{1,\dotsc,r\}$ and $\tilde J\subset E$. Define $J=\tilde J\cup(I_0-\tilde I)$. Then we have $\det A_{\tilde I,\tilde J}=\pm\Delta(J)$ if $\#J=r$ and $\det A_{\tilde I,\tilde J}=0$ otherwise. This shows that all minors of $A$ are in $P$ and that $A$ is a $P$-matrix.
\end{proof}

Given a partial field $P$, we denote by $P^\tract$ the tract $(P^\fuzz)^\tract$ associated with the fuzzy ring that is associated with $P$. An immediate consequence of Proposition \ref{prop: P-matroids from P-matrices} is the following. 

\begin{cor}\label{cor: a matroid is representable over a partial field iff it is the image of a P-matroid}
 Let $M$ be a matroid of rank $r$ on $E$ and $P$ a partial field. Then $M$ is representable over $P$ if and only if $M$ is contained in the image of $f_\ast:\cMat(r,E)(P^\tract)\to\cMat(r,E)(\K)$, where $f:P^\tract\to\K$ is the unique morphism into the Krasner hyperfield $\K$. \qed
\end{cor}

\subsubsection{Relation to regular matroids}
\label{subsubsection: relation to regular matroids}

A \emph{totally unimodular matrix} is an integral matrix $A$ whose minors are all in $\{0,1,-1\}$. This means, in particular, that the only possible coefficients of $A$ are $0$, $1$ and $-1$ and that the coefficients of the Pl\"ucker vector of $A$ are in $\{0,1,-1\}$ as well.

A \emph{regular matroid} is a matroid $M$ that has a representation over $\Z$ by a totally unimodular matrix $A$. It is well-known, and easy to prove, that a matroid $M$ is realizable by a totally unimodular matrix $A$ if and only if $M$ is realizable over the regular partial field $\Funpm$; cf.\ \cite[Prop.\ 4.3]{Semple-Whittle96}.

Corollary \ref{cor: a matroid is representable over a partial field iff it is the image of a P-matroid} allows for the following reformulation of this fact: a matroid is regular if and only if it is the underlying matroid of an $\Funpm$-matroid.
It is important to note, however, that different elements of $\cMat(r,E)(\Funpm)$ can give rise to the same regular matroid.
This is one of our motivations for studying rescaling classes of $F$-matroids: we show in section~\ref{subsection: rescaling classes} that two elements of $\cMat(r,E)(\Funpm)$ correspond to the same regular matroid if and only if they lie in the same rescaling class.

In section \ref{subsection: foundations of regular matroids}, we will reprove Tutte's characterization of regular matroids as matroids that are representable over every field.


\subsection{Matroids over \texorpdfstring{$\Funpm$}{F1+-}-algebras}
\label{subsection: Matroids over ordered blueprints}

We extend the definition of matroids from tracts to $\Funpm$-algebras in the following way.

\begin{df}
 Let $B$ be an $\Funpm$-algebra, $E$ a non-empty finite ordered set and $r$ a natural number. A \emph{Grassmann-Pl\"ucker function of rank $r$ on $E$ with coefficients in $B$} is a function
 \[\textstyle
  \Delta: \ \binom Er \ \longrightarrow \ B
 \]
 such that $\Delta_I\in B^\times$ for some $I\in\binom Er$ and $\Delta$ satisfies the Pl\"ucker relations
 \[
  0 \ \leq \ \sum_{k=0}^r \epsilon^k \ \Delta\big(I-\{i_k\}\big) \ \Delta\big(J\cup\{i_k\}\big) 
 \]
 for every $(r-1)$-subset $J$ of $E$ and every $(r+1)$-subset $I=\{i_0,\dotsc,i_r\}$ of $E$ with $i_0<\dotsb<i_r$, where we define $\Delta(J\cup\{i_k\})=0$ if $i_k\in J$.
 
 Two Grassmann-Pl\"ucker functions $\Delta$ and $\Delta'$ are \emph{equivalent} if $\Delta=a\Delta'$ for some element $a\in B^\times$. A \emph{$B$-matroid of rank $r$ on $E$} is an equivalence class $M$ of Grassmann-Pl\"ucker functions $\Delta$ of rank $r$ on $E$ with coefficients in $F$. We denote the set of all $F$-matroids of rank $r$ on $E$ by $\cMat(r,E)(B)$. 
\end{df}



\subsubsection{Pushforwards}
\label{subsubsection: pushforwards of matroids over ordered blueprints}

A morphism $f:B\to C$ of $\Funpm$-algebras induces a pushforward 
\[
f_\ast:\cMat(r,E)(B)\to\cMat(r,E)(C)
\]
which sends a $B$-matroid $M=[\Delta]$ to the $C$-matroid $f_\ast(M)=[f\circ\Delta]$. 




\subsubsection{Compatibility with matroids over tracts}
\label{subsubsection: compatibility with matroids over tracts}

In the following, we will explain the relation between the matroid theory of ordered blueprints and the matroid theory of tracts. In particular, Lemma \ref{lemma: matroids for ordered blue fields are the same as matroids over the underlying tract} shows that the matroid theory of an $\Fun$-algebra $B$ with unique weak inverses that is an ordered blue field with $B^+=\N[B^\times]$ is completely determined by the matroid theory of the underlying idyll $(B^\tract)^\oblpr$. Proposition \ref{prop: relation between matroids over Funpm-algebras and tracts} in turn shows that the matroid theory of an idyll is equal to the matroid theory of the associated tract. These results allow us to apply the results from \cite{Baker-Bowler19} to matroids over $\Funpm$-algebras that are blue fields.

Let $\eta_F:F\to (F^\oblpr)^\tract$ be the unit and $\epsilon_B:(B^\tract)^\oblpr\to B$ the counit of the adjunction between $\Tracts$ and $\OBlpr^\pm$ studied in section \ref{subsubsection: The adjunction between tracts and Funpm-algebras}. Let $\iota_F:F\to F^\oblpr$ be the identity map between the respective underlying sets and $\iota_B:B^\tract\to B$ the inclusion of the underlying set $B^\times\cup\{0\}$ of $B^\tract$ into $B$.

\begin{prop}\label{prop: relation between matroids over Funpm-algebras and tracts}
 Let $B$ be a nontrivial $\Funpm$-algebra with unique inverses and $F$ a tract. Let $E$ be a non-empty finite ordered set and $r$ a natural number. Then $\iota_F$ defines an injection
 \[
  \begin{array}{cccc}
   \iota_{F,\ast}: & \cMat(r,E)(F) & \longrightarrow & \cMat(r,E)(F^\oblpr), \\[5pt]
                   & [\Delta]       & \longmapsto     & [\iota_F\circ\Delta]
  \end{array}
 \]
 which is surjective if $\eta_F$ an isomorphism, and $\iota_B$ defines an injection
 \[
  \begin{array}{cccc}
   \iota_{B,\ast}: & \cMat(r,E)(B^\tract) & \longrightarrow & \cMat(r,E)(B), \\[5pt]
                   & [\Delta]       & \longmapsto     & [\iota_B\circ \Delta]
  \end{array}
 \]
 which is surjective if $\epsilon_B$ an isomorphism.
\end{prop}

\begin{proof}
 For better readability, we write $\overline{B}=(B^\tract)^\oblpr$ and $\overline{F}=(F^\oblpr)^\tract$.  Recall from Theorem \ref{thm: hyperfuzz to bluetracts} that the adjoint functors $(-)^\oblpr:\Tracts\to\OBlpr^\pm$ and $(-)^\tract:\OBlpr^\pm\to\Tracts$ restrict to mutually inverse equivalences between their respective images. This means that $\eta_{B^\tract}$ and $\epsilon_{F^\oblpr}$ are isomorphisms and that the maps $\iota_F$ and $\iota_B$ decompose as
 \[
  \beginpgfgraphicnamed{tikz/fig6x}
   \begin{tikzcd}[column sep=40pt,row sep=15pt]
    F \arrow{r}{\iota_F} \arrow{d}[swap]{\eta_F} & F^\oblpr \\
    \overline F \arrow{r}{\iota_{\overline F}} & \overline F^\oblpr \arrow{u}[swap]{\epsilon_{F^\oblpr}}
   \end{tikzcd}
  \endpgfgraphicnamed
  \qquad \text{and} \qquad
  \beginpgfgraphicnamed{tikz/fig7x}
   \begin{tikzcd}[column sep=40pt,row sep=15pt]
    B^\tract \arrow{r}{\iota_B} \arrow{d}[swap]{\eta_{B^\tract}} & B \\
    \overline B^\tract \arrow{r}{\iota_{\overline B}} & \overline B.\! \arrow{u}[swap]{\epsilon_{B}}
   \end{tikzcd}
  \endpgfgraphicnamed
 \]
 Consequently, we have $\iota_{F,\ast}=\epsilon_{F^\oblpr,\ast}\circ\iota_{\overline F,\ast}\circ\eta_{F,\ast}$ and $\iota_{B,\ast}=\epsilon_{B,\ast}\circ\iota_{\overline B,\ast}\circ\eta_{B^\tract,\ast}$ where $\epsilon_{F^\oblpr,\ast}$ and $\eta_{B^\tract,\ast}$ are bijections.
 
 By Theorem \ref{thm: hyperfuzz to bluetracts} and the definition of $(-)^\oblpr$, the map  $\iota_{\overline{F}}:\overline{F}\to \overline{F}^\oblpr$ is an isomorphism of monoids and a linear combination $\sum a_i$ of elements $a_i\in \overline{F}$ is in the nullset $N_{\overline{F}}$ of $\overline{F}$ if and only if $0\leq \sum a_i$ holds in $\overline{F}^\oblpr$. Therefore a map $\Delta:\binom Er\to \overline{F}$ is a Grassmann-Pl\"ucker function if and only if $\iota_{\overline{F}}\circ\Delta:\binom Er\to \overline{F}^\oblpr$ is a Grassmann-Pl\"ucker function. Since $\iota_{\overline{F}}$ restricts to an isomorphism $\overline{F}^\times\to (\overline{F}^\oblpr)^\times$, the classes of Grassmann-Pl\"ucker functions agree for $\overline{F}$ and $\overline{F}^\oblpr$, which shows that $\iota_{\overline{F},\ast}:\cMat(r,E)(\overline{F})\to \cMat(r,E)(\overline{B})$ is a bijection.
 
 An analogous argument shows that $\iota_{\overline B,\ast}$ is a bijection. Thus we can trace back the claims about the injectivity and surjectivity of $\iota_{F,\ast}$ and $\iota_{B,\ast}$ to the corresponding claims for $\eta_{F,\ast}$ and $\epsilon_{B,\ast}$, respectively.
 
 Since the morphism $\eta_F:F\to\overline{F}$ is a bijection that restricts to an isomorphism $F^\times\to\overline F^\times$ between the respective unit groups, the pushforward $\eta_{F,\ast}:\cMat(r,E)(F)\to\cMat(r,E)(\overline{F})$ is an injective map. It is surjective if $\eta_F$ is an isomorphism. This proves the first part of the proposition.
  
 Since the morphism $\epsilon_B:\overline{B}\to B$ is an injection that restricts to an isomorphism $\overline B^\times\to B^\times$ between the respective unit groups, the pushforward $\epsilon_{B,\ast}:\cMat(r,E)(\overline B)\to\cMat(r,E)(B)$ is an injective map. It is surjective if $\epsilon_B$ is an isomorphism. This proves the second part of the proposition.
\end{proof}

\begin{rem}
 Note that for the trivial ordered blueprint $B=\{0\}$, there is a unique matroid of rank $r$ on $E$, which is represented by the Grassmann-Pl\"ucker function $\Delta:\binom Er\to \{0\}$ sending every $r$-subset $I$ of $E$ to $0=1$. In this case, $B^\tract=(\{1\},\N-\{1\})$ is the Krasner hyperfield and the map $\cMat(r,E)(B^\tract)\to\cMat(r,E)(B)$ is the unique map into the one-point set, which is not injective if $0< r<\#E$.
\end{rem}

\begin{lemma}\label{lemma: matroids for ordered blue fields are the same as matroids over the underlying tract}
 Let $E$ be a non-empty finite ordered set, $r$ a natural number and $B$ an $\Funpm$-algebra that is a blue field. Then the canonical morphism $\epsilon_B:(B^\tract)^\oblpr\to B$ induces a bijection
 \[
  \begin{array}{ccc}
   \big\{\text{$B$-matroids of rank $r$ on $E$}\big\} & \longrightarrow & \big\{\text{$(B^\tract)^\oblpr$-matroids of rank $r$ on $E$}\big\}. \\[5pt]
        M                                           & \longmapsto     & \epsilon_{B,\ast}(M)
  \end{array}
 \]
\end{lemma}

\begin{proof}
 This follows at once from the fact that the Pl\"ucker relations are contained in $(B^\tract)^\oblpr$.
\end{proof}

\subsubsection{Examples}

\begin{ex}\label{ex: matroids for free algebras over tracts}
 Let $B$ be an idyll. Then it is easily seen that $B[T^{\pm1}]$ is also an idyll. Let $\iota:B[T]\to B[T^{\pm1}]$ be the canonical inclusion. The pushforward along $\iota$ defines a map
 \[
  \Phi: \quad \big\{B[T]\text{-matroids}\big\} \quad \longrightarrow \quad \big\{B[T^{\pm1}]\text{-matroids}\big\}.
 \]
 We claim that this map is an isomorphism. Indeed, its inverse $\Psi$ can be described as follows. Let $M$ be a $B[T^{\pm1}]$-matroid, represented by a Grassmann-Pl\"ucker function $\Delta:\binom Er\to B[T^{\pm1}]$. Define $i$ as the minimal exponent $j$ that occurs in a nonzero term $\Delta(I)=aT^j$ for some $r$-subset $I$ of $E$. Then the image of $T^{-i}\Delta$ is contained in $B[T]$ and $T^{-i}\Delta(I)\in B^\times$ for the $r$-subset $I$ for which the nonzero term $\Delta(I)=aT^i$ assumes the minimal exponent. We define $\Psi(M)=[T^{-i}\Delta]$. It is easily verified that $\Phi$ and $\psi$ are indeed mutually inverse bijections.
 
 This phenomenon is particular to ``rank $1$''. The inclusion $B[T_1,T_2]\to B[T_1^{\pm1},T_2^{\pm2}]$ defines an injection from the set of $B[T_1,T_2]$-matroids into the set of $B[T_1^{\pm1},T_2^{\pm2}]$-matroids which fails to be surjective. For instance, consider $E=\{1,2,3,4\}$ and the Grassmann-Pl\"ucker function $\Delta:\binom E2\to B[T_1^{\pm1},T_2^{\pm2}]$ with 
 \[
  \Delta_{1,2}=\Delta_{1,3}=T_1, \quad \Delta_{2,4}=T_2, \quad \Delta_{3,4}=\epsilon T_2 \quad \text{and} \quad \Delta_{1,4}=\Delta_{2,3}=0
 \]
 where $\Delta_{i,j}=\Delta(\{i,j\})$. Then the $B[T_1^{\pm1},T_2^{\pm2}]$-matroid $M=[\Delta]$ is not the pushforward of a $B[T_1,T_2]$-matroid, since it there is no unit $a$ in $B[T_1^{\pm1},T_2^{\pm2}]$ such that $aT_1,aT_2\in B[T_1,T_2]$ with one of $aT_1$ and $aT_2$ invertible in $B[T_1,T_2]$.

\end{ex}

\begin{ex}
 Let $\cO_\T$ be the \emph{hyperring of tropical integers}, which is the subhyperring of $\T$ whose underlying set is the unit interval $[0,1]$. As an ordered blueprint, it can be described as
 \[
  \cO_\T \ = \ \bpgenquot{[0,1]}{c\leq a+b\, | \, c=\max\{a,b\}\text{ or }a=b\text{ and }c\in[0,a]}. 
 \]
 For similar reasons as in Example \ref{ex: matroids for free algebras over tracts}, the inclusion $\cO_\T\to \T$ defines a bijection
 \[
  \big\{\cO_\T\text{-matroids}\big\} \quad \longrightarrow \quad \big\{\T\text{-matroids}\big\}.
 \]
\end{ex}


\bigskip
\part{Constructing moduli spaces of matroids}
\nopagebreak[4]
\bigskip

\section{Projective geometry for ordered blueprints}
\label{section: Projective geometry for ordered blueprints}

In this section, we review the definition of an ordered blue scheme from \cite{Lorscheid15} and extend the $\Proj$ functor from \cite{LopezPena-Lorscheid12} to ordered blue schemes. This makes it possible to define projective schemes, like Grassmannians, in terms of homogeneous algebras and to characterize morphisms to projective space in terms of invertible sheaves together with a fixed set of global sections.


\subsection{Ordered blue schemes}
\label{subsection: Ordered blue schemes}


In this section, we introduce the spectrum of an ordered blueprint as the space of its prime ideals together with a structure sheaf. Glueing the spectra of ordered blueprints in an appropriate sense leads to the notion of ordered blue schemes.


\subsubsection{Localizations}
\label{subsubsection: localizations}

Let $B$ be an ordered blueprint. Recall that we consider the monoid $B^\bullet$ as the underlying set of $B$. A \emph{multiplicative subset of $B$} is a multiplicatively closed subset of $B$ that contains $1$, i.e.\ a submonoid of $B^\bullet$.

Let $S$ be a multiplicative subset of $B$. The \emph{localization of $B$ at $S$} is the ordered blueprint $S^{-1}B=\bpquot{S^{-1}B^\bullet}{\cR_S}$ where $S^{-1}B^\bullet=\{\frac as \; | \; a\in B,s\in S\}$ is the localization of the monoid $B^\bullet$ at $S$, i.e.\ $\frac as=\frac{a'}{s'}$ if and only if there is a $t\in S$ such that $tsa'=ts'a$, and where 
\[
 \cR_S \ = \ \biggl\langle \, \sum \frac{a_i}1 \= \sum \frac{b_j}1 \, \biggl| \, \sum a_i\=\sum b_j \text{ in }B \, \biggr\rangle.
\]
The localization of $B$ at $S$ comes together with a morphism $\iota_S:B\to S^{-1}B$ that sends $a$ to $\frac a1$, and which satisfies the usual universal property: for every morphism $f:B\to C$ with $f(S)\subset C^\times$, there is a unique morphism $f_S:S^{-1}B\to C$ such that $f=f_S\circ\iota_S$.

\begin{ex}\label{ex: Laurent monomials as localization of monomial blueprints}
 Let $B[T]$ be the free algebra in $T$ over an ordered blueprint $B$ and $S=\{T^i\}_{i\geq0}$. The localization $S^{-1}B[T]$ is the ordered blueprint
 \[ \textstyle
  B[T^{\pm1}] \ = \ \bpbiggenquot{\big(B\times\{T^k\}_{k\in \Z}\big)}{\sum a_iT^k \leq \sum b_jT^k \, \big| \, \sum a_i\leq\sum b_j\text{ in }B, k\in\Z}.
 \]
 Analogously, we define $B[T_1^{\pm1},\dotsc,T_n^{\pm1}]$ as the localization of the free algebra $B[T_1,\dotsc,T_n]$ in several variables $T_1,\dotsc, T_n$ at the multiplicative subset generated by $\{T_1,\dotsc,T_n\}$.
\end{ex}


\subsubsection{Ideals}
\label{subsubsection: ideals}

Let $B$ be an ordered blueprint. A \emph{monoid ideal}, or \emph{$m$-ideal}, \emph{of $B$} is a subset $I$ of $B$ such that $0\in I$ and $IB=I$. 

An $m$-ideal $I\subset B$ is \emph{proper} if $I\neq B$. Note that every ordered blueprint $B$ has a unique maximal proper $m$-ideal $\fm=B-B^\times$. In this sense, every ordered blueprint is local with respect to $m$-ideals.

A \emph{prime $m$-ideal of $B$} is an $m$-ideal $\fp$ of $B$ such that its complement $S=B-\fp$ is a multiplicative subset. The \emph{localization of $B$ at $\fp$} is $B_\fp=S^{-1}B$. Note that the maximal $m$-ideal of an ordered blueprint is a prime $m$-ideal. The maximal $m$-ideal of the localization $B_\fp$ at a prime $m$-ideal is $\fp B_\fp$. 

Let $T=\{a_i\}_{i\in I}$ be a subset of $B$. The \emph{$m$-ideal $\gen{T}=\gen{a_i}_{i\in I}$ generated by $T$} is the smallest $m$-ideal that contains $T$, which is equal to
\[
 \gen{T} \quad = \quad \bigcap_{\substack{I\subset B\text{ $m$-ideal}\\\text{with }T\subset I}} I \quad = \quad \big\{\, ab \, \big| \, a\in T\cup\{0\}, b\in B \, \big\}.
\]

\begin{lemma} \label{lemma: generators of prime ideals}
 Let $B$ be an ordered blueprint and $T$ a subset of $B$ such that $B^\bullet$ is generated by $T\cup B^\times\cup\{0\}$ as a monoid. Then every prime $m$-ideal $\fp$ of $B$ is generated by a subset of $T$.
\end{lemma}

\begin{proof}
 Consider an element $b\in \fp$, which can be written as the product $b=ua_1\cdots a_n$ with $u\in B^\times$ and $a_1,\cdots,a_n\in T$ by the hypothesis on $T$. Then also $u^{-1}b=a_1\dotsc a_n$ is an element of $\fp$. Since the complement of $\fp$ in $B$ is multiplicatively closed, one of $a_1,\dotsc,a_n$ must be in $\fp$. This shows that $\fp$ is generated by a subset of $T$.
\end{proof}

\begin{ex}
 Let $B$ be an ordered blue field, e.g.\ an ordered blueprint associated with a tract. Then the underlying monoid of the free algebra $B[T_1,\dotsc,T_n]$ is generated by $T=\{T_1,\dotsc,T_n\}$ over $B=B^\times\cup\{0\}$, i.e.\ $T$ satisfies the hypothesis of Lemma \ref{lemma: generators of prime ideals}. Thus every prime $m$-ideal of $B[T_1,\dotsc,T_n]$ is generated by a subset $J$ of $T$. It is easily verified that $\fp_J=\gen{J}$ is indeed a prime $m$-ideal for every $J\subset T$.
\end{ex}


\subsubsection{Ordered blueprinted spaces}
\label{subsubsection: Ordered blueprinted spaces}

An \emph{ordered blueprinted space}, or for short an \emph{$\OBlpr$-space}, is a topological space $X$ together with a sheaf $\cO_X$ in $\OBlpr$. In practice, we suppress the \emph{structure sheaf $\cO_X$} from the notation and denote an $\OBlpr$-space by the same symbol $X$ as its \emph{underlying topological space}. 

For every point $x$ of $X$, the \emph{stalk at $x$} is the colimit $\cO_{X,x}=\colim\cO_X(U)$ over the system of all open neighbourhoods $U$ of $x$. 
Note that this colimit always exists since $\OBlpr$ contains all small limits and colimits.

A \emph{morphism of $\OBlpr$-spaces} is a continuous map $\varphi:X\to Y$ between the underlying topological spaces together with a morphism $\varphi^\#:\varphi^{-1}\cO_Y\to\cO_X$ of sheaves on $X$ that is \emph{local} in the following sense: for every $x\in X$ and $y=\varphi(x)$, the induced morphism $\cO_{Y,y}\to\cO_{X,x}$ of stalks sends non-units to non-units. This defines the category $\OBlprSp$ of ordered blueprinted spaces.


\subsubsection{The spectrum}
\label{subsubsection: the spectrum}

Let $B$ be an ordered blueprint. We define the \emph{spectrum $\Spec B$ of $B$} as the following ordered blueprinted space. The topological space of $X=\Spec B$ consists of the prime $m$-ideals of $B$ and comes with the topology generated by the \emph{principal opens}
\[
 U_h \ = \ \{\, \fp\in\Spec B \,|\, h\notin \fp\, \}
\]
where $h$ varies through the elements of $B$. Note that the principal opens form a basis of the topology for $X$ since $U_h\cap U_g=U_{gh}$. The structure sheaf $\cO_X$ on $X=\Spec B$ is determined by the following theorem:

\begin{prop}\label{prop: structure sheaf for the spectrum of an ordered blueprint}
 There is a unique sheaf of $B$-algebras $\cO_X$ on $X=\Spec B$ such that $\cO_X(U_h)=B[h^{-1}]$ for all $h \in B$. The stalk at a prime ideal $\fp$ of $B$ is $\cO_{X,\fp}=B_\fp$.
\end{prop}

\begin{proof}
 The proof is similar to the case for monoid schemes, cf.\ \cite[Prop.\ 2.1]{Deitmar05}. We briefly outline the arguments.
 
 The uniqueness of $\cO_X$ is clear since the open subsets $U_h$ form a basis for the topology of $X$ and $B$-linearity uniquely determines the restriction maps $B[h^{-1}]\to B[g^{-1}]$ for $U_g\subset U_h$. 

 For existence, we construct $\cO_X$ in the usual way. Let $U$ be an open subset of $X$. A {section on $U$} is a function $s:U\to \coprod_{\fp\in U}B_\fp$ such that: (a) $s(\fp)\in B_\fp$ for all $\fp$; and (b) there is a finite open covering $\{U_{h_i}\}$ of $U$ by principal open subsets $U_{h_i}$ and elements $a_i\in B[h_i^{-1}]$ whose respective images in $B_\fp$ equal $s(\fp)$ whenever $\fp\in U_{h_i}$. We define $O_{X}(U)$ to be the set of sections on $U$, which comes naturally with the structure of an ordered blueprint. The restriction of sections to subsets yield $B$-linear morphisms as desired.
 
 In order to see that $\cO_X(U_h)=B[h^{-1}]$, note that $U_h$ has a unique maximal point, which is the prime ideal $\fp_h$ consisting of all elements $a\in B$ that are not divisible by $h$. Thus every covering $U_i$ of $U_h$ must contain one subset $U_i$ that contains $\fp_h$, which means that $U_i=U_h$. Thus a section $s\in\cO_X(U_h)$ is represented by an element $a/h\in B[h^{-1}]$, which establishes the claim.
 
 From the construction of $\cO_X$, it is immediate that $\cO_{X,\fp}=B_\fp$.
\end{proof}

A morphism $f:B\to C$ of ordered blueprints defines a morphism $f^\ast:\Spec C\to \Spec B$ of $\OBlpr$-spaces by taking the inverse image of prime $m$-ideals and pulling back sections. This defines the contravariant functor
\[
 \Spec: \ \OBlpr \ \longrightarrow \ \OBlprSp.
\]
We call $\OBlpr$-spaces in the essential image of this functor \emph{affine ordered blue schemes}. 

\begin{ex}[Blue affine spaces]
 Let $k$ be an ordered blue field. The affine ordered blue scheme $\A^n_k=\Spec k[T_1,\dotsc,T_n]$ consists of the set $\A^n_k=\{\fp_I \; | \; I\subset\{1,\dotsc,n\}\}$ of all prime $m$-ideals $\fp_I=(T_i)_{i\in I}$ of $k[T_1,\dotsc,T_n]$. A subset $U$ of $\A^n_k$ is open if and only if for every $J\subset I$ such that $\fp_I\in U$ also $\fp_J\in U$. Let $h=\prod_{i\in I}T_i$ for some $I\subset\{1,\dotsc,n\}$. Then the value of the structure sheaf on the principal open
 \[
  U_h \ = \ \big\{ \, \fp_J \, \big| \, J\cap I=\emptyset \, \big\}
 \]
 is $k[T_1,\dotsc,T_n][T_i^{-1}]_{i\in I}$. We call $\A^n_k$ the \emph{(ordered blue) affine space} since it plays the analogous role for ordered blue schemes as affine spaces in classical algebraic geometry. We illustrate $\A^1_k$ and $\A^2_k$ in Figure \ref{figure: affine spaces}.
 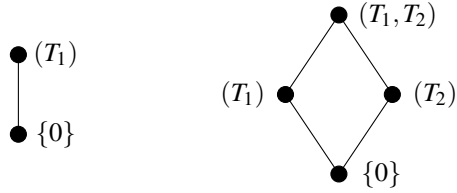
\begin{figure}[ht]
  \[
   \beginpgfgraphicnamed{tikz/fig9}
   \begin{tikzpicture}[inner sep=0,x=20pt,y=15pt,font=\small]
    \filldraw (0,1) circle (3pt);
    \filldraw (0,3) circle (3pt);
    \draw (0,1) -- (0,3);
    \node at (0.7,3) {$(T_1)$};
    \node at (0.7,1) {$\{0\}$};
    \filldraw (5,2) circle (3pt);
    \filldraw (6,0) circle (3pt);
    \filldraw (6,4) circle (3pt);
    \filldraw (7,2) circle (3pt);
    \draw (5,2) -- (6,0) -- (7,2) -- (6,4) -- (5,2);
    \node at (6.8,0) {$\{0\}$};
    \node at (7.1,4) {$(T_1,T_2)$};
    \node at (7.8,2) {$(T_2)$};
    \node at (4.2,2) {$(T_1)$};
   \end{tikzpicture}
   \endpgfgraphicnamed
  \] 
  \caption{The affine space $\A^1_k$ and $\A^2_k$ over an ordered blue field $k$}
  \label{figure: affine spaces}
 \end{figure}

\end{ex}


\subsubsection{Ordered blue schemes}
\label{subsubsection: ordered blue schemes}

An \emph{open subspace} of an $\OBlpr$-space $X$ is an open subset $U$ together with the restriction of the structure sheaf $\cO_X$ of $X$ to $U$. 
An \emph{open covering of $X$} is a collection of open subspaces $U_i$ such that $X$ is covered by the $U_i$ as a topological space. An open covering $\{U_i\}$ is \emph{affine} if every $U_i$ is affine.

An \emph{ordered blue scheme} is an $\OBlpr$-space that has an open covering by affine ordered blue schemes $U_i$. A \emph{morphism of ordered blue schemes} is a morphism of $\OBlpr$-spaces. We denote the category of ordered blue schemes by $\OBSch_\Fun$.

We mention some fundamental properties of ordered blue schemes. These facts can be proven in the same way as the corresponding facts in usual algebraic geometry, but the proof is substantially easier due to the fact that an affine ordered blue scheme $\Spec B$ has a unique closed point, which is the maximal ideal of $B$. We restrict ourselves to brief outlines of the main ideas of those proofs, which follow the same line of thoughts as their counterparts in usual algebraic geometry.

\begin{prop}
 Let $X$ be an ordered blue scheme. The collection of open subspaces of $X$ that are affine ordered blue schemes form a basis for the topology of $X$.
\end{prop}

\begin{proof}
 This is true for every affine open subscheme. Since $X$ can be covered by affine open subschemes, the result follows.
\end{proof}

We define the contravariant functor
\[
 \Gamma:\OBSch \to \OBlpr
\]
of taking global sections by $\Gamma X=\cO_X(X)$ for an ordered blue scheme $X$ and $\varphi^\#(Y):\Gamma Y\to \Gamma X$ for a morphism $\varphi:X\to Y$ of ordered blue schemes. 

\begin{prop}
 The functor $\Gamma$ is a left inverse to $\Spec$. In particular, sending an element $a\in B$ to the constant section on $\Spec B$ with value $a$ defines an isomorphism $B\to\Gamma\Spec B$. 
\end{prop}

\begin{proof}
 This follows from the fact that every ordered blueprint $B$ has a unique maximal $m$-ideal, which is $\fm=B-B^\times$. Therefore $X=\Spec B$ has a unique closed point and every open covering of $X$ must contain $X$ itself. Thus every global section of $X$ comes from $B$ and $B\to\Gamma\Spec B$ is an isomorphism.
\end{proof}

\begin{lemma} \label{lemma: adjunction between ordered blueprints and ordered blue schemes}
 Let $B$ be an ordered blueprint and $X$ an ordered blue scheme. Then the canonical map $\Hom(X,\Spec B)\to\Hom(B,\Gamma X)$ is a bijection.
\end{lemma}

\begin{proof}
 The canonical map $\Hom(X,\Spec B)\to\Hom(\Gamma B,\Gamma X)$ associates with a morphism $\varphi:X\to \Spec B$ the morphism $\varphi^\ast:\Gamma B\to\Gamma X$ between the respective ordered blueprints of global sections. Moreover, the composition of a morphism $f:B\to \Gamma X$ with the globalization $\sigma_B:B\to \Gamma B$ defines a bijection $\Hom(\Gamma B, \Gamma X)\to \Hom(B,\Gamma X)$. Thus we may without loss of generality replace $\Gamma B$ by $B$.
 
 Given a morphism $f:B\to \Gamma X$, we can cover $X$ by affine opens $U_i$ and obtain induced morphisms
 \[
  f_i: \ B \ \stackrel f \longrightarrow \ \Gamma X \ \stackrel{\res_{X,U_i}}\longrightarrow \ \Gamma U_i \qquad \text{and} \qquad f_{i,j}: \ B \ \stackrel f \longrightarrow \ \Gamma X \ \stackrel{\res_{X,U_{i,j}}}\longrightarrow \ \Gamma U_{i,j} 
 \]
 where $U_{i,j}=U_i\cap U_j$. This defines morphisms $\varphi_i=f_i^\ast:U_i\to \Spec B$ and $\varphi_{i,j}=f_{i,j}^\ast: U_{i,j}\to \Spec B$ between affine ordered blue schemes. Since the diagrams
 \[
  \beginpgfgraphicnamed{tikz/fig2}
   \begin{tikzcd}[column sep=2.5cm, row sep=-0.1cm]
                            & U_i \ar[rd,"\varphi_i" above] \\ 
    U_{i,j} \ar[ur] \ar[dr] &                              & \Spec B \\
                            & U_j \ar[ru,"\varphi_j" below]   
   \end{tikzcd}
  \endpgfgraphicnamed
 \]
 commute for all $i$ and $j$, the morphisms $\varphi_i$ glue to a morphism $\varphi:X\to \Spec B$. Since a global section $s\in\Gamma X$ is determined by its restrictions $s_i$ to $U_i$, we conclude that $\Gamma\varphi = f$.
 
 Conversely, every morphism $\varphi:X\to \Spec B$ is determined by its restrictions $\varphi_i:U_i\to \Spec B$. Since we have $\Gamma\varphi_i=\res_{X,U_i}\circ\Gamma\varphi$, we see that we reobtain $\varphi$ from the above construction applied to $f=\Gamma\varphi$. This verifies that the canonical map $\Hom(X,\Spec B)\to\Hom(B,\Gamma X)$ is a bijection.
\end{proof}

\begin{rem}
 Since the bijection $\Hom(X,\Spec B)\to\Hom(B,\Gamma X)$ is natural in $B$ and $X$, we have in fact proven that there is an adjunction
 \[
  \beginpgfgraphicnamed{tikz/fig3}
   \begin{tikzcd}[column sep=3cm]
    \OBlpr^\op \ar[shift left=1.0ex]{r}{\Spec\circ(-)^\op}[below=-1pt]{\perp}  &  \OBSch. \ar[shift left=1.0ex]{l}{(-)^\op\circ\Gamma}
   \end{tikzcd}
  \endpgfgraphicnamed
 \]
\end{rem}

\begin{prop}\label{prop: morphisms are locally algebraic}
 Let $\varphi:X\to Y$ be a morphism of ordered blue schemes and $\{V_j\}_{j\in J}$ an affine open covering of $Y$. Then there is an affine open covering $\{U_i\}_{i\in I}$ of $X$, a map $f:I\to J$ and for every $i\in I$ a morphism of ordered blueprints $f_i:\Gamma V_{f(i)}\to\Gamma U_i$ such that $\varphi(U_i)\subset V_{f(i)}$ and such that the restriction of $\varphi$ to $U_i\to V_{f(i)}$ is equal to $f_i^\ast$ for every $i\in I$.
\end{prop}

\begin{proof}
 For every $j\in J$, we can cover the open subset $\varphi^{-1}(V_j)$ of $X$ by affine opens $U_i$ of $X$ and define $f(i)=j$ for those indices $i$. Let $I$ be the set of all such indices $i$ (for varying $j$). Then $\{U_i\}_{i\in I}$ is an affine open covering of $X$, and we obtain a map $f:I\to J$ such that $\varphi(U_i)\subset V_{f(i)}$ for all $i\in I$. This defines blueprint morphisms $f_i:\Gamma V_{f(i)}\to\Gamma U_i$ for every $i\in I$. By Lemma \ref{lemma: adjunction between ordered blueprints and ordered blue schemes}, the restriction of $\varphi$ to $U_i\to V_{f(i)}$ is equal to $f_i^\ast$, which concludes the proof of the proposition.
\end{proof}

With these results at hand, we can transfer the usual construction of fibre products to the case of ordered blue schemes. We omit the proof of this result.

\begin{thm}
 $\OBSch_\Fun$ contains fibre products. In particular, we have $X\times_ZY\simeq \Spec (B\otimes_DC)$ for affine ordered blue schemes $X=\Spec B$, $Y=\Spec C$ and $Z=\Spec D$. \qed
\end{thm}

The most interesting class of non-affine blue schemes will be projective blue schemes, for which we introduce the $\Proj$-construction in section \ref{subsection: The Proj-construction}. For examples of projective blue schemes, we refer the reader to Example \ref{ex: projective line and projective plane}.


\subsubsection{Open and closed subschemes}
\label{subsubsection: open and closed subschemes}

Let $X$ be an ordered blue scheme. An \emph{open subscheme of $X$} is an open subspace of $X$ as an $\OBlpr$-space. An \emph{open immersion of ordered blue schemes} is a morphism of ordered blue schemes $\varphi:Y\to X$ that restricts to an isomorphism with an open subscheme of $X$.

A morphism $\varphi:Y\to X$ is \emph{affine} if for every affine open subscheme $U$ of $X$, the inverse image $\varphi^{-1}(U)$ in $Y$ is affine. As in usual scheme theory, this can be tested on an affine open covering of $X$, i.e.\ given an open and affine covering $\{U_i\}$ of $X$, then $\varphi$ is affine if and only if $\varphi^{-1}(U_i)$ is affine for all $i$.

A \emph{closed immersion of ordered blue schemes} is an affine morphism $Y\to X$ such that for every affine open subset $U$ of $X$ and its inverse image $V=\varphi^{-1}(U)$, the map $\varphi^\#(U):\cO_X(U)\to\cO_Y(V)$ of ordered blueprints is surjective. Again, this can be tested on an affine open covering of $X$, i.e.\ given an open and affine covering $\{U_i\}$ of $X$, an affine morphism $\varphi:Y\to X$ is a closed immersion if and only if $\Gamma U_i \to \Gamma \varphi^{-1}(U_i)$ is surjective for all $i$.

A \emph{closed subscheme of $X$} is an equivalence class of closed immersions $\varphi:Y\to X$, where two closed immersions $\varphi_i:Y_i\to X$ ($i=1,2$) are equivalent if there is an isomorphism $\psi:Y_1\to Y_2$ such that $\varphi_2\circ\psi=\varphi_1$.





\subsection{The \texorpdfstring{$\Proj$}{Proj}-construction}
\label{subsection: The Proj-construction}

Similar to usual algebraic geometry, it is possible to define a $\Proj$-construction that associates with a graded ordered blueprint a projective ordered blue scheme. The following is an adaptation of the $\Proj$-construction for $k$-ideals, as treated in \cite{LopezPena-Lorscheid12}, to $m$-ideals.

A \emph{graded ordered blueprint} is an ordered blueprint $B$ together with a family $\{B_i\}_{i\in\N}$ of subsets $B_i$ of $B$ such that $B=\bigcup_{i\in\N}B_i$, $B_i\cap B_j=\{0\}$ for $i\neq j$ and $ab\in B_{i+j}$ for all $a\in B_i$ and $b\in B_j$. The subset $B_i$ is called the \emph{$i$-th homogeneous part of $B$}. We write $B=\bigvee B_i$ if $B$ is a graded ordered blueprint with homogeneous parts $B_i$. A nonzero element of $B_i$ is called \emph{homogeneous of degree $i$}.

Let $S$ be a multiplicative subset of $B$. If $b/s$ is a nonzero element of the localization $S^{-1}B$ where $b$ is homogeneous of degree $i$ and $s$ is homogeneous of degree $j$, we say that $b/s$ is homogeneous of degree $i-j$. We define $(S^{-1}B)_0$ to be the subset of homogeneous elements of degree $0$. It is multiplicatively closed, and thus inherits the structure of an ordered blueprint from $S^{-1}B$. If $S$ is the complement of a prime $m$-ideal $\fp$, we write $B_{(\fp)}$ for the subblueprint $(B_\fp)_0$ of homogeneous elements of degree $0$ in $B_\fp$.

Let $B$ be a graded blueprint. We define $\Proj B$ to be the set of all homogeneous prime $m$-ideals $\fp$ of $B$ which are \emph{relevant}, i.e. which do not contain $B_{>0}=\bigcup_{i>0}B_i$. The set $X=\Proj B$ comes with the topology defined by the basis 
\[
 U_h \quad = \quad \{ \ \fp\in X\ | \ h\notin \fp \ \},
\]
where $h$ ranges through $B$, and with a structure sheaf $\cO_X$ that maps an open subset $U$ of $X$ to the set of \emph{locally represented sections on $U$}, which is the set of maps
$s: \ U \ \longrightarrow \ \coprod_{\fp\in U} \ B_{(\fp)}$
such that (a) for every $\fp\in U$, we have $s(\fp)\in B_{(\fp)}$; and (b) for every $\fp\in U$, there are an open neighbourhood $V\subset U$ of $\fp$ and elements $a,h\in B_i$ for some $i\in\N$ such that for all $\fq\in V$, we have $h\notin \fq$ and $s(\fq)=\frac ah$ in $B_{(\fq)}$.

The following theorem is proven as in the case of schemes; for instance, cf.\ \cite[Prop.\ 2.5]{Hartshorne77}:

\begin{thm}
 Let $B=\bigvee B_i$ be a graded ordered blueprint. Then $X=\Proj B$ is an ordered blue scheme. The stalk of $\cO_X$ at a point $\fp\in\Proj B$ is isomorphic to $B_{(\fp)}$. For every $h\in B_{>0}$, the open subscheme $U_h$ is isomorphic to $\Spec B[h^{-1}]_0$. \qed 
\end{thm}

Since none of the prime ideals $\fp$ in $\Proj B$ contain $B_{>0}$, we have $\Proj B=\bigcup_{h\in B_{>0}} U_h$. Consequently, the inclusions $B_0\hookrightarrow B[h^{-1}]_0$ yield morphisms $\Spec B[h^{-1}]_0\to \Spec B_0$, which glue to a morphism $\Proj B \to \Spec B_0$.

More generally, if $B$ is an ordered blueprint and $C=\bigvee C_i$ is a \emph{graded $B$-algebra}, i.e.\ a graded ordered blueprint together with morphism $B\to C_0$ of ordered blueprints, then $\Proj C$ comes together with the \emph{structural morphism} $\Proj C\to\Spec C_0\to\Spec B$. This construction is functorial in the following sense.

A \emph{morphism of graded $B$-algebras} is a morphism $f:C\to D$ of ordered blueprints between graded ordered $B$-algebras $C=\bigvee C_i$ and $D=\bigvee D_i$ with $f(C_i)\subset D_i$ for all $i\geq0$ that commutes with the morphisms $B\to C_0$ and $B\to D_0$. Given a morphism $f:C\to D$ of graded $B$-algebras, this defines a \emph{rational map} $f^\ast:\Proj D\to \Proj C$ of ordered blue $B$-schemes, i.e.\ a morphism $f^\ast:U\to\Proj C$ where $U$ is an open subet of $\Proj C$, by mapping a homogeneous prime $m$-ideal $\fp$ of $D$ to $f^{-1}(\fp)$ and by pulling back functions in the structure sheaf of $\Proj C$. Note that the domain of $f^\ast$ is the open subset $U\subset\Proj C$ of all homogeneous prime $m$-ideals $\fp$ such that $f^{-1}(\fp)$ is relevant.


\subsubsection{Projective space}
\label{subsubsection: projective space}

The functor $\Proj$ leads to the definition of the projective space $\P^n_B$ over an ordered blueprint $B$. Namely, the free algebra $C=B[T_0,\dotsc,T_n]$ over $B$ comes together with a natural grading where $C_i$ consists of all monomials $bT_0^{e_0}\dotsb T_n^{e_n}$ such that $e_0+\dotsb+e_n=i$ and $b\in B$. Note that $C_0=B$. The projective space $\P^n_B$ is defined as $\Proj B[T_0,\dotsc,T_n]$. It comes together with a structure morphism $\P^n_B\to\Spec B$. It is covered by the principal opens $U_i=U_{T_i}$ for $i=0,\dotsc,n$, which are isomorphic to affine $n$-spaces over $B$.

In the case $B=\Fun$, the projective space $\P^n_\Fun$ is the monoid scheme that is known from $\Fun$-geometry (see \cite[Section 3.1.4]{Chu-Lorscheid-Santhanam12}, \cite{Deitmar08} and \cite[Ex.\ 1.6]{Lorscheid18b}). The topological space of $\P^n_\Fun$ is finite. Its points correspond to the homogeneous prime $m$-ideals $(T_i)_{i\in I}$ of $\Fun[T_0,\dotsc,T_n]$, where $I$ ranges through all proper subsets of $\{0,\dotsc,n\}$. 

Note that when $B$ is a ring, the ordered blue projective space $\P^n_B$ does not coincide with the usual projective space, since the free ordered blueprint $B[T_0,\dotsc,T_n]$ is not a ring, but merely the ordered blueprint of all monomials of the form $bT_0^{e_0}\dotsb T_n^{e_n}$ with $b\in B$. However, the associated scheme $(\P^n_B)^+$ coincides with the usual projective space over $B$. The relation between the matroid space $\Mat(r,E)$, as defined in section \ref{subsection: The moduli space of matroids}, and the usual Grassmannian $\Gr(r,E)$ is similar, cf.\ Remark \ref{rem: relation of matroid spaces with Grassmannians} for details.

\begin{ex}[The projective line and the projective plane]\label{ex: projective line and projective plane}
 Let $k$ be an ordered blue field. We can label the points of $\P^n_k$ by homogeneous coordinates: we denote a homogeneous prime $m$-ideal $(T_i)_{i\in I}$ by $[a_0:\dotsb:a_n]$ with $a_i=0$ if $i\in I$ and $a_i=1$ otherwise. Note, however, that the real meaning of $a_i=1$ is $a_i\neq0$, i.e.\ a coefficient $a_i=1$ denotes a generic value. Therefore $[1:\dotsb:1]$ is the generic point of $\P^n_k$. We illustrate the points and their homogeneous coordinates for the projective line $\P^1_k$ and the projective plane $\P^2_k$ over $k$ in Figure \ref{figure: projective line and projective plane}. 
  \begin{figure}[ht]
  \[
   \beginpgfgraphicnamed{tikz/fig4}
   \begin{tikzpicture}[inner sep=0,x=30pt,y=20pt,font=\footnotesize,back line/.style={},cross line/.style={preaction={draw=white, -,line width=6pt}}]
    \filldraw (1,1) circle (3pt);
    \filldraw (0,3) circle (3pt);
    \filldraw (2,3) circle (3pt);
    \draw (0,3) -- (1,1) -- (2,3);
    \draw (6,4) -- (6,2) -- (8,4) -- (10,2) -- (10,4);
    \draw [cross line] (6,4) -- (8,2) -- (10,4);
    \draw (8,0) -- (6,2);
    \draw (8,0) -- (8,2);
    \draw (8,0) -- (10,2);
    \filldraw (8,0) circle (3pt);
    \filldraw (6,2) circle (3pt);
    \filldraw (8,2) circle (3pt);
    \filldraw (10,2) circle (3pt);
    \filldraw (6,4) circle (3pt);
    \filldraw (8,4) circle (3pt);
    \filldraw (10,4) circle (3pt);
    \node at (1.6,1) {$[1:1]$};
    \node at (-0.6,3) {$[1:0]$};
    \node at (2.6,3) {$[0:1]$};
    \node at (8.9,0) {$[1:1:1]$};
    \node at (5.2,2) {$[1:1:0]$};
    \node at (8.9,2) {$[1:0:1]$};
    \node at (10.8,2) {$[0:1:1]$};
    \node at (5.2,4) {$[1:0:0]$};
    \node at (8.9,4) {$[0:1:0]$};
    \node at (10.8,4) {$[0:0:1]$};
   \end{tikzpicture}
   \endpgfgraphicnamed
  \] 
  \caption{The projective line $\P^1_k$ and the projective plane $\P^2_k$ over $k$}
  \label{figure: projective line and projective plane}
 \end{figure}
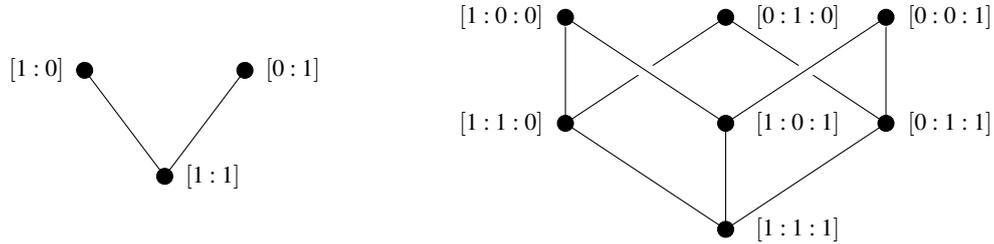
\end{ex}


\subsubsection{Closed subschemes of projective space}
\label{subsubsection: closed subschemes of projective space}

The closed subschemes of the ordered blue projective $n$-space $\P^n_B$ over an ordered blueprint $B$ correspond to quotients of the free $B$-algebra $B[T_0,\dotsc,T_n]$ by {homogeneous relations}, as we will explain in the following.

A \emph{homogeneous relation on $B[T_0,\dotsc,T_n]$} is a relation of the form $\sum a_{i}\leq \sum b_{j}$ where $a_i,b_j\in B_{k}$ are homogeneous of the same degree $k$. If $S$ is a set of homogeneous relations, then the quotient map $B[T_0,\dotsc,T_n]\to\bpgenquot{B[T_0,\dotsc,T_n]}{S}$ is a surjective morphism of graded $B$-algebras. For such morphisms, the following holds true.

\begin{lemma}\label{lemma: closed immersions to projective space}
 Let $f:B[T_0,\dotsc,T_n]\to C$ be a surjective morphism of graded $B$-algebras. Then $f^\ast:\Proj C\to \P^n_B$ is a closed immersion of ordered blue schemes.
\end{lemma}

\begin{proof}
 Let us write $D=B[T_0,\dotsc,T_n]$ and $\varphi=f^\ast$ for short. The property of $f$ being a closed immersion can be tested on the open affine covering of $\P^n_B$ by $U_i=\Spec D[T_i^{-1}]_0$. The inverse image $\varphi^{-1}(U_i)$ is the principal open $U_{h_i}\simeq \Spec C[h_i^{-1}]_0$ of $\Proj C$ where $h_i=f(T_i)$. Thus $\varphi$ is affine. The induced morphisms $f_i:C[h_i^{-1}]_0\to D[T_i^{-1}]_0$ are surjective since $f$ is. This shows that $\varphi$ is a closed immersion.
\end{proof}

\begin{rem}
 The converse of Lemma \ref{lemma: closed immersions to projective space} holds true as well, i.e.\ every closed $B$-subscheme of $\P^n_B$ comes from a surjective morphism of graded $B$-algebras. This can be proven as in the case of usual schemes. 
 %
\end{rem}

\subsection{Invertible sheaves}
\label{subsection: invertible sheaves}

Loosely speaking, an invertible sheaf on an ordered blue scheme $X$ is a sheaf that is locally isomorphic to the structure sheaf $\cO_X$ of $X$. To give this definition a precise meaning, we need to introduce ordered blue modules.

\subsubsection{Ordered blue modules}
\label{subsubsection: Ordered blue modules}

An \emph{ordered semigroup} is a commutative and unital semigroup $(M,+)$ together with a partial order $\leq$ that is compatible with the addition, i.e.\ $m\leq n$ and $p\leq q$ implies $m+p\leq n+q$. An \emph{ordered blue module} is an ordered semigroup $M^+$ with neutral element $0$ together with a subset $M^\bullet$ that contains $0$ and generates $M^+$ as a semigroup. We write $M=(M^\bullet,M^+,\leq)$ for an ordered blue module. A \emph{morphism of ordered blue modules} is an order preserving homomorphism $f:M^+\to N^+$ of semigroups with $f(0)=0$ and $f(M^\bullet)\subset N^\bullet$. We denote the category of ordered blue modules by $\OBMod$.

Let $B$ be an ordered blueprint. 
An \emph{ordered blue $B$-module} is a map $B^+\times M^+\to M^+$ that maps $(a,m)$ to $a.m$ and which satisfies, for all $a,b\in B$ and $m,n\in M$:
\begin{enumerate}
 \item $a.m\in M^\bullet$ if $a\in B^\bullet$ and $m\in M^\bullet$;
 \item $0.m=0$, $1.m=m$ and $a.0=0$;
 \item $(ab).m=a.(b.m)$, $(a+b).m=a.m+b.m$ and $a.(m+n)=a.m+a.n$;
 \item $a.m\leq b.n$ if $a\leq b$ and $m\leq n$.
\end{enumerate}
A \emph{morphism of ordered blue $B$-modules} is a morphism $f:M\to N$ of ordered blue modules such that the resulting diagram
\[
  \beginpgfgraphicnamed{tikz/fig5}
   \begin{tikzcd}[column sep=4em, row sep=1em]
    B\times M \ar{r}\ar{d} & M\ar{d} \\
    B\times N \ar{r} & N
   \end{tikzcd}
  \endpgfgraphicnamed
\]
commutes. We denote the category of ordered blue $B$-modules by $\OBMod B$.

Note that every ordered blue module has a unique structure as an $\Fun$-module. Thus $\OBMod$ is equivalent to $\OBMod\Fun$. 

We remark that the category $\OBMod B$ is complete and cocomplete for every ordered blueprint $B$. In particular, the categorical product of $M$ and $N$ is the Cartesian product $M\times N$.

\subsubsection{\texorpdfstring{$\cO_X$}{OX}-modules}
\label{subsubsection: OX-modules}

Let $X$ be an ordered blue scheme. A \emph{sheaf of ordered blue modules on $X$} is a sheaf on $X$ with values in $\OBMod$. For example, the structure sheaf $\cO_X$ is naturally a sheaf of ordered blue modules. Note that products of sheaves in ordered blue modules are calculated valuewise in $\OBMod$.

An \emph{$\cO_X$-module} is a sheaf $\cF$ on $X$ with values in $\OBMod$ together with a morphism of sheaves $\cO_X\times \cF\to \cF$ such that for every open subset $U$ of $X$, the map $\cO_X(U)\times \cF(U)\to \cF(U)$ endows $\cF(U)$ with the structure of an ordered blue module. 
A \emph{morphism $\cF\to \cG$ of $\cO_X$-modules} is a morphism of sheaves $\cF\to \cG$ such that $\cF(U)\to\cG(U)$ is a morphism of $\cO_X(U)$-modules for every open subset $U$ of $X$.

An \emph{invertible sheaf on $X$} is an $\cO_X$-module $\cL$ such that there are an open covering $\{U_i\}_{i\in I}$ and isomorphisms $\cL\vert_{U_i}\to\cO_{U_i}$ of $\cO_{U_i}$-modules for every $i\in I$.

\begin{ex}\label{ex: twisted sheaves on projective space}
 Recall that a principal open subset $U_h$ of $\P^n_B$ consists of all homogeneous prime ideals $\fp$ of $B[T_0,\dotsc,T_n]$ not containing a given element $h$ nor the ideal $(T_0,\dotsc, T_n)$. The \emph{twisted sheaf $\cO(d)$ on $\P^n_B$} is defined by $\cO(d)(U_h)=B[T_0,\dotsc,T_n][h^{-1}]_d$, together with the tautological inclusions $\cO(d)(U_h)\to \cO(d)(U_g)$ whenever $U_g\subset U_h$, i.e.\ $g$ divides $h$. This definition extends uniquely to a sheaf on $\P^n_B$.
 
 The sheaf $\cO(d)$ inherits the structure of an $\cO_X$-module from the natural action of $\cO_{\P^n_B}(U_h)=B[T_0,\dotsc,T_n][h^{-1}]_0$ on $\cO(d)(U_h)=B[T_0,\dotsc,T_n][h^{-1}]_d$ by multiplication. 
 For the canonical open subsets $U_i=U_{T_i}$, multiplication by $T_i^{-d}$ yields an isomorphism $\cO(d)(U_i)\to\cO_{\P^n_B}(U_i)$, which shows that $\cO(d)$ is an invertible sheaf.
 
 Note that $\cO(d)$ contains nontrivial global sections if and only if $d\geq0$. For instance, $\Gamma\cO(1)$ is freely generated over $B$ by the global sections $T_0,\dotsc,T_n$.
\end{ex}

\subsubsection{Tensor products of invertible sheaves}
\label{subsubsection: tensor products of invertible sheaves}

As for usual schemes, there is a notion of a tensor product $\cL\otimes\cL'=\cL\otimes_{\cO_X}\cL'$ of invertible sheaves, which is again an invertible sheaf; it is defined as the sheafification of the presheaf  $U \mapsto \cL(U)\otimes_{\cO_X(U)}\cL'(U)$.  
The dual $\cL^\vee=\Hom(\cL,\cO_X)$ of an invertible sheaf is an invertible sheaf, and $\cL\otimes\cL^\vee\simeq\cO_X$. This turns the set $\Pic X$ of isomorphism classes of invertible sheaves on $X$ into an abelian group with respect to tensor product.
 
For example, if $B$ is an ordered blue field, $\Pic\P^n_B$ is an infinite cyclic group generated by the isomorphism class of $\cO(1)$.
 
Note that two global sections $s\in\Gamma(X,\cL)$ and $s'\in\Gamma(X,\cL')$ have a product $ss'=s\otimes s'$ in $\Gamma(X,\cL\otimes\cL')$.

\subsubsection{Pullbacks of invertible sheaves}
\label{subsubsection: pullbacks of invertible sheaves}

Let $\varphi:X\to Y$ be a morphism of ordered blue schemes and $\cL$ an invertible sheaf on $Y$. Let $\varphi^{-1}(\cL)$ be the sheaf on $X$ that sends an open subset $U$ of $X$ to $\colim\cL(V)$ where $V$ runs through all open subsets of $Y$ containing $\varphi(U)$. The \emph{pullback of $\cL$ along $\varphi$} is the sheaf $\varphi^\ast(\cL)=\varphi^{-1}(\cL)\otimes_{\varphi^{-1}(\cO_Y)}\cO_X$ on $X$.

The pullback $\varphi^\ast(\cL)$ of $\cL$ is again an invertible sheaf, as can be seen as follows. Let $\{V_i\}$ be an open covering of $Y$ such that there are isomorphisms $\eta_i:\cL\vert_{V_i}\to \cO_{V_i}$. Then $X$ is covered by the open subsets $U_i=\varphi^{-1}(V_i)$ and $\cO_{U_i}\simeq \varphi^\ast(\cO_{V_i})$ Thus the pullbacks of the $\eta_i$ define isomorphisms $\varphi^\ast(\eta_i):\varphi^\ast(\cL)\vert_{U_i}\to \cO_{U_i}$, which verifies that $\varphi^\ast(\cL)$ is an invertible sheaf.

There is a natural morphism 
\[
 \varphi^\#_\cL: \ \Gamma(Y,\cL)\ \longrightarrow \ \Gamma\big(X,\varphi^\ast(\cL)\big)
\]
that sends a global section $s$ of $\cL$ to $\big(\colim \res_{X,V}(s)\big)\otimes 1$, where the colimit taken over the system of all restrictions of $s$ to open subsets $V$ of $Y$ containing $\varphi(X)$.

The pullback commutes with tensor products, i.e.\ we have $\varphi^\ast(\cL)\otimes\varphi^\ast(\cL')\simeq \varphi^\ast(\cL\otimes\cL')$. As a result, we obtain a group homomorphism $\varphi^\ast:\Pic Y\to \Pic X$.


\subsection{Morphisms to projective space}
\label{subsection: Morphisms to projective space}

A key point in our approach to matroid bundles is the characterization of morphisms into projective space, which can be described in complete analogy to classical algebraic geometry. We will prove the relevant results in this section. 

Let $X$ be an ordered blue scheme and $\cL$ an invertible sheaf on $X$. Let $k=\cO_X(X)$ be the ordered blueprint of global sections of $X$. We denote the ordered blue $k$-module of global sections of $\cL$ by $\Gamma\cL=\cL(X)$. We say that global sections \emph{$s_1,\dotsc,s_n\in\Gamma \cL$ generate $\cL$} if for every point $x\in X$ their images $s_{1,x},\dotsc,s_{n,x}$ in $\cL_x$ generate $\cL_x$ as an ordered blue $\cO_{X,x}$-module. This is the case if and only if not all of $s_{1,x},\dotsc,s_{n,x}$ are contained in $\fm_x\cL_x$, where $\fm_x$ is the maximal ideal of $\cO_{x,x}$.

\begin{ex}
 Let $\cO(1)$ be the twisted sheaf on $\P^n_B$, cf.\ Example \ref{ex: twisted sheaves on projective space}. Then the global sections $T_0,\dotsc,T_n\in\Gamma\cO(1)$ generate $\cO(1)$, since $\P^n_B=\Proj B[T_0,\dotsc,T_n]$ consists of all homogeneous prime $m$-ideals $\fp$ of $B[T_0,\dotsc,T_n]$ that do not contain all of $T_0,\dotsc,T_n$. Thus for every such $\fp$, there is at least one $T_i$ that is a unit in the stalk $\cO(1)_\fp$.
\end{ex}

\begin{lemma} \label{lemma: pullbacks of global sections generate the pullback sheaf}
 Let $\varphi:X\to Y$ be a morphism of ordered blue schemes and $\cL$ an invertible sheaf on $Y$ that is generated by global sections $s_1,\dotsc,s_n\in\Gamma\cL$. Then the images $\varphi^\#_\cL(s_1),\dotsc,\varphi^\#_\cL(s_n)$ generate the invertible sheaf $\varphi^\ast\cL$.
\end{lemma}

\begin{proof}
 As a first step, we observe that for every $x\in X$ and $y=\varphi(x)$, 
 \[
  (\varphi^\ast\cL)_x \ = \ \colim \varphi^\ast(\cL)(U) \ = \ \colim \big(\cL\otimes_{\cO_Y}\varphi_\ast\cO_X\big)(V) \ = \ \big(\cL\otimes_{\cO_Y}\varphi_\ast\cO_X\big)_{y} \ = \ \cL_{y}\otimes_{\cO_{Y,y}}\cO_{X,x}
 \]
 where the first colimit is taken over all open neighbourhoods $U$ of $x$ and the second colimit is taken over all open neighbourhoods $V$ of $y$. 
 
 If we write $t_i=\varphi^\#_\cL(s_i)$ for $i=1,\dotsc,n$, the above calculation shows that we can identify the image $t_{i,x}$ of $t_i$ in the stalk $(\varphi^\ast\cL)_x=\cL_{y}\otimes_{\cO_{Y,y}}\cO_{X,x}$ with $s_{i,y}\otimes 1$ for $i=1,\dotsc,n$. From this, it is clear that if $s_{1,y},\dotsc,s_{n,y}$ generate $\cL_{y}$ as an ordered blue $\cO_{Y,y}$-module, then $s_{1,y}\otimes 1,\dotsc,s_{n,y}\otimes 1$ generate $\cL_{y}\otimes_{\cO_{Y,y}}\cO_{X,x}$ as an ordered blue $\cO_{X,x}$-module. We conclude that $t_1,\dotsc,t_n$ generate $\varphi^\ast\cL$.
\end{proof}

\begin{lemma} \label{lemma: local units are global units}
 Let $X$ be an ordered blue scheme and $s\in\Gamma\cO_X$. If the image $s_x$ of $s$ in $\cO_{X,x}$ is a unit for all $x\in X$, then $s$ is a unit of $\Gamma\cO_X$.
\end{lemma}

\begin{proof}
 Since $\cO_{X,x}$ is the colimit over the local sections in the open neighbourhoods $U$ of $x$, an inverse $t_x$ of $s_x$ in $\cO_{X,x}$ must come from an inverse of the restriction of $s$ to some open neighbourhood $U$ of $x$. Thus we find an open covering $U_i$ of $X$ such that the restrictions $s_i$ of $s$ to $U_i$ have inverses $t_i$ in $\Gamma\cO_{U_i}$. On the intersections $U_i\cap U_j$, we get $\bar t_i=\bar t_i\, \bar s\, \bar t_j = \bar t_j$ where $\bar t_i$, $\bar t_j$ and $\bar s$ denote the respective restrictions of $t_i$, $t_j$ and $s$ to $U_i\cap U_j$. This shows that the local sections $t_i$ coincide on the overlaps of the $U_i$ and thus glue to a global section $t$ of $\cO_X$. Since $st$ is locally equal to $1$, we must have $st=1$ in $\Gamma\cO_X$, which shows that $s$ is a unit in $\Gamma\cO_X$.
\end{proof}

\begin{thm}\label{thm: morphisms to projective space}
 Let $B$ be an ordered blueprint, $X$ an ordered blue $B$-scheme and $n\geq 0$. Let $\P^n_B=\Proj B[x_0,\dotsc,x_n]$ be the projective space over $B$ and $\cO(1)$ the twisted sheaf of degree $1$.
 \begin{enumerate}
  \item \label{proj1} If $\varphi:X\to \P_B^n$ is a $B$-linear morphism, then $\varphi^\ast\big(\cO(1)\big)$ is an invertible sheaf on $X$, which is generated by the global sections $s_i=\varphi^\#_{\cO(1)}(x_i)$ for $i=0,\dotsc,n$.
  \item \label{proj2} If $\cL$ is an invertible sheaf on $X$ and $s_0,\dotsc,s_n\in \Gamma\cL$ are global sections that generate $\cL$, there is a unique $B$-linear morphism $\varphi:X\to \P^n_B$ such that there is an isomorphism $\cL\to \varphi^\ast\big(\cO(1)\big)$ that identifies $s_i$ with $\varphi^\#_{\cO(1)}(x_i)$ for $i=0,\dotsc,n$.
 \end{enumerate}
\end{thm}

\begin{proof}
 Part \eqref{proj1} of the theorem follows immediately from Lemma \ref{lemma: pullbacks of global sections generate the pullback sheaf}.
 
 Part \eqref{proj2} is trivial for the empty scheme $X$. Thus we may assume that $X$ is nonempty. Then the subsets $X_i=\{ \, x\in X\, | \, s_i\in\cL_x^\times \, \}$ are open, and they cover $X$ since $s_0,\dotsc,s_n$ generate $\cL$.
 
 Since $\cL$ is locally free of rank $1$, we can choose for every $x\in X$ a generator $\theta_x$ of $\cL_x$ as an $\cO_{X,x}$-module. Thus for every $i=0,\dotsc,n$ and $x\in X$, the image $s_{i,x}$ of $s_i$ in $\cL_x$ is a multiple of $\theta_x$, i.e.\ $s_{i,x}=\lambda_{i,x}\theta_x$ for some $\lambda_{i,x}$.
 
 For a fixed $i$, we have $s_i\in\cL_x^\times$ for every $x\in X_i$. Lemma \ref{lemma: local units are global units} implies that the restriction of $s_i$ to $X_i$ is invertible. Thus we can define for every $j$ the quotient $\frac{s_j}{s_i}$ in $\Gamma X_i$. Note that $\frac{s_j}{s_i}$ does not depend on the choice of the generators $\theta_x$. 
 
 Let $U_i=\Spec B[\frac{x_0}{x_i},\dotsc,\frac{x_n}{x_i}]$ be the $i$-th canonical open subset of $\P^n_B$. The association $\frac{x_j}{x_i}\mapsto \frac{s_j}{s_i}$ defines a morphism $f_i:B[\frac{x_0}{x_i},\dotsc,\frac{x_n}{x_i}]\to \Gamma X_i$. By Lemma \ref{lemma: adjunction between ordered blueprints and ordered blue schemes}, $f_i$ corresponds to a morphism $\varphi_i:X_i\to U_i$ of ordered blue schemes, which is obviously $B$-linear. Since $\varphi_i$ and $\varphi_j$ coincide on the intersection $X_i\cap X_j$, the morphisms $\varphi_i$ glue to a $B$-linear morphism $\varphi:X\to \P_B^n$.
 
 It is clear from the definition of $\varphi$ that $\varphi^\ast\big(\cO(1)\big)$ is isomorphic to $\cL$ and that the pullbacks of the coordinates $x_i$ coincide with the global sections $s_i$. It is also clear that there is a unique $\varphi$ with these properties. This finishes the proof of the theorem.
\end{proof}

Let $X$ be an ordered blue scheme. An \emph{invertible sheaf with $n$ generators} is an invertible sheaf $\cL$ together with global sections $s_1,\dotsc,s_n$ that generate $\cL$. Two invertible sheaves with $n$ generators $(\cL;s_1,\dotsc,s_n)$ and $(\cL';s_1',\dotsc,s_n')$ are isomorphic if there exists an isomorphism $\varphi:\cL\to \cL'$ of invertible sheaves such that $s_i=\varphi^\#_\cL(s_i')$ for $i=1,\dotsc,n$. Let $\PicGen(n,X)$ be the set of isomorphism classes of invertible sheaves on $X$ with $n$ generators. 

The following is an immediate consequence of Theorem \ref{thm: morphisms to projective space}.

\begin{cor}\label{cor: morphisms to projective space as invertible sheaves with generators}
 Let $B$ be an ordered blueprint, $X$ an ordered blue $B$-scheme and $n\geq0$. Then the map
 \[
  \begin{array}{ccc}
   \Hom_B(X,\P^n_B)   & \longrightarrow & \PicGen(n+1,X) \\
   \varphi:X\to\P^n_B & \longmapsto     & \big(\varphi^\ast\big(\cO(1)\big);\varphi^\#_{\cO(1)}(x_0),\dotsc,\varphi^\#_{\cO(1)}(x_n)\big)
  \end{array}
 \]
 is a bijection. In particular, we have $\P^n_B(C)=\PicGen(n+1,X)$ if $X=\Spec C$. \qed
\end{cor}


\section{Families of matroids and their moduli spaces}
\label{section: matroid bundles}

At this point, we have developed the necessary formalism to define matroid bundles over ordered blue $\Funpm$-schemes. Besides our immediate goal of constructing a (fine) moduli space of matroids, which requires the notion of a universal family, the concept of a matroid bundle should have applications to tropical geometry (and conceivably to classical matroid theory as well).  For example, we hope that matroid bundles will be useful for developing a theory of sheaf cohomology for tropical varieties.  In addition, a suitable notion of families of valuated matroids appears to be a necessary ingredient for defining families of tropical schemes in the sense of Maclagan and Rinc\'on (\cite{Maclagan-Rincon14}).

In this section, we will introduce matroid bundles, study their first properties, and use them to define a suitable moduli functor which will be represented by the moduli space of matroids.


\subsection{Families of matroids}
\label{subsection: Families of matroids}

An \emph{ordered blue $\Funpm$-scheme} is an ordered blue scheme $X$ together with a structure morphism $X\to \Spec\Funpm$. An $\Funpm$-linear morphism between two ordered blue $\Funpm$-schemes is a morphism $\varphi:X\to Y$ that commutes with the respective structure morphisms. This defines the category $\Sch_\Funpm$ of ordred blue $\Funpm$-schemes.

If $X$ is an ordered blue $\Funpm$-scheme, then $\Gamma(X,\cO_X)$ is an $\Funpm$-algebra. As usual, we denote the weak inverse of $1$ by $\epsilon$. Note that $\Gamma(X,\cL)$ is an ordered blue $\Gamma(X,\cO_X)$-module for every invertible sheaf $\cL$ over $X$. We define $\cL^{\otimes 2}=\cL\otimes\cL$ and recall from section \ref{subsubsection: tensor products of invertible sheaves} that the product $s\cdot s'$ of two elements $s,s'\in\Gamma(X,\cL)$ is an element of $\Gamma(X,\cL^{\otimes2})$.

\begin{df}\label{def: matroid bundles}
 Let $E$ be a non-empty finite ordered set, $r$ a natural number and $X$ an ordered blue $\Funpm$-scheme. A \emph{Grassmann-Pl\"ucker function of rank $r$ on $E$ over $X$} is an invertible sheaf $\cL$ over $X$ together with a map
 \[
  \begin{array}{cccc}
   \Delta: & \binom Er & \longrightarrow & \Gamma(X,\cL) \\
      &  I   & \longmapsto     &   \Delta(I)
  \end{array}
 \]
 such that $\{\Delta(I)\}_{I\in\binom Er}$ generate $\cL$ and satisfy the \emph{Pl\"ucker relations}
 \[
  0 \quad \leq \quad \sum_{k=1}^{r+1} \ \epsilon^k \ \cdot \ \Delta({I\cup\{i_k\}}) \ \cdot \ \Delta({I'-\{i_k\}})
 \]
 in $\Gamma(X,\cL^{\otimes 2})$ for every $(r-1)$-subset $I$ and every $(r+1)$-subset $I'=\{i_1,\dotsc,i_{r+1}\}$ of $E$, where $i_1<\dotsb<i_{r+1}$ and we define $\Delta({I\cup\{i_k\}})=0$ if $I\cup\{i_k\}$ is of cardinality $r-1$.
 
 Two Grassmann-Pl\"ucker functions $(\cL,\Delta)$ and $(\cL',\Delta')$ on $X$ are \emph{isomorphic} if there exists an isomorphism $\varphi:\cL\to\cL'$ of invertible sheaves such that $\Delta'=\Gamma\varphi\circ \Delta$, where $\Gamma\varphi:\Gamma(X,\cL)\to\Gamma(X,\cL')$ is the induced isomorphism between the ordered blueprints of global sections. 
 
 A \emph{matroid bundle over $X$} is an isomorphism class of Grassmann-Pl\"ucker functions.
\end{df}

\subsubsection{Pullbacks of matroid bundles}
\label{subsubsection: pullbacks of matroid bundles}

The pullback $\varphi^\ast(\cM)$ of a matroid bundle $\cM$ on $Y$ along a morphism $\varphi:X\to Y$ of ordered blue $\Funpm$-schemes is defined by the following lemma.

\begin{lemma}\label{lemma: definition of the pullback of matroids}
 Let $\varphi:X\to Y$ be a morphism in $\OBSch_\Funpm$ and $\cL$ an invertible sheaf on $Y$. Let $E$ be a non-empty finite ordered set, $r$ a natural number and $\cM$ a matroid bundle over $Y$ represented by a Grassmann-Pl\"ucker function $\Delta:\binom Er\to\Gamma(Y,\cL)$. Then $\varphi^\ast(\Delta)=\varphi_\cL^\#\circ \Delta:\binom Er\to \Gamma(X,\varphi^\ast(\cL))$ is a Grassmann-Pl\"ucker function over $X$ and the matroid bundle $\varphi^\ast(\cM)=[\tilde \Delta]$ over $X$ does not depend on the choice of representative $\Delta$ of $\cM$.
\end{lemma}

\begin{proof}
 As a first step, we verify that $\varphi^\ast(\Delta)= \varphi_\cL^\#\circ \Delta$ is a Grassmann-Pl\"ucker function. Since $\{\Delta(I)\}_{I\in\binom Er}$ generates $\cL$, Lemma \ref{lemma: pullbacks of global sections generate the pullback sheaf} implies that $\{\varphi^\#_\cL(\Delta(I))\}_{I\in\binom Er}$ generates $\varphi^\ast(\cL)$. The identification $\varphi^\ast(\cL^{\otimes 2})=\varphi^\ast(\cL)^{\otimes 2}$ yields a morphism $\varphi^\#_{\cL^{\otimes2}}:\Gamma(Y,\cL^{\otimes2})\to \Gamma(X,\varphi^\ast(\cL)^{\otimes2})$ of ordered blueprints. Thus the validity of the Pl\"ucker relations in $\Gamma(Y,\cL^{\otimes2})$ implies the validity of the Pl\"ucker relations in $\Gamma(X,\varphi^\ast(\cL)^{\otimes2})$. This shows that $\varphi^\ast(\Delta)$ is a Grassmann-Pl\"ucker function.
 
 Next we show independence from the choice of representative $\Delta$. Let $\Delta':\binom Er\to \Gamma(Y,\cL')$ be another Grassmann-Pl\"ucker function representing $\cM$, i.e.\ there is an isomorphism $\eta:\cL\to\cL'$ such that $\Delta'=\Gamma(Y,\eta)\circ \Delta$. This yields an isomorphism $\varphi^\ast(\eta):\varphi^\ast(\cL)\to\varphi^\ast(\cL')$ and $\varphi^\ast(\Delta')=\Gamma(Y,\varphi^\ast(\eta))\circ \varphi^\ast(\Delta)$, as desired. 
\end{proof}


\subsection{The moduli functor of matroids}
\label{subsubsection: the moduli functor of matroids}

Let $E$ be a non-empty finite ordered set and $r$ a natural number. We extend the functor $\cMat(r,E):\OBlpr_\Funpm\to\Sets$ to the functor
\[
  \begin{array}{cccc}
  \cMat(r,E): & \OBSch_\Funpm & \longrightarrow & \Sets \\[5pt]
                         &   X          & \longmapsto     & \big\{\,\text{matroid bundles of rank $r$ on $E$ over $X$}\, \big\} \\[5pt]
                         & \varphi:X\to Y     & \longmapsto     & \varphi^\ast:\cMat(r,E)(Y)\to \cMat(r,E)(X)
 \end{array}
\]
Thanks to Proposition \ref{prop: compatibility of B-matroids with matroid bundles over Spec B}, we have $\cMat(r,E)(\Spec B)=\cMat(r,E)(B)$ for every $\Funpm$-algebra and $\cMat(r,E)(\Spec f)=\cMat(r,E)(f)$ for every morphism $f:B\to C$ in $\OBlpr_\Funpm$.



\subsection{Compatibility with matroids over ordered blueprints}
\label{subsection: compatibility with matroids over ordered blueprints}

In the following, we verify that matroid bundles over $\Spec B$ correspond bijectively to $B$-matroids in a functorial way.

\begin{prop}\label{prop: compatibility of B-matroids with matroid bundles over Spec B}
 Let $B$ be an $\Funpm$-algebra, $E$ a non-empty finite ordered set, $r$ a natural number and $X=\Spec B$. Then the map 
 \[
  \begin{array}{cccc}
   \Phi_B: & \big\{\text{$B$-matroids of rank $r$ on $E$}\big\} & \longrightarrow & \big\{\text{matroid bundles of rank $r$ on $E$ over $X$}\big\} \\[5pt]
         & M=[\Delta:\binom Er\to B]                                 & \longmapsto     & \widetilde{M}=[\iota_B\circ \Delta:\binom Er\to\Gamma(X,\cO_X)]
  \end{array}
 \]
 is a bijection, where $\iota_B:B\to \Gamma(X,\cO_X)$ is the inclusion as constant sections. If $f:B\to C$ is a morphism of $\Funpm$-algebras, $\varphi=f^\ast:\Spec C\to \Spec B$ the induced morphism and $M$ a $B$-matroid, then $\Phi_C(f_\ast(M))=\varphi^\ast(\Phi_B(M))$. In other words, we have a commutative diagram of functors
  \[
  \beginpgfgraphicnamed{tikz/fig10}
   \begin{tikzcd}[row sep=0pt, column sep=80pt]
    \OBlpr_\Funpm \arrow{dr}[above]{\cMat(r,E)} \arrow{dd}[swap]{\Spec} \\
      & \Sets. \\
    \OBSch_\Funpm \arrow{ur}[below]{\cMat(r,E)}
   \end{tikzcd}
  \endpgfgraphicnamed
 \]
\end{prop}

\begin{proof}
  To begin with, we verify that $\Phi_B$ is well-defined. Let $\Delta:\binom Er\to B$ be a Grassmann-Pl\"ucker function. Then $\Delta(I)\in B^\times$ for some $r$-subset $I$ of $E$.  Therefore $\iota_B(\Delta(I))\in \Gamma(X,\cO_X)^\times$, which shows that $\{\iota_B\circ \Delta(I)\}_{I\in \binom Er}$ generates $\cO_X$. The Pl\"ucker relations for $\Delta$ imply the Pl\"ucker relations for $\iota_B\circ\Delta$. Thus $\iota_B\circ\Delta$ is a Grassmann-Pl\"ucker function over $X$. Since every $a\in B^\times$ defines an automorphism of $\cO_X$, the map $\Phi_B$ is independent of the choice of representative, which shows that $\Phi_B$ is well-defined.
 
 The injectivity of $\Phi_B$ can be verified as follows. The inclusion $\iota_B:B\to \Gamma(X,\cO_X)$ as constant sections is an isomorphism of ordered blueprints, which implies that any two Grassmann-Pl\"ucker functions $\Delta,\Delta':\binom Er\to B$ are different if $\iota_B\circ \Delta$ and $\iota_B\circ \Delta'$ are different. Moreover, this implies that the automorphisms of $\cO_X$ are equal to the automorphisms of $B$ as a $B$-module, which are given by multiplication with a unit, i.e.\ $\Aut(\cO_X)=B^\times$. Thus different $B$-matroids yield different matroid bundles over $X$, which proves the injectivity of $\Phi_B$.
 
 The surjectivity of $\Phi_B$ can be verified as follows. It is obvious if $B$ is trivial. If $B$ is nontrivial, then $B$ has a unique maximal ideal, which is $\fm=B-B^\times$. Therefore the only open subset of $X$ containing $\fm$ is $X$ itself. Thus there are no nontrivial coverings of $X$ and consequently every invertible sheaf on $X$ is isomorphic to $\cO_X$. This shows that we can represent every matroid bundle $\cM$ over $X$ by a Grassmann-Pl\"ucker function $\tilde \Delta:\binom Er\to \Gamma(X,\cO_X)$. Composing $\tilde \Delta$ with the inverse $\iota_B^{-1}$ of $\iota_B$ yields a map $\Delta=\iota_B^{-1}\circ\tilde \Delta:\binom Er\to B$. Since $\{\iota_B\circ \Delta(I)\}_{I\in \binom Er}$ generates $\cO_X$, it generates the stalk $\cO_{X,\fm}=B$ as an ordered blue $B$-module, which means that $\Delta(I)\in B^\times$ for some $r$-subset $I$ of $E$. The Pl\"ucker relations for $\tilde \Delta$ imply the Pl\"ucker relations for $\Delta$. Thus $\Delta$ is a Grassmann-Pl\"ucker function with coefficients in $B$ and $\tilde \Delta=\Phi_B(\Delta)$, as desired. This shows that $\Phi_B$ is bijective.
 
 To verify the final claim of the proposition, let $\Delta:\binom Er\to B$ be a Grassmann-Pl\"ucker function representing $M$. Then $f_\ast(M)$ is represented by the Grassmann-Pl\"ucker function $f\circ \Delta:\binom Er\to C$. The matroid bundle $\widetilde M=\Phi_B(M)$ is represented by the Grassmann-Pl\"ucker function $\iota_B\circ \Delta:\binom Er\to \Gamma(X,\cO_X)$. By Lemma \ref{lemma: definition of the pullback of matroids}, the pullback $\varphi^\ast(\widetilde M)$ is represented by the Grassmann-Pl\"ucker function $\varphi_{\cO_X}^\#\circ \iota_B\circ \Delta:\binom Er\to \Gamma(Y,\varphi^\ast(\cO_Y))$ where $Y=\Spec C$. 
 
 The result now follows from the commutativity of the diagram
 \[
  \beginpgfgraphicnamed{tikz/fig8}
   \begin{tikzcd}[row sep=15pt, column sep=50pt]
    B \arrow{r}{f} \arrow{d}[swap]{\iota_B} & C \arrow{d}{\iota_C} \\
    \Gamma(X,\cO_X) \arrow{r}{\varphi_{\cO_X}^\#} & \Gamma(Y,\varphi^\ast(\cO_Y))
   \end{tikzcd}
  \endpgfgraphicnamed
 \]
 where $\iota_C:C\to \Gamma(Y,\cO_Y)$ is the canonical isomorphism.
\end{proof}

\subsubsection{Example of a matroid bundle over the projective line over $\K$}

In this example, we investigate matroid bundles of rank $2$ on $E=\{1,2,3,4\}$ over the projective line $\P^1_\K=\Proj\big(\K[T_0,T_1]\big)$. We review some general facts that we will use below. 
 
Since $\K^\bullet=\{0,1\}$, the underlying monoid of $\K[T_0,T_1]$ is $\{0\}\cup\{T_0^{e_0}T_1^{e_1}|e_0,e_1\in\N\}$. Thus the homogeneous prime ideals of $\K[T_0,T_1]$ not containing both $T_0$ and $T_1$ are $(0)$, $(T_0)$ and $(T_1)$, cf.\ Figure \ref{figure: projective line and projective plane} for an illustration.
 
As in the classical case, every invertible sheaf on $\P^1_\K$ is isomorphic to a twisted sheaf $\cO(d)$ for some $d\in \Z$ and every automorphism of $\cO(d)$ is given by the multiplication by a unit of $\K$, i.e. $\Aut(\cO(d))=\K^\times=\{1\}$. This means that every matroid bundle $\cM$ of rank $2$ on $E$ over $\P^1_\K$ is represented by a unique Grassmann-Pl\"ucker function of the form $\Delta:\binom E2\to\Gamma(X,\cO(d))$. Note that there is only one Pl\"ucker relation in this case, which is
\[
 0 \ \leq \ \Delta_{1,2}\Delta_{3,4} \ + \ \Delta_{1,3}\Delta_{2,4} \ + \ \Delta_{1,4}\Delta_{2,3}
\]
where we write $\Delta_{i,j}=\Delta(\{i,j\})$.

We have $\Gamma(\P^1_\K,\cO(d))=\{0\}$ for $d<0$, which means that $\cO(d)$ cannot be generated by global sections for $d<0$. For $d\geq0$, we have $\Gamma(\P^1_\K,\cO(d))=\{0\}\cup\{T_0^d,T_0^{d-1}T_1,\dotsc,T_1^d\}$. Since $T_0$ is contained in the maximal ideal of $\K[T_0,T_1]_{(T_0)}$ and $T_1$ is contained in the maximal ideal of $\K[T_0,T_1]_{(T_1)}$, there is a unique minimal set of global sections that generates $\cO(d)$: for $d=0$, this set is $\{1\}$ and for $d>0$, this set is $\{T_0^d,T_1^d\}$.

For every $d\geq0$, there is exists a nonempty set of Grassmann-Pl\"ucker functions $\Delta:\binom E2\to\Gamma(\P^1_\K,\cO(d))$. We classify them for $d=0$ and $d=1$ in the following.

The case $d=0$ ties to $\K$-matroids as follows: the pullback along the structure morphism $\chi:\P^1_\K\to\Spec \K$ yields a bijection 
\[\textstyle
 \chi^\ast: \quad \left\{\begin{array}{c}\text{Grassmann-Pl\"ucker functions}\\\Delta:\binom E2\to\K\end{array}\right\} \quad \longrightarrow \quad \left\{\begin{array}{c}\text{Grassmann-Pl\"ucker functions}\\\Delta:\binom E2\to\Gamma(\P^1_\K,\cO_{\P^1_\K})\end{array}\right\},
\]
which realizes $\K$-matroids as ``constant'' matroid bundles over $\P^1_\K$. The inverse is given by the pullback $\xi^\ast(\Delta)$ along an arbitrary morphism $\xi:\Spec \K\to\P^1_\K$. (Note that there are three such morphisms, which are characterized by their image, which can be each of the three points of $\P^1_\K$.)

The $\K$-matroids of rank $2$ on $E$ correspond to the functions $\Delta:\binom E2\to\{0,1\}$ for which at least two of the products $\Delta_{1,2}\Delta_{3,4}$, $\Delta_{1,3}\Delta_{2,4}$ and $\Delta_{1,4}\Delta_{2,3}$ are equal to $1$, or for which all three products are equal to $0$ but $\Delta_{i,j}=1$ for at least one $2$-subset $\{i,j\}$ of $E$.

The case $d=1$ is more involved and reveals some novel phenomena. We have $\Gamma(\P^1_\K,\cO(1))=\{0,T_0,T_1\}$. Let $\Delta:\binom E2\to\{0,T_0,T_1\}$ be a function. Since $\{\Delta_{i,j}\}_{\{i,j\}\in\binom E2}$ has to generate $\cO(1)$ in order for $\Delta$ to be a Grassmann-Pl\"ucker function, we must have $\Delta_{i,j}=T_0$ and $\Delta_{k,l}=T_1$ for some $2$-subsets $\{i,j\}$ and $\{k,l\}$ of $E$. Moreover, at least two of the products $\Delta_{1,2}\Delta_{3,4}$, $\Delta_{1,3}\Delta_{2,4}$ and $\Delta_{1,4}\Delta_{2,3}$ must be equal to each other, while the third might be equal to the other two or equal to $0$. This allows for the following Grassmann-Pl\"ucker functions:
\begin{itemize}
 \item $\Delta_{1,2}\Delta_{3,4}=\Delta_{1,3}\Delta_{2,4}=\Delta_{1,4}\Delta_{2,3}=T_0T_1$;
 \item $\Delta_{i,j}\Delta_{k,l}=\Delta_{i,k}\Delta_{j,l}=T_0^2$, $\Delta_{i,l}=0$, and $\Delta_{j,k}=T_1$ for some $\{i,j,k,l\}=E$;
 \item $\Delta_{i,j}\Delta_{k,l}=\Delta_{i,k}\Delta_{j,l}=T_1^2$, $\Delta_{i,l}=0$, and $\Delta_{j,k}=T_0$ for some $\{i,j,k,l\}=E$;
 \item $\Delta_{i,j}\Delta_{k,l}=\Delta_{i,k}\Delta_{j,l}=T_0T_1$, $\Delta_{i,l}=0$, and $\Delta_{j,k}\in\{0,T_0,T_1\}$ for some $\{i,j,k,l\}=E$;
 \item $\Delta_{1,2}\Delta_{3,4}=\Delta_{1,3}\Delta_{2,4}=\Delta_{1,4}\Delta_{2,3}=0$, $\Delta_{i,j}=T_0$, and $\Delta_{i,k}=T_1$ for some pairwise distinct $i,j,k$.
\end{itemize}

The cases $d\geq2$ become increasingly more involved.


\subsection{The moduli space of matroids}
\label{subsection: The moduli space of matroids}

We define the \emph{matroid space of rank $r$ on $E$} as the ordered blue scheme
\[ \textstyle
 \Mat(r,E) \quad = \quad \Proj\Big( \, \bpquot{\Funpm\big[ \, x_I \, \big| \, I\in\binom Er \, \big]}{\cPl(r,E)} \, \Big),
\]
where $\cPl(r,E)$ is generated by the Pl\"ucker relations
\[
 0 \quad \leq \quad \sum_{k=1}^{r+1} \ \epsilon^k \ \cdot \ x_{I\cup\{i_k\}} \ \cdot \ x_{I'-\{i_k\}} 
\]
for every $(r-1)$-subset $I$ and every $(r+1)$-subset $I'=\{i_1,\dotsc,i_{r+1}\}$ of $E$. By definition, it comes with a closed immersion into projective space
\[\textstyle
 \iota: \ \Mat(r,n) \ \longrightarrow \ \P^N_{\Funpm} \ = \ \Proj\Big( \, \Funpm\big[ \, x_I \, \big| \, I\in\binom Er \, \big] \, \Big)
\]
where $N=\#\binom Er - 1$. We denote the pullback of the tautological bundle $\cO(1)$ of $\P^N_{\Funpm}$ to $\Mat(r,E)$ by $\cL_\univ=\iota^\ast\cO(1)$. The pullbacks of the canonical sections $x_I$ of $\cO(1)$ define a map
\[
 \begin{array}{cccc}
  \Delta_\univ: & \binom Er & \longrightarrow & \Gamma(\Mat(r,E),\cL_\univ). \\
           & I    & \longmapsto     & \Delta_\univ(I)=\iota^\#_{\cO(1)}(x_I)
 \end{array}
\]
Since $\{x_I\}_{I\in\binom Er}$ generates $\cO(1)$, Lemma \ref{lemma: pullbacks of global sections generate the pullback sheaf} implies that $\{\Delta_\univ(I)\}_{I\in\binom Er}$ generates $\cL_\univ$. Since $\Mat(r,E)$ satisfies the Pl\"ucker relations, the function $\Delta_\univ$ is a Grassmann-Pl\"ucker function on $\Mat(r,E)$. The \emph{universal matroid bundle} is the class $\cM_\univ$ of $\Delta_\univ$, which is a matroid bundle of rank $r$ on $E$ over $\Mat(r,E)$.

\begin{rem}\label{rem: relation of matroid spaces with Grassmannians}
 The matroid space $\Mat(r,E)$ should be thought of as an analogue of the Grassmannian $\Gr(r,E)_R$ over a ring $R$ from usual algebraic geometry. In fact, we can recover the Grassmannian as the scheme $\Mat(r,E)_B^+$ associated to the base extension $\Mat(r,E)_B$ of the matroid space to the ordered blueprint $B$ associated with $R$. The functor $(-)^+$ respects many of the standard `decorations' of the Grassmannian (e.g. its Pl{\"u}cker embedding into $\P^N$)
in an obvious sense, linking them to their respective avatars in classical algebraic geometry.
\end{rem}

The following theorem shows that the pair $(\Mat(r,E),\cM_\univ)$ represents the moduli functor $\cMat(r,E)$:

\begin{thm}\label{thm: moduli space of matroids}
 Let $E$ be a non-empty finite ordered set and let $r$ be a natural number. The ordered blue $\Funpm$-scheme $\Mat(r,E)$, together with its universal matroid bundle $\cM_\univ$, is the fine moduli space of all matroid bundles of rank $r$ on $E$, i.e. the map
 \[
  \begin{array}{cccc}
   \Phi: & \Hom_{\Funpm}\big(X,\Mat(r,E)\big)                 & \longrightarrow & \cMat(r,E)(X) \\[5pt]
         & \varphi:X\to \Mat(r,E) & \longmapsto & \varphi^\ast(\cM_\univ)\\[5pt]
  \end{array}
 \]
 is a bijection for every ordered blue $\Funpm$-scheme $X$.
\end{thm}

\begin{proof}
 Note that every morphism $\varphi:\Spec B\to \Mat(r,E)$ is automatically $\Funpm$-linear since the morphism $\Funpm\to B$ is unique. Therefore we can omit $\Funpm$-linearity from the notation for the morphism set $\Hom(X,\Mat(r,E))$.

 As a first step, we define a map $\Psi:\cMat(r,E)(X)\to \Hom(X,\Mat(r,E))$ in the opposite direction of $\Phi$. Let $\cM$ be a matroid bundle over $X$ that is represented by a Grassmann-Pl\"ucker function $\Delta:\binom Er\to \Gamma(X,\cL)$ for some invertible sheaf $\cL$ over $X$. Let $N=\#\binom Er - 1$. By Theorem \ref{thm: morphisms to projective space} \eqref{proj2}, there is a unique $\Funpm$-linear morphism $\varphi_0:X\to \P_{\Funpm}^{N}$ such that $\Delta(I)=\varphi^\#_0(x_I)$ for all $I\in\binom Er$. Since $\Delta$ satisfies the Pl\"ucker relations, $\varphi_0$ factors uniquely into a morphism $\varphi:X\to \Mat(r,E)$ followed by the closed immersion $\iota:\Mat(r,E)\to \P_{\Funpm}^{N}$. 
 
 That $\Phi$ and $\Psi$ are mutually inverse bijections follows at once from Corollary \ref{cor: morphisms to projective space as invertible sheaves with generators}.
 
 Consider a morphism $\psi:Y\to X$ in $\OBSch_\Funpm$. Then by the definition of the pullback of a matroid bundle, we have $(\varphi\circ\psi)^\ast(\cM_\univ)=\psi^\ast(\varphi^\ast(\cM_\univ))$, which establishes the functoriality of the bijection $\Phi$. This completes the proof of the theorem.
\end{proof}


\subsection{Duality}
\label{subsection: duality}

One of the fundamental features of matroid theory is that every matroid (with coefficients) comes with a canonical dual matroid. This extends to matroid bundles, and, in fact, the duality is derived from a duality between the moduli spaces.

\begin{thm}\label{thm: duality of matoid spaces}
 Let $E$ be a non-empty finite ordered set, $r\leq\#E$ a natural number and $r^{\smallvee}=\#E-r$. Let $I^c=E-I$ denote the complement of a subset $I$ of $E$. Then the association $x_I\mapsto x_{I^c}$ defines a graded $\Funpm$-linear isomorphism
 \[\textstyle
  \alpha^\smallvee: \ \bpquot{\Funpm\big[ \, x_I \, \big| \, I\in\binom Er \, \big]}{\cPl(r,E)} \quad \stackrel\sim\longrightarrow \quad \bpquot{\Funpm\big[ \, x_I \, \big| \, I\in\binom E{r^\smallvee} \, \big]}{\cPl(r^\smallvee,E)}
 \]
 and thus an isomorphism
 \[
  \varphi^\smallvee: \ \Mat(r^\smallvee,E) \ \stackrel\sim\longrightarrow \ \Mat(r,E)
 \]
 of ordered blue $\Funpm$-schemes.
\end{thm} 
 
\begin{proof}
 Clearly $x_I\mapsto x_{I^c}$ defines a graded $\Funpm$-linear isomorphism 
  \[\textstyle
  \tilde\alpha^\smallvee: \ \Funpm\big[ \, x_I \, \big| \, I\in\binom Er \, \big] \ \stackrel\sim\longrightarrow \ \Funpm\big[ \, x_I \, \big| \, I\in\binom E{r^\smallvee} \, \big].
 \]
 Thus we are left with verifying that $\tilde\alpha^\smallvee$ preserves the respective Pl\"ucker relations.
 
 For this verification, we rewrite the Pl\"ucker relations in a form that is more symmetric with respect to duality. For $I\subset E$ and $i\in E$, we define $\sigma(i,I)=\#\{j\in I|j\leq i\}$. Then the Pl\"ucker relation given by an $(r-1)$-subset $I$ and an $(r+1)$-subset $J$ of $E$ is
 \[
  0 \quad \leq \quad \sum_{i\in J-I} \ \epsilon^{\sigma(i,I)+\sigma(i,J)} \ \cdot \ x_{I\cup\{i\}} \ \cdot \ x_{J-\{i\}}.
 \]
 Note that
 \[
  (I\cup\{i\})^c=I^c-\{i\}, \quad (J-\{i\})^c=J^c\cup\{i\} \quad \text{and} \quad J-I=I^c-J^c.
 \]
 The last equality implies that $\sigma(i,J-I)=\sigma(i,I^c-J^c)$. Since $\sigma(i,J-I)=\sigma(i,J)-\sigma(i,I\cap J)$, and likewise for $\sigma(i,I^c-J^c)$, we obtain
 \[
  \sigma(i,J)-\sigma(i,I\cap J) \ = \ \sigma(i,I^c)-\sigma(i,J^c\cap I^c).
 \]
 Exchanging the roles of $I$ and $J$ yields an analogous equation. Adding both equations yields
 \[
  \sigma(i,I)+\sigma(i,J)-2\sigma(i,I\cap J) \ = \ \sigma(i,J^c)+\sigma(i,I^c)-2\sigma(i,J^c\cap I^c).
 \]
 This shows that $\epsilon^{\sigma(i,J^c)+\sigma(i,I^c)}=\epsilon^{\sigma(i,I)+\sigma(i,J)}$. Thus applying $\tilde\alpha^\smallvee$ to the Pl\"ucker relation for $I$ and $J$ yields
 \[
  0 \quad \leq \quad \sum_{i\in I^c-J^c} \ \epsilon^{\sigma(i,J^c)+\sigma(i,I^c)} \ \cdot \ x_{J^c\cup\{i\}} \ \cdot \ x_{I^c-\{i\}},
 \]
 which is the Pl\"ucker relation for the $(r^\smallvee-1)$-subset $J^c$ and the $(r^\smallvee+1)$-subset $I^c$ of $E$. We conclude that $\tilde\alpha^\smallvee$ maps $\cPl(r,E)$ to $\cPl(r^\smallvee,E)$, which completes the proof of the theorem.
\end{proof}

\begin{df}
 Let $X$ be an ordered blue $\Funpm$-scheme endowed with an involution (generalizing the involution on a tract $F$ from Section~\ref{subsubsection: cryptomorphisms for matroids over tracts}) $\iota:X\to X$, $\cL$ a line bundle on $X$, and $\cM$ a matroid bundle on $X$ that is represented by the Grassmann-Pl\"ucker function $\Delta:\binom Er\to\Gamma(X,\cL)$. The \emph{dual of $\Delta$ with respect to $\iota$} is the function
 \[\textstyle
  \begin{array}{cccc}
   \Delta_\iota^\smallvee: & \binom E{r^\smallvee} & \longrightarrow & \Gamma(X,\cL) \\
                      &     I            & \longmapsto     & \iota_\cL^\#\circ\Delta(I^c)
  \end{array}
 \]
 where 
 $\iota_\cL^\#:\Gamma(X,\cL) \to\Gamma(X,\cL)$ is the involution induced by $\iota$. 
 
 The \emph{dual of $\cM$ with respect to $\iota$} is the isomorphism class $\cM_\iota^\smallvee$ of $\Delta_\iota^\smallvee$.
\end{df}

In the following proposition, we verify that $\Delta_\iota^\smallvee$ is a Grassmann-Pl\"ucker function and thus that $\cM_\iota^\smallvee$ is a matroid bundle on $X$. Moreover, we will see that the duality of matroid bundles is compatible with the duality of the moduli spaces of matroids from Theorem \ref{thm: duality of matoid spaces}.

Note that in case that $X=\Spec F$ for an idyll $F$, the duality of the matroid bundle $\cM$ on $\Spec F$ coincides with the duality of the corresponding $F^\tract$-matroid $M$ from \cite[Thm.\ 2.24]{Baker-Bowler19}.

Given a matroid bundle $\cM$ on $X$, we call the morphism $\chi_\cM:X\to\Mat(r,E)$ that corresponds to $\cM$ under the bijection from Theorem \ref{thm: moduli space of matroids} the \emph{characteristic morphism of $\cM$}.

\begin{prop}\label{prop: dual matroid bundle}
 Let $X$ be an ordered blue $\Funpm$-scheme with involution $\iota:X\to X$ and $\cL$ a line bundle on $X$. Let $\Delta:\binom Er\to\Gamma(X,\cL)$ be a Grassmann-Pl\"ucker function that represents a matroid bundle $\cM$ on $X$ with characteristic morphism $\chi_\cM:X\to \Mat(r,E)$. Then the dual $\Delta_\iota^\smallvee$ of $\Delta$ with respect to $\iota$ is a Grassmann-Pl\"ucker function and $\cM_\iota^\smallvee$ is the matroid bundle on $X$ whose characteristic morphism is
 \[
  \chi_{\cM_\iota^\smallvee} \ = \ \varphi^\smallvee\circ\chi_\cM\circ\iota: \ X \ \stackrel{\iota}\longrightarrow \ X \ \stackrel{\chi_\cM}\longrightarrow \ \Mat(r,E) \ \stackrel{\varphi^\smallvee}\longrightarrow \ \Mat(r^\smallvee,E)
 \]
 where $r^\smallvee=\#E-r$ and $\varphi^\smallvee$ is the isomorphism from Theorem \ref{thm: duality of matoid spaces}.
\end{prop}

\begin{proof}
 That $\Delta_\iota^\smallvee$ is a Grassmann-Pl\"ucker function can be shown directly by an analogous calculation to that from the proof of Theorem \ref{thm: duality of matoid spaces}. Alternatively, we can show this by applying the result from Theorem \ref{thm: duality of matoid spaces} in the following way. 
 
 The direct sum $\bigoplus_{i\geq0}\Gamma(X,\cL^{\otimes i})$ can be given the structure of an ordered blueprint $(B^\bullet,B^+,\leq)$ as follows: 
 \begin{itemize}
 \item The ambient semiring $B^+$ is the direct sum of the semigroups $\Gamma(X,\cL^{\otimes i})^+$ for all $i\geq0$, which comes with a natural multiplication. 
 \item The monoid $B^\bullet$ is the union of all the subsets $\Gamma(X,\cL^{\otimes i})^\bullet$ in $B^+$. 
 \item The partial order $\leq$ is the smallest additive and multiplicative partial order that contains the partial order of $\Gamma(X,\cL^{\otimes i})^+$ for every $i\geq0$. 
 \end{itemize}
 
 Since $\Delta$ is a Grassmann-Pl\"ucker function, the association $x_I\mapsto \Delta(I)$ defines a morphism
 \[\textstyle
  \xi_\Delta: \ \bpquot{\Funpm\big[ \, x_I \, \big| \, I\in\binom Er \, \big]}{\cPl(r,E)} \ \longrightarrow \ \bigoplus_{i\geq0}\Gamma(X,\cL^{\otimes i}).
 \]
 Composing this morphism with the map $j_{r,E}:\binom Er\to \bpquot{\Funpm\big[x_I\big|I\in\binom Er\big]}{\cPl(r,E)}$ that sends $I$ to $x_I$ gives the Grassmann-Pl\"ucker function $\Delta:\binom Er\to\Gamma(X,\cL)$, where we consider $\Gamma(X,\cL)$ as a subset of $\bigoplus_{i\geq0}\Gamma(X,\cL^{\otimes i})$. 
 
 Precomposing $\xi_\Delta$ with the inverse of the isomorphism $\alpha^\smallvee$ from Theorem \ref{thm: duality of matoid spaces} yields a morphism $\xi_{\Delta_\iota^\smallvee}$, whose composition with $j_{r^\smallvee,E}$ is the dual $\Delta_\iota^\smallvee$ of $\Delta$. This means that we obtain a commutative diagram
 \[\textstyle
  \beginpgfgraphicnamed{tikz/fig11}
   \begin{tikzcd}[row sep=0pt, column sep=50pt]
                                                                                                                               & \binom Er \arrow{dl}[swap]{j_{r,E}}\arrow{rrd}{\Delta} \\
    \bpquot{\Funpm\big[ \, x_I \, \big| \, I\in\binom Er \, \big]}{\cPl(r,E)} \arrow{ddd}[swap]{\alpha^\smallvee} \arrow{rrr}[swap]{\xi_\Delta} &                                          & & \bigoplus_{i\geq0}\Gamma(X,\cL^{\otimes i})\arrow{ddd}{\iota} \\
    \ \\
    \ \\
    \bpquot{\Funpm\big[ \, x_I \, \big| \, I\in\binom E{r^\smallvee} \, \big]}{\cPl(r^\smallvee,E)} \arrow{rrr}{\xi_{\Delta_\iota^\smallvee}} &                                                      & & \bigoplus_{i\geq0}\Gamma(X,\cL^{\otimes i}).  \\
                                                                                                                               & \binom Er \arrow{ul}{j_{r^\smallvee,E}}\arrow{rru}[swap]{\Delta_\iota^\smallvee}
   \end{tikzcd}
  \endpgfgraphicnamed
 \]
 Since $\Delta_\iota^\smallvee$ factors through $\xi_{\Delta_\iota^\smallvee}$, it satisfies the Pl\"ucker relations. Thus $\Delta_\iota^\smallvee$ is a Grassmann-Pl\"ucker function, which verifies the first part of the proposition.
 
 Note that the characteristic morphism $\chi_\cM:X\to \Mat(r,E)$ is induced from the graded morphism $\xi_\Delta$. Thus applying the $\Proj$-functor to inner square of the above diagram yields a commutative diagram
 \[\textstyle
  \beginpgfgraphicnamed{tikz/fig12}
   \begin{tikzcd}[row sep=15pt, column sep=80pt]
    X \arrow{r}{\chi_\cM} \arrow{d}[swap]{\iota} & \Mat(r,E) \arrow{d}{\varphi^\smallvee} \\
    X \arrow{r}{\chi_{\cM_\iota^\smallvee}} & \Mat(r^\smallvee,E),
   \end{tikzcd}
  \endpgfgraphicnamed
 \]
 which verifies the second part of the proposition.
\end{proof}


\subsection{Contraction and deletion}
\label{subsection: contraction and deletion}

The operations of contracting and deleting an element $e$ of the ground set $E$ of an $F$-matroid $M$, as introduced in \cite[section 3.9]{Baker-Bowler19}, are defined on the level of maps between subvarieties of the matroid spaces for $E$ and $E'=E-\{e\}$ and appropriate ranks. To this end we define the closed subschemes
\[\textstyle
 V_{/e} \ = \ \Proj\Big(\bpgenquot{\Funpm[T_I\mid I\in\binom Er]}{\cPl(r,E)\cup\{T_I\mid e\in I\}}\Big)
\]
and 
\[\textstyle
 V_{\backslash e} \ = \ \Proj\Big(\bpgenquot{\Funpm[T_I\mid I\in\binom Er]}{\cPl(r,E)\cup\{T_I\mid e\notin I\}}\Big)
\]
of $\Mat(r,E)$, as well as their respective set-theoretic complements
\[
 U_{/e} \ = \ \Mat(r,E) - V_{/e} \qquad \text{and} \qquad U_{\backslash e} \ = \ \Mat(r,E) - V_{\backslash e},
\]
which we consider as open subschemes of $\Mat(r,E)$. 

The graded morphism
\[
 \begin{array}{ccc}
  \bpquot{\Funpm[T_J\mid J\in\binom{E'}{r-1}]}{\cPl(r-1,E')} & \longrightarrow & \bpquot{\Funpm[T_I\mid I\in\binom Er]}{\cPl(r,E)} \\
  T_J                                                                  & \longmapsto     & T_{J\cup\{e\}}
 \end{array}
\]
defines a rational map $\Mat(r,E)\dashedrightarrow\Mat(r-1,E')$ whose domain is $U_{/e}$, since the inverse image of a homogeneous prime ideal $\gen{T_I\mid I\in\cI}$ is relevant if and only if $\cI$ contains an $I$ with $e\notin I$. This yields a morphism $\Psi^o_{/e}: U_{/e}\to\Mat(r-1,E')$. The graded morphism
\[
 \begin{array}{ccc}
  \bpquot{\Funpm[T_J\mid J\in\binom{E'}{r}]}{\cPl(r,E')} & \longrightarrow & \bpgenquot{\Funpm[T_I\mid I\in\binom Er]}{\cPl(r,E)\cup\{T_I\mid e\in I\}} \\
  T_J                                                              & \longmapsto     & T_{J}
 \end{array}
\]
is an isomorphism and defines an isomorphism $\Psi^c_{/e}: V_{/e}\to\Mat(r,E')$ of ordered blue schemes. Combining these morphisms, we obtain the diagram
\[
 \Mat(r,E) \ \stackrel{\iota_{/e}}{\longleftarrow} \ U_{/e} \, \amalg \, V_{/e} \ \stackrel{\Psi_{/e}}{\longrightarrow} \ \Mat(r-1,E') \, \amalg \, \Mat(r,E'),
\]
where $\iota_{/e}$ is the disjoint union of the inclusions of the subschemes $U_{/e}$ and $V_{/e}$ into $\Mat(r,E)$ and $\Psi_{/e}$ is the disjoint union of $\Psi^o_{/e}$ and $\Psi^c_{/e}$.

Similarly, the graded morphism
\[
 \begin{array}{ccc}
  \bpquot{\Funpm[T_J\mid J\in\binom{E'}{r}]}{\cPl(r,E')} & \longrightarrow & \bpquot{\Funpm[T_I\mid I\in\binom Er]}{\cPl(r,E)} \\
  T_J                                                                  & \longmapsto     & T_{J}
 \end{array}
\]
and the graded isomorphism
\[
 \begin{array}{ccc}
  \bpquot{\Funpm[T_J\mid J\in\binom{E'}{r-1}]}{\cPl(r-1,E')} & \longrightarrow & \bpgenquot{\Funpm[T_I\mid I\in\binom Er]}{\cPl(r,E)\cup\{T_I\mid e\in I\}} \\
  T_J                                                              & \longmapsto     & T_{J\cup\{e\}}
 \end{array}
\]
define morphisms $\Psi^o_{\backslash e}: U_{\backslash e}\to\Mat(r,E')$ and $\Psi^c_{\backslash e}: V_{\backslash e}\to\Mat(r-1,E')$
of ordered blue schemes, and combining these yields the diagram
\[
 \Mat(r,E) \ \stackrel{\iota_{\backslash e}}{\longleftarrow} \ U_{\backslash e} \, \amalg \, V_{\backslash e} \ \stackrel{\Psi_{\backslash e}}{\longrightarrow} \ \Mat(r,E') \, \amalg \, \Mat(r-1,E').
\]

The following theorem explains how these morphisms extend the usual operations of contraction and deletion to the level of moduli spaces. Since we will not use this result in the paper, we omit a proof.

\begin{thm}\label{thm: contraction and contraction}
 Let $F$ be an idyll and $M$ an $F$-matroid of rank $r$ on $E$ with characteristic morphism $\chi_M:\Spec F\to \Mat(r,E)$. Let $e\in E$ and $E'=E-\{e\}$. Define $r_{/e}=r$ if $e$ is a loop and $r_{/e}=r-1$ if not, and define $r_{\backslash e}=r-1$ if $e$ is a coloop and $r_{\backslash e}=r$ if not. Let $\chi_{M/e}:\Spec F\to\Mat(r_{/e},E')$ and $\chi_{M\backslash e}:\Spec F\to\Mat(r_{\backslash e},E')$ be the characteristic morphisms of the contraction $M/e$ and the deletion $M\backslash e$, respectively. Then the following holds:
 \begin{enumerate}
  \item The morphism $\chi_M$ factors into a uniquely determined morphism $\chi_{M,/e}:\Spec F\to U_{/e}\amalg V_{/e}$ composed with $\iota_{/e}$, as well as into a uniquely determined morphism $\chi_{M,\backslash e}:\Spec F\to U_{\backslash e}\amalg V_{\backslash e}$ composed with $\iota_{\backslash e}$. 
  \item The morphism $\chi_{M\backslash e}$ is the unique morphism from $\Spec F$ to $\Mat(r_{/e},E')$ that makes the diagram
  \[
   \beginpgfgraphicnamed{tikz/fig14}
   \begin{tikzcd}[column sep=1cm]
               & \Spec F \ar[dl,"\chi_M"'] \ar[d,"\chi_{M,/e}"] \ar[r,"\chi_{M/e}"] & \Mat(r_{/e},E') \ar[d] \\
    \Mat(r,E) & U_{/e} \ \amalg \ V_{/e} \ar[l,"\iota_{/e}"'] \ar[r,"\Psi_{/e}"] & \Mat(r-1,E') \ \amalg \ \Mat(r,E')
   \end{tikzcd}
   \endpgfgraphicnamed
  \]
  commute, and $\chi_{M/e}$ is the unique morphism from $\Spec F$ to $\Mat(r_{\backslash e},E')$ that makes the diagram
  \[
   \beginpgfgraphicnamed{tikz/fig15}
   \begin{tikzcd}[column sep=1cm]
               & \Spec F \ar[dl,"\chi_M"'] \ar[d,"\chi_{M,\backslash e}"] \ar[r,"\chi_{M\backslash e}"] & \Mat(r_{\backslash e},E') \ar[d] \\
    \Mat(r,E) & U_{\backslash e} \ \amalg \ V_{\backslash e} \ar[l,"\iota_{\backslash e}"'] \ar[r,"\Psi_{\backslash e}"] & \Mat(r,E') \ \amalg \ \Mat(r-1,E')
   \end{tikzcd}
   \endpgfgraphicnamed
  \]
  commute, where the vertical arrows on the right hand side are the respective canonical inclusions into the coproduct.
  \item Let $r^\vee=\# E-r$ and let $U_{\backslash e}^\vee$ and $V^\vee_{\backslash e}$ be the obvious respective variants of $U_{\backslash e}$ and $V_{\backslash e}$ for $\Mat(r^\vee,E)$. Then there are unique isomorphisms $\varphi^{o,\vee}:U_{/e}\to U_{\backslash e}^\vee$ and $\varphi^{c,\vee}:V_{/e}\to V_{\backslash e}^\vee$ that make the diagram
  \[
   \beginpgfgraphicnamed{tikz/fig16}
   \begin{tikzcd}[column sep=1cm]
    \Mat(r,E) \ar[d,"\varphi^\vee"',"\sim"] & U_{/e} \ \amalg \ V_{/e} \ar[l,"\iota_{/e}"'] \ar[r,"\Psi_{/e}"] \ar[d,"\varphi^{o,\vee}\amalg\varphi^{c,\vee}","\sim"'] & \Mat(r-1,E') \ \amalg \ \Mat(r,E') \ar[d,"\varphi^\vee\amalg\varphi^\vee","\sim"'] \\
    \Mat(r^\vee,E) & U_{\backslash e}^\vee \ \amalg \ V^\vee_{\backslash e} \ar[l,"\iota_{\backslash e}^\vee"'] \ar[r,"\Psi_{\backslash e}^\vee"] & \Mat(r^\vee,E') \ \amalg \ \Mat(r^\vee-1,E')
   \end{tikzcd}
   \endpgfgraphicnamed
  \]
  commute, where $\iota_{\backslash e}^\vee$ and $\Psi_{\backslash e}^\vee$ are the obvious variants of $\iota_{\backslash e}$ and $\Psi_{\backslash e}$, respectively, and where the vertical morphisms denoted by $\varphi^\vee$ are the isomorphisms from Theorem \ref{thm: duality of matoid spaces}.
 \end{enumerate}
 
\end{thm}


\subsection{Rational point sets}
\label{subsection: rational point sets}

In this section, we explain how the matroid space recovers classical objects like the Grassmannian, the Dressian and the MacPhersonian as rational point sets. 

Let $B$ be an $\Funpm$-algebra. By the universal property of the matroid space, $\Mat(r,E)(B)$ corresponds to the set of $B$-matroids of rank $r$ on $E$. If $B$ carries a topology, then $\Mat(r,E)(B)$ inherits the so-called \emph{fine topology} from $B$. The fine topology is defined by a general categorical construction, which has been exhibited first in \cite{Lorscheid-Salgado16} and which has been transferred to rational point sets of ordered blue schemes in \cite{Lorscheid15}. Instead of recalling the definition of the fine topology in full generality, we will provide an equivalent characterization in Theorem \ref{thm :fine topology for topological idylls}. 

A \emph{topological idyll} is an idyll $F$ together with a topology such that the multiplication $F\times F\to F$  (where $F\times F$ carries the product topology) is a continuous map, and such that $F^\times$ is an open subset of $F$ and the inversion map $F^\times\to F^\times$, sending $a$ to $a^{-1}$, is continuous. 

\begin{rem}
 It might appear strange at first sight that the definition of a topological idyll does not involve any continuity condition for addition. Thus a topological idyll that is a field is not necessarily a topological field.  However, our definition is guided by properties of the fine topology on rational point sets, as described in Theorem \ref{thm :fine topology for topological idylls} below. The proof of these properties does not require any continuity conditions for addition, in contrast to the corresponding proof for topological fields. This difference in the proofs can be traced back to the fact that free algebras in the world of ordered blueprints consist of monomials (as opposed to more general polynomials).
 \end{rem}

Given an ordered blue scheme $X$ and a topological idyll $F$, the fine topology for the rational point set $X(F)$ is determined in terms of the following theorem.

\begin{thm}\label{thm :fine topology for topological idylls} 
 Let $F$ be a topological idyll. Then there is a unique way to endow the rational point sets $X(F)$ for all ordered blue schemes $X$ with a topology such that the following properties hold true:
 \begin{enumerate}
  \item the canonical bijection $F\to \A_F^1(F)$ is a homeomorphism;
  \item the canonical bijection $(X\times Y)(F)\to X(F)\times Y(F)$ is a homeomorphism;
  \item for every morphism $Y \to X$, the canonical map $Y(F)\to X(F)$ is continuous;
  \item for every open / closed immersion $Y\to X$, the canonical inclusion $Y(F)\to X(F)$ is an open / closed topological embedding;
  \item for every covering of $X$ by ordered blue open subschemes $U_i$, a subset $W$ of $X(F)$ is open if and only if $W\cap U_i(F)$ is open in $U_i(F)$ for every $i$.
 \end{enumerate}
 Moreover, if $F\to F'$ is a continuous morphism of idylls and $X$ an ordered blue scheme, the induced map $X(F)\to X(F')$ is continuous.
\end{thm}

\begin{proof}
 This is a special case of Theorem 5.2 in \cite{Lorscheid15}.
\end{proof}

\begin{ex}
 Every topological field is a topological idyll in a tautological way. Anderson and Davis have extended this notion to hyperfields in \cite{Anderson-Davis19}. It turns out that a topological hyperfield is the same as a topological idyll if identified with the associated idyll via the functor $(-)^\oblpr:\HypFields\to\OBlpr^\pm$. In the following, we consider the following topological idylls:
 \begin{itemize}
  \item the reals $\R$ with the usual topology;
  \item the Krasner hyperfield $\K$ together with the topology that consists of the open subsets $\emptyset$, $\{1\}$, $\K$;
  \item the sign hyperfield $\S$ together with the topology that consists of the open subsets $\emptyset$, $\{1\}$, $\{-1\}$, $\{\pm 1\}$, $\S$;
  \item the tropical hyperfield $\T$ together with the topology coming from the identification of $\T$ with $\R_{\geq0}$ and its embedding into $\R$;
  \item the regular partial field $\Funpm$ together with the topology that consists of the open subsets $\emptyset$, $\{1\}$, $\{-1\}$, $\{\pm 1\}$, $\Funpm$.
 \end{itemize}
 Note that \cite{Anderson-Davis19} contains reasons why it might be better to exclude all neighborhoods of $0$ in the topology of $\T$; we refer to section 2.3.2 of {\it loc.~cit.}, but ignore this issue in the following. The topological spaces $X(F)$ appearing in Theorem \ref{thm :fine topology for topological idylls} are also closely related to Jun's considerations in \cite{Jun17}. Namely for $F=\K$, $F=\T$ or $F=\S$, a blue scheme $X$ and its associated scheme $X^+$, it is not hard to show that the topological space $X(F)$ coincides with the topological space $X^+(F)$ from \cite{Jun17}.
\end{ex}

\subsubsection{Matroids}

A matroid is the same as a $\K$-matroid where $\K=\bpgenquot{\{0,1\}}{0\leq 1+1,0\leq 1+1+1}$ is the Krasner hyperfield. Thus $\Mat(r,E)(\K)$ is the set of all matroids of rank $r$ on $E$. The topology on $\K$ turns $\Mat(r,E)(\K)$ into a contractible topological space, cf.\ \cite[section 6]{Anderson-Davis19} for details.

\subsubsection{Oriented matroids and the MacPhersonian}
\label{subsubsection: oriented matroids and the MacPhersonian}

Note that as an idyll, the sign hyperfield turns into $\S=\bpquot{\{0,1,\epsilon\}}{\cR}$ where $\cR$ is generated by relations $0\leq 1+\dotsb+1+\epsilon+\dotsb+\epsilon$ that contain at least one $1$ and one $\epsilon$.

An oriented matroid is the same thing as a $\S$-matroid. Thus $\Mat(r,E)(\S)$ is the set of all oriented matroids of rank $r$ on $E$. The topology of $\S$ turns $\Mat(r,E)(\S)$ into a topological space, which is, by definition, the MacPhersonian $\MacPh(r,E)$ of rank $r$ on $E$; cf.\ \cite[section 6]{Anderson-Davis19} for details.

\subsubsection{Subspaces and the Grassmannian}

Let $k$ be a field, which we identify with the idyll $\bpgenquot{k^\bullet}{0\leq \sum a_i|\sum a_i=0\text{ in }k}$. (Note that this results from considering $k$ as a partial field and applying the functor $\PartFields\to\OBlpr^\pm$ or, equivalently, from considering $k$ as a hyperfield and applying the functor $\HypFields\to\OBlpr^\pm$. This allows us to consider fields as objects of either subcategory $\PartFields$ and $\HypFields$ of $\OBlpr^\pm$.)

It is immediate that the class of a Grassmann-Pl\"ucker function $\Delta:\binom Er\to k$ corresponds to the point $\big[\Delta(I)\big|I\in\binom Er\big]$ of the Grassmannian $\Gr(r,E)(k)$ and vice-versa. This yields an identification $\Mat(r,E)(k)=\Gr(r,E)(k)$ and shows that a $k$-matroid is the same thing as an $r$-dimensional subspace of $k^E$.


\subsubsection{The oriented matroid of real subspaces}

The topology of $\R$ endows $\Mat(r,E)(\R)$ with a topology that coincides with the usual topology of the real Grassmannian. The hyperfield morphism $\sign:\R\to\S$ is continuous and therefore induces a continuous map 
\[
 \Gr(r,E)(\R) \ = \ \Mat(r,E)(\R) \ \longrightarrow \ \Mat(r,E)(\S) \ = \ \MacPh(r,E).
\]
This map sends an $r$-dimensional subspace $V$ of $\R^E$ to its associated oriented matroid $M_V$, which is the class of the Grassmann Pl\"ucker function $\sign\circ\Delta:\binom Er\to \S$, where $\Delta$ is defined by the Pl\"ucker coordinates $\big[\Delta(I)\big|I\in\binom Er\big]$ of $V$.

This map is closely connected to the MacPhersonian conjecture, as formulated by Mn\"ev and Ziegler in \cite{Mnev-Ziegler93}, which asserts a relation between the homotopy type of $\Gr(r,E)(\R)$ and the MacPhersonian $\MacPh(r,E)$. For more details on these connections, see \cite[section 7]{Anderson-Davis19}. Note that certain cases of this conjecture have recently been disproven by Liu in \cite{Liu17}.

\subsubsection{Valuated matroids and the Dressian}

A valuated matroid is the same thing as a $\T$-matroid, where $\T$ is the tropical hyperfield. Thus $\Mat(r,E)(\T)$ is the set of all valuated matroids of rank $r$ on $E$. An \emph{$r$-dimensional tropical linear space in $\R^E$} is the geometric realization of a valuated matroid as a subspace of $\R^E$, analogous to the Bergman fan of a matroid; cf.\ \cite{Speyer08} for a precise definition. The \emph{Dressian $\Dress(r,E)$} is the set of $r$-dimensional tropical linear spaces in $\R^E$.

By definition, the $r$-dimensional tropical linear spaces in $\R^E$ correspond bijectively to the valuated matroids of rank $r$ on $E$. This yields an identification $\Dress(r,E)=\Mat(r,E)(\T)$ of the Dressian with the $\T$-rational points of the matroid space. Note that the topology of $\T$ endows the Dressian $\Dress(r,E)$ with a natural topology.

\subsubsection{Regular matroids}

It follows from our explanations in section \ref{subsubsection: relation to regular matroids} that the subset of regular matroids in $\Mat(r,E)(\K)$ is equal to the image of the map $\Mat(r,E)(\Funpm)\to\Mat(r,E)(\K)$ induced by the unique morphism $\Funpm\to\K$. Note that the topology of $\Funpm$ endows the set of $\Funpm$-matroids $\Mat(r,E)(\Funpm)$ with a topology. Since the morphism $\Funpm\to\K$ is continuous, the map $\Mat(r,E)(\Funpm)\to\Mat(r,E)(\K)$ is continuous. 

Note that this map is in general not injective, as the following example shows. Let $E=\{1,2\}$. Then the Grassmann-Pl\"ucker functions $\Delta_1:\binom E1\to\Funpm$ and $\Delta_2:\binom E1\to\Funpm$  with
\[
 \Delta_1(\{1\}) \ = \ \Delta_1(\{2\}) \ = \ \Delta_2(\{1\})=1 \quad\text{and}\quad \Delta_2(\{2\}) \ = \ \epsilon
\]
define different $\Funpm$-matroids $M_1=[\Delta_1]$ and $M_2=[\Delta_2]$ with the same underlying matroid.


\bigskip
\part{Applications to matroid theory}
\bigskip

\section{Realization spaces and the Tutte group}
\label{section: Applications to realization spaces}

A new feature that comes along with the matroid space is the universal idyll associated with a matroid. We will introduce this notion and explain how it interacts with questions about the representability of matroids and realization spaces. We will 
also discuss the analogous invariant for weak matroids and its relation to the Tutte group.

Throughout the entire section, we fix a totally ordered non-empty finite set $E$ and a natural number $r\leq\#E$.


\subsection{The universal idyll of a matroid}
\label{subsection: the universal idyll of a matroid}

We can associate with every matroid its universal idyll, which is derived from a certain residue field of the matroid space. We will define the universal idyll and describe its basic properties in this section.

Let $N=\#\binom Er-1$. Recall from section \ref{subsection: The moduli space of matroids} that the matroid space comes with a closed immersion
\[
 \iota: \ \Mat(r,E) \ \longrightarrow \ \P^N_\Funpm
\]
into ordered blue projective space.

\begin{lemma}
 The closed immersion $\iota$ is a homeomorphism between the respective underlying topological spaces.
\end{lemma}

\begin{proof}
 Since $\iota$ is a closed immersion, it is clearly injective and continuous. 
 Since the Pl\"ucker relations in the definition of $\Mat(r,E)$ are merely inequalities, they do not identify any elements of the underlying monoid of $\Funpm\big[x_I \; \big | \; I\in\binom Er\big]$. As a result, the underlying topological space of $\Mat(r,E)$ is the same as that of its image in $\P_{\Funpm}^N$. 
\end{proof}

In section \ref{subsubsection: projective space}, we have defined the points of $\P_{\Funpm}^N$ as the relevant homogeneous prime ideals $\fp_\cI=(T_I)_{I\in\cI}$ of $\Funpm\big[x_I \; \big | \; I\in\binom Er\big]$, where $\cI$ can be any proper subset of $\binom Er$. This means that the underlying points of $\Mat(r,E)$ are of the form $\fp_\cI$ for $\cI\subset\binom Er$.

Fix a proper subset $\cI$ of $\binom Er$.  The \emph{stalk at $\fp_\cI$} is the ordered blueprint
\[\textstyle
 \cO_{\Mat(r,E),\fp_\cI} \ = \ \big(\bpquot{\Funpm\big[ \, x_I^\pm, x_J \, \big| \, I\in\cI,J\in\cI^c \, \big]}{\cPl(r,E)}\big)_0
\]
where $(-)_0$ refers to the degree-$0$ part of the graded ordered blueprint in brackets and where $\cPl(r,E)$ is generated by the Pl\"ucker relations
\[
 0 \quad \leq \quad \sum_{k=1}^{r+1} \ \epsilon^k \ \cdot \ x_{I\cup\{i_k\}} \ \cdot \ x_{I'-\{i_k\}} 
\]
for every $(r-1)$-subset $I$ and every $(r+1)$-subset $I'=\{i_1,\dotsc,i_{r+1}\}$ of $E$ with $i_1<\dotsb<i_{r+1}$. The \emph{residue field at $\fp_\cI$} is 
\[\textstyle
 k(\fp_\cI) \ = \ \bpgenquot{\cO_{\Mat(r,E),\fp_\cI}}{x_Jx_I^{-1}\=0 \; | \; J\in\cI^c}
\]
where $I\in\cI$ is an arbitrary fixed index that allows us to express the equation $x_J=0$ in terms of elements of $\cO_{\Mat(r,E),\fp_J}$, which have degree $0$.

Note that the residue field $k(\fp_\cI)$ is not a field in the classical sense. For the matroid space, it turns out that residue fields are always ordered blue fields, but in general it happens that some residue fields are the trivial ordered blueprint with $0=1$; cf.\ \cite[section 5.9]{Lorscheid18} for more details.

\begin{df}
 Let $F$ be an idyll and $M$ an $F$-matroid. The \emph{terminal map of $F$} is the unique morphism $t_F:F\to\K$ into the Krasner hyperfield $\K$. The \emph{characteristic morphism of $M$} is the morphism $\chi_M:\Spec F\to\Mat(r,E)$ determined by the bijection in Theorem \ref{thm: moduli space of matroids}. The \emph{underlying matroid of $M$} is the push-forward $t_{F,\ast}(M)$ of $M$ along the terminal map $t_F:F\to\K$.
\end{df}
 
\begin{df}
 Let $M$ be a matroid with characteristic morphism $\chi_M:\Spec\K\to\Mat(r,E)$. The \emph{support of $M$} is the image point $x_M$ of $\chi_M$ and the \emph{universal idyll of $M$} is $k_M=k(x_M)^\pm$.
\end{df}

More explicitly, we have
\[\textstyle
 k_M \ = \ \big(\bpquot{\Funpm\big[ \, x_I^\pm \, \big| \, I\in\cI\, \big]}{\cPl(r,E)}\big)_0^\pm
\]
where $\cPl(r,E)$ is generated by the Pl\"ucker relations
\[
 0 \quad \leq \quad \sum_{k=1}^{r+1} \ \epsilon^k \ \cdot \ x_{I\cup\{i_k\}} \ \cdot \ x_{I'-\{i_k\}} 
\]
for every $(r-1)$-subset $I$ and every $(r+1)$-subset $I'=\{i_1,\dotsc,i_{r+1}\}$ of $E$ with $i_1<\dotsb<i_{r+1}$. 

\begin{rem}\label{rem: the universal idyll is an idyll}
 The universal idyll is indeed an idyll, which can be seen as follows. From the above description, it is clear that $k_M$ is a purely positive ordered blue field with unique weak inverses. The residue field $k(\fp_\cI)$ satisfies the required property $k(\fp_\cI)^+=\N[k(\fp_\cI)^\times]$, and this property is inherited by $k_M$ since the quotient $k_M^+$ of $k(\fp_\cI)^+$ is defined by relations between elements of the underlying monoid of $k(\fp_\cI)$ (elements $a$ and $b$ of $k(\fp_\cI)^\bullet$ become identified in $k_M$ whenever there is a Pl\"ucker relation of the form $0\leq a+b$). 
\end{rem}

We denote the underlying topological space of $\Mat(r,E)$ by $\Mat(r,E)^\top$.

\begin{prop}\label{prop: embedding of matroids into the matroid space and its image}
 The map
 \[
  \begin{array}{cccc}
   \Phi: & \Mat(r,E)(\K) & \longrightarrow & \Mat(r,E)^\top \\
         &     \chi      & \longmapsto     & \im\chi
  \end{array}
 \]
 is injective, where we identify $\im\chi=\{x\}$ with the point $x$. The points $x$ of $\Mat(r,E)^\top$ that are supports of matroids are characterized by the following equivalent assertions:
 \begin{enumerate}
  \item\label{rep0} $x$ is the support of a matroid;
  \item\label{rep1} $x$ is in the image of $\Phi$;
  \item\label{rep2} there is a unique morphism $k(x)^\pm\to\K$;
  \item\label{rep3} $k(x)^\pm\neq\{0\}$;
  \item\label{rep4} $k(x)^\pm$ is an idyll.
 \end{enumerate}
\end{prop}

\begin{proof}
 By the definition of the support of a matroid, \eqref{rep0} and \eqref{rep1} are equivalent. Thus we are left with proving the equivalence of \eqref{rep1} with the latter affirmations. By Theorem \ref{thm: moduli space of matroids}, a morphism $\chi:\Spec\K\to\Mat(r,E)$ corresponds to a matroid $M$. Let $\Delta:\binom Er\to\K$ be the unique Grassmann-Pl\"ucker function that represents $M$ and $\fp_\cI$ the image point of $\chi$. Then we have $\Delta(I)=0$ if and only if $I\in\cI$. This shows that $\Delta$ and $M$ are determined by $\chi$ and that $\Phi$ is injective. This proves the first part of the theorem.
 
 
 Let $x$ be a point of $\Mat(r,E)^\top$. We begin with \eqref{rep1}$\Rightarrow$\eqref{rep2}. Assume that $x$ is the image point of a morphism $\chi:\K\to\Mat(r,E)$. This means that $\chi$ factors into a uniquely determined morphism $\Spec \K\to\Spec k(x)$ followed by $\Spec k(x)\to\Mat(r,E)$. This yields a morphism $k(x)\to\K$, which extends uniquely to a morphism $k(x)^\pm\to\K$. Thus \eqref{rep2}.
 
 The existence of a morphism $k(x)^\pm\to\K$ implies that $k(x)^\pm\neq\{0\}$, thus \eqref{rep2}$\Rightarrow$\eqref{rep3}.
 
 We continue with \eqref{rep3}$\Rightarrow$\eqref{rep4}. If $k(x)^\pm\neq\{0\}$, then it is an ordered blue field. It is an $\Funpm$-algebra with unique weak inverses by definition, and the partial order of $k(x)^\pm$ is generated by relations of the form $0\leq\sum a_i$, which shows that $k(x)^\pm$ is an idyll. Thus \eqref{rep4}.
 
 We continue with \eqref{rep4}$\Rightarrow$\eqref{rep1}. Assume that $k(x)^\pm$ is an idyll and let $t_{x}:k(x)^\pm\to\K$ be the terminal map. Then we obtain a morphism
 \[
  \Spec\K \ \stackrel{t_{x}^\ast}\longrightarrow \ \Spec k(x)^\pm \ \longrightarrow \ \Spec k(x) \ \longrightarrow \ \Mat(r,E),
 \]
 which shows that $x$ is in the image of $\Phi$ and thus \eqref{rep1}. This concludes the proof.
\end{proof}

\begin{cor}\label{cor: supports of matroids are image points of matroids}
 Let $M$ be a matroid with support $x_M$ and with characteristic morphism $\chi_M:\Spec\K\to\Mat(r,E)$. Let $t_M:k_M\to\K$ be the terminal map. Then $\chi_M$ equals the composition
 \[
  \Spec\K \ \stackrel{t_M^\ast}\longrightarrow \ \Spec k_M \ \longrightarrow \ \Spec k(x_M) \ \longrightarrow \ \Mat(r,E)
 \]
 where the middle morphism is induced by the canonical map $k(x_M)\to k(x_M)^\pm=k_M$ and the last morphism is the canonical inclusion of the spectrum of the residue field.
\end{cor}

\begin{proof}
 By the latter part of Proposition \ref{prop: embedding of matroids into the matroid space and its image}, $k_M$ is an idyll, thus $k_M$ comes with a terminal map $t_M:k_M\to\K$. By the first part of Proposition \ref{prop: embedding of matroids into the matroid space and its image}, there is at most one morphism $\Spec\K\to\Mat(r,E)$ with given image point. Thus the morphism resulting from $t_M^\ast$ with the canonical morphism $\Spec k_M\to\Mat(r,E)$ must be equal to $\chi_M$.
\end{proof}

\begin{rem}
 As a consequence of Proposition \ref{prop: embedding of matroids into the matroid space and its image} and Corollary \ref{cor: supports of matroids are image points of matroids}, we see that only the points $x$ in the image of $\Phi$ are supports of matroids. 
\end{rem}

\begin{lemma}\label{lemma: number of bases of a point of the matroid space}
 Let $k$ be a field and $N=\#\binom Er-1$. Let $x$ be a point of $\Mat(r,E)$. Then $k(x)^\times$ is the product of $\{1,\epsilon\}$ with a free abelian group whose rank is equal to the dimension of the closed subvariety $\big(\overline{\{x\}}\otimes_{\Funpm}k\big)^+$ of $\P^N_k$, where $\overline{\{x\}}$ is the closure of $x$ in $\Mat(r,E)$. 
\end{lemma}

\begin{proof}
 Let $\cI$ be the subset of $\binom Er$ such that $x=\fp_\cI$ and $\cI^c$ its complement in $\binom Er$. Then 
 \[\textstyle
  k(x)^\times \ = \ \{1,\epsilon\} \ \times \ \big\{ \, \prod_{I\notin\cI^c}x_I^{e_I} \, \big| \, e_I\in\Z\text{ with }\sum_{I\in\cI^c} e_I=0 \, \big\},
 \]
 whose second factor is a free abelian group of rank $\#\cI^c-1=N-\#\cI$. This rank is equal to the dimension of 
 \[
  \big(\overline{\{x\}}\otimes_{\Funpm}k\big)^+ \ = \ \Proj \big( \, k[ \, x_I \, | \, I\in\cI^c \, ] \, \big),
 \]
 which proves the assertion of the lemma. 
\end{proof}

\begin{rem}
 Note that if $M$ is a matroid with support $x_M=\fp_\cI$, then the complement $\cB=\cI^c$ of $\cI$ in $\binom Er$ is the set of bases of the matroid $M$. Thus the rank of $k(x)^\times$ is one less than the number of bases of the matroid.
 
 Note however that structure of the universal idyll $k_M$ is more complicated. In general, $k_M^\times$ is a proper quotient of $k(x)^\times$, which means that the free rank of the unit group drops when we identify different weak inverses. This happens, for instance, in the cases $\beta(x)=4$ and $\beta(x)=5$ in section \ref{subsection: universal idylls for rank 2 and the four element set}.
\end{rem}

\begin{cor}\label{cor: universal idyll of a closed point}
 A point $x$ of $\Mat(r,E)$ is a closed point if and only if $k(x)=\Funpm$. Thus every closed point is the support of a matroid.
\end{cor}

\begin{proof}
 A point $x$ is closed if and only if $\beta(x)=1$. By Lemma \ref{lemma: number of bases of a point of the matroid space}, this is equivalent to $k(x)^\times=\{1,\epsilon\}$, or $k(x)=\Funpm$, as claimed. Proposition \ref{prop: embedding of matroids into the matroid space and its image} implies that $x$ is the support of a matroid.
\end{proof}

\begin{ex}[Support of the uniform matroid]
 The \emph{uniform matroid of rank $r$ on $E$} is the matroid represented by the Grassmann-Pl\"ucker function $\Delta:\binom Er\to\K$ with $\Delta(I)=1$ for all $r$-subsets $I$ of $E$. The support of the uniform matroid is the generic point of $\Mat(r,E)$.
\end{ex}

To summarize, all closed points and the generic point of $\Mat(r,E)$ are supports of matroids. But if $2\leq r \leq \# E-2$, then there are points of $\Mat(r,E)$ that are not the support of matroids. In other words, the map $\Phi$ from Proposition \ref{prop: embedding of matroids into the matroid space and its image} is not surjective in general. The following section will exhibit points of the matroid space that are not the support of a matroid for $r=2$ and $\# E=4$. This can be easily generalized to any $r$ and $E$ with $2\leq r \leq \# E-2$.

\begin{rem}
 Let $F$ be an idyll and $M$ an $F$-matroid with characteristic morphism $\chi_M:\Spec F\to\Mat(r,E)$. Then the image point $x_M$ of $\chi_M$ is the support of the underlying matroid of $M$. Together with our previous observation about points of $Mat(r,E)$ that are not the support of matroids, this shows that these points are not the image for any morphism $\Spec F\to \Mat(r,E)$ for any idyll $F$. 
 
 But every point $x$ of $\Mat(r,E)$ occurs in the support for some matroid bundle. Namely, let $\cO_{X,x}$ be the stalk of $\Mat(r,E)$ at $x$. Then the canonical inclusion $\Spec\cO_{X,x}\to\Mat(r,E)$ defines an $\cO_{X,x}$-matroid that has support at $x$, together with all more general points of $\Mat(r,E)$.
\end{rem}


\subsection{Universal idylls for rank 2-matroids on the four element set}
\label{subsection: universal idylls for rank 2 and the four element set}

 In the following, we characterize the different universal idylls that can occur for $\Mat(2,E)$ where $E=\{1,2,3,4\}$. Note that $\Mat(2,E)$ is defined by a single Pl\"ucker relation, namely
 \[
  0 \ \leq \ x_{1,2}x_{3,4} \ + \ \epsilon \cdot x_{1,3}x_{2,4} \ + \ x_{1,4}x_{2,3}
 \]
 where we write $x_{i,j}$ for $x_{\{i,j\}}$. We systematically determine the residue field $k(x)$ and $k(x)^\pm$ for every point $x$ of $\Mat(r,E)$, in increasing order of the number $\beta(x)=\rk(k(x)^\times)+1$ where $\rk(k(x)^\times)$ is the free rank of the abelian group $k(x)^\times$. Note that if $M$ is a matroid with support $x_M=\fp_\cI$, then $\beta(x)$ equals the number of bases of $M$; cf.\ Lemma \ref{lemma: number of bases of a point of the matroid space}.
  The complement $\cB=\cI^c$ of $\cI$ in $\binom Er$ is the set of bases of the matroid $M$. 
 
 \subsubsection*{Case $\beta(x)=1$} 
 By Corollary \ref{cor: universal idyll of a closed point}, we have $k(x)^\pm=k(x)=\Funpm$. In particular, $x$ is the support of a matroid.
 
 \subsubsection*{Case $\beta(x)=2$} 
 By Lemma \ref{lemma: number of bases of a point of the matroid space}, we have $x=\fp_\cI$ for a $2$-subset $\cI=\{I,J\}$ of $\binom Er$. There are two cases. If $I$ and $J$ intersect nontrivially, then
 \[
  k(x)^\pm \ = \ k(x) \ = \ \Funpm[x_{I}^{\pm1},x_{J}^{\pm1}]_0
 \]
 and $x$ is the support of a matroid. If $I\cap J=\emptyset$, then the Pl\"ucker relation of $\Mat(r,E)$ yields $0\leq x_{I}x_{J}$ after substituting all other terms by $0$. Multiplication with $x_{I}^{-1}x_{J}^{-1}$ yields $0\leq 1$ and we obtain 
  \[
  k(x) \ = \ \big(\bpgenquot{\Funpm[x_{I}^{\pm1},x_{J}^{\pm1}]}{0\leq 1}\big)_0.
 \]
 Since $0\leq1=1+0$, we see that $0$ is a weak inverse of $1$. We conclude that $k(x)^\pm=\{0\}$ and that $x$ is not the support of a matroid. 

 \subsubsection*{Case $\beta(x)=3$} 
 By Lemma \ref{lemma: number of bases of a point of the matroid space}, we have $x=\fp_\cI$ for a $3$-subset $\cI=\{I,J,K\}$ of $\binom Er$. As in the rank $1$-case, we are are confronted with two cases. If each two of $I$, $J$ and $K$ have nonempty intersection, the Pl\"ucker relation is trivial. Thus we have 
  \[
  k(x)^\pm \ = \ k(x) \ = \ \Funpm[x_{I}^{\pm1},x_{J}^{\pm1},x_K^{\pm1}]_0
 \]
 and $x$ is the support of a matroid. If not---for instance, $I\cap J=\emptyset$--- then 
 \[
  k(x) \ = \ \big(\bpgenquot{\Funpm[x_{I}^{\pm1},x_{J}^{\pm1},x_{K}^{\pm1}]}{0\leq 1}\big)_0,
 \]
 as in the rank $1$-case. Thus $k(x)^\pm=\{0\}$ and $x$ is not the support of a matroid.

 \subsubsection*{Case $\beta(x)=4$} 
 By Lemma \ref{lemma: number of bases of a point of the matroid space}, we have $x=\fp_\cI$ for a $4$-subset $\cI=\{I,J,K,L\}$ of $\binom Er$. There is at least one pair of subsets with empty intersection, say $I\cap J=\emptyset$. We differentiate two cases. If $K$ and $L$ intersect nontrivially, then we have 
 \[
  k(x) \ = \ \big(\bpgenquot{\Funpm[x_{I}^{\pm1},x_{J}^{\pm1},x_{K}^{\pm1},x_{L}^{\pm1}]}{0\leq 1}\big)_0
 \]
 and $k(x)^\pm=\{0\}$, i.e.\ $x$ is not the support of a matroid. If $K\cap L=\emptyset$, then
 \[
  k(x) \ = \ \big(\bpgenquot{\Funpm[x_{I}^{\pm1},x_{J}^{\pm1},x_{K}^{\pm1},x_{L}^{\pm1}]}{0\leq x_Ix_J+\epsilon^ix_Kx_L}\big)_0
 \]
 where $i=0$ or $1$, depending on $I$, $J$, $K$ and $L$. In this case, $k(x)$ has multiple weak inverses, and 
 \[
  k(x)^\pm \ = \ \bpgenquot{k(x)}{\epsilon\= \epsilon^ix_Kx_Lx_I^{-1}x_J^{-1}} \ \simeq \ \Funpm[x_{I}^{\pm1},x_{J}^{\pm1},x_{K}^{\pm1}]_0. 
 \]
 Since $k(x)^\pm\neq\{0\}$, the point $x$ is the support of a matroid.
 
 \subsubsection*{Case $\beta(x)=5$} 
 By Lemma \ref{lemma: number of bases of a point of the matroid space}, we have $x=\fp_\cI$ for a $5$-subset $\cI=\{I,J,K,L,N\}$ of $\binom Er$, which contains two pairs of subsets with empty intersection, say $I\cap J=K\cap L=\emptyset$. Then we have
 \[
  k(x) \ = \ \big(\bpgenquot{\Funpm[x_{I}^{\pm1},x_{J}^{\pm1},x_{K}^{\pm1},x_{L}^{\pm1},x_{N}^{\pm1}]}{0\leq x_Ix_J+\epsilon^ix_Kx_L}\big)_0
 \]
 where $i=0$ or $1$, depending on $I$, $J$, $K$ and $L$. As in the rank $4$-case, we have
 \[
  k(x)^\pm \ = \ \bpgenquot{k(x)}{\epsilon\= \epsilon^ix_Kx_Lx_I^{-1}x_J^{-1}} \ \simeq \ \Funpm[x_{I}^{\pm1},x_{J}^{\pm1},x_{K}^{\pm1},x_{N}^{\pm1}]_0. 
 \]
 Thus $x$ is the support of a matroid.

 \subsubsection*{Case $\beta(x)=6$} 
 By Lemma \ref{lemma: number of bases of a point of the matroid space}, we have $x=\fp_\cI$ for $\cI=\binom Er$ and
 \[\textstyle
  k(x)^\pm \ = \ k(x) \ = \ \big(\bpgenquot{\Funpm[x_{I}^{\pm1}\,|\,I\in\binom Er]}{0\leq x_{1,2}x_{3,4}+\epsilon x_{1,3}x_{2,4} + x_{1,4}x_{2,3}}\big)_0.
 \]
 The point $x$ is the support of the uniform matroid.


\subsection{Realization spaces}
\label{subsection: realization spaces}

Let $k$ be a field. The realization space of a matroid $M$ is the subset of the Grassmannian over $k$ that consists of the subspaces whose associated matroid is $M$. These realization spaces have been used for proving that several moduli spaces, such as Hilbert schemes and moduli spaces of curves, can become arbitrarily complicated, cf.\ \cite{Vakil06}. In this section, we show that realization spaces are the same as morphism sets from universal idylls.

Let $\Delta:\binom Er\to \K$ be a Grassmann-Pl\"ucker function, $M=[\Delta]$ the corresponding matroid and $\chi_M:\K\to\Mat(r,E)$ its characteristic morphism. Let $F$ be an idyll. The terminal map $t_F:F\to\K$ induces a map
\[
 \Phi_{r,E,F}: \ \Mat(r,E)(F) \ \longrightarrow \ \Mat(r,E)(\K)
\]
that sends a morphism $\chi:\Spec F\to\Mat(r,E)$ to $\chi\circ t_F^\ast:\Spec\K\to\Mat(r,E)$. In other words, $\Phi_{r,E,F}$ maps an $F$-matroid to its underlying matroid. 

\begin{df}
 The \emph{realization space of $M$ over $F$} is the fibre 
 \[
  \cX_M(F) \ = \ \Phi_{r,E,F}^{-1}(M) \ = \ \big\{ \, \chi:\Spec F\to\Mat(r,E) \, \big| \, \chi\circ t_F^\ast=\chi_M \, \big\}
 \]
 of $\Phi_{r,E,F}$ over $\chi_M$. 
\end{df}

Note that the realization space of $M$ is functorial in $F$: a morphism $f:F\to F'$ of idylls induces a map
\[
 \begin{array}{ccc}
   \cX_M(F) & \longrightarrow & \cX_M(F'). \\
   \chi    & \longmapsto     & \chi\circ f^\ast
 \end{array}
\]

\begin{ex}
 In the case of a field $k$, $\cX_M(k)$ is the subset of $\Gr(r,E)(k)=\Mat(r,E)(k)$ that consists of all subspaces $V$ of $k^E$ whose associated matroid is $M$, i.e.\ $t_F\circ\Delta_V(I)=\Delta(I)$ for all $I\in\binom Er$, where $\Delta_V(I)$ are the Pl\"ucker coordinates of $V$. Note that $\cX_M(k)$ comes with the structure of a locally closed subvariety of $\Gr(r,E)(k)$ since it is defined by the equations $x_I=0$ whenever $\Delta(I)=0$ and $x_I\neq 0$ whenever $\Delta(I)\neq 0$.
\end{ex}

It turns out that $\cX_M$ is represented by $k_M$ as a functor from idylls to sets. In other words, $\Spec k_M$ is the fine moduli space of realization spaces for $M$. In down-to-earth terms, this means the following:

\begin{thm}\label{thm: realization space as morphism set from the universal idyll}
 Let $M$ be a matroid and $\iota_M:\Spec k_M\to\Mat(r,E)$ the inclusion of the universal idyll $k_M$ of $M$ into the matroid space. Let $F$ be an idyll. The map
 \[
  \begin{array}{cccc}
   \iota_{M,\ast}: & \Hom(k_M,F)  & \longrightarrow & \cX_M(F) \\
         & f            & \longmapsto     & \iota_M\circ f^\ast
  \end{array}
 \]
 is a bijection that is functorial in $F$.
\end{thm}

\begin{proof}
 We will show that every morphism $\chi:\Spec F\to \Mat(r,E)$ in $\cX_M(F)$ factors uniquely through $\iota_M$. As a first step, we observe that the equality $\chi_M=\chi\circ t_F^\ast$ implies that $\im\chi=\im\chi_M=\{x_M\}$ where $\chi_M$ is the characteristic morphism of $M$, $x_M$ its image point and $t_F:F\to \K$ the terminal map.
 
 We conclude that $\chi$ factors uniquely into a morphism $\overline\chi:\Spec F\to\Spec k(x_M)$ followed by the inclusion $\iota:\Spec k(x_M)\to \Mat(r,E)$, where $k(x_M)$ is the residue field of $x_M$. Taking global sections yields a morphism $f=\Gamma\overline\chi:k(x_M)\to F$ with $\overline\chi=f^\ast$. 
 
 Since $F$ is an idyll, $f$ factors into the canonical map $k(x_M)\to k(x_M)^\pm=k_M$ followed by a uniquely determined morphism $f^\pm:k_M\to F$ of idylls. We conclude that $\varphi=(f^\pm)^\ast:\Spec F\to \Spec k_M$ is the unique morphism such that $\chi=\iota_M\circ\varphi$. 
 
 This shows that $\iota_{M,\ast}$ is a bijection for every $F$. The functoriality of $F$ is clear from the definitions. This completes the proof of the theorem.
\end{proof}

\begin{rem}
 Roughly speaking, the universality theorem of Mn{\"e}v says that the realization spaces for oriented matroids $M$ can become arbitrarily complex for varying $M$, cf.\ \cite{Mnev88}. Lafforgue adapts in \cite{Lafforgue03} Mn{\"e}v's proof to realization spaces for matroids. Lee and Vakil explain in \cite{Lee-Vakil13} that arbitrarily complex means, in particular, that every type of singularity can occur in a realization space of a matroid.

 The universality theorem, paired with Theorem \ref{thm: realization space as morphism set from the universal idyll}, implies (loosely speaking) that universal idylls can be ``arbitrarily complex''.  It would be interesting to have a precise formulation of this.
 \end{rem}


\subsection{The weak matroid space}
\label{subsection: the weak matroid space}


There is a variant of the matroid space for weak matroids, which leads to the notion of the universal pasture of a matroid. Although the weak matroid space is not a moduli space for weak matroids, it turns out that the universal idyll is a very useful object for matroid theory because of its connections to the Tutte group and rescaling classes.

Since an idyll $F$ can be considered as a tract, we gain the notion of a weak $F$-matroid, as explained in section \ref{subsubsection: weak Grassmann-Plucker functions}. This means that in contrast to a strong $F$-matroid, which is defined by all Pl\"ucker relations, a weak $F$-matroid is represented by a function $\Delta:\binom Er \to F$ whose support is the set of bases of a matroid and that is only required to satisfy the $3$-term Pl\"ucker relations 
\[
 0 \ \leq \ \Delta(I_{1,2}) \, \Delta(I_{3,4}) \ + \ \epsilon \, \Delta(I_{1,3}) \, \Delta(I_{2,4}) \ + \ \Delta(I_{1,4}) \, \Delta(I_{2,3})
\]
for every $(r-2)$-subset $I$ of $E$ and all $i_1<i_2<i_3<i_4$ with $i_1,i_2,i_3,i_4\notin I$, where $I_{k,l}=I\cup\{i_k,i_l\}$.

\begin{rem}\label{rem: perfect tracts}
 Note that a strong Grassmann-Pl\"ucker function is evidently a weak Grassmann-Pl\"ucker function. Thus every strong $F$-matroid is a weak $F$-matroid. For many tracts of interest, the reverse implication is also true. For instance, this holds for the class of \emph{perfect tracts}, which are tracts for which covectors are orthogonal to vectors, cf.\ \cite[section 3]{Baker-Bowler19} for details. Examples of perfect tracts are $\Funpm$, $\K$, $\S$, $\T$, partial fields and doubly-distributive hyperfields. (A hyperfield $K$ is \emph{doubly-distributive} if $(a\hyperplus b)(c\hyperplus d) \ = \ ac \hyperplus bc \hyperplus ad \hyperplus bd$ for all $a,b,c,d\in K$.)
 Not every tract, or even hyperfield, is perfect. For instance, the phase hyperfield, which is the hyperfield quotient of $\C$ by $\R_{>0}$, admits weak matroids that are not strong. See Example 2.36 in \cite{Baker-Bowler19} for details. 
 
 In the following, we shall call an idyll $F$ \emph{perfect} if the associated tract $F^\tract$ is perfect. Since the matroid theories of $F$ and $F^\tract$ coincide by Proposition \ref{prop: relation between matroids over Funpm-algebras and tracts}, which is also true for weak matroids, we conclude that every weak matroid over a perfect idyll is a strong matroid. This justifies our abuse of terminology.
\end{rem}

\begin{df}
 The \emph{weak matroid space of rank $r$ on $E$} is the ordered blue scheme
\[ \textstyle
 \Mat^w(r,E) \quad = \quad \Proj\Big( \, \bpquot{\Funpm\big[ \, x_I \, \big| \, I\in\binom Er \, \big]}{\cPl^w(r,E)} \, \Big),
\]
where $\cPl^w(r,E)$ is generated by the $3$-term Pl\"ucker relations
 \[
  0 \ \leq \ x_{I,1,2} \, x_{I,3,4} \ + \ \epsilon \, x_{I,1,3} \, x_{I,2,4} \ + \ x_{I,1,4} \, x_{I,2,3}
 \]
 for every $(r-2)$-subset $I$ of $E$ and all $i_1<i_2<i_3<i_4$ with $i_1,i_2,i_3,i_4\notin I$ where $x_{I,k,l}=x_{I\cup\{i_k,i_l\}}$.
\end{df}

By definition, the weak matroid space comes with a closed immersion into projective space
\[\textstyle
 \iota: \ \Mat^w(r,E) \ \longrightarrow \ \P^N_{\Funpm} \ = \ \Proj\Big( \, \Funpm\big[ \, x_I \, \big| \, I\in\binom Er \, \big] \, \Big)
\]
where $N=\#\binom Er-1$. Let $\cL^w_\univ=\iota^\ast(\cO(1))$ be the pullback of the tautological bundle $\cO(1)$ on $\P^N_\Funpm$ to $\Mat^w(r,E)$.

The identity map induces a morphism
\[\textstyle
 \bpquot{\Funpm\big[ \, x_I \, \big| \, I\in\binom Er \, \big]}{\cPl^w(r,E)} \ \longrightarrow \ \bpquot{\Funpm\big[ \, x_I \, \big| \, I\in\binom Er \, \big]}{\cPl(r,E)}
\]
of graded ordered blueprints, which in turn induces a morphism
\[
 \gamma^w: \ \Mat(r,E) \ \longrightarrow \Mat^w(r,E)
\]
of ordered blue schemes. Since the underlying monoids of the graded ordered blueprints above are equal, $\gamma^w$ is a homeomorphism between the respective underlying topological spaces.

In order to capture the analogue of the unversal idyll for the weak matroid space, we introduce the concept of a pasture. Please note that this definition is equivalent with the corresponding notion in the authors' follow-up paper \cite{Baker-Lorscheid20}.

\begin{df} \label{def:pasture}
 A \emph{pasture} is an idyll $B$ whose partial order is generated by relations of the form $0\leq a+b+c$ with $a,b,c\in B^\bullet$.
\end{df}

Note that since a pasture $B$ is an idyll, its partial order is in fact generated by $0\leq 1+(-1)$ and by terms of the form $0\leq a+b+c$ with $a,b,c\in B^\times$.

\begin{df}
 Let $M$ be a matroid with characteristic morphism $\chi_M$ and support $x_M$. The \emph{weak characteristic morphism} is the morphism $\chi_M^w=\gamma^w\circ\chi_M:\Spec\K\to\Mat^w(r,E)$. The \emph{weak support of $M$} is the image $x_M^w=\gamma^w(x_M)$ of $x_M$ in $\Mat^w(r,E)$. The \emph{universal pasture of $M$} is $k_M^w=k(x_M^w)^\pm$ where $k(x^w)$ is the residue field of $x_M^w$. 
\end{df}

More explicitly, we have
\[\textstyle
 k_M^w \ = \ \big(\bpquot{\Funpm\big[ \, x_I^\pm \, \big| \, I\in\cI\, \big]}{\cPl^w(r,E)}\big)_0^\pm
\]
where $\cPl^w(r,E)$ is generated by the $3$-term Pl\"ucker relations
\[
 0 \ \leq \ \Delta(I_{1,2}) \, \Delta(I_{3,4}) \ + \ \epsilon \, \Delta(I_{1,3}) \, \Delta(I_{2,4}) \ + \ \Delta(I_{1,4}) \, \Delta(I_{2,3})
\]
for every $(r-2)$-subset $I$ of $E$ and all $i_1<i_2<i_3<i_4$ with $i_1,i_2,i_3,i_4\notin I$. Note that $k_M^f$ is a pasture for the same reasons that $k_M$ is an idyll; cf.\ Remark \ref{rem: the universal idyll is an idyll}.

\begin{df}
 Let $M$ be a matroid with weak characteristic morphism $\chi_M^w$ and $F$ an idyll with terminal map $t_F:F\to \K$. The \emph{weak realization space of $M$ over $F$} is the set
 \[
  \cX_M^w(F) \ = \ \big\{ \, \chi:\Spec F\to\Mat^w(r,E) \, \big| \, \chi\circ t_F^\ast=\chi_M^w \, \big\}
 \]
 of all weak $F$-matroids that represent $M$.
\end{df}

Recall from Remark \ref{rem: perfect tracts} the definition of a perfect idyll.

\begin{lemma}\label{lemma: weak realization space as morphism set from the universal pasture}
 Let $M$ be a matroid and $F$ an idyll. The map
 \[
  \begin{array}{cccc}
   \gamma^w_\ast: & \cX_M(F) & \longrightarrow & \cX_M^w(F) \\
                  & \chi     & \longmapsto     & \gamma^w\circ\chi
  \end{array}
 \]
 is injective. If $F$ is a perfect idyll then $\gamma^w_\ast$ is bijective.
\end{lemma}

\begin{proof}
 The map $\gamma^w_\ast$ identifies an $F$-matroid with the corresponding weak $F$-matroid and is obviously injective. If $F$ is a perfect idyll, then every weak $F$-matroid is a strong $F$-matroid, cf.\ Remark \ref{rem: perfect tracts}. Thus $\gamma^w_\ast$ is a bijection in this case.
\end{proof}

\begin{prop} \label{prop: weak realization space as morphism set from the universal pasture}
 Let $M$ be a matroid and $\iota_M^w:\Spec k_M^w\to\Mat^w(r,E)$ the inclusion of the universal pasture $k_M^w$ of $M$ into the weak matroid space. Let $F$ be an idyll. Then the map
 \[
  \begin{array}{cccc}
   \iota_{M,\ast}^w: & \Hom(k_M^w,F)  & \longrightarrow & \cX_M^w(F) \\
                   & f              & \longmapsto     & \iota_M^w\circ f^\ast
  \end{array}
 \]
 is a bijection.
\end{prop}

\begin{proof}
 The proof is analogous to that of the corresponding result for strong $F$-matroids, see Theorem \ref{thm: realization space as morphism set from the universal idyll}. For completeness, we outline the idea of the proof.
 
 We need to show that every morphism $\chi:\Spec F\to \Mat^w(r,E)$ in $\cX_M^w(F)$ factors uniquely through $\iota_M^w$. Since $M$ is the underlying matroid of $\chi$, the image point of $\chi$ is the weak support $x_M^w$ of $M$. Thus we obtain a unique morphism $k(x_M^w)\to F$, which extends uniquely to a morphism $k_M^w\to F$. This association provides an inverse bijection to $\iota_{M,\ast}^w$.
\end{proof}

\begin{rem}\label{rem: no moduli space for weak matroids}
 It seems unlikely that the functor $\cMat^w(r,E)$ can be represented by an ordered blue scheme. The obstacle is that in the definition of a weak $F$-matroid $M=[\Delta]$ with representing Grassmann-Pl\"ucker function $\Delta:\binom Er\to F$, it is required that the support of $\Delta$ is the basis set of a matroid, i.e.\ $t_F \circ\Delta$ satisfies all Pl\"ucker relations where $t_F:F\to\K$ is the terminal map. Since the locus of points of $\Mat^w(r,E)$ supporting matroids is not locally closed, but merely constructible in general, this locus does not inherit a scheme structure from $\Mat^w(r,E)$ in an obvious way.
 
 For instance, it is a well-known fact that the $3$-term Pl\"ucker relations do not suffice, in general, to define classical Grassmann varieties. In fact, the same holds true for any idyll, as the following example shows.
 
\end{rem}

\begin{ex}\label{ex: function that satisfies the 3-term relations but not all}
 Let $F$ be an idyll. In this example, we exhibit a function $\Delta:\binom Er\to F$ that satisfies all $3$-term Pl\"ucker relations, but is not a weak Grassmann-Pl\"ucker function since it fails to satisfy all Pl\"ucker relations over $\K$. 

 Let $E=\{1,2,3,4,5,6\}$ and $J$ and $J^c$ a pair of disjoint $3$-subsets of $E$. Let $\Delta:\binom E3\to F$ be the function with $\Delta(J)=\Delta(J^c)=1$ and $\Delta(I)=0$ for all other $3$-subsets $I$ of $E$. Consider the $3$-term Pl\"ucker relation
 \[
  0 \ \leq \ \Delta(I_{1,2}) \, \Delta(I_{3,4}) \ + \ \epsilon \, \Delta(I_{1,3}) \, \Delta(I_{2,4}) \ + \ \Delta(I_{1,4}) \, \Delta(I_{2,3})
 \]
 for $I=\{i_0\}$ and $i_1<i_2<i_3<i_4$ with $i_1,i_2,i_3,i_4\notin \{i_0\}$, where $I_{k,l}=\{i_0,i_k,i_l\}$. In order for some term in this equation to be nonzero, we have to have that $\Delta(I_{k,l})=\Delta(I_{k',l'})=1$ where $\{k,l,k',l'\}=\{1,2,3,4\}$. Since $i_0$ is contained in both $I_{k,l}$ and $I_{k',l'}$, this means that the elements $i_0,\dotsc,i_4$ have to all be contained in $J$ or else all be contained in $J^c$, which is impossible since $\# J=\#J^c=3$. This shows that $\Delta$ satisfies all $3$-term Pl\"ucker relations.

 To show that $\Delta$ is not a weak Grassmann-Pl\"ucker function, let $\overline\Delta=t_F\circ\Delta:\binom E3\to \K$ where $t_F:F\to \K$ is the terminal map. Let $j\in J$ and define $I=J-\{j\}$ and $I'=J^c\cup\{j\}$. Then $I'=\{j_1,j_2,j_3,j_4\}$ for $j_1<j_2<j_3<j_4$ and $j=j_{l}$ for some $l\in\{1,2,3,4\}$. The Pl\"ucker relation for $I$ and $I'$ is
 \[
  0 \ \leq \ \sum_{k=1}^4 \epsilon^k \cdot \overline\Delta(I\cup\{j_k\}) \cdot \overline\Delta(I'-\{j_k\})
 \] 
 where $\epsilon=1$. The sum on the right hand side has precisely one nonzero term, namely 
 \[
  \overline\Delta(I\cup\{j_l\}) \cdot \overline\Delta(I'-\{j_l\}) \ = \ \overline\Delta(J) \cdot \overline\Delta(J^c) \ = \ 1. 
 \]
 But the relation $0\leq 1$ does not hold in $\K$, which shows that $\overline\Delta$ does not satisfy all Pl\"ucker relations. Therefore $\Delta$ is not a weak Grassmann-Pl\"ucker function.

 Let $\cI$ be the complement of $\{J,J^c\}$ in $\binom Er$ and $x^w=\fp_\cI$ the corresponding point of the weak matroid space $\Mat^w(3,E)$. Since all $3$-term Pl\"ucker relations for $\Delta$ are trivial, the residue field of $x$ is $k(x^w)=\Funpm[x_J^{\pm1},x_{J^c}^{\pm1}]_0$. Note that $k(x^w)$ is an idyll, i.e.\ $k(x^w)^\pm=k(x^w)$ is nonzero. This shows that the weak supports of matroids cannot be characterized by the nonvanishing of $k(x^w)^\pm$, in contrast to the corresponding result for (strong) supports of matroids, cf.\ Proposition \ref{prop: embedding of matroids into the matroid space and its image}.
\end{ex}


\subsection{The Tutte group}
\label{subsection: the Tutte group}

The Tutte group is introduced by Dress and Wenzel in \cite{Dress-Wenzel89} and used as a tool to study the representability of matroids and to provide cryptomorphisms for matroids over fuzzy rings, cf.\ \cite{Dress-Wenzel91}. In this section, we show that the Tutte group is precisely the unit group of the universal pasture.

For the following characterization of the Tutte group, see Definition 1.2 and Theorem 1.1 in \cite{Dress-Wenzel89}.

\begin{df}\label{def: Tutte group}
 Let $M$ be a matroid of rank $r$ on $E$ and $\cB$ the set of bases of $M$. Consider the quotient $G_M$ of the free abelian group generated by symbols $\epsilon$ and $X_{(i_1,\dotsc,i_r)}$ for every $(i_1,\dotsc,i_r)\in E^r$ such that $\{i_1,\dotsc,i_r\}\in\cB$ modulo the subgroup generated by 
 \[
  \epsilon^2, \qquad \epsilon^{\sign(\sigma)} X_{(\sigma(i_1),\dotsc,\sigma(i_r))} X_{(i_1,\dotsc,i_r)}^{-1} 
 \]
 for every permutation $\sigma$ of $\{1,\dotsc,r\}$ and
 \[
  X_{(i_1,\dotsc,i_{r-2},k_1,l_1)} X_{(i_1,\dotsc,i_{r-2},k_2,l_2)} X_{(i_1,\dotsc,i_{r-2},k_1,l_2)}^{-1} X_{(i_1,\dotsc,i_{r-2},k_2,l_1)}^{-1}
 \]
 whenever $\{i_1,\dotsc,i_{r-2},k_1,k_2\}\notin\cB$. The \emph{Tutte group $\T_M$ of $M$} is defined as the subgroup of $G_M$ that is generated by $\epsilon$ and elements of the form $X_{(i_1,\dotsc,i_r)}X_{(j_1,\dotsc,j_r)}^{-1}$ with $\{i_1,\dotsc,i_r\},\{j_1,\dotsc,j_r\}\in\cB$.
\end{df}

Let $x_M$ be the support of $M$ and $x_M^w$ its weak support. The natural map $k(x_M^w)\to k(x_M)$ between the respective residue fields is a bijection because the difference between the two is the validity of the higher Pl\"ucker relations in $k(x_M)$, but these relations do not identify any elements. Thus $k(x_M^w)^\times=k(x_M)^\times$. As explained in the proof of Lemma \ref{lemma: number of bases of a point of the matroid space}, we have $x_M=\fp_{\cB^c}$ where $\cB^c$ is the complement of $\cB$ in $\binom Er$, and 
\[
 k(x_M^w)^\times \ = \ k(x_M)^\times \ = \ \Big\{ \, \epsilon^i \, \cdot \, \prod_{I\in\cB} \ x_I^{e_I} \, \Big| \, i\in\{0,1\},e_I\in\Z\text{ and }\sum_{I\in\cB} e_I=0 \, \Big\}.
\]

\begin{thm}\label{thm: the unit group of the universal pasture is the Tutte group}
 Let $M$ be a matroid with universal pasture $k_M^w$ and Tutte group $\T_M$. For a basis $I=\{i_1,\dotsc,i_r\}$ of $M$ with $i_1<\dotsb<i_r$, we define $X_I=X_{(i_1,\dotsc,i_r)}$. Then the association $\epsilon^i\prod x_I^{e_I}\mapsto \epsilon^i\prod X_I^{e_I}$ defines an isomorphism $(k_M^w)^\times\to \T_M$ of groups.
\end{thm}

\begin{proof}
 Let $x_M^w$ be the weak support and $\cB$ the set of bases of $M$. To enable ourselves to work with degree-$0$ elements, we will work with two graded abelian groups $G_M'$ and $H_M$, which contain $\T_M$ and $(k_M^w)^\times$, respectively, as subquotients.
 
 Namely, we define $G_M'$ as the abelian group generated by the symbols $\epsilon$ and $X_{(i_1,\dotsc,i_r)}$ for every $(i_1,\dotsc,i_r)\in E^r$ such that $\{i_1,\dotsc,i_r\}\in\cB$ modulo the subgroup generated by 
 \[
  \epsilon^2 \qquad \text{and} \qquad \epsilon^{\sign(\sigma)} X_{(\sigma(i_1),\dotsc,\sigma(i_r))} X_{(i_1,\dotsc,i_r)}^{-1} 
 \]
 for every permutation $\sigma$ of $\{1,\dotsc,r\}$, where $\epsilon$ is of degree $0$ and $X_{(i_1,\dotsc,i_r)}$ is of degree $1$. 

 The second group is
 \[
  H_M \ = \ \{1,\epsilon\} \times \Big\{ \, \prod_{I\in\cB} x_I^{e_I} \, \Big| \,e_I\in\Z\, \Big\}, 
 \]
 where $\epsilon$ is of degree $0$, $\epsilon^2=1$, and $x_I$ is of degree $1$.
 
 Since $\epsilon^2=1$ in both $H_M$ and $G_M'$ and since $X_{(\sigma(i_1),\dotsc,\sigma(i_r))} =\epsilon^{\sign(\sigma)} X_{(i_1,\dotsc,i_r)}$ for every permutation $\sigma$ of $\{1,\dotsc,r\}$, the association $\epsilon^i\prod x_I^{e_I}\mapsto \epsilon^i\prod X_I^{e_I}$ defines a degree-preserving group isomorphism $f:H_M\to G_M'$.
 
 Let $G_M$ be the group from Definition \ref{def: Tutte group} and $g:G_M'\to G_M$ the quotient map. Then the kernel of $g$ is generated by the elements
 \[
  X_{(i_1,\dotsc,i_{r-2},k_1,l_1)} X_{(i_1,\dotsc,i_{r-2},k_2,l_2)} X_{(i_1,\dotsc,i_{r-2},k_1,l_2)}^{-1} X_{(i_1,\dotsc,i_{r-2},k_2,l_1)}^{-1}
 \]
 for which $\{i_1,\dotsc,i_{r-2},k_1,k_2\}\notin\cB$. Consequently, $\T_M$ is the degree-$0$ subgroup of the quotient group $G_M'/\ker g$.
 
 Let $h:k(x_M^w)^\times\to (k_M^w)^\times$ be the quotient map. Then $(k_M^w)^\times$ is the degree-$0$ subgroup of the quotient group $H_M/\ker h$. Thus the theorem follows if we can show that the isomorphism $f$ identifies $\ker h$ with $\ker g$.
 
 As our next step, we exhibit a set of generators for $\ker h$. The kernel of $h$ consists of all weak inverses of $1$ in $k(x_M^w)$. Such elements must come from the $3$-term Pl\"ucker relation
 \[
  0 \leq x_{I,1,2} \, x_{I,3,4} \ + \ \epsilon \, x_{I,1,3} \, x_{I,2,4} \ + \ x_{I,1,4} \, x_{I,2,3}
 \]
 for an $r-2$-subset $I$ and $i_1<i_2<i_3<i_4$ with $i_1,i_2,i_3,i_4\notin I$, where $x_{I,p,q}=x_{I\cup\{i_p,i_q\}}$.
 
 If none or all of the products $x_{I,p,q}x_{I,p',q'}$ are zero, then the $3$-term Pl\"ucker relation does not define new weak inverses in $k(x_M^w)$. The case where precisely one of the products $x_{I,p,q}x_{I,p',q'}$ is nonzero cannot happen since then $k_M=k_M^w=\{0\}$, which contradicts Proposition \ref{prop: embedding of matroids into the matroid space and its image}.
 
 Thus we are left with the case that precisely two of products $x_{I,p,q}x_{I,p',q'}$ are nonzero. There are three cases:
 \begin{itemize}
 \item If the first two products in the Pl\"ucker relation are nonzero, then the relation $0\leq x_{I,1,2} x_{I,3,4}+\epsilon x_{I,1,3} x_{I,2,4}$ in $k(x_M^w)$ implies that $\epsilon x_{I,1,3}x_{I,2,4}$ is a weak inverse of $x_{I,1,2} x_{I,3,4}$. Thus $x_{I,1,2} x_{I,3,4}x_{I,1,3}^{-1}x_{I,2,4}^{-1}$ is in $\ker h$.
 \item If the first and the third product are nonzero, then $x_{I,1,4}x_{I,2,3}$ is a weak inverse of $x_{I,1,2} x_{I,3,4}$ and $\epsilon x_{I,1,2} x_{I,3,4}x_{I,1,4}^{-1}x_{I,2,3}^{-1}$ is in $\ker h$. 
 \item If the last two products are nonzero, then $x_{I,1,3} x_{I,2,4}x_{I,1,4}^{-1}x_{I,2,3}^{-1}$ is in $\ker h$. 
 \end{itemize}
 
We have thus exhibited a complete set of generators for $\ker h$ in all cases. In the following, we will show that $f$ maps this set to the set of generators for $\ker g$ that we used in the definition of $G_M=G_M'/\ker g$.
 
 Let $j_1,\dotsc,j_{r-2},k_1,k_2,l_1,l_2$ be pairwise different elements of $E$, and define $I=\{j_1,\dotsc,j_{r-2}\}$ and $\{i_1,i_2,i_3,i_4\}=\{k_1,k_2,l_1,l_2\}$ with $i_1<i_2<i_3<i_4$. Then $\{j_1,\dotsc,j_{r-2},k_p,l_q\}\in\cB$ for all $p,q\in\{1,2\}$ and $\{j_1,\dotsc,j_{r-2},k_1,k_2\}\notin\cB$ is equivalent to the fact that $x_{I,k_1,l_1}x_{I,k_2,l_2}$ and $x_{I,k_1,l_2}x_{I,k_2,l_1}$ are nonzero and that $x_{I,k_1,k_2}$ is zero in $k(x_M^w)$. 
 
 To compare both types of generators, we need to relate the terms $X_{(j_1,\dotsc,j_{r-2},k_p,l_q)}$ and $X_{I\cup\{k_p,l_q\}}$ for $p,q\in\{1,2\}$, which differ by the power of $\epsilon$ that arises from permuting the elements $I\cup\{k_p,l_q\}$. For $k,l\in E$, define $\mu(k,l)=0$ if $k<l$ and $\mu(k,l)=1$ if $l<k$. It is easily verified that there is an $N$ such that for all $p,q\in\{1,2\}$ and $p'=3-p$, $q'=3-q$, we have
 \[
  X_{(j_1,\dotsc,j_{r-2},k_p,l_q)} \, X_{(j_1,\dotsc,j_{r-2},k_{p'},l_{q'})} \ = \ \epsilon^{N+\mu(k_p,l_p)+\mu(k_{p'},l_{q'})} \, X_{I\cup\{k_p,l_q\}} \, X_{I\cup\{k_{p'},l_{q'}\}}.
 \]
Indeed, we can define $N$ as the number of transpositions needed to bring all elements into increasing order, up to exchanging $k_p$ with $l_q$ and $k_{p'}$ with $l_{q'}$ if necessary. For instance, if $j_1<\dotsb<j_{r-2}$, then we have
 \[
  N \ = \ \sum_{i\in\{k_1,k_2,l_1,l_2\}} \#\{j\in I|i<j\}.
 \]
From these considerations, we obtain the equality
 \[
  \frac{X_{(j_1,\dotsc,j_{r-2},k_p,l_q)} \, X_{(j_1,\dotsc,j_{r-2},k_{p'},l_{q'})}} {X_{(j_1,\dotsc,j_{r-2},k_p,l_{q'})} \, X_{(j_1,\dotsc,j_{r-2},k_{p'},l_{q})}} \ = \ \epsilon^{\mu(k_1,k_2;l_1,l_2)} \frac{X_{I\cup\{k_p,l_q\}} \, X_{I\cup\{k_{p'},l_{q'}\}}} {X_{I\cup\{k_p,l_{q'}\}} \, X_{I\cup\{k_{p'},l_{q}\}}}
 \]
 where $\mu(k_1,k_2;l_1,l_2)=\mu(k_1,l_2)+\mu(k_1,l_2)+\mu(k_2,l_1)+\mu(k_2,l_2)$. Note that the term $\epsilon^N$ disappears since it appears in both nominator and denominator.
 
 We are led to an inspection of the different possible orderings of $k_1$, $k_2$, $l_1$ and $l_2$, i.e.\ the different identifications $\{i_1,i_2,i_3,i_4\}=\{k_1,k_2,l_1,l_2\}$ where $i_1<i_2<i_3<i_4$. Note that permuting $k_1$ and $k_2$ and permuting $l_1$ and $l_2$ changes neither $\mu(k_1,k_2;l_1,l_2)$ nor the nonzero terms of the $3$-term Pl\"ucker relations. Since $\epsilon^{\mu(k_1,k_2;l_1,l_2)}$ depends only on the parity of $\mu(k_1,k_2,l_1,l_2)$ and both $x_{I,k_1,k_2}$ and $x_{I,l_1,l_2}$ appear in the same product of the $3$-term Pl\"ucker relation, a simultaneous exchange of $k_1$ and $k_2$ with $l_1$ and $l_2$ will leave the validity of our arguments below unchanged. Up to these permutations, we are left with three cases.
 
 We begin with the case $(i_1,i_2,i_3,i_4)=(k_1,k_2,l_1,l_2)$, i.e.\ $k_1<k_2<l_1<l_2$. Then we have $\mu(k_1,k_2;l_1,l_2)=0$ and $x_{I,i_1,i_2}x_{I,i_3,i_4}$ is zero, while the last two terms of the corresponding $3$-Pl\"ucker relation are nonzero. We obtain
 \[
  \frac{X_{(j_1,\dotsc,j_{r-2},k_1,l_1)} \, X_{(j_1,\dotsc,j_{r-2},k_{2},l_{2})}} {X_{(j_1,\dotsc,j_{r-2},k_1,l_{2})} \, X_{(j_1,\dotsc,j_{r-2},k_{2},l_{1})}} \quad = \quad \frac{X_{I\cup\{i_1,i_3\}} \, X_{I\cup\{i_2,i_4\}}} {X_{I\cup\{i_1,i_4\}} \, X_{I\cup\{i_2,i_3\}}}.
 \]
 The inverse image of the right-hand side under $f$ is the element $x_{I,1,3} x_{I,2,4} x_{I,1,4}^{-1} x_{I,2,3}^{-1}$ of $H_M$. We have seen in our discussion of the relations of $\ker h$ that this is the generator in the case that the last two products of the $3$-term Pl\"ucker relations are nonzero.
 
 In the case $(i_1,i_2,i_3,i_4)=(k_1,l_1,k_2,l_2)$, we have $\mu(k_1,k_2;l_1,l_2)=1$ and the zero term of the $3$-term Pl\"ucker relation is $\epsilon x_{I,i_1,i_3}x_{I,i_2,i_4}$. Thus 
 \[
  \frac{X_{(j_1,\dotsc,j_{r-2},k_1,l_1)} \, X_{(j_1,\dotsc,j_{r-2},k_{2},l_{2})}} {X_{(j_1,\dotsc,j_{r-2},k_1,l_{2})} \, X_{(j_1,\dotsc,j_{r-2},k_{2},l_{1})}} \quad = \quad \epsilon \, \cdot \  \frac{X_{I\cup\{i_1,i_2\}} \, X_{I\cup\{i_3,i_4\}}} {X_{I\cup\{i_1,i_4\}} \, X_{I\cup\{i_2,i_3\}}},
 \]
 whose inverse image under $f$ is $\epsilon x_{I,1,2} x_{I,3,4} x_{I,1,4}^{-1} x_{I,3,4}^{-1}$, which coincides with the generator of $\ker h$ exhibited in the case that the first and last product of the $3$-term Pl\"ucker relation are nonzero.
 
 In the case $(i_1,i_2,i_3,i_4)=(k_1,l_1,l_2,k_2)$, we have $\mu(k_1,k_2;l_1,l_2)=2$ and the zero term of the $3$-term Pl\"ucker relation is $\epsilon x_{I,i_1,i_4}x_{I,i_3,i_4}$. Since $\epsilon^2=1$, we have
 \[
  \frac{X_{(j_1,\dotsc,j_{r-2},k_1,l_1)} \, X_{(j_1,\dotsc,j_{r-2},k_{2},l_{2})}} {X_{(j_1,\dotsc,j_{r-2},k_1,l_{2})} \, X_{(j_1,\dotsc,j_{r-2},k_{2},l_{1})}} \quad = \quad \frac{X_{I\cup\{i_1,i_2\}} \, X_{I\cup\{i_3,i_4\}}} {X_{I\cup\{i_1,i_3\}} \, X_{I\cup\{i_2,i_4\}}},
 \]
 whose inverse image under $f$ is $x_{I,1,2} x_{I,3,4} x_{I,1,3}^{-1} x_{I,2,4}^{-1}$, which coincides with the generator of $\ker h$ exhibited in the case that the first two products of the $3$-term Pl\"ucker relation are nonzero.

 This establishes the claimed bijection between the generating sets of $\ker h$ and $\ker g$ and concludes the proof of the theorem. 
\end{proof}


\section{Cross ratios and rescaling classes}
\label{section: cross ratios and rescaling classes}

In this section, we will define and study the properties of the foundation of a matroid, which is a subidyll of the universal pasture that is closely related to the inner Tutte group from \cite{Dress-Wenzel89} and the universal partial field from \cite{Pendavingh-vanZwam10a}. 

The key notions in this section are cross ratios, rescaling classes, fundamental elements, and the foundation of a matroid. We will study their interdependencies and apply our theory to reprove some classical results, e.g. the characterization of regular matroids as matroids which are representable over every field. We refer to section \ref{subsubsection: intro - foundations of binary and regular matroids} of the introduction for a list of such results.

\subsection{Cross ratios}
\label{subsection: cross ratios}

The study of cross ratios of four points on a line belongs to the oldest themes in mathematics and finds its earliest traces in the writings of Pappus of Alexandria (\cite{Pappus7}). Its main property is that it is invariant under projective transformation and that it characterizes the ratios of the pairwise differences between the four points.

Four points on a projective line correspond to a point of the Grassmannian $\Gr(2,4)$, and the cross ratio can be reformulated as an invariant of the Pl\"ucker coordinates of this point. This reinterpretation allows for a generalization of cross ratios to higher Grassmannians and subsequently found its entrance into matroid theory. 
For instance, cf.\ the papers \cite{Dress-Wenzel90}, \cite{Dress-Wenzel91} and \cite{Dress-Wenzel92} of Dress and Wenzel and \cite{Wenzel91} of Wenzel, \cite{Gelfand-Rybnikov-Stone95} by Gelfand, Rybnikov and Stone, \cite{Pendavingh-vanZwam10a} and \cite{Pendavingh18} by Pendavingh--van Zwam and Pendavingh, respectively, and  \cite{Delucchi-Hoessly-Saini15} by Delucchi, Hoessly and Saini. For more details on the developments of cross ratios in general and explanations of their relevance for matroid theory, we refer to the book \cite{Richter-Gebert11} of Richter-Gebert.

Let $F$ be an idyll and $M$ a matroid of rank $r$ on $E$. The cross ratios of $M$ in $F$ are indexed by certain quadrangles or $4$-cycles in the basis exchange graph of $M$.

We formalize these quadrangles as tuples $\cI=(I,i_1,i_2,i_3,i_4)\in\binom E{r-2}\times E^4$ for which $I_{1,3}$, $I_{1,4}$, $I_{2,3}$ and $I_{2,4}$ are bases of $M$, where $I_{k,l}=I\cup\{i_k,i_l\}$. We denote the collection of such tuples $\cI$ by $\Omega_M$. We say that $\cI=(I,i_1,i_2,i_3,i_4)$ is \emph{non-degenerate} if  also $I_{1,2}$ and $I_{3,4}$ are bases of $M$. Otherwise we call $(I,i_1,i_2,i_3,i_4)$ \emph{degenerate}. We define
\[
 \mu(\cI) \ = \ \mu(i_1,i_2;i_3,i_4) \ = \ \#\big\{(k,l)\in\{1,2\}\times\{3,4\} \, \big| \, i_k>i_l \big\},
\]
which is the same function that appears in the proof of Theorem \ref{thm: the unit group of the universal pasture is the Tutte group}. Note that $\epsilon^{\mu(i_1,i_2;i_3,i_4)}=\epsilon$ if and only if
 \[
  \big\{\{1,2\},\{3,4\}\big\} \ = \ \big\{\{\sigma(1),\sigma(3)\},\{\sigma(2),\sigma(4)\}\big\}
 \]
 where $\sigma\in S_4$ is the permutation with $i_{\sigma(1)}<i_{\sigma(2)}<i_{\sigma(3)}<i_{\sigma(4)}$. In particular, we have
 \[
  \mu(1,2;3,4)+\mu(1,3;2,4)+\mu(1,4;2,3) \ \equiv \ 1 \pmod{2}.
 \]

\begin{df}
 Let $F$ be an idyll and $\Delta:\binom Er\to F$ a weak Grassmann-Pl\"ucker function with underlying matroid $M$. The \emph{cross ratio function of $\Delta$} is the function $\Cr_\Delta:\Omega_M\to F^\times$ that sends an element $\cI=(I,i_1,i_2,i_3,i_4)$ of $\Omega_M$ to
 \[
  \Cr_\Delta(\cI) \ = \ \epsilon^{\mu(\cI)} \cdot\frac{\Delta_{I,1,3} \cdot \Delta_{I,2,4}}{\Delta_{I,1,4} \cdot \Delta_{I,2,3}}
 \]
 where $\Delta_{I,k,l}=\Delta(I\cup\{i_k,i_l\})$.
\end{df}

\begin{rem}
 From the perspective of Grassmann-Pl\"ucker functions as functions $\Delta:\binom Er\to F$, the factor $\epsilon^\mu(\cI)$ might appear unmotivated, but it appears naturally if one considers Grassmann-Pl\"ucker functions as alternating functions $\Delta:E^r\to F$ instead; in particular, our definition is compatible with \cite{Dress-Wenzel90}, \cite{Pendavingh-vanZwam10a} and \cite{Baker-Lorscheid20}. The formulas of Lemma~\ref{lemma: relations for cross ratios from the 3-term Pluecker relations} reflect the relevance of this factor.
\end{rem}

The following properties are immediate from the definition. Let $\cI=(I,i_1,i_2,i_3,i_4)\in\Omega_M$. For a permutation $\sigma$ of $\{1,2,3,4\}$, we define $\sigma.\cI=(I,i_{\sigma(1)},i_{\sigma(2)},i_{\sigma(3)},i_{\sigma(4)})$. Then the cross ratio satisfies the relations
\[
 \Cr_\Delta\big(\sigma.\cI\big) \ = \ \Cr_\Delta(\cI) \qquad \text{and} \qquad \Cr_\Delta\big(\tau.\cI\big) \ = \ \Cr_\Delta(\cI)^{-1}
\]
for every $\sigma$ in the Klein four group $V=\big\{e,(12)(34),(13)(24),(14)(23)\big\}$ and every $\tau$ in the coset $V.(34)$. In particular, all of these cross ratios are defined. If $\cI$ is non-degenerate, then $\Cr_\Delta(\sigma.\cI)$ is defined for all permutations $\sigma$ and we have the identity
\[
 \Cr_\Delta(\cI) \ \cdot \ \Cr_\Delta\big((234).\cI\big) \ \cdot \ \Cr_\Delta\big((243).\cI\big) \ = \ \epsilon.
\]
Finally, we observe that $\Cr_{a\Delta}(\cI)=\Cr_\Delta(\cI)$ for every $a\in F^\times$. In other words, the cross ratio depends only on the weak $F$-matroid $[\Delta]$ defined by $\Delta$. 

\begin{lemma}\label{lemma: relations for cross ratios from the 3-term Pluecker relations}
 Let $F$ be an idyll and $\Delta:\binom Er\to F$ a weak Grassmann-Pl\"ucker function with underlying matroid $M$. Let $\cI\in\Omega_M$. If $\cI$ is degenerate, then $\Cr_\Delta(\cI)=1$. If $\cI$ is non-degenerate, then
 \[
  0 \ \leq \ \Cr_\Delta(\cI) \ + \ \Cr_\Delta\big((23).\cI\big) \ + \ \epsilon.
 \]
\end{lemma}

\begin{proof}
 Let $\cI=(I,i_1,i_2,i_3,i_4)\in\Omega_M$ and let $\sigma\in S_4$ be the permutation with $i_{\sigma(1)}<i_{\sigma(2)}<i_{\sigma(3)}<i_{\sigma(4)}$. Since $\epsilon^{\mu(i_1,i_2;i_3,i_4)}=\epsilon$ if and only if
 \[
  \big\{\{1,2\},\{3,4\}\big\} \ = \ \big\{\{\sigma(1),\sigma(3)\},\{\sigma(2,),\sigma(4)\}\big\},
 \]
 the $3$-term Grassmann-Pl\"ucker relation for $\Delta$ become
 \begin{align*}
   0 \ &\leq \ \Delta_{I,\sigma(1),\sigma(2)} \ \Delta_{I,\sigma(3),\sigma(4)} \ + \ \epsilon \ \Delta_{I,\sigma(1),\sigma(3)} \ \Delta_{I,\sigma(2),\sigma(4)} \ + \ \Delta_{I,\sigma(1),\sigma(4)} \ \Delta_{I,\sigma(2),\sigma(3)} \\
       &= \ \epsilon^{\mu(1,2;3,4)} \Delta_{I,1,2}\Delta_{I,3,4} \ + \ \epsilon^{\mu(1,3;2,4)} \Delta_{I,1,3}\Delta_{I,2,4} \ + \ \epsilon^{\mu(1,4;2,3)} \Delta_{I,1,4}\Delta_{I,2,3}
 \end{align*}
 where $\Delta_{I,k,l}=\Delta\big(I\cup\{i_k,i_l\}\big)$.
 
 If $\cI$ is degenerate, then the term $\Delta_{I,1,2}\Delta_{I,3,4}$ is zero and thus the quantities $\epsilon^{\mu(1,3;2,4)}\Delta_{I,1,3}\Delta_{I,2,4}$ and $\epsilon^{\mu(1,4;2,3)}\Delta_{I,1,4}\Delta_{I,2,3}$ are mutually weakly inverse to each other. Thus
 \[
  \Cr_\Delta(\cI) \ = \ \epsilon^{\mu(1,2;3,4)}\frac{\Delta_{I,1,3}\Delta_{I,2,4}}{\Delta_{I,1,4}\Delta_{I,2,3}} \ = \ \epsilon^{\mu(1,2;3,4)+\mu(1,3;2,4)+\mu(1,4;2,3)+1}\frac{\Delta_{I,1,4}\Delta_{I,2,3}}{\Delta_{I,1,4}\Delta_{I,2,3}} \ = \ 1,
 \]
 where we use that $\mu(1,2;3,4)+\mu(1,3;2,4)+\mu(1,4;2,3)+1$ is even.
 
 If $\cI$ is non-degenerate, then all three terms of the Pl\"ucker relation are nonzero. After dividing by $\epsilon^{\mu(1,4;2,3)+1}\Delta_{I,1,4}\Delta_{I,2,3}$ and interchanging the order of the first two terms, we obtain
 \begin{align*}
   0 \ &\leq \ \epsilon^{\mu(1,3;2,4)+\mu(1,4;2,3)+1} \cdot \frac{\Delta_{I,1,3}\Delta_{I,2,4}}{\Delta_{I,1,4}\Delta_{I,2,3}} \ + \ \epsilon^{\mu(1,2;3,4)+\mu(1,4;2,3)+1} \cdot \frac{\Delta_{I,1,2}\Delta_{I,3,4}}{\Delta_{I,1,4}\Delta_{I,2,3}} \ + \ \epsilon \\
       &= \ \Cr_\Delta(\cI) \ + \ \Cr_\Delta\big((23).\cI\big)  \ + \ \epsilon,
 \end{align*}
 as claimed.
\end{proof}

\begin{rem}
 Let $\Delta:\binom Er\to F$ be a Grassmann-Pl\"ucker function with values in $F$. We can extend the cross ratio $\Cr_\Delta(\cI)$ to tuples $\cI=(I,i_1,i_2,i_3,i_4)$ such that only one of
 \[
  \Delta_{I,1,3} \cdot \Delta_{I,2,4} \qquad \text{and} \qquad \Delta_{I,1,4} \cdot \Delta_{I,2,3}
 \]
 is nonzero, which is the case if $\#I\cup\{i_k,i_l\}=r$ for all distinct $k,l\in\{1,\dotsc,4\}$ and if no three of $i_1,\dotsc,i_4$ are identical. In this case, we define $\Cr_\Delta(\cI)$ to be $0$ if the numerator is $0$ and to be $\infty$ if the denominator is $0$. This extended notion of cross-ratio gives a function with values in $\P^1(F)$, generalizing the cross ratio of four points (no three of which coincide) on a line in classical projective geometry.
\end{rem}


\subsection{Foundations}
\label{subsection: foundations}

Pendavingh and van Zwam exhibit in \cite{Pendavingh-vanZwam10a} the role of fundamental elements for the representability of matroids over partial fields. In this section, we extend this concept to $\Funpm$-algebras, which makes this theory applicable to matroids over all idylls. 

Recall from section \ref{subsubsection: subblueprints} the definition of a subblueprint of $B$ as a submonoid $C^\bullet$ of $B^\bullet$ together with the structure of an ordered blueprint that is induced from $B$. 

\begin{df}
 Let $B$ be an $\Funpm$-algebra. A \emph{fundamental element of $B$} is an element $a\in B$ such that there exists an element $b\in B$ with $0\leq a+b+\epsilon$. The \emph{foundation of $B$} is the subblueprint $B^\found$ of $B$ generated by the fundamental elements over $\Funpm$. We call $B$ a \emph{foundation} if $B^\found=B$.
\end{df}

Note that $B^\found$ is a foundation and that $0$ and $1$ are always fundamental elements since $0\leq 0+1+\epsilon$. Note further that the definition of the foundation is functorial: if $f:B\to C$ is a morphism of $\Funpm$-algebras and $0\leq a+b+\epsilon$ in $B$, then $0\leq f(a)+f(b)+\epsilon$ in $C$. Thus $f$ restricts to a morphism $f^\found:B^\found\to C^\found$. This defines an idempotent endofunctor $(-)^\found:\OBlpr_\Funpm\to\OBlpr_\Funpm$. 

\begin{rem}
If $R$ is an integral domain of characteristic zero which is finitely generated over $\Z$, then for every subidyll $F$ of $R$, the set of fundamental elements of $F$ (which generates the foundation of $F$ as an ordered blueprint) is  {\em finite}.  (This general result applies, for example, to many of partial fields appearing in \cite{vanZwam09}.)

Indeed, Serge Lang proves in \cite{Lang60} that an integral domain $R$ of characteristic zero which is finitely generated over $\Z$ contains only a finite number of pairs of elements $a,b\in R^\times$ with $a+b=1$. 
The original proof in \cite{Lang60} was ineffective, but \cite{Evertse-Gyory13} contains an effective proof. An effective upper bound for the number of fundamental elements, depending only on the rank of $R^\times$ as an abelian group, is given in \cite{Beukers-Schlickewei96}. There is a similar result for characteristic $p>0$ if one counts solutions up to $p$-th powers; cf.\ \cite{Koymans-Pagano17}.
\end{rem}

The relevance of fundamental elements and foundations for matroid theory is that the foundation of $B$ contains all cross ratios of all Grassmann-Pl\"ucker functions in $B$. More precisely, we have the following:

\begin{lemma}\label{lemma: every cross ratio is a fundamental element}
 Let $F$ be an idyll, $M$ a matroid, $\cI\in\Omega_M$ and $\Delta:\binom Er\to F$ a Grassmann-Pl\"ucker function representing $M$. Then $\Cr_\Delta(\cI)$ is a fundamental element in $F$. 
\end{lemma}

\begin{proof}
 This follows at once from the relations for the cross ratios exhibited in Lemma \ref{lemma: relations for cross ratios from the 3-term Pluecker relations}.
\end{proof}

\begin{df}\label{def: foundation}
 Let $M$ be a matroid and $k_M^w$ its universal pasture. The \emph{foundation of $M$} is the subidyll $k_M^f=(k_M^w)^\found$ of $k_M^w$. 
 
 Let $\Delta:\binom Er\to k_M^w$ be the weak Grassmann-Pl\"ucker function with $\Delta(I)=x_I/x_{I_0}$ for some fixed basis $I_0$ of $M$. The \emph{universal cross ratio function of $M$} is the function $\Cr_M^\univ:\Omega_M\to k_M^f$ that sends $\cI\in\Omega_M$ to the \emph{universal cross ratio $\Cr_M^\univ(\cI)=\Cr_\Delta(\cI)$ of $\cI$}.
\end{df}

Note that multiplying $\Delta$ with a nonzero scalar $a\in k_M^w$ does not change the value of the cross ratio. Thus $\Cr_M^\univ$ is an invariant of the matroid $M$.

\begin{lemma}\label{lemma: the foundation of a matroid is generated by the universal cross ratios}
 Let $M$ be a matroid. Then the foundation $k_M^f$ of $M$ is generated by the universal cross ratios $\Cr_M^\univ(\Omega_M)$ over $\Funpm$. 
\end{lemma}

\begin{proof}
 Since all additive relations of $k_M^w$ come from the $3$-term Pl\"ucker relations, every $3$-term relation in $k_M^w$ must be a multiple of a Pl\"ucker relation. If we ask that a specified term of such a multiple is equal to $\epsilon^i$ for $i=0$ or $1$, then this multiple $0\leq a+b+\epsilon^i$ is uniquely determined and must be of the form of the relations occurring in Lemma \ref{lemma: relations for cross ratios from the 3-term Pluecker relations}, up to a factor of $\epsilon$. Thus we conclude that $a$ and $b$ must be zero or a cross ratio, up to a possible factor of $\epsilon$. Since $\epsilon\in\Funpm$, this verifies the claim of the lemma.
\end{proof}

\begin{ex}[Foundations for rank-$2$ matroids on a four element set]\label{ex: foundations for rank 2-matroids on a four element set}
 We calculate all foundations $k_M^f$ of rank-$2$ matroids $M$ on the set $E=\{1,2,3,4\}$, leaning on the classification of the universal idylls in section \ref{subsection: universal idylls for rank 2 and the four element set}. Recall that there is a unique Pl\"ucker relation in this case, which is
 \[
  0 \ \leq \ x_{1,2}x_{3,4} \ + \ \epsilon \cdot x_{1,3}x_{2,4} \ + \ x_{1,4}x_{2,3}
 \]
 where we write $x_{i,j}$ for $x_{\{i,j\}}$. If any of the three terms is zero, then all cross ratios are $1$ by Lemma \ref{lemma: relations for cross ratios from the 3-term Pluecker relations} and thus $k_M^f=\Funpm$. From our results in section \ref{subsection: universal idylls for rank 2 and the four element set}, we see that this is the case for all matroids $M$ except for the uniform matroid. In particular this implies that all these matroids are regular, cf.\ Theorem \ref{thm: characteriztaion of regular matroids}.
 
 In case of the uniform matroid $M$, it is easily seen that its foundation is
 \[
  k_M^f \ = \ \bpgenquot{\Funpm[T_1^\pm,T_2^\pm]}{0\leq T_1+T_2+\epsilon}
 \]
 where $T_1$ and $T_2$ stand for the cross ratios
 \[
  T_1 \ = \ \frac{x_{1,3}x_{2,4}}{x_{1,4}x_{2,3}} \quad \text{and} \quad T_2 \ = \ \epsilon \cdot \frac{x_{1,2}x_{3,4}}{x_{1,4}x_{2,3}},
 \]
 which do not satisfy any multiplicative relation. Note that $k_M^f$ admits a morphism into every field with more than $2$ elements. This shows that the uniform matroid of rank $2$ on $4$ elements is representable over every field but $\F_2$.
\end{ex}


\subsection{The inner Tutte group}
\label{subsection: the inner Tutte group}

Let $M$ be a matroid of rank $r$ on $E$ and $\cB$ the set of bases of $M$. As a consequence of Theorem \ref{thm: the unit group of the universal pasture is the Tutte group}, the Tutte group $\T_M$  of $M$ is isomorphic to the abelian group generated by $\epsilon$ and $\prod_{I\in\cB} X_I^{e_I}$ with $\sum e_I=0$ modulo the relations $\epsilon^2=1$ and 
\[
 \epsilon^{\mu(i_1,i_2;i_3,i_4)} \cdot \frac{X_{I,1,3}\ X_{I,2,4}}{X_{I,1,4}\ X_{I,2,3}} \ = \ 1
\]
for every degenerate $(I,i_1,i_2,i_3,i_4)\in\Omega_M$, where $X_{I,k,l}=X_{I\cup\{i_k,i_l\}}$.
 
We recall the definition of the inner Tutte group from \cite[Def.\ 1.6]{Dress-Wenzel89}. For a subset $I$ of $E$, let $\delta_I:E\to \Z$ be the characteristic function on $I$, i.e.\ $\delta_I(i)=1$ if $i\in I$ and $\delta_I(i)=0$ otherwise. Since 
\[
 \delta_{I\cup\{i_1,i_2\}} \ + \ \delta_{I\cup\{i_3,i_4\}} \ - \ \delta_{I\cup\{i_2,i_3\}} \ - \ \delta_{I\cup\{i_4,i_1\}} \ = \ 0
\]
for every degenerate $(I,i_1,i_2,i_3,i_4)\in\Omega_M$, we obtain a group homomorphism
\[
 \begin{array}{cccc}
  \deg_E: & \T_M           & \longrightarrow & \Z^E. \\
        &\prod X_I^{e_I} & \longmapsto     & \sum e_I\delta_I
 \end{array}
\]

\begin{df}
 The \emph{inner Tutte group $\T_M^{(0)}$} is the kernel of $\deg_E$.
\end{df}

The following result is Proposition 6.4 in \cite{Wenzel91}.  Its proof relies on Tutte's ``fundamental theorem on linear subclasses'' (Theorem~4.34 in \cite{Tutte65}), which is significantly easier to prove than the relatively deeper parts of Tutte's homotopy theory for matroids found in \cite{Tutte58a} and \cite{Tutte58b} (and also exposited in \cite{Tutte65}).

\begin{thm}\label{thm: the inner Tutte group is generated by cross ratios}
 The inner Tutte group $\T_M^{(0)}$ is generated by $\epsilon$ and the elements
 \[
  \epsilon^{\mu(i_1,i_2;i_3,i_4)} \cdot \frac{X_{I,1,3}\ X_{I,2,4}}{X_{I,1,4}\ X_{I,2,3}}
 \]
 for every non-degenerate $(I,i_1,i_2,i_3,i_4)\in\Omega_M$.
\end{thm}

This theorem has a series of consequences for our theory that we will explain in the following. Recall from Theorem \ref{thm: the unit group of the universal pasture is the Tutte group} that the association $\prod x_I^{e_I}\mapsto\prod X_I^{e_I}$ defines an isomorphism $(k_M^w)^\times\to\T_M$ between the units of the universal pasture and the Tutte group of $M$.

\begin{cor}\label{cor: the inner Tutte group is the unit group of the foundation}
 The isomorphism $(k_M^w)^\times\to\T_M$ restricts to an isomorphism $(k_M^f)^\times\to\T_M^{(0)}$.
\end{cor}

\begin{proof}
 Let $\Delta:\binom Er\to k_M^w$ be the Grassmann-Pl\"ucker function defined by $\Delta(I)=x_I/x_{I_0}$ for some fixed basis $I_0\in\cB$. By Lemma \ref{lemma: relations for cross ratios from the 3-term Pluecker relations}, we have $\Cr_\Delta(\cI)\in\{1,\epsilon\}$ for degenerate $\cI\in\Omega_M$. Therefore $k_M^f$ is generated by $\epsilon$ and the cross ratios $\Cr_\Delta(\cI)$ for non-degenerate $\cI\in\Omega_M$.
 
 For non-degenerate $I=(I,i_1,i_2,i_3,i_4)$, the image of $\Cr_\Delta(\cI)$ in $\T_M$ is $\epsilon^{\mu(i_1,i_2;i_3,i_4)} \cdot \frac{X_{I,1,3}\ X_{I,2,4}}{X_{I,1,4}\ X_{I,2,3}}$. Thus the generators of $(k_M^f)^\times$ and $\T_M^{(0)}$ agree, which yields the promised isomorphism.
\end{proof}

Let $B$ be an ordered blueprint. Recall from Example \ref{ex: Laurent monomials as localization of monomial blueprints} the definition of $B[T_1^{\pm1},\dotsc,T_s^{\pm1}]$ as the localization of the free algebra $B[T_1,\dotsc,T_s]$ at the multiplicative subset generated by $\{T_1,\dotsc,T_s\}$. The canonical isomorphism
\[
 B[T_1^{\pm1},\dotsc,T_s^{\pm1}] \ \stackrel\sim\longrightarrow \ B\otimes_\Fun\Fun[T_1^{\pm1},\dotsc,T_s^{\pm1}]
\]
makes clear that all additive relations of $B[T_1^{\pm1},\dotsc,T_s^{\pm1}]$ come from $B$. In particular, the ordered blueprint $B[T_1^{\pm1},\dotsc,T_s^{\pm1}]$ is an $\Funpm$-algebra with unique weak inverses if and only if $B$ is so. By the very construction of $B[T_1^{\pm1},\dotsc,T_s^{\pm1}]$, we have $B[T_1^{\pm1},\dotsc,T_s^{\pm1}]^\times\simeq B^\times\times\Z^s$.

\begin{cor}\label{cor: the universal pasture as Laurent seris over the foundation}
 The universal pasture $k_M^w$ is isomorphic to $k_M^f[T_1^{\pm1},\dotsc,T_s^{\pm1}]$ for some $s\geq0$.
\end{cor}

\begin{proof}
 The additive relations of $k_M^w$ are generated by the $3$-term Pl\"ucker relations. We have seen in Lemma \ref{lemma: relations for cross ratios from the 3-term Pluecker relations} that the $3$-term Pl\"ucker relations lead to relations in $k_M^f$. It is clear that we can recover the $3$-term Pl\"ucker relations from these relations between the cross ratios by multiplying with an appropriate element of $k_M^w$.
 
 Since both $k_M^w$ and $k_M^f$ are ordered blue fields, we are done if we can show that $(k_M^w)^\times$ is isomorphic to the product of $(k_M^f)^\times$ with a free abelian group. By Theorem \ref{thm: the unit group of the universal pasture is the Tutte group}, we have $(k_M^w)^\times\simeq\T_M$, and by Corollary \ref{cor: the inner Tutte group is the unit group of the foundation}, this isomorphism restricts to an isomorphism $(k_M^f)^\times\simeq\T_M^{(0)}$. By definition, $\T_M^{(0)}$ is the kernel of $\deg_E:\T_M\to\Z^E$. Thus the quotient $\T_M/\T_M^{(0)}$ is isomorphic to a subgroup of $\Z^E$ and is therefore free abelian. This proves our claim.
\end{proof}

\begin{rem}
 We will see in Corollary \ref{cor: the universal pasture as Laurent seris over the foundation with s=n-c} that the number $s$ of variables in $k_M^f[T_1^{\pm1},\dotsc,T_s^{\pm1}]$ is equal to $n - c$, where $n=\#E$ and $c$ is the number of connected components of $M$.
\end{rem}

\begin{df}
 Let $F$ be an idyll and $t_F:F\to\K$ the terminal map. A matroid $M$ is \emph{weakly (resp. strongly) representable over $F$} if there is a weak (resp. strong) $F$-matroid $M'$ whose pushforward under $t_F$ is $M$.
 \end{df}

Equivalently, $M$ is weakly (resp. strongly) representable over $F$ if and only if $\cX_M^w(F)$ (resp. $\cX_M(F)$) is nonempty. The following result is motivated by a theorem of Pendavingh and van Zwam, cf.\ \cite[Thm.\ 2.27]{Pendavingh-vanZwam10a}.

\begin{thm}\label{thm: weakly representable over F and morphism from the fundament}
 Let $M$ be a matroid and $F$ an idyll. Then the following are equivalent:
 \begin{enumerate}
  \item\label{weakrep1} $M$ is weakly representable over $F$; 
  \item\label{weakrep2} $M$ is weakly representable over $F^\found$; 
  \item\label{weakrep3} there exists a morphism $k_M^f\to F$.
 \end{enumerate}
\end{thm}

\begin{proof}
 The inclusion $F^\found\to F$ induces a map $\cX_M^w(F^\found)\to\cX_M^w(F)$, which shows that if $M$ is representable over $F^\found$, then it is representable over $F$. Thus \eqref{weakrep2}$\Rightarrow$\eqref{weakrep1}.
 
 If $M$ is weakly representable over $F$, then there exists a morphism $k_M^w\to F$ by Proposition \ref{prop: weak realization space as morphism set from the universal pasture}. Composing it with the inclusion $k_M^f\to k_M^w$ yields the desired morphism $k_M^f\to F$. Thus \eqref{weakrep1}$\Rightarrow$\eqref{weakrep3}.
 
 If there is a morphism $k_M^f\to F$, then its image is contained in $F^\found$. By Corollary \ref{cor: the universal pasture as Laurent seris over the foundation}, we have $k_M^w\simeq k_M^f[T_1^{\pm1},\dotsc,T_s^{\pm1}]$, and thus there exists a morphism $k_M^w\to k_M^f$, for instance by extending the identity on $k_M^f$ by $T_i\mapsto 1$. Thus we obtain a morphism $k_M^w\to F^\found$. By Proposition \ref{prop: weak realization space as morphism set from the universal pasture}, $M$ is weakly representable over $F^\found$. Thus \eqref{weakrep3}$\Rightarrow$\eqref{weakrep2}.
\end{proof}


\subsection{Rescaling classes}
\label{subsection: rescaling classes}

Let $B$ be an $\Funpm$-algebra and $T(B)$ the set of functions $t:E\to B^\times$, which comes with the structure of an abelian group with respect to the product $t\cdot t'(i)=t(i)\cdot t'(i)$. For a subset $I$ of $E$, we define $t_I=\prod_{i\in I}t(i)$. For $t\in T(B)$ and a weak Grassmann-Pl\"ucker function $\Delta:\binom Er\to B$, we define
\[
 t.\Delta(I) \ = \ t_I\cdot\Delta(I).
\]
It is evident that $t.\Delta:\binom Er\to B$ is also a weak Grassmann-Pl\"ucker function. This defines an action
\[
 \begin{array}{ccc}
  T(B) \times \cMat^w(r,E)(B) & \longrightarrow & \cMat^w(r,E)(B) \\
         (t,[\Delta])        & \longmapsto     & [t.\Delta]
 \end{array}
\]
of $T(B)$ on the set $\Mat^w(r,E)(B)$ of weak $B$-matroids. 

Note that $t.[\Delta]=[\Delta]$ if $t$ is a constant function, i.e.\ $t(i)=t(j)$ for all $i,j\in E$. Thus the action of $T(B)$ on $\Mat^w(r,E)(B)$ factors through an action of the quotient of $T(B)$ by the subgroup of constant functions. See Corollary \ref{cor: stabilizer and orbit of a matroid under the torus action} for a detailed description of the stabilizer and the orbit of a weak $F$-matroid under this action in case of an idyll $F$.

\begin{df}
 Let $M$ be a $B$-matroid. The \emph{rescaling class of $M$} is the $T(B)$-orbit of $M$ in $\Mat^w(r,E)(B)$. Two Grassmann-Pl\"ucker functions $\Delta$ and $\Delta'$ are \emph{rescaling equivalent} if $[\Delta]$ and $[\Delta']$ lie in the same rescaling class of $M$. 
\end{df}

\begin{rem}
 The ``rescaling'' equivalence relation on $B$-matroids appears under different names in the literature. While we borrow the terminology ``rescaling class'' from \cite{Delucchi-Hoessly-Saini15}, the term ``projective equivalence class'' is used in \cite{Wenzel91}. Rescaling classes of oriented matroids appear as ``reorientation classes'' in \cite{Gelfand-Rybnikov-Stone95}. In \cite{Pendavingh-vanZwam10a}, projective equivalence is called ``strong equivalence'' in the context of partial fields (cf.\ Remark \ref{rem: equivalence with PvZ's definition of the universal partial field}).

In the context of matroid representations over fields, there are additional notions of equivalence which one encounters in the literature.  We briefly explain the connection with the notion of equivalent representations that one finds in Oxley's book \cite{Oxley92}.
Consider a matroid $M$ with Grassmann-Pl\"ucker function $\Delta:\binom Er\to \K$. Recall from section \ref{subsubsection: representations of matroids over a partial field} that a representation of $M$ over a field $k$ is an $r\times E$-matrix $A$ whose $r\times r$-minors vanish for precisely those $r$-subsets $I$ of $E$ for which $\Delta(I)=0$. Equivalence of representations $A$ and $A'$, in the sense of \cite{Oxley92}, is defined in terms of realizations as incidence geometries inside $\P^{r-1}(k)$: $A$ and $A'$ are said to be \emph{equivalent} if there is an automorphism of $\P^{r-1}(k)$ (as an incidence geometry) which identifies the respective realizations of $A$ and $A'$.  

In our language this boils down to the following (using \cite[Cor.\ 6.3.11]{Oxley92}).
The representations $A$ and $A'$ define Grassmann-Pl\"ucker functions $\Delta:\binom Er\to k$ and $\Delta':\binom Er\to k$, respectively. Then $A$ and $A'$ are equivalent in the sense of  \cite{Oxley92} if and only if $r\leq 2$ or there exists a field automorphism $\tau: k\to k$ such that the $k$-matroids $[\Delta']$ and $[\tau\circ\Delta]$ are rescaling equivalent.
(Matroids of rank $2$ behave exceptionally in this approach due to the lack of rigidity of $\P^1$ as an incidence geometry.)
In view of this rephrasing, we see that equivalence of representations over $k$ in the sense of \cite{Oxley92} is weaker than the notion of rescaling equivalence of $k$-matroids.
\end{rem}

\begin{rem}
From the point of view of complex algebraic geometry, rescaling classes are related to the natural action of the diagonal torus $({\mathbb C}^\times)^n$ on the Grassmannian $\Gr(r,n)$.  
The closure $X_p$ of the torus orbit of a point $p \in \Gr(r,n)$ is a toric variety which depends only on the matroid $M_p$ whose bases correspond to the non-zero Pl\"ucker coordinates of $p$, and the polytope corresponding to $X_p$ under the moment map is the {\em matroid polytope} of $M_p$, cf.\ \cite{Gelfand-Goresky-MacPherson-Serganova87}.
\end{rem}

\begin{rem}
 Note that $t.\Delta$ is a (strong) Grassmann-Pl\"ucker function if $\Delta$ is so. Thus the action of $T(B)$ restricts to an action on the matroid space $\Mat(r,E)(B)$. This action is, in fact, defined on the level of ordered blue schemes.
  
To see this, let us (for notational purposes) identify $E$ with $\{1,\dotsc,n\}$. The ordered blue scheme $\G_{m,\Funpm}^n=\Spec\Funpm[T_1^{\pm1},\dotsc, T_n^{\pm1}]$ is a group object in $\OBSch_\Funpm$ with respect to the comultiplication given by $T_i\mapsto T_i\otimes T_i$. It acts on $\Mat(r,E)$ via the coaction given by $x_I\mapsto (\prod_{i\in I}T_i)\otimes x_I$. The action of $T(B)$ on $\Mat(r,E)(B)$ results from applying $\Hom(\Spec B,-)$ to the morphism $T\times\Mat(r,E)\to\Mat(r,E)$.
\end{rem}

Note that a morphism $f:F\to F'$ of idylls induces a map $f_\ast:\cX^w(F)\to\cX^w(F')$ since $f\circ(t.\Delta)=(f\circ t).(f\circ\Delta)$ for all $t\in T(F)$ and all weak Grassmann-Pl\"ucker functions $\Delta:\binom Er\to F$. This defines a functor $\cX^f:\OBlpr_\Funpm\to\Sets$ that sends an $\Funpm$-algebra to the set $\cX^f(F)$ of rescaling classes in $\Mat^w(r,E)(F)$.

\begin{lemma}\label{lemma: cross ratios are well-defined on rescaling classes}
 Let $\Delta$ and $\Delta'$ be rescaling equivalent Grassmann-Pl\"ucker functions in $B$. Then they correspond to the same $B$-matroid $M$ and $\Cr_\Delta(\cI)=\Cr_{\Delta'}(\cI)$ for all $\cI\in\Omega_M$.
\end{lemma}

\begin{proof}
 Since $\Delta$ and $\Delta'$ are rescaling equivalent, there are elements $a\in B^\times$ and $t\in T(B)$ such that $\Delta'=a(t.\Delta)$. Since $at_I\in B^\times$ for all subsets $I$ of $E$, we have $\Delta(I)=0$ if and only if $\Delta'(I)=a(t_I\Delta(I))=0$. This shows that $\Delta$ and $\Delta'$ correspond to the same matroid $M$.
 
 Let $\cI=(I,i_1,i_2,i_3,i_4)\in\Omega_M$. Then for $I_{k,l}=I\cup\{i_k,i_l\}$, we have $t_{I_{1,3}}t_{I_{2,4}}=t_{I_{1,4}}t_{I_{2,3}}$ and thus
 \[
  \Cr_{\Delta'}(\cI) \quad = \quad\frac{\Delta'_{I,1,3}\Delta'_{I,2,4}}{\Delta'_{I,1,4}\Delta'_{I,2,3}} \quad = \quad\frac{at_{I_{1,3}}at_{I_{2,4}}}{at_{I_{1,4}}at_{I_{2,3}}} \ \cdot \ \frac{\Delta_{I,1,3}\Delta_{I,2,4}}{\Delta_{I,1,4}\Delta_{I,2,3}} \quad = \quad\Cr_{\Delta}(\cI) 
 \]
 as desired, where $\Delta'_{I,k,l}=\Delta'\big(I\cup\{i_k,i_l\}\big)$ and $\Delta_{I,k,l}=\Delta\big(I\cup\{i_k,i_l\}\big)$.
\end{proof}

Let $M$ be a matroid. Recall from section \ref{subsection: the inner Tutte group} that the group homomorphism $\deg_E:\T_M\to\Z^E$ sends an element $\prod X_I^{e_I}$ of the Tutte group $\T_M$ to $\sum e_I\delta_I$, where $\delta_I$ is the characteristic function of $I$. For the proof of Theorem \ref{thm: the foundation represents rescaling classes}, we require the following fact, which follows from Theorem 1.5 in \cite{Dress-Wenzel89}. For completeness, we include a short proof.

\begin{lemma}\label{lemma: the cokernel of deg_E is free}
 Let $c$ be the number of connected components of $M$. The cokernel of $\deg_E:\T_M\to\Z^E$ is a free abelian group of rank $c$. 
\end{lemma}

\begin{proof}
 As a first step, we claim that the Tutte group $\T_M$ is the direct sum of the Tutte groups of the connected components of $M$, modulo the identification of the weak inverses $\epsilon$ of all summands (also cf.\ \cite[Prop.\ 5.1]{Dress-Wenzel89}). Recall that a connected component of $M$ is the restriction of $M$ to an equivalence class of the equivalence relation $\sim$ on $E$ generated by the relations $i\sim j$ whenever there are bases $I$ and $J$ such that $I-J=\{i\}$ and $J-I=\{j\}$. Let $(I,i_1,i_2,i_3,i_4)\in\Omega_M$ be degenerate and $\epsilon^{\mu(i_1,i_2;i_3,i_4)}X_{I,1,3}X_{I,2,4}X_{I,1,4}^{-1} X_{I,2,3}^{-1}=1$ the corresponding relation of the Tutte group. The bases involved in this relation imply that $i_1\sim i_2\sim i_3\sim i_4$. Thus the relation $\epsilon^{\mu(i_1,i_2;i_3,i_4)}X_{I,1,3}X_{I,2,4}X_{I,1,4}^{-1} X_{I,2,3}^{-1}=1$ comes from the restriction of $M$ to the equivalence class of $i_1$, $i_2$, $i_3$ and $i_4$. This establishes the claim.

 Thus the morphism $\deg_E:\T_M\to\Z^E$ is the direct sum of its restrictions to the support of the connected components of $M$ and the rank of the cokernel is the sum of the ranks for each summand. Therefore we may assume that $M$ is connected, and we are left to show that under this assumption the cokernel of $\deg_E$ is free of rank $c=1$.
 
 Given bases $I$ and $J$ with $I-J=\{i\}$ and $J-I=\{j\}$, the quantity $\delta_i-\delta_j=\deg_E(X_I/X_J)$ is in the image of $\deg_E$. By the multiplicativity of $\deg_E$ and the connectedness of $M$, we conclude that $\delta_i-\delta_j$ is in the image of $\deg_E$ for all $i,j\in E$. These elements generate the degree-$0$ hyperplane in $\Z^E$, and thus the quotient of $\Z^E$ by the image of $\deg_E$ is $\Z$. This completes the proof of the lemma.
\end{proof}

This lemma leads to the following strengthening of Corollary \ref{cor: the universal pasture as Laurent seris over the foundation}.

\begin{cor}\label{cor: the universal pasture as Laurent seris over the foundation with s=n-c}
 Let $M$ be a matroid of rank $r$ on $E$ with $c$ connected components. Then $k_M^w\simeq k_M^f[T_1^{\pm1},\dotsc,T_s^{\pm1}]$ for $s=\# E-c$.
\end{cor}

\begin{proof}
 We know from Corollary \ref{cor: the universal pasture as Laurent seris over the foundation} that $k_M^w\simeq k_M^f[T_1^{\pm1},\dotsc,T_s^{\pm1}]$ for some $s\geq0$. By Theorem \ref{thm: the unit group of the universal pasture is the Tutte group}, $(k_M^w)^\times\simeq \T_M$ and by Corollary \ref{cor: the inner Tutte group is the unit group of the foundation}, $(k_M^f)^\times\simeq\T_M^{(0)}$. By Lemma \ref{lemma: the cokernel of deg_E is free}, the quotient $\T_M/\T_M^{(0)}$ is a free group of rank $\# E-c$. Thus $s=\# E-c$.
\end{proof}

\begin{thm}\label{thm: the foundation represents rescaling classes}
 Let $F$ be an idyll and $M$ a matroid. Let $\Delta,\Delta':\binom Er\to F$ be two weak Grassmann-Pl\"ucker functions representing $M$ with respective characteristic morphisms $\chi_{[\Delta]},\chi_{[\Delta']}:k_M^w\to F$. Then the following assertions are equivalent.
 \begin{enumerate}
  \item \label{equiv1} $\Delta$ and $\Delta'$ are rescaling equivalent.
  \item \label{equiv2} $\Cr_\Delta$ and $\Cr_{\Delta'}$ are equal as functions $\Omega_M\to F^\times$.
  \item \label{equiv3} The restrictions of $\chi_{[\Delta]}$ and $\chi_{[\Delta']}$ to the foundation $k_M^f$ are equal.
 \end{enumerate}
\end{thm}

\begin{proof}
 The implication \eqref{equiv1}$\Rightarrow$\eqref{equiv2} follows from Lemma \ref{lemma: cross ratios are well-defined on rescaling classes}. We continue with \eqref{equiv2}$\Rightarrow$\eqref{equiv3}. Let $\Cr_M^\univ:\Omega_M\to k_M^f$ be the universal cross ratio function, cf.\ Definition \ref{def: foundation}. Since $\chi_{[\Delta]}$ sends $\prod x_I^{e_I}$ to $\prod\Delta(I)^{e_I}$, and similarly for $\chi_{[\Delta']}$, we have $\Cr_\Delta=\chi_{[\Delta]}\circ \Cr_M^\univ$ and $\Cr_{\Delta'}=\chi_{[\Delta']}\circ \Cr_M^\univ$. Thus for every $\cJ\in\Omega_M$ we have
 \[
  \chi_{[\Delta]}\vert_{k_M^f}\big(\Cr_M^\univ(\cI)\big) \ = \ \Cr_\Delta(\cI) \ = \ \Cr_{\Delta'}(\cI) \ = \ \chi_{[\Delta']}\vert_{k_M^f}\big(\Cr_M^\univ(\cI)\big).
 \]
 By definition, $k_M^f$ is generated by the cross ratios in $k_M$, i.e.\ by the image of $\Cr_M^\univ$. Thus $\chi_{[\Delta]}\vert_{k_M^f}=\chi_{[\Delta']}\vert_{k_M^f}$, which establishes the implication \eqref{equiv2}$\Rightarrow$\eqref{equiv3}.
 
 We conclude with \eqref{equiv3}$\Rightarrow$\eqref{equiv1}. Recall that the inner Tutte group $\T_M^{(0)}$ is defined as the kernel of $\deg_E:\T_M\to \Z^E$. By the assumption from \eqref{equiv2}, the restrictions of $\chi_{[\Delta]}$ and $\chi_{[\Delta']}$ to $\T_M^{(0)}=(k_M^f)^\times$ are equal. Thus there is a group homomorphism $t:\im\deg_E\to F^\times$ such that 
 \[
  \chi_{[\Delta]}(x) \ = \ t\big(\deg_E(x)\big) \, \cdot \, \chi_{[\Delta']}(x)
 \]
 for every element $x\in\T_M$. Since the cokernel of $\deg_E$ is free, $\Z^E$ is the direct sum of the image of $\deg_E$ with a free abelian group. Therefore we can extend $t$ to a group homomorphism $t:\Z^E\to F^\times$. Then we have
 \[\textstyle
  \prod \Delta(I)^{e_I} \ = \ \chi_{[\Delta]}\big( \prod x_I^{e_I}\big) \ = \ t\big(\sum e_I\delta_I\big) \ \chi_{[\Delta']}\big( \prod x_I^{e_I}\big) \ = \ \prod t_I^{e_I} \,  \cdot \, \prod \Delta'(x_I)^{e_I} \ = \ (t.\Delta')(I)^{e_I}
 \]
 for all $e_I$ with $\sum e_I=0$, where $I$ varies through the bases of $M$. This shows that $\Delta$ and $t.\Delta'$ are proportional, i.e.\ $[\Delta]=[t.\Delta']$, and establishes the implication \eqref{equiv3}$\Rightarrow$\eqref{equiv1}.
\end{proof}

\begin{cor}\label{cor: stabilizer and orbit of a matroid under the torus action}
 Let $M$ be a matroid with $c$ connected components, $F$ an idyll and $\Delta:\binom Er\to F$ a Grassmann-Pl\"ucker function representing $M$. Then the stabilizer of $[\Delta]$ in $T(F)$ is isomorphic to $(F^\times)^c$ and the rescaling class of $[\Delta]$ is in bijection with $(F^\times)^{n-c}$ where $n=\# E$.
 \end{cor}

\begin{proof}
 An element of $T(F)$, which is a function $t:E\to F^\times$, can be extended linearly to a group homomorphism $t:\Z^E\to F^\times$, which we denote by the same symbol $t$. From the proof of Theorem \ref{thm: the foundation represents rescaling classes}, it is clear that the stabilizer of $[\Delta]$ consists of those group homomorphisms $t:\Z^E\to F^\times$ that are trivial on the image of $\deg_E$. By Lemma \ref{lemma: the cokernel of deg_E is free}, the cokernel of $\deg_E$ is a free abelian group of rank $h$. Thus the stabilizer of $[\Delta]$ is a subgroup of $T(F)$ isomorphic to $(F^\times)^c$. Consequently, the orbit of $[\Delta]$ corresponds to $(F^\times)^E/(F^\times)^c\simeq(F^\times)^{n-c}$. 
\end{proof}

\begin{df}
 Let $M$ be a matroid and $F$ be an idyll. The \emph{rescaling class space $\cX^f_M(F)$} is the set of rescaling classes of weak $F$-matroids with underlying matroid $M$.
\end{df}

Note that $\cX_M^f(F)$ is a subset of $\cX^f(F)$ and that $\cX_M^f(F)$ is functorial in $F$.

\begin{cor}\label{cor: rescaling classes over M as morphisms from the foundation}
 Let $M$ be a matroid and $F$ be an idyll. Then the map
 \[
  \begin{array}{cccc}
   \Phi: & \cX_M^f(F) & \longrightarrow & \Hom(k_M^f,F) \\
         & [\Delta]   & \longmapsto     & \chi_{[\Delta]}\vert_{k_M^f}
  \end{array}
 \]
 is a bijection that is functorial in $F$.
\end{cor}

\begin{proof}
 By Theorem \ref{thm: the foundation represents rescaling classes}, two Grassmann-Pl\"ucker functions $\Delta$ and $\Delta'$ that represent $M$ are rescaling equivalent if and only if $\chi_{[\Delta]}\vert_{k_M^f}=\chi_{[\Delta']}\vert_{k_M^f}$. This shows that $\Phi$ is well-defined and injective. The functoriality of $\Phi$ in $F$ is clear.
 
 To show surjectivity, consider a morphism $\chi^f:k_M^f\to F$. By Corollary \ref{cor: the universal pasture as Laurent seris over the foundation}, there is an isomorphism $k_M^w\simeq k_M^f[T_1^{\pm1},\dotsc,T_s^{\pm1}]$. Thus sending $T_1,\dotsc,T_s$ to any choice of units of $k_M^f$ defines an extension of the identity map $k_M^f\to k_M^f$ to a morphism $k_M^w\to k_M^f$. Composing this morphism with $\chi^f:k_M^f\to F$ yields a morphism $\chi:k_M^w\to F$ and thus an $F$-matroid $[\Delta]$. By construction, it is clear that $\Phi([\Delta])=\chi\vert_{k_M^f}=\chi^f$. This verifies the surjectivity of $\Phi$ and concludes the proof.
\end{proof}

The following corollary is somewhat surprising, given that the foundation of $F$ is intimately connected with {\it weak representability} but {\it a priori} has little to do with {\it strong representability}:

\begin{cor}\label{cor: strongly representable over B iff only if strongly representable over its fundament}
 Let $M$ be a matroid and $F$ an idyll. Then $M$ is strongly representable over $F$ if and only if it is strongly representable over $F^\found$.
\end{cor}

\begin{proof}
 The inclusion $F^\found\to F$ induces a morphism $\cX_M(F^\found)\to\cX_M(F)$, which shows that $M$ is strongly representable over $F$ if it is strongly representable over $F^\found$.
 
 Conversely, assume that $M$ is strongly representable over $F$ by a Grassmann-Pl\"ucker function $\Delta:\binom Er\to F$ and let $\chi_{[\Delta]}:k_M\to F$ be the characteristic morphism. Then the composition with $k_M^f\to k_M^w\to k_M$ induces a morphism $\chi_{[\Delta]}^f:k_M^f\to F$, which corresponds to the rescaling class of $[\Delta]$ by Corollary \ref{cor: rescaling classes over M as morphisms from the foundation}. Since the image of $\chi_{[\Delta]}^f$ is contained in $F^\found$, the rescaling class of $[\Delta]$ comes from a class of an $F^\found$-matroid $[\Delta']$ that is represented by a Grassmann-Pl\"ucker function $\Delta':\binom Er\to F^\found$. 
 
 Thus $\Delta$ and $\iota\circ\Delta'$ represent the same rescaling class over $F$, where $\iota:F^\found\to F$ is the inclusion, i.e.\ $[\iota\circ\Delta']=t.[\Delta]$ for some $t\in T(F)$. Since the $T(F)$-action maps strong matroids to strong matroids, $\Delta'$ satisfies all Pl\"ucker relations. This shows that $M$ is representable over $F^\found$.
\end{proof}

A consequence of Corollaries~\ref{cor: the universal pasture as Laurent seris over the foundation} and~\ref{cor: rescaling classes over M as morphisms from the foundation} is that if $F$ is a finite idyll (e.g. a finite field $\F_q$), the number of lifts of a matroid $M$ to $F$ is determined in terms of the number of rescaling classes of $M$ over $F$.

\begin{cor}\label{cor: numer of weak matroids and rescaling classes over a finite idyll}
 Let $F$ be a finite idyll with $q$ elements and $M$ a matroid of rank $r$ on $E$ with $c$ connected components.  Let $n = \#E$. Then 
 \[
  \#\cX^w_M(F) \ = \ (q-1)^{{n-c}} \, \cdot \, \#\cX_M^f(F).
 \]
\end{cor}

\begin{proof}
 By Corollary \ref{cor: the universal pasture as Laurent seris over the foundation with s=n-c}, we have $k_M^w=k_M^f[T_1^{\pm1},\dotsc,T_s^{\pm1}]$ for some elements $T_1,\dotsc,T_s\in k_M^w$ where $s=n-c$. Therefore every morphism $f:k_M^f\to F$ has $(q-1)^s$ extensions to a morphism $g:k_M^w\to F$, corresponding to the choices of images $g(T_i)\in F^\times$. Thus by Corollary~\ref{cor: rescaling classes over M as morphisms from the foundation} and Proposition~\ref{prop: weak realization space as morphism set from the universal pasture}, we have
 \[
  \#\cX_M^w(F) \ = \ \#\Hom(k_M^w,F) \ = \ (q-1)^s \, \cdot \, \#\Hom(k_M^f,F) \ = \ (q-1)^s\, \cdot \, \#\cX_M^f(F). \qedhere
 \]
\end{proof}

\begin{rem}
 Over a field $k$, the weak realization space $\cX^w_M(k)$ is naturally identified with the $k$-rational points of a locally closed subscheme $X^w_M$ of the Grassmannian $\Gr(r,n)$.  By Corollary~\ref{cor: stabilizer and orbit of a matroid under the torus action}, the natural action of the diagonal torus $T \subset {\mathrm GL}_n$ on $X^w_M$ factors through a free action by a quotient torus $T'$.  
Thus there exists a GIT quotient $X^f_M=X^w_M/T'$ as a scheme, and this quotient satisfies $X^f_M(k)=\cX^f_M(k)$ for every field $k$.
\end{rem}



\subsection{Foundations of binary matroids}
\label{subsection: foundations of binary matroids}

A \emph{binary matroid} is a matroid that is representable\footnote{Since fields are perfect idylls, cf.\ Remark \ref{rem: perfect tracts}, it does not matter here if we talk about weak or strong representability.} over the finite field $\F_2$ with two elements. In this section, we classify the foundations of binary matroids, and draw conclusions about the representability of binary matroids. 

In the following, we consider $\F_2$ as the ordered blueprint $\F_2=\bpgenquot{\{0,1\}}{0\leq 1+1}$, which results from the embedding $(-)^\oblpr:\PartFields\to\OBlpr$. Note that by Lemma \ref{lemma: matroids for ordered blue fields are the same as matroids over the underlying tract}, a matroid is representable over $\F_2$ if and only if it is representable over $\bpgenquot{\{0,1\}}{0\= 1+1}$, which would be the alternative choice of realizing $\F_2$ as an ordered blueprint; cf.\ section \ref{subsubsection: relation to the universal pasture} for more details.

\begin{thm}\label{thm: foundations of binary matroids}
 A matroid is binary if and only if its foundation is either $\Funpm$ or $\F_2$.
\end{thm}

\begin{proof}
Let $M$ be a matroid and $k_M^f$ its foundation. If $k_M^f$ is $\Funpm$ or $\F_2$, then there is a morphism $k_M^f\to\F_2$. Thus $M$ is binary by Theorem~\ref{thm: weakly representable over F and morphism from the fundament}.
 
 Conversely, let $M$ be a binary matroid that is represented by a Grassmann-Pl\"ucker function $\Delta:\binom Er\to\F_2$, i.e.\ $\Delta(I)=1$ if $I$ is a basis of $M$ and $\Delta(I)=0$ otherwise. A $3$-term Pl\"ucker relation is satisfied over $\F_2$ if and only if an even number of its terms are nonzero. This means that either all terms are zero, and thus there are not cross ratios for the corresponding tuple of $\Omega_M$, or precisely two terms are nonzero, in which case all involved cross ratios are $1$ or $\epsilon$, cf.\ Lemma \ref{lemma: relations for cross ratios from the 3-term Pluecker relations}. By Lemma \ref{lemma: the foundation of a matroid is generated by the universal cross ratios}, the foundation of $M$ is generated by the cross ratios over $\Funpm$. 

 This leads to the following two possibilities for the foundation $k_M^f$ of $M$: all nontrivial Pl\"ucker relations are of the form $0\leq 1+\epsilon$ (up to a scalar multiple) and $k_M^f=\Funpm$; or there exists a Pl\"ucker relation that is of the form $0\leq 1+1$ (up to a scalar multiple). In the latter case, $1$ is a weak inverse of $1$ and becomes identified with $\epsilon$ in $k_M^w=k(x_M^w)^\pm$. Thus $k_M^f=\F_2$. This proves the theorem.
\end{proof}

Note that both situations of Theorem \ref{thm: foundations of binary matroids} occur. The case $k_M^f=\Funpm$ occurs for regular matroids $M$, which is investigated in more detail in section \ref{subsection: foundations of regular matroids}. And we have $k_M^f=\F_2$ if $M$ is a binary matroid that is not regular. An example of such a matroid is the Fano plane, which is not representable over any field of characteristic different from $2$, cf.\ \cite[para.\ 16]{Whitney35}.

In fact, it is a classical result that a binary matroid fails to be regular only if it contains the Fano plane or its dual as a minor, cf.\ \cite[(4.5)]{Tutte58b}. Thus binary matroids are either representable over every field or only over fields of characteristic $2$. Theorem \ref{thm: foundations of binary matroids} provides a proof of this result which does not require us to consider minors or the Fano plane.

\begin{cor}\label{cor: representability of binary matroids}
 A binary matroid is either representable over every field or it is not representable over any field of characteristic different from $2$.
\end{cor}

\begin{proof}
 Let $M$ be a binary matroid. By Theorem \ref{thm: foundations of binary matroids}, the foundation of $M$ is either $\Funpm$ or $\F_2$. If $k_M^f=\Funpm$ and $k$ is a field, then there exists a map $\Funpm\to k$. Thus $M$ is representable over $k$ by Theorem \ref{thm: weakly representable over F and morphism from the fundament}. 
 
 If $k_M^f=\F_2$ and $k$ is a field of characteristic different from $2$, then there exists no morphism $\F_2\to k$. By Theorem \ref{thm: weakly representable over F and morphism from the fundament}, $M$ is not representable over $k$. 
\end{proof}

The following has already been observed in \cite{Brylawski-Lucas73} in the context of fields and has been extended in \cite[Thm.\ 6.9]{Wenzel91} to fuzzy rings:

\begin{cor}\label{cor: binary matroids allow at most one rescaling class over any idyll}
 A binary matroid has at most one rescaling class over every idyll.
\end{cor}

\begin{proof}
 Let $M$ be a binary matroid and $F$ an idyll. By Corollary \ref{cor: rescaling classes over M as morphisms from the foundation}, the classes in $\cX^f_M(F)$ correspond bijectively to the morphisms $k_M^f\to F$. For both possibilities of $k_M^f$, there is at most one such morphism. Thus the theorem follows.
\end{proof}


\subsection{Foundations of regular matroids}
\label{subsection: foundations of regular matroids}

Our results on binary matroids lead to a short proof of Tutte's characterization of regular matroids. While Tutte's original proof was based on his homotopy theory for matroids, cf.\ \cite{Tutte58a} and \cite{Tutte58b}, shorter and more elementary proofs were found later on, cf.\ \cite{Gerards89} and \cite[Thm.\ 3.1.6]{vanZwam09}.

Our proof has some ingredients in common with these latter approaches, but in contrast to other proofs, it is based on the observation that the universal idyll of a regular matroid is $\Funpm$.

\begin{thm}\label{thm: characteriztaion of regular matroids}
 Let $M$ be a matroid. Then the following are equivalent.
 \begin{enumerate}
  \item\label{reg1} $M$ is regular;
  \item\label{reg2} the foundation of $M$ is $\Funpm$;
  \item\label{reg3} $M$ is weakly representable over every idyll;
  \item\label{reg4} $M$ is binary and weakly representable over an idyll with $1\neq\epsilon$.
 \end{enumerate}
\end{thm}

\begin{proof}
 In the following, we show \eqref{reg1}$\Rightarrow$\eqref{reg2}$\Rightarrow$\eqref{reg3}$\Rightarrow$\eqref{reg1} and \eqref{reg3}$\Rightarrow$\eqref{reg4}$\Rightarrow$\eqref{reg2}. The implications \eqref{reg3}$\Rightarrow$\eqref{reg1} and \eqref{reg3}$\Rightarrow$\eqref{reg4} are trivial.

 We continue with \eqref{reg1}$\Rightarrow$\eqref{reg2}. Assume that $M$ is regular. Then $M$ is binary and its foundation is $\Funpm$ or $\F_2$ by Theorem \ref{thm: foundations of binary matroids}. The latter is not possible since there is no morphism from $\F_2$ to $\Funpm$. This establishes \eqref{reg1}$\Rightarrow$\eqref{reg2}.
 
 We continue with \eqref{reg2}$\Rightarrow$\eqref{reg3}. Assume that $k_M^f=\Funpm$. Let $F$ an idyll. Then there exists a morphism $k_M^f\to F$, and thus $M$ is weakly representable over $F$ by Theorem \ref{thm: weakly representable over F and morphism from the fundament}. This establishes \eqref{reg2}$\Rightarrow$\eqref{reg3}.

 We continue with \eqref{reg4}$\Rightarrow$\eqref{reg2}. Assume that $F$ is binary and weakly representable over an idyll $F$ with $1\neq\epsilon$. By Theorem \ref{thm: foundations of binary matroids}, the foundation of $M$ is $\Funpm$ or $\F_2$. Since $1\neq\epsilon$ in $F$, there is no morphism from $\F_2$ to $F$. Thus the foundation of $M$ must be $\Funpm$. This establishes \eqref{reg4}$\Rightarrow$\eqref{reg2} and concludes the proof of the theorem.
\end{proof}

A matroid is \emph{orientable} if it is representable over the sign hyperfield $\S$. 
Since both $\S$ and $\F_q$ (for $q$ odd) are idylls with $1\neq\epsilon$, we reobtain the following well-known consequences of Tutte's characterization of regular matroids (cf.\ 7.52 in \cite{Tutte65}) as immediate consequences of Theorem \ref{thm: characteriztaion of regular matroids}.

\begin{cor}\label{cor: regular iff. binary and orientable}
 A matroid is regular if and only if it is binary and orientable.\qed
\end{cor}

\begin{cor}\label{cor: regular iff. binary and ternary}
 A matroid is regular if and only if it is binary and representable over $\F_q$ for some odd prime-power $q$.\qed
\end{cor}

\begin{cor}\label{cor: regular iff. exactly one rescaling class over every idyll}
 A matroid is regular if and only if it has precisely one rescaling class over every idyll.
\end{cor}

\begin{proof}
 By Theorem \ref{thm: characteriztaion of regular matroids}, a regular matroid is representable over every idyll. Thus there is at least one rescaling class for every idyll. A regular matroid is binary and thus has precisely one rescaling class over every idyll by Corollary \ref{cor: binary matroids allow at most one rescaling class over any idyll}.
  
 Conversely, assume that $M$ is a matroid that has exactly one rescaling class over every idyll. Then $M$ is, in particular, representable over every idyll. Thus $M$ is regular by Theorem \ref{thm: characteriztaion of regular matroids}.
\end{proof}

\begin{cor}
 Let $M$ be a regular matroid of rank $r$ on $E$ with $c$ connected components. Then $\#\cX_M(\Funpm)=\#\cX_M^w(\Funpm)=2^{\# E-c}$.
\end{cor}

\begin{proof}
 This follows at once from Corollaries~\ref{cor: numer of weak matroids and rescaling classes over a finite idyll} and~\ref{cor: regular iff. exactly one rescaling class over every idyll} and the fact that $\Funpm$ is perfect.
\end{proof}


\subsection{Uniqueness for rescaling classes over \texorpdfstring{$\F_3$}{GF(3)}}
\label{subsection: uniqueness for rescaling classes over F_3}

A matroid is called \emph{ternary} if it is representable over the field $\F_3$ with $3$ elements. Brylawski and Lucas show in \cite{Brylawski-Lucas73} that every ternary matroid has a unique rescaling class over $\F_3$. We find the following short proof of this result, employing the foundation of a matroid.

\begin{thm}\label{thm: uniqueness of rescaling classes over  F_3}
 Every matroid admits at most one rescaling class over $\F_3$.
\end{thm}

\begin{proof}
 Let $M$ be a matroid with foundation $k_M^f$. By Corollary \ref{cor: rescaling classes over M as morphisms from the foundation}, the rescaling classes of $M$ over $\F_3$ correspond bijectively to the morphisms $f:k_M^f\to \F_3$. Thus we aim to show that such a morphism is uniquely determined.
 
 The foundation $k_M^f$ of $M$ is defined as the subidyll of $k_M^w$ that is generated by the fundamental elements $a$ of $k_M^w$, which satisfy a relation of the form
 \[
  0 \ \leq \ a + b + \epsilon
 \]
 for some $b\in k_M^w$. Since $f$ maps nonzero elements of $k_M^f$ to $\F_3^\times$, we encounter the following possibilities. If $a=0$, then $f(a)=0$. If $b=0$, then $a=1$ and $f(a)=1$. If both $a$ and $b$ are nonzero, then $f(a)$ and $f(b)$ are units in $\F_3$ and $f(a)+f(b)=1$. This is only possible if $f(a)=f(b)=-1$. This shows that $f:k_M^f\to \F_3$ is uniquely determined if it exists and concludes the proof of the theorem.
\end{proof}

\begin{cor}
 Let $M$ be a ternary matroid of rank $r$ on $E$ with $c$ connected components. Then $\#\cX_M(\F_3)=\#\cX_M^w(\F_3)=2^{\# E-c}$.
\end{cor}

\begin{proof}
 This follows at once from Theorem \ref{thm: uniqueness of rescaling classes over F_3}, Corollary \ref{cor: numer of weak matroids and rescaling classes over a finite idyll} and the fact that $\Funpm$ is perfect.
\end{proof}


\subsection{The bracket ring and the universal partial field}
\label{subsection: the bracket ring and the universal partial field}

In \cite{Pendavingh-vanZwam10a}, Pendavingh and Van Zwam introduce new techniques for establishing representability theorems for matroids. Their central tools are the bracket ring and the universal partial field. We will explain in this section how these two objects relate to the universal pasture and the foundation.

\subsubsection{The bracket ring}

Pendavingh and van Zwam's bracket ring is a variation of White's bracket ring from \cite[Def.\ 3.1]{White75}. We recall the definition from \cite[Def.\ 4.1]{Pendavingh-vanZwam10a}.

\begin{df}
 Let $M$ be a matroid with representing Grassmann-Pl\"ucker function $\Delta:\binom Er\to\K$. Let
 \[\textstyle
  \cB \ = \ \big\{ \, I\in\binom Er \, \big| \, \Delta(I)\neq 0 \, \big\}
 \]
 be the set of bases of $M$. The \emph{bracket ring of $M$} is the ring $B_M=\Z[x_J^{\pm1}|J\in\cB]/I_M$ where $I_M$ is the ideal of $\Z[x_J^{\pm1}|J\in\cB]$ generated by the $3$-term Pl\"ucker relations
 \[
  x_{I,1,2} \, x_{I,3,4} \ - \ x_{I,1,3} \, x_{I,2,4} \ + \ x_{I,1,4} \, x_{I,2,3}
 \]
 for every $(r-2)$-subset $I$ of $E$ and all $i_1<i_2<i_3<i_4$ with $i_1,i_2,i_3,i_4\notin I$, where $x_{I,k,l}=x_{I\cup\{i_k,i_l\}}$ if $I\cup\{i_k,i_l\}\in\cB$ and $x_{I,k,l}=0$ otherwise.
\end{df}

In order to relate the bracket ring of a matroid to its universal pasture, we require some auxiliary definitions. The units of the bracket ring are
\[\textstyle
 B_M^\times \ = \ \big\{\pm \prod_{I\in\cB} x_I^{e_I} \, \big| \, e_I\in\Z \, \big\}.
\]
We define the \emph{partial bracket field} as the partial field $P_M=(P_M^\times,\pi_{P_M})$ where $P_M^\times=B_M^\times$ and $\pi_{P_M}:\Z[P_M^\times]\to B_M$ is the canonical projection. Note that $P_M$ is indeed a partial field if $B_M$ is nontrivial since $\ker \pi_{P_M}$ is generated by the $3$-term Pl\"ucker relations. This partial field has been considered by Pendavingh and van Zwam without being given a distinctive name, cf.\ \cite[Lemma 4.4]{Pendavingh-vanZwam10a}. Note that there are matroids with trivial bracket ring, cf.\ Remark \ref{rem: matroids with trivial bracket ring}.

Since the bracket ring $B_M$ is $\Z$-graded by putting $\deg x_J=1$, we can consider its degree-$0$ subring $B_{M,0}$, which we call the \emph{degree-$0$ bracket ring of $M$}. Analogously, we define the \emph{partial degree-$0$ bracket field of $M$} as $P_{M,0}=(B_{M,0}^\times,\pi_{P_{M,0}})$ where $\pi_{P_{M,0}}:\Z[B_{M,0}^\times]\to B_{M,0}$ is the restriction of $\pi_{P_M}$ to the degree-$0$ elements.

\subsubsection{Relation to the universal pasture}
\label{subsubsection: relation to the universal pasture}


The bracket ring can be derived from the universal pasture in a functorial way. 
The most conceptually satisfying way to see this involves first modifying the way in which we realize partial fields as ordered blueprints.
Namely, let $P=(P^\times,\pi_P)$ be a partial field, where $\pi_P:\Z[P^\times]\to R_P$ is the surjection onto the ambient ring $R_P$ of $P$. The construction used primarily in this paper associates with $P$ the idyll
\[
 P^\oblpr \ = \ \bpgenquot{P}{0\leq a+b+c \,|\,\pi_P(a+b+c)=0}.
\]
Alternatively, we can view $P$ as the $\Funpm$-algebra
\[\textstyle
 P^\blpr \ = \ \bpgenquot{P}{\sum a_i\=\sum b_j \,|\, \sum\pi_P(a_i)=\sum \pi_P(b_j)}.
\]
Both associations are functorial and define fully faithful embeddings into $\OBlpr_\Funpm$, which can be transformed into each other in the following way. 

Given an ordered blueprint $B$, we define the \emph{associated purely positive ordered blueprint} as
\[ \textstyle
 B^\ppos \ = \ \bpbiggenquot{B^\bullet}{0\leq\sum a_i \; \big| \; 0\leq\sum a_i\text{ holds in }B}.
\]
Then the identity map on $P$ induces a natural morphism $P^\oblpr\to P^\blpr$, which in turn induces isomorphisms 
\[
 P^\oblpr\otimes_\Funpm\Funsq \ \stackrel\sim\longrightarrow \ P^\blpr \quad \text{and} \quad P^\oblpr \ \stackrel\sim\longrightarrow \ (P^\blpr)^\ppos
\]
since $P^\blpr$ is an $\Funsq$-algebra and $P^\oblpr$ is purely positive. 
Note that we can recover the ambient ring of $P$ as $R_P=(P^\oblpr\otimes_\Funpm\Funsq)^+=(P^\blpr)^+$; recall from section \ref{subsection: ordered blueprints} the notation $B^+$ for the ambient semiring of an ordered blueprint $B$.

Recall from section \ref{subsection: the weak matroid space} the definition of the universal pasture $k_M^w$ of $M$ as 
\[\textstyle
 k_M^w \ = \ \big(\bpquot{\Funpm[x_J^{\pm1} \; | \; J\in\binom Er]}{\cPl^w(r,E)}\big)_0 
\]
where $\cPl^w(r,E)$ is generated by the $3$-term Pl\"ucker relations. Since these relations are satisfied in $P_{M,0}$, there is a canonical morphism of idylls
\[
 k_M^w \ \longrightarrow \ P_{M,0}^\oblpr.
\]
Note that this morphism is in general not an isomorphism. In particular, $P_{M,0}$ can be trivial while $k_M^w$ is not; cf.\ Remark \ref{rem: matroids with trivial bracket ring}.

\begin{lemma}\label{lemma: relation between the universal pasture and the partial degree 0-bracket field}
 Let $M$ be a matroid whose set of bases is $\cB$. Let $k_M^w$ be its universal pasture and $P_{M,0}$ its partial degree-$0$ bracket field. Then $k_M^w\otimes_\Funpm\Funsq\simeq P_{M,0}^\blpr$ and there are natural bijections
 \[
  \cX_M(P^\oblpr) \ \stackrel{1:1}\longleftrightarrow \ \Hom(k_M^w,P^\oblpr) \ \stackrel{1:1}\longleftrightarrow \ \Hom(k_M^w\otimes_\Funpm\Funsq,P^\blpr) \ \stackrel{1:1}\longleftrightarrow \ \Hom(P_{M,0}, P)
 \]
 for every partial field $P$.
\end{lemma}

\begin{proof}
 The first claim follows readily: since both $k_M^w\otimes_\Funpm\Funsq$ and $P_{M,0}^\blpr$ are with $-1$, the ambient semiring of both ordered blueprints is a ring and is thus trivially ordered. In both cases, the ambient ring is generated by Laurent monomials in the $x_J$ of degree $0$, and all relations between the Laurent monomials are generated by the $3$-term Pl\"ucker relations. Thus $(k_M^w\otimes_\Funpm\Funsq)^+\simeq B_{M,0} \simeq (P_{M,0}^\blpr)^+$ is the degree-$0$ bracket ring. The underlying monoids of both $k_M^w\otimes_\Funpm\Funsq$ and $P_{M,0}^\blpr$ are generated by the terms $\pm x_J^{\pm1}$. Thus $k_M^w\otimes_\Funpm\Funsq\simeq P_{M,0}^\blpr$.
 
 We turn to the proof of the second claim. A partial field is a doubly-distributive partial hyperfield and thus perfect by \cite[Cor.\ 3.3]{Baker-Bowler19}. Thus $\cX_M(P^\oblpr)=\cX_M^w(P^\oblpr)$ by Lemma \ref{lemma: weak realization space as morphism set from the universal pasture} and $\cX_M^w(P^\oblpr)=\Hom(k_M^w,P^\oblpr)$ by Proposition \ref{prop: weak realization space as morphism set from the universal pasture}. This establishes the first bijection.
 
 The second bijection is obtained by applying the functors $(-)^\ppos$ and $-\otimes_\Funpm\Funsq$, which define mutually inverse bijections between the two morphism sets in question. The last bijection follows from the isomorphism $k_M^w\otimes_\Funpm\Funsq\simeq P_{M,0}^\blpr$ and the fact that $(-)^\oblpr:\PartFields\to\OBlpr^\pm$ is fully faithful.
\end{proof}

As a consequence of Lemma \ref{lemma: relation between the universal pasture and the partial degree 0-bracket field}, we can reprove Theorem 4.6 from \cite{Pendavingh-vanZwam10a}, which is the following assertion.

\begin{cor}\label{cor: M is representable over a partial field if and only if B_M is nontrivial}
 Let $M$ be a matroid with bracket ring $B_M$. Then $M$ is representable over some partial field if and only if $B_M$ is nontrivial.
\end{cor}

\begin{proof}
 By Corollary \ref{cor: a matroid is representable over a partial field iff it is the image of a P-matroid} and Theorem \ref{thm: moduli space of matroids}, $M$ is representable over a partial field $P$ if and only if $\cX_M(P^\oblpr)$ is nonempty.
 
 Assume that $\cX_M(P^\oblpr)$ is nonemtpy. By Lemma \ref{lemma: relation between the universal pasture and the partial degree 0-bracket field}, $\cX_M(P^\oblpr)=\Hom(P_{M,0}, P)$, i.e.\ there is a morphism $P_{M,0}\to P$, where $P_{M,0}$ is the partial degree-$0$ bracket field of $M$. This induces a morphism $B_{M,0}\to R_P$ between the respective ambient rings. This shows that the degree-$0$ bracket ring $B_{M,0}$ of $M$ is nontrivial, and as a consequence $B_M$ is nontrivial.
 
 If $B_M$ is nontrivial, then $P_M$ is a partial field and the canonical morphism $k_M^w\to P_{M,0}\to P_M$ yields a point in $\cX_M(P^\oblpr)$ by Lemma \ref{lemma: relation between the universal pasture and the partial degree 0-bracket field}. Thus $\cX_M(P^\oblpr)$ is nonempty and $M$ is representable over the partial field $P_M$.
\end{proof}

The following fact shows that the class of representable matroids does not change if we ask for representability over fields or partial fields. This was already observed in \cite[Cor.\ 5.2]{Pendavingh-vanZwam10b}. Note that for a partial field, strong and weak matroids coincide, so we do not have to distinguish these two classes.

\begin{lemma}\label{lemma: a matroid representable over a partial field is representable over a field}
 Let $M$ be a matroid. Then $M$ is representable over a partial field if and only if $M$ is representable over a field.
\end{lemma}

\begin{proof}
 Since a field is a partial field, one implication is trivial. Assume that $M$ is representable over a partial field $P=(P^\times,\pi_P)$, i.e.\ there is a morphism $\chi:k_M\to P$. Then the ambient ring $R_P$ of $P$ is nontrivial and admits therefore a morphism $f:R_P\to k$ into a field $k$. The composition 
 \[
  k_M \ \stackrel\chi\longrightarrow \ P^\blpr \ \longrightarrow \ \Z[P^\times] \ \stackrel{\pi_P}\longrightarrow \ R_P \ \stackrel f\longrightarrow k
 \]
 yields a representation of $M$ over $k$ where we consider all objects as ordered blueprints. This proves the reverse implication.
\end{proof}

\begin{rem}\label{rem: matroids with trivial bracket ring}
 By Lemma \ref{lemma: a matroid representable over a partial field is representable over a field}, a matroid $M$ that is not representable over any field is not representable over any partial field. By Corollary \ref{cor: M is representable over a partial field if and only if B_M is nontrivial}, such matroids are characterized by the property that their bracket ring is trivial. Since such matroids exist, for instance the V\'amos matroid, this means that there are matroids $M$ with trivial bracket ring $B_M$. 
 
On the other hand, the (weak) universal idyll is always nontrivial. This shows, in particular, that the canonical morphism $k_M^w\to P_{M,0}$ from the universal pasture of $M$ to the partial degree-$0$ bracket field is not injective in general.
\end{rem}

\subsubsection{The universal partial field}
\label{subsubsection: the universal partial field}

The universal partial field of a matroid $M$ was introduced by Pendavingh and van Zwam in \cite{Pendavingh-vanZwam10a} as a device for proving representability theorems for matroids over partial fields. The definition in \cite{Pendavingh-vanZwam10a} is somewhat technical, which is perhaps an inevitable consequence of their approach to matroid representations over partial fields in terms of concrete matrix manipulations. Unravelling their definitions leads to the following short characterization of the universal partial field. The interested reader will find in Remark \ref{rem: equivalence with PvZ's definition of the universal partial field} an outline of the equivalence of our definition with that of \cite{Pendavingh-vanZwam10a}.

\begin{df}
 Let $M$ be a matroid and $P_M$ its partial bracket field. Let $\Delta:\binom Er\to P_M$ be the weak Grassmann-Pl\"ucker function defined by $\Delta(I)=x_I$ for $I\in\cB$ and $\Delta(I)=0$ otherwise. The \emph{universal partial field of $M$} is the partial subfield $\P_M$ of $P_M$ generated by the cross ratios $\Cr_\Delta(\cI)$ for $\cI\in\Omega_M$.
\end{df}

Note that since all cross ratios are expressions in the $x_I$ of degree $0$, the universal partial field $\P_M$ is contained in the partial degree-$0$ bracket field $P_{M,0}$. Recall from Lemma \ref{lemma: relation between the universal pasture and the partial degree 0-bracket field} that the canonical morphism $k_M^w\to P_{M,0}^\oblpr$ induces an isomorphism $k_M^w\otimes_\Funpm\Funsq\to P_{M,0}^\blpr$. The following lemma recovers \cite[Cor.\ 4.12]{Pendavingh-vanZwam10a} in a more concise form.

\begin{lemma}\label{lemma: relation between the foundation and the universal partial field}
 The isomorphism $k_M^w\otimes_\Funpm\Funsq\to P_{M,0}^\blpr$ restricts to an isomorphism $k_M^f\otimes_\Funpm\Funsq\to \P_M^\blpr$ and there are natural bijections
 \[
  \cX_M^f(P^\oblpr) \ \stackrel{1:1}\longleftrightarrow \ \Hom(k_M^f,P^\oblpr) \ \stackrel{1:1}\longleftrightarrow \ \Hom(k_M^f\otimes_\Funpm\Funsq,P^\blpr) \ \stackrel{1:1}\longleftrightarrow \ \Hom(\P_{M}, P)
 \]
 for every partial field $P$.
\end{lemma}

\begin{proof}
 The former claim is immediate from the definitions of $k_M^f$ and $\P_M$ as the subobjects of $k_M^w$ and $P_M$, respectively, that are generated by the cross ratios.
  
 We turn to the proof of the latter claim. The first bijection is established in Corollary \ref{cor: rescaling classes over M as morphisms from the foundation}. The second bijection comes from applying $-\otimes_\Funpm\Funsq$ and $(-)^\ppos$ to the morphism sets, cf.\ the proof of Lemma \ref{lemma: relation between the universal pasture and the partial degree 0-bracket field}. The third bijection follows from $k_M^f\otimes_\Funpm\Funsq\simeq \P_{M}^\blpr$ and the fact that $(-)^\oblpr:\PartFields\to\OBlpr^\pm$ is fully faithful.
\end{proof}

\begin{cor}\label{cor: a matroid is representable over a partial field iff it receives a morphism from the universal partial field}
 Let $P$ be a partial field and $M$ a matroid. Then $M$ is representable over $P$ if and only if there is a morphism $\P_M\to P$.
\end{cor}

\begin{proof}
 The matroid $M$ is representable over $P$ if and only if there is a rescaling class over $P$, i.e.\ $\cX_M^f(P^\oblpr)$ is nonempty. By Lemma \ref{lemma: relation between the foundation and the universal partial field}, this is equivalent with $\Hom(\P_M,P)$ being nonempty.
\end{proof}

\begin{cor}\label{cor: if the foundation is a partial field then it is the universal partial field}
 Let $M$ be a matroid with foundation $k_M^f$ and universal partial field $\P_M$. If $k_M^f$ is a partial field, then the canonical morphism $k_M^f\to k_M^f\otimes_\Funpm\Funsq\to\P_M$ is an isomorphism.
\end{cor}

\begin{proof}
 The latter morphism $k_M^f\otimes_\Funpm\Funsq\to\P_M$ is the isomorphism from Lemma \ref{lemma: relation between the foundation and the universal partial field}. Since $B\otimes_{\Funpm}\Funsq=\bpgenquot B{1+\epsilon\=0}$ and since $1+\epsilon\=0$ holds in every partial field, the canonical inclusion $k_M^f\to\bpgenquot{k_M^f}{1+\epsilon\=0}=k_M^f\otimes_{\Funpm}\Funsq$ is an isomorphism, and so is the composition $k_M^f\to \P_M$ of these two isomorphisms.
\end{proof}

Let $P=(P^\times,\pi_P)$ be a partial field with quotient map $\pi_P:\Z[P^\times]\to R_P$. Then we define $P[T_1^{\pm1},\dotsc,T_s^{\pm1}]$ as the partial field $(P^\times\times\{\prod T_i^{e_i}\}_{e_i\in\Z}, \hat\pi_P)$ where 
\[\textstyle
 \hat\pi_P: \ \Z[P^\times][T_1^{\pm1},\dotsc,T_s^{\pm1}] \ \longrightarrow \ R_P[T_1^{\pm1},\dotsc,T_s^{\pm1}]
\]
is the extension of $\pi_P$ that maps $T_i$ to $T_i$. Note that $P[T_1^{\pm1},\dotsc,T_n^{\pm1}]^\blpr=P^\blpr[T_1^{\pm1},\dotsc,T_s^{\pm1}]$.

\begin{cor}\label{cor: the bracket partial field as Laurent series over the universal partial field}
 Let $M$ be a matroid. Then $P_{M,0}\simeq \P_M[T_1^{\pm1},\dotsc,T_s^{\pm1}]$ for some $s\geq0$.
\end{cor}

\begin{proof}
 By Corollary \ref{cor: the universal pasture as Laurent seris over the foundation}, we have $k_M^w\simeq k_M^f[T_1^{\pm1},\dotsc,T_s^{\pm1}]$ for some $s\geq 0$. Using the isomorphisms $k_M^w\otimes_\Funpm\Funsq\to P_{M,0}^\blpr$ from Lemma \ref{lemma: relation between the universal pasture and the partial degree 0-bracket field} and $k_M^f\otimes_\Funpm\Funsq\to \P_M^\blpr$ from Lemma \ref{lemma: relation between the foundation and the universal partial field}, we obtain a sequence of isomorphisms
 \begin{multline*}
  P_{M,0}^\blpr \ \simeq \ k_M^w\otimes_\Funpm\Funsq \ \simeq \ k_M^f[T_1^{\pm1},\dotsc,T_s^{\pm1}]\otimes_\Funpm\Funsq \\
  \ \simeq \ (k_M^f\otimes_\Funpm\Funsq)[T_1^{\pm1},\dotsc,T_s^{\pm1}] \ \simeq \ \P_M^\blpr[T_1^{\pm1},\dotsc,T_s^{\pm1}] \ \simeq \ \P_M[T_1^{\pm1},\dotsc,T_n^{\pm1}]^\blpr.
 \end{multline*}
 Since $(-)^\blpr$ is fully faithful, this yields the desired isomorphism $P_{M,0}\simeq \P_M[T_1^{\pm1},\dotsc,T_s^{\pm1}]$.  
\end{proof}

\begin{rem}\label{rem: equivalence with PvZ's definition of the universal partial field}
 We indicate how it can be seen that our definition of the universal partial field is equivalent to that of Pendavingh and van Zwam in \cite{Pendavingh-vanZwam10a}. This equivalence is not hard to establish, but requires unravelling a series of definitions. This remark is meant as a guide for the reader who wants to do this exercise.
 
 Let $P$ be a partial field and $M$ a matroid of rank $r$ on $E$. The matrix representations $A$ of $M$ considered in \cite{Pendavingh-vanZwam10a} are assumed to be normalized in the sense that they contain a square submatrix of maximal size that is an identity matrix, which corresponds to fixing a canonical affine open subset of the matroid space. Note that in the case of a field, the matroid space is nothing else than a Grassmann variety, which might give the reader some geometric intuition. Moreover, the identity matrix is omitted from $A$ and only the truncated part of $A$ is considered.
 
 Strong equivalence of two such truncated normalized matrices $A$ and $A'$ is defined by three elementary operations: pivoting, permuting rows and columns, and scaling rows and columns by nonzero elements of $P$. We explain the effect on the corresponding Pl\"ucker coordinates. Pivoting incorporates the effect of a change of the affine open of the Grassmannian on the truncated matrix $A$, but has no effect on the Pl\"ucker coordinates, except for possible sign changes. Exchanging rows and columns has no effect on the Pl\"ucker coordinates except for sign changes. Scaling rows corresponds to multiplying the truncated columns with the inverse scalar. Thus all operations that generate the strong equivalence relation come from scaling columns of a (non-truncated and possibly non-normalized) matrix representation $A$ of $M$. This corresponds to the torus action appearing in the definition of rescaling classes.

 Let $A$ be a truncated and normalized matrix representation of $M$. With the help of column and row scaling, every $2\times 2$-submatrix of $A$ with nonzero entries can be brought into the shape $\binom{1\ 1}{p\ 1}$. The element $p$ is called the cross ratio of the submatrix. In \cite{Pendavingh-vanZwam10a}, the universal partial field is defined as the subfield of $P_M$ that is generated by all cross ratios $p$ that occur in submatrices of the form $\binom{1\ 1}{p\ 1}$ for some matrix $A'$ that is strongly equivalent to $A$. 
 
 Since $A$ is normalized, all entries in $A$ are Pl\"ucker coordinates of $A$. A $2\times 2$-submatrix with nonzero entries corresponds to a tuple $\cI = (I,i_1,i_2,i_3,i_4)\in\Omega_M$ where $i_1,i_3$ label the rows and $i_2,i_4$ label the columns of $\binom{1\ 1}{p\ 1}$. Thus $p$ is the cross ratio $\Cr_\Delta(\cI)$, up to a possible difference in signs. Since cross ratios are invariant under the action of $T(P)$ and since every $2\times2$-submatrix can be brought into the form $\binom{1\ 1}{p\ 1}$ by scaling rows and columns, this establishes a bijective correspondence between cross ratios in the sense of \cite{Pendavingh-vanZwam10a} and the cross ratios $\Cr_\Delta(\cI)$ for $\cI \in \Omega_M$. The difference in signs that occur do not affect the universal partial field $\P_M$ since it contains all (weak) inverses.
\end{rem}

\subsubsection{The universal partial fields of binary and regular matroids}
\label{subsubsection: the universal partial fields of binary and regular matroids}

Our classification of binary and regular matroids in terms of their foundation, cf.\ section \ref{subsection: foundations of binary matroids} and \ref{subsection: foundations of regular matroids}, yields a classification of such matroids in terms of their universal partial fields.

\begin{cor}\label{cor: classification of binary and regular matroids by their universal partial fields}
 A matroid is binary if and only if its universal partial field is $\Funpm$ or $\F_2$. A matroid is regular if and only if its universal partial field is $\Funpm$.
\end{cor}

\begin{proof}
 Let $M$ be a binary matroid. By Theorem \ref{thm: foundations of binary matroids}, its foundation is $\Funpm$ or $\F_2$. Since both of these idylls are partial fields, the foundation is isomorphic to the universal partial field by Corollary \ref{cor: if the foundation is a partial field then it is the universal partial field}. If $M$ is regular, then its foundation, and hence its universal partial field, is $\Funpm$ by Theorem \ref{thm: characteriztaion of regular matroids}.

 Conversely, assume that $M$ is a matroid with universal partial field $\Funpm$. Since there is a morphism from $\Funpm$ to every field, it follows from Corollary \ref{cor: a matroid is representable over a partial field iff it receives a morphism from the universal partial field} that $M$ is regular and binary. If the universal partial field of $M$ is $\F_2$, on the other hand, then clearly $M$ is binary.
\end{proof}


\begin{small}
 \bibliographystyle{plain}
 \bibliography{matroid}

\begin{thebibliography}{10}

\bibitem{Akian-Gaubert-Guterman14}
Marianne Akian, St\'{e}phane Gaubert, and Alexander Guterman.
\newblock Tropical {C}ramer determinants revisited.
\newblock In {\em Tropical and idempotent mathematics and applications}, volume
  616 of {\em Contemp. Math.}, pages 1--45. Amer. Math. Soc., Providence, RI,
  2014.

\bibitem{Anderson19}
Laura Anderson.
\newblock Vectors of matroids over tracts.
\newblock {\em J. Combin. Theory Ser. A}, 161:236--270, 2019.

\bibitem{Anderson-Davis19}
Laura Anderson and James~F. Davis.
\newblock Hyperfield {G}rassmannians.
\newblock {\em Adv. Math.}, 341:336--366, 2019.

\bibitem{Baker-Bowler19}
Matthew Baker and Nathan Bowler.
\newblock Matroids over partial hyperstructures.
\newblock {\em Adv. Math.}, 343:821--863, 2019.

\bibitem{Baker-Lorscheid20}
Matthew Baker and Oliver Lorscheid.
\newblock Foundations of matroids. {P}art 1: Matroids without large uniform
  minors.
\newblock Preprint, \arxiv{2008.00014}, 2020.

\bibitem{Beukers-Schlickewei96}
Frits Beukers and Hans~P. Schlickewei.
\newblock The equation {$x+y=1$} in finitely generated groups.
\newblock {\em Acta Arith.}, 78(2):189--199, 1996.

\bibitem{Bjorner-LasVergnas-Sturmfels-White-Ziegler99}
Anders Bj\"orner, Michel Las~Vergnas, Bernd Sturmfels, Neil White, and
  G\"unter~M. Ziegler.
\newblock {\em Oriented matroids}, volume~46 of {\em Encyclopedia of
  Mathematics and its Applications}.
\newblock Cambridge University Press, Cambridge, second edition, 1999.

\bibitem{Brylawski-Lucas73}
Tom Brylawski and Dean Lucas.
\newblock Uniquely representable combinatorial geometries.
\newblock In {\em Proceedings of the International Colloquium on Combinatorial
  Theory}. Rome, 1973.

\bibitem{Chu-Lorscheid-Santhanam12}
Chenghao Chu, Oliver Lorscheid, and Rekha Santhanam.
\newblock Sheaves and {$K$}-theory for {$\mathbb F_1$}-schemes.
\newblock {\em Adv. Math.}, 229(4):2239--2286, 2012.

\bibitem{Connes-Consani11}
Alain Connes and Caterina Consani.
\newblock The hyperring of ad\`ele classes.
\newblock {\em J. Number Theory}, 131(2):159--194, 2011.

\bibitem{Deitmar05}
Anton Deitmar.
\newblock Schemes over {$\mathbb F\sb 1$}.
\newblock In {\em Number fields and function fields---two parallel worlds},
  volume 239 of {\em Progr. Math.}, pages 87--100. Birkh\"auser Boston, Boston,
  MA, 2005.

\bibitem{Deitmar08}
Anton Deitmar.
\newblock {$\mathbb F_1$}-schemes and toric varieties.
\newblock {\em Beitr\"age Algebra Geom.}, 49(2):517--525, 2008.

\bibitem{Delucchi-Hoessly-Saini15}
Emanuele Delucchi, Linard Hoessly, and Elia Saini.
\newblock Realization spaces of matroids over hyperfields.
\newblock Preprint, \arxiv{1504.07109v4}, 2015.

\bibitem{Dress-Wenzel91}
Andreas Dress and Walter Wenzel.
\newblock Grassmann-{P}l\"{u}cker relations and matroids with coefficients.
\newblock {\em Adv. Math.}, 86(1):68--110, 1991.

\bibitem{Dress86}
Andreas W.~M. Dress.
\newblock Duality theory for finite and infinite matroids with coefficients.
\newblock {\em Adv. in Math.}, 59(2):97--123, 1986.

\bibitem{Dress-Wenzel89}
Andreas W.~M. Dress and Walter Wenzel.
\newblock Geometric algebra for combinatorial geometries.
\newblock {\em Adv. Math.}, 77(1):1--36, 1989.

\bibitem{Dress-Wenzel90}
Andreas W.~M. Dress and Walter Wenzel.
\newblock On combinatorial and projective geometry.
\newblock {\em Geom. Dedicata}, 34(2):161--197, 1990.

\bibitem{Dress-Wenzel92}
Andreas W.~M. Dress and Walter Wenzel.
\newblock Perfect matroids.
\newblock {\em Adv. Math.}, 91(2):158--208, 1992.

\bibitem{Evertse-Gyory13}
Jan-Hendrik Evertse and K\'alm\'an Gy\H{o}ry.
\newblock Effective results for unit equations over finitely generated integral
  domains.
\newblock {\em Math. Proc. Cambridge Philos. Soc.}, 154(2):351--380, 2013.

\bibitem{Folkman-Lawrence78}
Jon Folkman and Jim Lawrence.
\newblock Oriented matroids.
\newblock {\em J. Combin. Theory Ser. B}, 25(2):199--236, 1978.

\bibitem{Gelfand-Goresky-MacPherson-Serganova87}
I.~M. Gel\cprime{}fand, R.~M. Goresky, R.~D. MacPherson, and V.~V. Serganova.
\newblock Combinatorial geometries, convex polyhedra, and {S}chubert cells.
\newblock {\em Adv. in Math.}, 63(3):301--316, 1987.

\bibitem{Gelfand-Rybnikov-Stone95}
Israel~M. Gelfand, Grigori~L. Rybnikov, and David~A. Stone.
\newblock Projective orientations of matroids.
\newblock {\em Adv. Math.}, 113(1):118--150, 1995.

\bibitem{Gerards89}
Albertus M.~H. Gerards.
\newblock A short proof of {T}utte's characterization of totally unimodular
  matrices.
\newblock {\em Linear Algebra Appl.}, 114/115:207--212, 1989.

\bibitem{Giansiracusa-Jun-Lorscheid16}
Jeffrey Giansiracusa, Jaiung Jun, and Oliver Lorscheid.
\newblock On the relation between hyperrings and fuzzy rings.
\newblock {\em Beitr. Algebra Geom.}, 58(4):735--764, 2017.

\bibitem{Hartshorne77}
Robin Hartshorne.
\newblock {\em Algebraic geometry}.
\newblock Springer-Verlag, New York, 1977.
\newblock Graduate Texts in Mathematics, No. 52.

\bibitem{Jun17}
Jaiung Jun.
\newblock Geometry of hyperfields.
\newblock Preprint, \arxiv{1707.09348}, 2017.

\bibitem{Jun18}
Jaiung Jun.
\newblock Algebraic geometry over hyperrings.
\newblock {\em Adv. Math.}, 323:142--192, 2018.

\bibitem{Katz16}
Eric Katz.
\newblock Matroid theory for algebraic geometers.
\newblock In {\em Nonarchimedean and tropical geometry}, Simons Symp., pages
  435--517. Springer, [Cham], 2016.

\bibitem{Koymans-Pagano17}
Peter Koymans and Carlo Pagano.
\newblock On the equation {$X_1+X_2=1$} in finitely generated multiplicative
  groups in positive characteristic.
\newblock {\em Q. J. Math.}, 68(3):923--934, 2017.

\bibitem{Krasner56}
Marc Krasner.
\newblock Approximation des corps valu\'{e}s complets de caract\'{e}ristique
  {$p\not=0$} par ceux de caract\'{e}ristique {$0$}.
\newblock In {\em Colloque d'alg\`ebre sup\'{e}rieure, tenu \`a {B}ruxelles du
  19 au 22 d\'{e}cembre 1956}, Centre Belge de Recherches Math\'{e}matiques,
  pages 129--206. \'{E}tablissements Ceuterick, Louvain; Librairie
  Gauthier-Villars, Paris, 1957.

\bibitem{Lafforgue03}
Laurent Lafforgue.
\newblock {\em Chirurgie des grassmanniennes}, volume~19 of {\em CRM Monograph
  Series}.
\newblock American Mathematical Society, Providence, RI, 2003.

\bibitem{Lang60}
Serge Lang.
\newblock Integral points on curves.
\newblock {\em Inst. Hautes \'Etudes Sci. Publ. Math.}, (6):27--43, 1960.

\bibitem{Lee-Vakil13}
Seok~Hyeong Lee and Ravi Vakil.
\newblock Mn\"ev-{S}turmfels universality for schemes.
\newblock In {\em A celebration of algebraic geometry}, volume~18 of {\em Clay
  Math. Proc.}, pages 457--468. Amer. Math. Soc., Providence, RI, 2013.

\bibitem{Liu17}
Gaku Liu.
\newblock A counterexample to the extension space conjecture for realizable
  oriented matroids.
\newblock {\em S\'em. Lothar. Combin.}, 78B:Art. 31, 7, 2017.

\bibitem{LopezPena-Lorscheid12}
Javier L\'opez~Pe{\~n}a and Oliver Lorscheid.
\newblock Projective geometry for blueprints.
\newblock {\em C. R. Math. Acad. Sci. Paris}, 350(9-10):455--458, 2012.

\bibitem{Lorscheid15}
Oliver Lorscheid.
\newblock Scheme theoretic tropicalization.
\newblock Preprint, \arxiv{1508.07949v3}, 2015.

\bibitem{Lorscheid18b}
Oliver Lorscheid.
\newblock The geometry of blueprints part {II}: {T}its-{W}eyl models of
  algebraic groups.
\newblock {\em Forum Math. Sigma}, 6:e20, 90, 2018.

\bibitem{Lorscheid18}
Oliver Lorscheid.
\newblock Blueprints and tropical scheme theory.
\newblock Lecture notes,
  \url{http://oliver.impa.br/2018-Blueprints/versions/lecturenotes180521.pdf},
  version from May 21, 2018.

\bibitem{Lorscheid-Salgado16}
Oliver Lorscheid and Cec\'ilia Salgado.
\newblock A remark on topologies for rational point sets.
\newblock {\em J. Number Theory}, 159:193--201, 2016.

\bibitem{Maclagan-Rincon14}
Diane Maclagan and Felipe Rinc\'{o}n.
\newblock Tropical schemes, tropical cycles, and valuated matroids.
\newblock {\em J. Eur. Math. Soc. (JEMS)}, 22(3):777--796, 2020.

\bibitem{Marshall96}
Murray~A. Marshall.
\newblock {\em Spaces of orderings and abstract real spectra}, volume 1636 of
  {\em Lecture Notes in Mathematics}.
\newblock Springer-Verlag, Berlin, 1996.

\bibitem{Marty36}
F.~Marty.
\newblock Sur les groupes et hypergroupes attaches \`a une fraction
  rationnelle.
\newblock {\em Ann. Sci. \'{E}cole Norm. Sup. (3)}, 53:83--123, 1936.

\bibitem{Mnev-Ziegler93}
Nicolai~E. Mn\"ev and G\"unter~M. Ziegler.
\newblock Combinatorial models for the finite-dimensional {G}rassmannians.
\newblock {\em Discrete Comput. Geom.}, 10(3):241--250, 1993.

\bibitem{Mnev88}
Nikolai~E. Mn\"ev.
\newblock The universality theorems on the classification problem of
  configuration varieties and convex polytopes varieties.
\newblock In {\em Topology and geometry---{R}ohlin {S}eminar}, volume 1346 of
  {\em Lecture Notes in Math.}, pages 527--543. Springer, Berlin, 1988.

\bibitem{Pappus7}
Pappus of~Alexandria.
\newblock {\em Book 7 of the {\it {C}ollection\/}}, volume~8 of {\em Sources in
  the History of Mathematics and Physical Sciences}.
\newblock Springer-Verlag, New York, 1986.
\newblock Part 1. Introduction, text, and translation, Part 2. Commentary,
  index, and figures, Edited and with translation and commentary by Alexander
  Jones.

\bibitem{Oxley92}
James~G. Oxley.
\newblock {\em Matroid theory}.
\newblock Oxford Science Publications. The Clarendon Press, Oxford University
  Press, New York, 1992.

\bibitem{Pendavingh18}
Rudi Pendavingh.
\newblock Field extensions, derivations, and matroids over skew hyperfields.
\newblock Preprint, \arxiv{1802.02447}, 2018.

\bibitem{Pendavingh-vanZwam10a}
Rudi~A. Pendavingh and Stefan H.~M. van Zwam.
\newblock Confinement of matroid representations to subsets of partial fields.
\newblock {\em J. Combin. Theory Ser. B}, 100(6):510--545, 2010.

\bibitem{Pendavingh-vanZwam10b}
Rudi~A. Pendavingh and Stefan H.~M. van Zwam.
\newblock Lifts of matroid representations over partial fields.
\newblock {\em J. Combin. Theory Ser. B}, 100(1):36--67, 2010.

\bibitem{Richter-Gebert11}
J\"urgen Richter-Gebert.
\newblock {\em Perspectives on projective geometry}.
\newblock Springer, Heidelberg, 2011.
\newblock A guided tour through real and complex geometry.

\bibitem{Rowen16}
Louis Rowen.
\newblock Algebras with a negation map.
\newblock Preprint, \arxiv{1602.00353}, 2016.

\bibitem{Semple-Whittle96}
Charles Semple and Geoff Whittle.
\newblock Partial fields and matroid representation.
\newblock {\em Adv. in Appl. Math.}, 17(2):184--208, 1996.

\bibitem{Speyer08}
David~E. Speyer.
\newblock Tropical linear spaces.
\newblock {\em SIAM J. Discrete Math.}, 22(4):1527--1558, 2008.

\bibitem{Tutte58a}
William~T. Tutte.
\newblock A homotopy theorem for matroids, {I}.
\newblock {\em Trans. Amer. Math. Soc.}, 88:144--160, 1958.

\bibitem{Tutte58b}
William~T. Tutte.
\newblock A homotopy theorem for matroids, {II}.
\newblock {\em Trans. Amer. Math. Soc.}, 88:161--174, 1958.

\bibitem{Tutte65}
William~T. Tutte.
\newblock Lectures on matroids.
\newblock {\em J. Res. Nat. Bur. Standards Sect. B}, 69B:1--47, 1965.

\bibitem{Vakil06}
Ravi Vakil.
\newblock Murphy's law in algebraic geometry: badly-behaved deformation spaces.
\newblock {\em Invent. Math.}, 164(3):569--590, 2006.

\bibitem{vanZwam09}
Stefan H.~M. van Zwam.
\newblock Partial fields in matroid theory.
\newblock PhD thesis, Eindhoven, 2009. Online available at
  \url{www.math.lsu.edu/~svanzwam/pdf/thesis-online.pdf}.

\bibitem{Viro11}
Oleg~Ya. Viro.
\newblock On basic concepts of tropical geometry.
\newblock {\em Proc. Steklov Inst. Math.}, 273(1):252--282, 2011.

\bibitem{Wenzel91}
Walter Wenzel.
\newblock Projective equivalence of matroids with coefficients.
\newblock {\em J. Combin. Theory Ser. A}, 57(1):15--45, 1991.

\bibitem{White75}
Neil~L. White.
\newblock The bracket ring of a combinatorial geometry. {I}.
\newblock {\em Trans. Amer. Math. Soc.}, 202:79--95, 1975.

\bibitem{Whitney35}
Hassler Whitney.
\newblock On the abstract properties of linear dependence.
\newblock {\em Amer. J. Math.}, 57(3):509--533, 1935.

\end{thebibliography}
\end{small}

\end{document}